\documentclass[10pt,reqno]{amsart}
\usepackage{amsfonts}
\usepackage{enumerate}
\usepackage{url}
\usepackage{amssymb,amscd,epsfig}
\usepackage{mathtools}
\usepackage{graphicx}
\usepackage{caption}
\usepackage{psfrag}
\usepackage{mathptmx}
\usepackage{helvet}
\usepackage{courier}
\usepackage{multicol}
\usepackage[bottom]{footmisc}
\usepackage{tikz}\usetikzlibrary{arrows,snakes,backgrounds,patterns,matrix,shapes,fit,calc,shadows,plotmarks,calc,shapes,decorations.pathreplacing,3d}
\usepackage{enumitem}
\usepackage{color}
\usepackage{xfrac}

\makeatletter
\def\subsection{\@startsection{subsection}{2}%
  \z@{.5\linespacing\@plus.7\linespacing}{.1\linespacing}%
  {\normalfont\bfseries}}
\makeatother

\newcommand{\Ban}{E} 
\newcommand{\Dom}{{\mathcal D}}
\newcommand{\Dif}{{\mathrm{Dif}}}
\newcommand{\JDif}{\mathrm{Dif}(\Dom,\Jcal)}
\newcommand{\Dcal}{{\mathcal D}}
\newcommand{\Hcal}{{\mathcal H}}
\newcommand{\Kcal}{{\mathcal K}}

\newcommand{\Jcal}{{\mathcal J}}
\newcommand{\Lcal}{{\mathcal L}}
\newcommand{\Ocal}{{\mathcal O}}

\newcommand{\Rcal}{{\mathcal R}}
\newcommand{\N}{\mathbb{N}}
\newcommand{\Z}{\mathbb{Z}}
\newcommand{\R}{\mathbb{R}}

\newcommand{\C}{\mathbb{C}}
\newcommand{\rn}{\R^d}

\newcommand{\e}{e}
\newcommand{\rd}{\mathrm{d}}

\newcommand{\T}{\mathbb{T}}

\newcommand{\Gal}{G_\mathrm{Gal}}
\newcommand{\Smu}{\Sigma_\mu}
\newcommand\restr[2]{{
  \left.\kern\nulldelimiterspace 
  #1 
  \vphantom{|} 
  \right|_{#2} 
  }}
  
\newcommand{\scalh}{\langle\cdot,\cdot\rangle}
\newcommand{\scal}[2]{{\langle#1,#2\rangle}}

\newcommand{\wt}{\widetilde}

\newcommand{\vertiii}[1]{{\left\vert\kern-0.25ex\left\vert\kern-0.25ex\left\vert #1 
    \right\vert\kern-0.25ex\right\vert\kern-0.25ex\right\vert}}
\newcommand{\nl}{f}

\newtheorem{theorem}{Theorem}[section]
\newtheorem{proposition}[theorem]{Proposition}

\newtheorem{lemma}[theorem]{Lemma}
\theoremstyle{definition}
\newtheorem{definition}[theorem]{Definition}
\newtheorem{example}[theorem]{Example}
\theoremstyle{remark}
\newtheorem{remark}[theorem]{Remark}
\numberwithin{equation}{section}

\newcommand\RR{\mathbb R} 
\newcommand\CC{\mathbb C}

\newcommand\ZZ{\mathbb Z}  
\newcommand\TT{\mathbb T} 

\newcommand\vb{|}
\newcommand\norm{\|}

\newcommand\vctl[2]{\left(#1,\,#2\right)}
\newcommand{\sech}{\mathrm{sech}}
\newcommand{\en}{\xi}
\newcommand{\muak}{\mu_{\alpha,k}}
\newcommand{\mua}{\mu_{\alpha}}

\newcommand{\nat}{\mathbb{N}}
\newcommand{\real}{\mathbb{R}}
\newcommand{\complex}{\mathbb{C}}
\newcommand{\intrn}{\int_{\R^d}}
\newcommand{\intr}{\int_{\R}}

\newcommand{\lam}{\xi}

\newcommand{\ly}{\xi_\infty}

\newcommand{\la}{\langle}
\newcommand{\ra}{\rangle}
\newcommand{\diff}{\,\mathrm{d}}
\newcommand{\dif}{\mathrm{d}}
\newcommand{\Ve}{\Vert}

\renewcommand\emptyset{\mbox{\Large \o}}

\newcommand{\calJ}{\mathcal{J}}
\newcommand{\calL}{\mathcal{L}}

\newcommand{\disp}{\displaystyle}
\newcommand\txt{\textstyle}

\DeclareMathOperator \sgn{sgn}

\DeclareMathOperator \im{Im}
\DeclareMathOperator \re{Re}

\makeindex

\begin{document}

\title{Orbital stability: analysis meets geometry}
\author[S. {De Bi\`{e}vre}]{Stephan {De Bi\`{e}vre}$^{1,2}$}
\address{$^1$Laboratoire Paul Painlev\'e, CNRS, UMR 8524 et UFR de Math\'ematiques,
Universit\'e Lille~1, Sciences et Technologies
F-59655 Villeneuve d'Ascq Cedex, France.}
\email{Stephan.De-Bievre@math.univ-lille1.fr}
\email{Simona.Rota-Nodari@math.univ-lille1.fr}

\address{$^2$
Equipe-Projet MEPHYSTO,
Centre de Recherche INRIA Futurs,
Parc Scientifique de la Haute Borne, 40, avenue Halley B.P. 70478,
F-59658 Villeneuve d'Ascq cedex, France.}

\author[F. Genoud]{Fran\c{c}ois Genoud$^3$}
\address{$^3$Faculty of Mathematics, University of Vienna,
Oskar-Morgenstern-Platz 1,
1090 Vienna, Austria.}
\email{francois.genoud@univie.ac.at}

\author[S. {Rota Nodari}]{Simona {Rota Nodari}$^1$}

\date{\today}

\begin{abstract}
We present an introduction to the orbital stability of relative equilibria of Hamiltonian dynamical systems on (finite and infinite dimensional) Banach spaces. A convenient formulation of the theory of Hamiltonian dynamics with symmetry and the corresponding momentum maps is proposed that allows us to highlight the interplay between (symplectic) geometry and (functional) analysis in the proofs of orbital stability of relative equilibria via the so-called energy-momentum method. The theory is illustrated with examples from finite dimensional systems, as well as from Hamiltonian PDE's, such as  solitons, standing and plane waves for the nonlinear Schr\"odinger equation, for the wave equation, and for the Manakov system.
\end{abstract}

\maketitle

\tableofcontents

\newpage

\section{Introduction}\label{s:intro}
The purpose of these notes is to provide an introduction to the theory of orbital stability of relative equilibria, a notion from the theory of (mostly Hamiltonian) dynamical systems with symmetry that finds its origins in the study of planetary motions~\cite{am}. In more recent times it has proven important in two new ways at least. It has on the one hand found an elegant reformulation in the modern framework of Hamiltonian mechanics of finite dimensional systems with symmetry in terms of symplectic geometry. It can indeed be phrased and studied in terms of  the theory of momentum maps and of symplectic reduction~\cite{am,ma,pat,mon,lersi,or,patrobwulff,rob,monrod}. On the other hand, it also underlies the stability analysis of plane waves, of travelling wave solutions and of solitons in infinite dimensional nonlinear Hamiltonian PDE's, which has received considerable attention over the last fourty years or so, and continues to be a very active area of research. We will give a brief historical 
account of the notion of orbital stability in the context of nonlinear PDE's in Section~\ref{history.sec}.

It is clear that in this field nonlinear analysis can be expected to meet geometry in interesting and beautiful ways. It nevertheless appears that in the literature on Hamiltonian PDE's, the simple and elegant geometric ideas underlying the proofs of orbital stability aren't emphasized. 
The goal of these notes is to provide a unified formulation of the theory in a sufficiently general but not too abstract framework that allows one to treat finite and infinite dimensional systems on the same footing. In this manner, one may hope to harness the geometric intuition readily gained from treating finite dimensional systems and use it as a guide when dealing with  the infinite dimensional ones that are the main focus of our interest, but that demand more sophisticated technical tools from functional analysis and PDE theory.  
The text is of an introductory nature and suitable for young researchers wishing to familiarize themselves with the field. It is aimed at analysts not allergic to geometry and at geometers with a taste for analysis, and written in the hope such people exist.  
\subsection{Notions of stability}
There are many notions of stability for dynamical systems. One may in particular consider stability with respect to perturbations in the vector field generating the dynamics, or stability with respect to a variation in the initial conditions. It is the latter one we shall be considering here. For a sampling of possible definitions in this context, one can consult Section~6.3 of Abraham and Marsden [AM], who give nine different ones and mention there exist others still\dots  We start by introducing the ones of interest to us in these notes. 

The simplest possible one is presumably the following. Let $\Ban$ be a normed vector space,  $\rd$ the corresponding metric on $\Ban$, and $X$ a vector field on $\Ban$.  Let $u\in \Ban$ and $t\in\R\to u(t)\in \Ban$ a flow line of $X$ (\emph{i.e.} $\dot u(t) = X(u(t))$, with $u(0)=u$). Let us assume the flow is well-defined globally, with $u(t)=\Phi_t^X(u)$. Then one says that the initial condition $u$ is stable if for all $\epsilon>0$, there exists a $\delta>0$ so that, for all $v\in E$,
\begin{equation}\label{eq:stab0}
\rd(v, u)\leq \delta\Rightarrow\sup_{t\in\R} \rd(v(t), u(t))\leq \epsilon.
\end{equation}
Here $v(t)=\Phi_t^X(v)$. This can be paraphrased as follows: once close, forever not too far. Note that, if $u$ is stable in this sense, then so is $u(t)$ for all $t\in\R$.  There exists one situation where proving stability is straightforward. It is the case where $u=u_*$ is a fixed point of the dynamics, meaning $u(t)=u_*$, for all $t\in\R$, and where $u_*$ is a local non-degenerate minimum of a constant of the motion, that is a 
function $\Lcal:\Ban\to\R$, referred to as a \emph{Lyapunov function}, satisfying $\Lcal(v(t))=\Lcal(v)$ for all $t\in\R$, and for all $v$ in a neighbourhood of $u_*$. Let us sketch the argument, which is classic. Supposing $\Lcal\in C^2(\Ban, \Ban)$ and that $D^2_{u_*}\Lcal$ is positive definite, one obtains from a Taylor expansion of $\Lcal$ about $u_*$ an estimate of the type
\begin{equation}\label{eq:localmin}
c\rd(v, u_*)^2\leq \Lcal(v)-\Lcal(u_*)\leq C\rd(v, u_*)^2,
\end{equation}
for all $v$ in a neighbourhood of $u_*$. Then, for $v$ sufficiently close to $u_*$, one can easily show, using an argument by contradiction, that $v(t)$ stays in this neighbourhood  and hence, for all $t$,
\begin{equation}\label{eq:lyapmethod}
c\rd(v(t), u_*)^2\leq \Lcal(v(t))-\Lcal(u_*)=\Lcal(v)-\Lcal(u_*)\leq C\rd(v, u_*)^2,
\end{equation}
from which~\eqref{eq:stab0} follows immediately.  This approach is known as the Lyapunov method for proving stability.\footnote{Remark that $\Lcal(v(t))\leq \Lcal(v)$ would suffice in~\eqref{eq:lyapmethod}. But in these notes we will exclusively work with constants of the motion.}

In Hamiltonian systems, at least one constant of the motion always exists, namely the Hamiltonian itself. The above argument leads therefore to the perfectly standard result that local minima of the Hamiltonian are stable fixed points of the dynamics. All orbital stability results that we shall discuss below are, 
{\it in fine}, based on this single argument, appropriately applied and combined with additional geometric properties of (Hamiltonian) systems with symmetry, and, of course, with an appropriate dose of (functional) analysis. Let us finally point out that when this approach does not work, and this is very often the case, one is condemned to resort to considerably more sophisticated techniques, involving the KAM theorem or Nekhoroshev estimates, for example.

A stronger version of stability than~\eqref{eq:stab0} is an asymptotic one, and goes as follows:
there exists a $\delta>0$ so that, for all $v\in E$,
\begin{equation*}\label{eq:asymstab}
\rd(v, u)\leq \delta\Rightarrow\lim_{t\to +\infty} \rd(v(t), u(t))=0.
\end{equation*}
This phenomenon can only occur in dissipative systems. When $u$ is a fixed point of the dynamics, it corresponds to requiring it is attractive. If the flow line issued from $u$ is periodic, one obtains a limit cycle. So in this second definition, the idea is that, if two points start close enough, they end up together. Since our focus here is on Hamiltonian systems, where such behaviour cannot occur (because volumes are preserved), we shall not discuss it further. 
Note, however, that another notion of ``asymptotic stability'' has
been introduced and studied in the context of Hamiltonian nonlinear dispersive PDE's. 
We shall briefly comment on this in Section~\ref{history.sec}.

There are several cases when definition~\eqref{eq:stab0} is too strong, and a weaker notion is needed, referred to as \emph{orbital stability}. The simplest definition of this notion goes as follows. Suppose $t\in\R\to u(t)\in E$ is a flow line of the dynamics and consider the dynamical orbit
$$
\gamma=\{u(t) \mid t\in \R\}.
$$ 
We say $u=u(0)$ is orbitally stable if the following holds. For all $\epsilon>0$, there exists $\delta>0$, so that
\begin{equation}\label{eq:orbstab0}
\rd(v, u)<\delta \Rightarrow \forall t\in\R, \rd(v(t), \gamma)\leq \epsilon. 
\end{equation}
The point here is that the new dynamical orbit $\tilde \gamma=\{v(t)\mid t\in\R\}$ stays close to the initial one, while possibly $v(t)$ can drift away from $u(t)$, for the same value of the time $t$. As we will see, this can be expected to be the rule since the nearby orbit may no longer be periodic even if the original one was, or have a different period. A simple example that can be understood without computation is this. Think of two satellites on circular orbits around the earth. Imagine the radii are very close. Then the periods of both motions will be close but different. Both satellites will eternally move on their respective circles, which are close, but they will find themselves on opposite sides of the earth after a long enough time, due to the difference in their angular speeds. In addition, a slight perturbation in the initial condition of one of the satellites will change its orbit, which will become elliptical, and again have a different period. But the new orbit will stay close to the original circle.  So here the idea is this: if an initial condition $v$ is chosen close to $u$, then at all later times $t$, $v(t)$ is close to \emph{some} point on $\gamma$, but not necessarily close to $u(t)$, for the same value of $t$. We will treat this illustrative example in detail in Section~\ref{s:circular}. 
\subsection{Symmetries and relative equilibria}
The definition of orbital stability in~\eqref{eq:orbstab0} turns out to be too strong still for many applications, in particular in the presence of symmetries of the dynamics. This is notably the case in the study of solitons and standing or travelling wave solutions of nonlinear Hamiltonian differential or partial differential equations. We will therefore present an appropriate generalization of the notion of orbital stability in the presence of symmetries in Section~\ref{s:orbstab}. For that purpose, we introduce in Section~\ref{s:dynamics}  dynamical systems $\Phi_t^X$, $t\in\R$ on  Banach spaces $\Ban$, which admit an invariance group $G$ with an action $\Phi_g, g\in G$ on $\Ban$, \emph{i.e.} $\Phi_g\Phi_t^X=\Phi_t^X\Phi_g$. We then say $u\in\Ban$ is a \emph{relative equilibrium} if, for all $t\in\R$, $\Phi_t^X(u)\in \Ocal_u$, where $\Ocal_u=\Phi_G(u)$ is the group orbit of $u$ under the action of $G$. As we will see, solitons, travelling waves and plane waves are relative equilibria. We say a relative equilibrium  $u$ is \emph{orbitally stable}  if initial conditions $v\in\Ban$ close to $u$ have the property that for all $t\in\R$, $\Phi_t^X(v)$ remains close to $\Ocal_u$. Note that the larger the symmetry group $G$ is, the weaker is the corresponding notion of stability.

The main goal of these notes is to present a general framework allowing to establish orbital stability of  such \emph{relative equilibria}  of (both finite and infinite) dynamical systems with symmetry, using an appropriate generalization of the Lyapunov method sketched above. This approach to stability is often referred to
as the ``energy-momentum'' method.
In the process, we wish to clearly separate the part of the argument which is abstract and very general, from the part that is model-dependent. We will also indicate for which arguments one needs the dynamics to be Hamiltonian and which ones go through more generally.

In Section~\ref{s:spherpotstab}, we treat the illustrative example of the relative equilibria of the motion in a spherical potential, allowing us to present four variations of the proof of orbital stability, which are later extended to a very general setting in Section~\ref{s:orbstabproof}. The main hypothesis of the proofs, which work for general dynamical systems on Banach spaces, is the existence of a \emph{coercive Lyapunov function} $\Lcal$, which is a group-invariant constant of the motion satisfying an appropriately generalized coercive estimate of the type~\eqref{eq:localmin} (see~\eqref{eq:coerciverestricted}). In applications, the proof of orbital stability is thus reduced to the construction of such a function.

It is in this step that the geometry of Hamiltonian dynamical systems with symmetry plays a crucial role. Indeed, the construction of an appropriate Lyapunov function for such systems exploits the special link that exists between their constants of the motion $F$ and their symmetries, as embodied in Noether's theorem and the theory of the momentum map. This is explained in Sections~\ref{s:hamdyninfinite} and~\ref{s:identifyreleq}. The crucial observation is then that 
in Hamiltonian systems, relative equilibria tend to come in families $u_\mu\in\Ban$, indexed by the value $\mu$ of the constants of the motion at $u_\mu$. In fact, it turns out that $u_\mu\in\Ban$ is a relative equilibrium of a Hamiltonian system if (and only if) $u_\mu$ is a critical point of the restriction of the Hamiltonian to the level surface $\Sigma_\mu=\{u\in\Ban\mid F(u)=\mu\}$ of these constants of the motion (Theorem~\ref{thm:relequicritical}). This observation at once yields the candidate Lyapunov function $\Lcal_\mu$ (see~\eqref{eq:lagfunction}). 

We finally explain (Proposition~\ref{thm:hessianestimate}) how the proof of the coercivity of the Lyapunov function can be obtained from a suitable lower bound on its second derivatives $D^2\Lcal_\mu(w,w)$, with $w$ restricted to an appropriate subspace of $\Ban$,  using  familiar arguments from the theory of Lagrange multipliers (Section~\ref{s:orbstabproof}). This ends the very general, geometric and abstract part of the theory. To control  $D^2\Lcal_\mu(w,w)$ finally requires an often difficult, problem-dependent, and detailed spectral analysis of the Hessian of the Lyapunov function, as we will show in the remaining sections. 

\subsection{Examples}
We illustrate the theory in Section~\ref{s:nlsetorus1d} on a first simple example. We consider the plane waves $u_{\alpha,k}(t,x)=\alpha e^{-ikx}e^{i\en t}$, $\en\in \RR$, $k\in {2\pi}\ZZ$ and $\alpha \in \R$, which are solutions of
the cubic  nonlinear Schr\"odinger equation on the one-dimensional torus $\TT$,
$$ 	
i{\partial_t}u(t,x)+\beta{\partial_{xx}^2}u(t,x)+\lambda\vb u(t,x)\vb^2u(t,x)=0,
$$
provided $\en +\beta k^2=\lambda\vb\alpha\vb^2$. This equation is (globally) well-posed on $\Ban=H^1(\TT,\CC)$ and its dynamical flow is invariant under the globally Hamiltonian action $\Phi$ of the group $G=\R\times\R$ defined by $\left(\Phi_{a,\gamma}(u)\right)(x)=e^{i\gamma}u(x-a)$ (see Section \ref{ss:hampde}). The plane waves $u_{\alpha,k}(t,x)$ are $G$-relative equilibria. 
We establish (Theorem~\ref{thmstability}) their orbital stability
when $\beta\left({2\pi}\right)^2>2\lambda |\alpha|^2$.
Although the linear stability analysis for this model is sketched in many places, and the nonlinear (in)stability results seem to be known to many, we did not find a complete proof of nonlinear orbital stability in the literature. A brief comparison between our analysis and related results  (\cite{zhidkov01,galhar07a,galhar07b}) ends Section~\ref{s:nlsetorus1d}. Note that the analysis of orbital stability of plane waves of the cubic nonlinear Schr\"odinger equation on a torus of  dimension $d>1$ is much more involved (see for example \cite{faogaulub2013}). 

In Section~\ref{curves.sec} we will present orbital stability results pertaining to curves (\emph{i.e.} one-dimensional families) of standing waves
of nonlinear Schr\"odinger equations on $\rn$ with a space-dependent coefficient $f$:
\begin{equation}\label{nls0}
i\partial_t u(t,x)+\Delta u(t,x)+f(x,|u|^2(t,x))u(t,x)=0.
\end{equation}
Imposing a 
non-trivial spatial dependence has two major consequences. First, the space-translation 
symmetry of the equation is destroyed, and one is left with the reduced one-parameter symmetry group 
$G=\R$, acting on the Sobolev space $E=H^1(\rn)$ via
$\Phi_\gamma(u)=e^{i\gamma} u$.
Note that the associated group orbits are of the simple form 
$\Ocal_u=\{e^{i\gamma} u : \gamma\in\R\}\subset H^1(\rn)$.
Now, standing waves are, by definition, solutions of~\eqref{nls0} of the form
$u(x,t)=e^{i\xi t} w(x)$,
which are therefore clearly relative equilibria. Such standing waves are sometimes referred to as
``solitons'' due to the spatial localization of the profile $w(x)$, and to their stability.

Second, constructing curves of standing wave solutions of \eqref{nls0} is now a hard problem, 
and we will outline the 
bifurcation theory developed in \cite{dcds,jde,ans,eect} to solve it. This powerful
approach allows one to deal with power-type nonlinearities
$f(x,|u|^2)=V(x)|u|^{\sigma-1}$ (under an approriate decay assumption 
on the coefficient $V:\rn\to\real$) 
but also with more general nonlinearities, for instance the asymptotically linear
$f(x,|u|^2)=V(x)\frac{|u|^{\sigma-1}}{1+|u|^{\sigma-1}}$. This will give a good illustration
of how involved the detailed analysis of $D^2\Lcal(w,w)$ required by the model can be.
As we shall see, this analysis
turns out to be deeply connected with the bifurcation behaviour of the standing waves.

In the pure power (space-independent) case $f(x,|u|^2)=|u|^{\sigma-1}$, the appropriate
notion of stability is that associated with the action of the full group $G=\rn\times\real$,
$\left(\Phi_{a,\gamma}(u)\right)(x)=e^{i\gamma}u(x-a)$.
The stability
of standing waves in this context was proved in the seminal paper of Cazenave and Lions \cite{cazlions}
for $1<\sigma<1+\frac{4}{d}$, and this result is sharp ({\it i.e.} stability does not hold at 
$\sigma=1+\frac{4}{d}$). The contribution \cite{cazlions} is one of the first rigorous results
on orbital stability for nonlinear dispersive equations, and is based on variational arguments using the
concentration-compactness principle (see for instance \cite{zhidkov01,hajstu} for more recent 
results in this direction). This line of argument is conceptually very different from the 
energy-momentum approach developed here, so we shall not say more about it.

The modern treatment of Hamiltonian dynamical systems with symmetries uses the language of symplectic geometry, as for example in~\cite{am, ar, ma, souriau1997}. But we don't need the full power of this theory, since we will work exclusively with linear symplectic structures on (infinite dimensional) symplectic vector spaces. For the reader not familiar with Hamiltonian mechanics, Lie group theory and symplectic group actions, elementary self-contained introductions to these subjects sufficient for our purposes are provided in the Appendix. 

\medskip
\noindent{\bf Acknowledgments.}  This work was supported in part by the Labex CEMPI (ANR-11-LABX-0007-01). F.G. thanks CEMPI and the Lab. Paul Painlev\'e for their hospitality during his one-month visit to the Universit\'e Lille 1 in September 2013. He also acknowledges the support 
of the ERC Advanced Grant ``Nonlinear studies of water flows with vorticity''.
The authors are grateful to V.~Combet, A.~De Laire, S.~Keraani, G.~Rivi\`ere, B.~Tumpach and G.~Tuynman for stimulating discussions on the subject matter of these notes.

\section{Dynamical systems, symmetries and relative equilibria}\label{s:dynamics}
\subsection{Dynamical systems on Banach spaces.}\label{ss:dynamicsbanach}
Let $\Ban$ be a Banach space.  A domain $\Dom$\index{domain} is a dense subset of $E$; in the examples presented in these notes, it will be a dense linear subspace of $E$.  
\begin{definition}\label{def:dynsystem}
A \emph{dynamical system}\index{dynamical system} on $\Ban$ is a separately continuous map
\begin{equation}\label{flow1}
\Phi^X:(t, u)\in\R\times \Ban\to \Phi^X_t(u):=\Phi^X(t,u)\in \Ban, 
\end{equation}
with the following properties: 
\begin{enumerate}
\item[(i)] For all $t,s\in\R$,
\begin{equation}\label{flow2}
\Phi_t^X\circ\Phi_s^X=\Phi_{t+s}^X, \quad  \Phi_0^X(u)=\mathrm{Id}_\Ban.
\end{equation}
\item[(ii)] For all $t\in\R$, $\Phi_t^X(\Dom)=\Dom$.
\item[(iii)]  $X:{\mathcal D}\subset \Ban\to \Ban$ is a vector field that generates the dynamics in the sense that, when $u\in\mathcal D$,  $\Phi_t^X(u):=u(t)\in\Dcal$ is a solution of the differential equation
\begin{equation}\label{eq:diffeq}
\dot u(t)=X(u(t)),\quad u(0)=u.
\end{equation}
\end{enumerate}
By this we mean that the curve $t\in\R\to u(t)\in \Ban$ is differentiable as a map from $\R$ to $E$. 
\end{definition}
In infinite dimensional problems, the vector fields are often only defined on a domain $\Dom$, where they may not even be continuous.   But note that we always assume that the flows themselves are defined on all of $\Ban$ (or on an open subset of $\Ban$). For examples illustrating these subtleties, see Section~\ref{s:dynsysexamples}. Local flows can be defined in the usual manner. In that case the domains are dense in some open subset of $E$, but we shall not deal with such situations in these notes since we will always assume the flows to be globally defined.

Suppose  there exists a function $F:\Ban\to\R^m$ so that 
\begin{equation}\label{eq:constantofmotion}
F\circ \Phi_t^X=F, \quad \forall t\in \R.
\end{equation}
We then say that the  vector field $X$ or its associated flow $\Phi_t^X$ admits $m$ \emph{constants of the motion}\index{constants of the motion}, which are the components $F_i$ of $F$. In that case, one may consider the restriction of the flow $\Phi_t^X$ to the level sets of $F$: for $\mu\in\R^m$, we define
\begin{equation}\label{eq:levelsurface}
\Sigma_\mu=\{u\in \Ban\mid F(u)=\mu\},
\end{equation}
and one has that $\Phi_t^X\Sigma_\mu=\Sigma_\mu$, for all $\mu\in\R^m$.

\begin{remark}
The role of and the need for a domain $\Dcal$ with the properties (ii) and (iii) in the definition of a dynamical system above will become clear in Sections~\ref{s:hamdyninfinite} and~\ref{s:identifyreleq}. They are in particular needed to prove~\eqref{eq:constantofmotion} for suitable $F$. Some of the stability results that are our main focus can be obtained without those conditions, as we will further explain in Section~\ref{s:orbstabproof}. Similarly, global existence is not strictly needed: it can for example be replaced by a weaker ``blow-up alternative.'' We will not further deal with these issues here.
\end{remark}
\subsection{Symmetries, reduced dynamics and relative equilibria}\label{ss:invariance}
We now define the notion of an invariance group for a dynamical system.  For that purpose, we need to say a few words about group actions. Let $G$ be a topological group acting on $E$. By this we mean there exists a separately continuous map
$$
\Phi : (g,u)\in G\times \Ban \to \Phi_g(u)\in \Ban,
$$
satisfying $\Phi_e=\mathrm{Id}$, $\Phi_{g_1g_2}=\Phi_{g_1}\circ\Phi_{g_2}$. We will call
\begin{equation}\label{eq:Gorbit}
\Ocal_u=\{\Phi_g(u) \mid g\in G\}
\end{equation}
the orbit of $G$\index{orbit} through $u\in E$. For later reference, we define the isotropy group\index{isotropy group} of $u$, $G_u$, as follows
\begin{equation}\label{eq:isotropyu}
G_u=\{g\in G\mid \Phi_g(u)=u\}.
\end{equation}

We can then introduce the notion of an invariance group for $\Phi_t^X$. 
\begin{definition}\label{def:symgroup} We say $G$ is an invariance group \index{invariance group} (or symmetry group)\index{symmetry group|seeonly {invariance group}} for the dynamical system $\Phi_t^X$ if, for all $g\in G$, and for all $t\in\R$,
\begin{equation}\label{eq:symgroup}
\Phi_g\circ \Phi^X_t=\Phi_t^X\circ \Phi_g.
\end{equation}
\end{definition}
Remark that  $G=\R$ is always an invariance group of the dynamical system, with action $\Phi_t^X$ on $E$. While this is correct, this is not of any particular use, as one can suspect from the start. Indeed, the flow $\Phi_t^X$ is in applications obtained by integrating a nonlinear differential or partial differential equation, and is not explicitly known. In fact, it is the object of study. ``Useful'' symmetries are those that help to simplify this study; they need to have a simple and explicit action on $\Ban$. They are often of a clearcut  geometric origin: translations, rotations, gauge transformations, \emph{etc.} Several examples are provided in the following sections. 

Finally, it should be noted we did not define ``the'' symmetry group for $\Phi_t^X$, but ``a'' symmetry group. Depending on the problem at hand and the questions addressed, different symmetry groups may prove useful for the same dynamical system, as we shall also illustrate. In particular, any subgroup of an invariance group is also an invariance group, trivially. 

It follows immediately from~\eqref{eq:Gorbit} and~\eqref{eq:symgroup} that, for all $x\in \Ban$, 
\begin{equation}\label{eq:orbdyn}
\Phi_t^X\Ocal_u=\Ocal_{\Phi_t^X(u)}.
\end{equation}
In other words, if $G$ is an invariance group, then the dynamical system maps $G$-orbits into $G$-orbits.  This observation lies at the origin of the following construction which is crucial for the definitions of relative equilibrium and orbital stability that we shall introduce. We give the general definitions here, and refer to the coming sections for examples.  Defining an equivalence relation on $\Ban$ through 
$$
u\sim u'\Leftrightarrow \Ocal_u=\Ocal_{u'},
$$
we consider the corresponding quotient space that we denote by $\Ban_G=\Ban/\sim$ and that we refer to as the reduced phase space. We will occasionally use the notation
\begin{equation}\label{eq:projection}
\pi:u\in \Ban \to \Ocal_u\in \Ban_G
\end{equation}
for the associated projection. So the elements of $\Ban_G$ are just the $G$-orbits in $\Ban$. It is then clear from~\eqref{eq:orbdyn} that the dynamical system $\Phi_t^X$ on $\Ban$ naturally induces reduced dynamics\index{reduced dynamics} on the orbit space $\Ban_G$: it ``passes to the quotient'' in the usual jargon. We will use the same notation for these reduced dynamics and write $\Phi_t^{X}\Ocal=\Ocal(t)$ for any $\Ocal\in E_G$. Note that $\Phi^X_t\Ocal_u=\Ocal_{u(t)}$ (See Fig.~\ref{figure1}).  

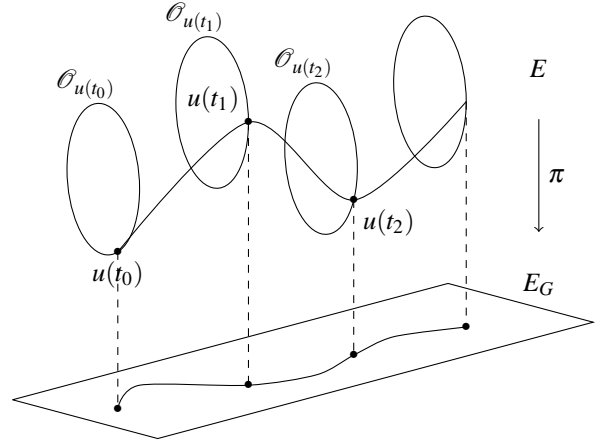
\begin{figure}
	\begin{center}
 	\begin{tikzpicture}[x={(0.965cm,-0.258cm)}, y={(0.965cm,0.258cm)}, z={(0cm,1cm)}]
		\begin{scope}[canvas is yx plane at z=0]
		\draw (0.5,0.5)--(6.5,0.5)--(6.5,2.5)--(0.5,2.5)--(0.5,0.5);
		\draw plot [smooth, tension=0.5] coordinates {(1,1.45) (1.75,1) (2.5,1.75) (3.25,1.95) (4,1.7) (4.75,1.5) (5.5,1.75)};
		\node  at (1,1.45)  {\scriptsize$\bullet$};
		\node  at (2.5,1.75)  {\scriptsize$\bullet$};
		\node  at (4,1.7)  {\scriptsize$\bullet$};
		\node  at (5.5,1.75)  {\scriptsize$\bullet$};
		\end{scope}
		\begin{scope}[canvas is xz plane at y=1]
		\draw (1.25,3) ellipse (0.5cm and 1cm);
		\draw[dashed] (1.45,2.083) -- (1.45,0);
		\end{scope}
		\begin{scope}[canvas is xz plane at y=2.5]
		\draw (1.25,3.5) ellipse (0.5cm and 1cm);
		\draw[dashed] (1.75,3.5) -- (1.75,0);
		\end{scope}
		\begin{scope}[canvas is xz plane at y=4]
		\draw (1.25,2.5) ellipse (0.5cm and 1cm);
		\draw[dashed] (1.7,2.064) -- (1.7,0);
		\end{scope}
		\begin{scope}[canvas is xz plane at y=5.5]
		\draw (1.25,3) ellipse (0.5cm and 1cm);
		\draw[dashed] (1.75,3) -- (1.75,0);
		\end{scope}
		\draw plot [smooth, tension=0.5] coordinates {(1.45,1,2.083) (1.75,2.5,3.5) (1.7,4,2.064) (1.75,5.5,3)};
		\node  at (1,1,4.2) {$\Ocal_{u(t_0)}$};
		\node  at (1.45,1,2.083) [anchor=north] {${u(t_0)}$};
		\node  at (1.45,1,2.083) {\scriptsize$\bullet$};
		\node  at (1,2.5,4.7) {$\Ocal_{u{(t_1)}}$};
		\node  at (1.75,2.5,3.5)[anchor=south east] {$u{(t_1)}$};
		\node  at (1.75,2.5,3.5) {\scriptsize $\bullet$};
		\node  at (1,4,3.7) {$\Ocal_{u{(t_2)}}$};
		\node  at (1.7,4,2.064) [anchor=north west] {$u{(t_2)}$};
		\node  at (1.7,4,2.064) {\scriptsize $\bullet$};
		\node at (1.25,7,2.5) [anchor=south,inner sep=0.3cm] {$E$};
		\draw[->] (1.25,7,2.25) --  (1.25,7,0.75) node[midway, anchor=west] {$\pi$};
		\node at (1.25,7,0.5) [anchor=north,inner sep=0.3cm] {$E_G$};
	\end{tikzpicture}
	\end{center}
	\caption{A dynamical orbit $t\to u(t)$ and its ``attached'' $G$-orbits, with the projection into $E_G$. \label{figure1} }
\end{figure}

As a general rule of thumb, one may hope that the reduced dynamics are simpler than the original ones, since they take place on a lower dimensional (or in some sense smaller) quotient space. This idea can sometimes provide a useful guideline, notably in the study of stability properties of fixed points or periodic orbits of the original dynamical system, as will be illustrated in the coming sections. Implementing it concretely can nevertheless be complicated, in particular because the quotient itself may be an unpleasant object to do analysis on, even in finite dimensions, as its topology or differential structure may be pathological and difficult to deal with. Conditions on $G$ and on the action $\Phi$ are needed, for example, to ensure the quotient topology on $E_G$ is Hausdorff, or that it has a differentiable structure~\cite{am, ma, patrobwulff}. In addition, concrete computations on models are more readily done on $E$ directly, than in the abstract quotient space, particularly in infinite dimensional problems. We will avoid these difficulties, in particular because we will work almost exclusively with \emph{isometric} group actions.  Their orbits have simplifying features that we will repeatedly use: see Proposition~\ref{prop:isometry} below.

We are now in a position to introduce the notion of \emph{relative equilibrium}, as follows. 
\begin{definition}\label{def:orbfixedpoint} Let $u\in E$. Let $\Phi_t^X$ be a dynamical system on $E$ and let $G$ be a symmetry group for $\Phi_t^X$. We say $u$ is a $G$-\emph{relative equilibrium}\index{relative equilibrium}\footnote{In~\cite{ma}, the term \emph{stationary motion} is used for this concept.} for $\Phi_t^X$ if, for all $t\in\R$, $u(t)\in\Ocal_u$. Or, equivalently, if for all $t\in\R$, $\Phi_t^{X}\Ocal_u=\Ocal_u$. When there is no ambiguity about the dynamical system $\Phi_t^X$ and the group $G$ considered, we will simply say $u$ is a \emph{relative equilibrium}. 
\end{definition}

With the language introduced, $u$ is a relative equilibrium if $\Ocal_u$ is a fixed point of the reduced dynamics on $\Ban_G$.  Again, we refer to the following sections for examples. We are interested in these notes in the stability of such relative equilibria. Roughly speaking, we will say a relative equilibrium is orbitally stable if it is stable as a fixed point of the reduced dynamics; we give a precise definition in Section~\ref{s:orbstab}. 

We end this section with two comments. First, the above terminology comes from  the literature on Hamiltonian dynamical systems in finite dimensions. We will see in the following sections what the many specificities are of that situation. We refer to~\cite{ar, am, ma} for textbook treatments and historical background and to~\cite{pat, mon, lersi, patrobwulff, rob, monrod} for more recent developments. Second, we will often need to deal with the restriction of the dynamical systems under consideration to the level sets $\Sigma_\mu\subset\Ban$ of a family of constants of the motion $F$, as defined in~\eqref{eq:levelsurface}. Note that $\Sigma_\mu$ is a metric space. We define
\begin{equation}\label{eq:gsimgamu}
G_{\Sigma_\mu}=\{ g\in G\mid \forall u\in \Sigma_\mu, \Phi_g(u)\in\Sigma_\mu\}.
\end{equation}
This is clearly a subgroup of $G$, which is a symmetry group of the dynamical system restricted to $\Sigma_\mu$. We will often deal with isometric group actions on such $\Sigma_\mu$, or on the full Banach space $E$. The following simple proposition collects some of the essential properties of their orbits that we shall repeatedly need and use. We first recall the definition of the Hausdorff metric. Let $(\Sigma, \rd)$ be a metric space and let $S, S'\subset (\Sigma, \rd)$. Then
\begin{equation}\label{eq:hausdorff}
\Delta(S, S')=\max\{\sup_{u\in S} \rd(u, S'), \sup_{u'\in S'}\rd(S, u')\}.
\end{equation}
Notice that this is only a pseudometric\footnote{$\Delta(S, S')=0$ does not imply $S=S'$. In particular, $\Delta(S, \overline S)=0$. } and that $\Delta(S, S')=+\infty$ is possible. 
\begin{proposition}\label{prop:isometry} Let $G$ be a group, $(\Sigma,\rd)$ a metric space  and $\Phi:G\times \Sigma\to \Sigma$ an action of $G$ on $\Sigma$. Suppose that for each $g\in G$,  $\Phi_g$ is an isometry: $\forall u,u'\in \Sigma, \rd (\Phi_g(u),\Phi_g(u'))=\rd (u,u')$. Let $\Ocal, \Ocal'$ be two $G$-orbits in $\Sigma$. Then
\begin{enumerate}[label=({\roman*})]
\item $\forall u_1, u_2\in\Ocal, \forall u'_1, u'_2\in \Ocal',\ \rd (u_1, \Ocal')=\rd(u_2, \Ocal'),\quad \rd(u'_1, \Ocal)=\rd (u'_2, \Ocal),$
\item $\forall u\in \Ocal, u'\in \Ocal', \quad \rd(u, \Ocal')=\Delta(\Ocal, \Ocal')=\rd(u', \Ocal),$
\item $\forall u\in \Ocal, u'\in \Ocal', \quad   \Delta(\Ocal, \Ocal')\leq \rd(u,u').$
\end{enumerate}
\end{proposition}
\begin{proof} The first statement follows from the existence of $g\in G$ so that $\Phi_g(u_1)=u_2$. For the second, we proceed by contradiction. Suppose first that, $\forall u\in \Ocal, u'\in \Ocal'$, $\rd(u,\Ocal')<\rd(u',\Ocal)$. Let $u\in\Ocal, u'\in\Ocal'$. Then we know there exists $v\in\Ocal'$ (depending on $u,u'$) so that $\rd(u,\Ocal')\leq \rd(u,v)<\rd(u',\Ocal)$. But since, by the first part of the proposition, $\rd(v,\Ocal)=\rd(u', \Ocal)$, this implies $\rd(u,v)<\rd(v,\Ocal)$, which is a contradiction. So we conclude, using the first part again, that  $\forall u\in \Ocal, u'\in \Ocal'$, $\rd(u,\Ocal')\geq\rd(u',\Ocal)$. Repeating the argument with the roles of $\Ocal, \Ocal'$ inverted, the result follows.
\end{proof}
If the action is not isometric, it is quite possible for all the statements of the theorem to fail. For example, consider on $\Ban=\R^2$ the action $\Phi_a(q,p)=(\exp(a) q, \exp(-a) p)$, $a\in \R$.

\section{Examples}
\subsection{Motion in a spherical potential}\label{s:spherpot}
In this section, we illustrate the preceding notions on a simple Hamiltonian mechanical system: a particle in a spherical potential. We will make free use of the concepts and notation of  Appendices~\ref{s:liegroups} and~\ref{s:hammech} that we invite the reader unfamiliar with Hamiltonian mechanics or Lie group theory to peruse. 
  
By a spherical potential we mean  a function $V:\R^3\to \R$,  
satisfying $V(Rq)=V(q)$, for all $R\in\,$SO$(3)$. With a slight abuse of notation, we write $V(q)=V(\|q\|)$, for a smooth function $V:\R^+\to \R$. 
We consider on $E=\R^6$ the Hamiltonian 
\begin{equation}\label{eq:spherpotham}
H(u)=H(q,p)=\frac12p^2 +V(\|q\|)
\end{equation} 
and the corresponding Hamiltonian equations of motion
\begin{equation}\label{eq:eqmotionspherical}
\dot q=p,\quad \dot p=-V'(\|q\|)\hat q,
\end{equation}
where we introduce the notation $\hat b=\frac{b}{\|b\|}$ for any $b\in\R^3$.   Integrating those, we obtain the Hamiltonian flow $\Phi_t^H(u)=u(t)$, where $u=(q,p)\in\R^6$. Introducing the angular momentum 
\begin{equation}\label{eq:angularmomentum}
L(q,p)=q\wedge p,
\end{equation}
one checks immediately that, for any solution $t\in \R\to (q(t), p(t))\in\R^6$, one has 
\begin{equation}\label{eq:angmomentumcons}
\frac{\rd}{\rd t} L(q(t), p(t))=0.
\end{equation}
In other words, angular momentum is conserved during the motion in a central potential: its three components are constants of the motion. This implies the familiar result that the motion takes place in the plane perpendicular to $L$ and passing through $0$.  

We will now use Noether's Theorem (Theorem~\ref{thm:nother2}) to show this system is SO$(3)$-invariant. We start with the following observations.  
First, the action of the group $G=$SO$(3)$ on $\Ban=\R^6$ given by 
\begin{equation}\label{eq:so3}
\Phi_R(u)=(Rq, Rp)
\end{equation}
is easily checked to be globally Hamiltonian\footnote{See Definition~\ref{def:globhamaction}.}. Indeed, for each $\xi\in\mathrm{so}(3)$, 
$$
\Phi_{\exp (t\xi)}=\Phi^{F_\xi}_t,
$$
where
\begin{equation}\label{eq:so3mommap}
F_{\xi}(q,p)= \xi \cdot L(q,p)
\end{equation}
(recall that we can identify $\mathrm{so}(3)$ with $\R^3$ via \eqref{xirepr}). In other words,``angular momentum generates rotations.''
Next, it is clear that the Hamiltonian satisfies $H\circ \Phi_R=H$.
As a result, it follows from  Theorem~\ref{thm:nother2}~(iii) that the dynamical flow is rotationally invariant:
$$
\Phi_t^H\circ \Phi_R=\Phi_R\circ \Phi_t^H,\qquad\forall t\in\R, \  R\in \mathrm{SO}(3).
$$ 
Note that, here and in what follows, we are using, apart from the symplectic, also the standard euclidean structure on $\R^6$. 

We now wish to identify the relative equilibria of these systems. For that purpose, consider first $u\in\R^6$ with $L(u)=\mu\not=0$. Then the ensuing dynamical trajectory $u(t)$ lies in the 
surface
\begin{equation}\label{eq:sphericalsigmamu}
\Sigma_\mu=\{u\in\R^6\mid L(u)=\mu\}.
\end{equation}
Now, if $u$ is a relative equilibrium, then, for each $t$, there exists $R(t)\in\mathrm{SO}(3)$ so that $\Phi_{R(t)}u=u(t)$. Hence $\mu=L(u(t))=L(\Phi_{R(t)}u)=R(t)L(u)=R(t)\mu$.  In other words, $R(t)$ belongs to 
 \begin{equation*}
 G_\mu=\{R\in\mathrm{SO}(3)\mid R\mu=\mu\}\simeq \mathrm{SO}(2),
 \end{equation*}
which is the subgroup of rotations about the  $\mu$-axis. It follows that $\|q(t)\|=\|q\|$, for all $t$. Since $q(t)$ is perpendicular to $\mu$, this means that $q(t)$ lies on the circle of radius $\|q\|$ centered at $0$ and perpendicular to $\mu$. The orbit is therefore circular and, in particular, for all $t$, $q(t)\cdot p(t)=0$. Conversely, it is clear that all circular dynamical orbits are relative equilibria. 
The initial conditions corresponding to such circular orbits are easily seen to be of the form
\begin{equation}\label{eq:circincond}
q=\rho_*\hat q,\quad p=\sigma_*\hat p, \quad \sigma_*^2=\rho_*V'(\rho_*), \quad \hat q\cdot \hat p=0,
\end{equation}
with $\rho_*,\sigma_*>0$ and hence $V'(\rho_*)>0$.
We will discuss in Section~\ref{s:spherpotstab} under what conditions they are orbitally stable in the sense of~\eqref{eq:orbstab0}. 

Now, let $u=(q,p)\in\R^6$ be such that $L(u)=0$. In this case $q$ and $p$ are parallel and this remains true at all times. But if $p(t)\not=0$ at any time $t$, $u$ cannot be a relative equilibrium. Indeed, the motion is then along a straight line passing through the origin and such a straight line cannot lie in an SO$(3)$ orbit since the SO$(3)$ action preserves norms.   If on the other hand  $u=(\rho_*\hat q,0)=u(t)$ is a fixed point of the dynamics, it is a fortiori a relative equilibrium. This occurs if and only if $V'(\rho_*)=0$ as is clear from the equations of motion. Note that these fixed points fill the sphere of radius $\rho_*$.

 It is clear these fixed points cannot be stable in the sense of definition~\eqref{eq:stab0} or~\eqref{eq:orbstab0}.  
Indeed, any initial condition $u'$ close to such fixed point $u$, but with $p'\not=0$ gives rise to a trajectory in the plane spanned by $q'$ and $p'$: when $q'$ and $p'$ are not parallel, the trajectory will wind around the origin in this plane, moving away from the initial condition. What we will prove in Section~\ref{s:spherpotstab} is that, provided $V''(\rho_*)>0$, these trajectories all stay close to 
\begin{equation}\label{eq:fixorbit}
\Ocal_{\rho_*, 0,0}=\{u\in\R^6\mid q\cdot q=\rho_*^2, \ p\cdot p=0, \ q\cdot p=0\},
\end{equation}
which is the SO$(3)$ orbit through the fixed point $u=(\rho_*\hat q, 0)$. Those fixed points are therefore SO$(3)$-orbitally stable, in the sense of Definition~\ref{def:orbstabgen}~(i) below.

To end this section, we list, for later purposes, all SO$(3)$-orbits in $\Ban=\R^6$. 
Those are easily seen to be the hypersurfaces  $\Ocal_{\rho, \sigma, \alpha}$ of the form
\begin{equation}\label{eq:rhosigmaalpha}
\Ocal_{\rho, \sigma, \alpha}=\{(q,p)\in\R^6\mid q\cdot q=\rho^2,\ p\cdot p=\sigma^2, \ q\cdot p=\alpha\},
\end{equation}
with $\rho, \sigma\geq 0, \alpha\in \R$. Note that $|\alpha|\leq \rho\sigma$. 
Those orbits are three-dimensional smooth submanifolds of $\R^6$, except on the set where the angular momentum $L$ vanishes, \emph{i.e.} on
$$
\Sigma_0=\{(q,p)\in\R^6\mid L(q,p)=0\}.
$$
This surface (which is not a submanifold of $E$) is itself SO$(3)$-invariant and foliated by group orbits as follows:
$$
\Sigma_0=\bigcup_{\rho\sigma=|\alpha|}\ \Ocal_{\rho,\sigma,\alpha}=
\{(0,0)\} \ \cup\bigcup_{\substack{ \rho\sigma=|\alpha| \\ (\rho,\sigma)\not=(0,0) }}\ \Ocal_{\rho,\sigma,\alpha}.
$$
On the latter orbits, $q$ and $p$ are parallel, but do not both vanish, so that these orbits can be identified with two-dimensional spheres.  


\subsection{The nonlinear Schr\"odinger equation}\label{s:dynsysexamples}

An important example of an infinite dimensional dynamical system is the nonlinear Schr\"odinger\index{equation!nonlinear Schr\"odinger} equation 

\begin{equation}
	\label{eq:nlsgen}
	\left\{
	\begin{aligned}
		&i\partial_t u(t,x) +\Delta u(t,x)  +\nl(x,u(t,x))=0,\\
		&u(0,x)=u_0(x),
	\end{aligned}
	\right.
\end{equation}
with $u(t,x):\R\times\R^d\to \C$. Here $\Delta$ denotes the usual Laplace operator and $\nl$ is a local nonlinearity. More precisely, consider $\nl : (x,u)\in\R^d\times \R^+\to \nl(x,u)\in\R$ such that $\nl$ is measurable in $x$ and continuous in $u$. Assume that 
\begin{equation}
	\label{eq:nlscond1nl}
	\nl(x,0)=0\ \mbox{a.e. in }\R^d
\end{equation}
and that for every $K>0$ there exists $L(K)<+\infty$ such that
\begin{equation}
	\label{eq:nlscond2nl}
	|\nl(x,u)-\nl(x,v)|\le L(K)|u-v|
\end{equation}
a.e. in $\R^d$ and for all $0\leq u,v\le K$. Assume further that
\begin{equation}
	\label{eq:nlscond3nl}
	\left\{
	\begin{aligned}
		& L(\cdot)\in C^0([0,+\infty))&\mbox{if } d=1,\\
		& L(K)\le C(1+K^{\alpha})\ \mbox{with }  0\le \alpha <\frac{4}{d-2}&\mbox{if } d\ge 2,
	\end{aligned}
	\right.
\end{equation}
and extend $\nl$ to the complex plane by setting
\begin{equation}
	\label{eq:nlsextnl}
	f(x,u)=\frac{u}{|u|}f(x,|u|),
\end{equation}
for all $u\in \C$, $u\neq 0$.

Finally, let $H$ be the Hamiltonian of the system defined by 
\begin{equation}
	\label{eq:nlshamiltonian}
	H(u)=\frac{1}{2}\int_{\R^d}|\nabla u|^2(x) \diff x-\int_{\R^d}\int_{0}^{|u|(x)}f(x,s)\diff s\diff x.
\end{equation}
 We now explain how the Schr\"odinger equation defines an infinite dimensional dynamical system with symmetries, within the framework of Sections \ref{ss:dynamicsbanach} and \ref{ss:invariance}. The sense in which the Schr\"odinger equation defines a Hamiltonian dynamical system will be explained in Section \ref{s:hamdyninfinite}.
 
For that purpose, we need the following results on local and global existence of solutions to
\eqref{eq:nlsgen}. First, concerning local existence, we have:

\begin{theorem}[\cite{cazenave2003}]\label{thm:nlslocalflowH1} 
If $\nl$ is as above, then for every $u_0\in H^{1}(\R^d,\C)$ there exist numbers $T_{\mathrm{min}}, T_{\mathrm{max}}>0$ and a unique maximal solution $u:t\in (-T_{\mathrm{min}}, T_{\mathrm{max}})\to u(t)\in H^1(\R^d,\C)$ of~\eqref{eq:nlsgen} satisfying
	$$u\in C^0((-T_{\mathrm{min}}, T_{\mathrm{max}}),H^{1}(\R^d))\cap C^1((-T_{\mathrm{min}}, T_{\mathrm{max}}),H^{-1}(\R^d)).$$ Moreover, $u$ depends continuously on $u_0$ in the following sense: if $u_0^k\to u_0$ in $H^1(\R^d,\C)$  and if $u_k$ is the maximal solution of \eqref{eq:nlsgen} with the initial value $u_0^k$, then $u_k\to u$ in $ C^0([-S,T],H^1(\R^d))$ for every interval $[-S,T]\subset (-T_{\mathrm{min}}, T_{\mathrm{max}})$. In addition, there is conservation of charge and energy, that is
	\begin{equation}
		\label{eq:nlsconservationlaw}
		\|u(t)\|_{L^2}=\|u_0\|_{L^2},\quad H(u(t))=H(u_0)
	\end{equation}
	for all $t\in  (-T_{\mathrm{min}}, T_{\mathrm{max}})$.
\end{theorem}
For global existence of solutions, one needs a growth condition on $f$ in its second variable. 
\begin{theorem}[\cite{cazenave2003}]\label{thm:nlsglobalflowH1} 
Let $\nl$ be as in Theorem \ref{thm:nlslocalflowH1}. Suppose in addition that there exist 
$A\ge 0$ and $0\le \nu <\frac{4}{d}$ such that
	\begin{equation}
		\label{eq:nlscondglobalex}
		\int_{0}^{|u|}f(x,s)\diff s\le A|u|^2(1+|u|^{\nu}), \quad x\in\rn, \ u\in\complex.
	\end{equation}
	It follows that for every $u_0\in H^{1}(\R^d,\C)$, the maximal strong $H^1$-solution $u$ of  \eqref{eq:nlsgen} given by Theorem \ref{thm:nlslocalflowH1} is global and $\sup_{t\in\R} \|u(t)\|_{H^1}<+\infty$.
\end{theorem}
Note that the condition on $f$ is always satisfied when $f$ is negative. This result implies that one can define $\Phi_t^X$ on $\Ban=H^1(\R^d,\C)$ by $\Phi_t^X(u)=u(t)\in \Ban$ and that $\Phi_t^X$ satisfies~\eqref{flow1}--\eqref{flow2}. Note however that, whereas the flow lines $t\to u(t)\in \Ban$ are guaranteed to be continuous by the above theorems, they are $C^1$ only when viewed as taking values in $E^*=H^{-1}(\R^d,\C)$. The following ``propagation of regularity'' theorem allows one to identify the appropriate domain $\Dom$ on which the stronger condition~\eqref{eq:diffeq} holds.

\begin{theorem}[\cite{cazenave2003}] \label{thm:nlsregularityH2Hm} 
Let $\nl$ be as in Theorem \ref{thm:nlslocalflowH1},
and consider $u_0\in H^1(\R^d,\C)$ and $u\in  C^0((-T_{\mathrm{min}}, T_{\mathrm{max}}),H^{1}(\R^d))$ 
the solution of the problem \eqref{eq:nlsgen} given by Theorem~\ref{thm:nlslocalflowH1}.
Then the following statements hold.
\begin{itemize}
\item[(i)] If $u_0\in H^2(\R^d,\C)$, then  $u\in  C^0((-T_{\mathrm{min}}, T_{\mathrm{max}}),H^{2}(\R^d))$. If, in addition, $\nl(x,\cdot)\in C^1 (\C,\C)$, then  $u$ depends continuously on $u_0$ in the following sense: if $u_0^k\to u_0$ in $H^2(\R^d,\C)$  and if $u_k$ is the maximal solution of \eqref{eq:nlsgen} with the initial value $u_0^k$, then $u_k\to u$ in $ C^0([-S,T],H^2(\R^d))$ for every interval $[-S,T]\subset (-T_{\mathrm{min}}, T_{\mathrm{max}})$.
\item[(ii)] If $u_0\in H^m(\R^d,\C)$ for some integer $m>\max\left\{\frac{d}{2},2\right\}$ and if $\nl(x,\cdot)\in C^m (\C,\C)$, then $u\in C^0((-T_{\mathrm{min}}, T_{\mathrm{max}}),H^{m}(\R^d))$. In addition, $u$ depends continuously on $u_0$ in the following sense: if $u_0^k\to u_0$ in $H^m(\R^d,\C)$  and if $u_k$ is the maximal solution of \eqref{eq:nlsgen} with the initial value $u_0^k$, then $u_k\to u$ in $L^{\infty}([-S,T],H^m(\R^d))$ for every interval $[-S,T]\subset (-T_{\mathrm{min}}, T_{\mathrm{max}})$.
\end{itemize}
\end{theorem}
Note that the derivatives of $f$ should be understood in the real sense here.

\begin{remark}\label{strongdiff.rem}
It follows from Theorem~\ref{thm:nlsregularityH2Hm} that, if we take $\Dom=H^m(\R^d,\C)$, 
with $m\geq 3$, then~\eqref{eq:diffeq} is satisfied, and so the flow
is differentiable as a map from $\real$ to $E=H^1(\R^d,\C)$.
\end{remark}

\begin{example}
	\label{ex:nlspowernonlin}
	A typical example of local nonlinearity which satisfies \eqref{eq:nlscond1nl}, \eqref{eq:nlscond2nl}, \eqref{eq:nlscond3nl} and \eqref{eq:nlsextnl} is the pure power nonlinearity
	\begin{equation}
	 	\label{eq:nlspowernonlin}
	 	f(u)=\lambda |u|^{\sigma-1} u
	 \end{equation} 
	with 
	\begin{equation}
		\label{eq:nlscondpowernl}
		\begin{aligned}
			& 1\le \sigma <+\infty&\mbox{for } d=1,\\
			& 1\le \sigma <1+\frac{4}{d-2}&\mbox{for } d\ge 2,
		\end{aligned}
	\end{equation}
	and $\lambda\in \R$. The standard ``cubic'' Schr\"odinger equation corresponds to $\sigma=3$, which is an allowed value of $\sigma$ only if $1\leq d\leq 3$. The Hamiltonian is then given by
	\begin{equation}
		\label{eq:nlshamiltonianpower}
		H(u)=\frac{1}{2}\int_{\R^d}|\nabla u|^2(x)\diff x-\frac{\lambda}{\sigma+1}\int_{\R^d}|u|^{\sigma+1}(x)\diff x.
	\end{equation}
	In this case, the nonlinear Schr\"odinger equation reads
	\begin{equation}
	\label{eq:nlspurepower}
	\left\{
	\begin{aligned}
		&i\partial_t u(t,x) +\Delta u(t,x) +\lambda |u|^{\sigma-1}(t, x) u(t,x)=0,\\
		&u(0,x)=u_0(x).
	\end{aligned}
	\right.
	\end{equation}
	Theorem \ref{thm:nlslocalflowH1} then ensures the existence of a local solution
	\begin{equation}
	u\in C^0((-T_{\mathrm{min}}, T_{\mathrm{max}}),H^{1}(\R^d))\cap C^1((-T_{\mathrm{min}}, T_{\mathrm{max}}),H^{-1}(\R^d))
	\end{equation}
	and the conservation of the Hamiltonian energy $H$. To guarantee the existence of a global flow, we have to distinguish the focusing ($\lambda>0$) and the defocusing case ($\lambda<0$). More precisely, Theorem~\ref{thm:nlsglobalflowH1} implies the flow is globally defined on $H^1(\R^d,\C)$, i.e.
	\begin{equation}\label{eq:schrodflow}
		\Phi^{X}: \R \times H^1(\R^d,\C)\to H^1(\R^d,\C), 
	\end{equation}
	if $\sigma$ satisfies \eqref{eq:nlscondpowernl} in the defocusing case or if $1\le \sigma<1+\frac{4}{d}$ in the focusing case. Note that, in the latter situation, $\sigma=3$ is allowed only if $d=1$. 

	Next, we recall that 
	\begin{align*}
		&\sigma\in \N,\ \sigma\ \mbox{odd}\ \Rightarrow f\in  C^\infty (\C,\C),\\
		&\sigma\in \N,\ \sigma\ \mbox{even}\ \Rightarrow (f\in  C^m (\C,\C)\Leftrightarrow m\le \sigma-1),\\
		&\sigma\notin \N\Rightarrow (f\in  C^m (\C,\C)\Leftrightarrow m\le [\sigma-1]+1),
	\end{align*}
	and, in particular, $f\in  C^1 (\C,\C)$ for all $\sigma\ge 1$.
Hence Theorem \ref{thm:nlsregularityH2Hm} applies and the flow can be restricted to $H^2(\R^d,\C)$
	\begin{equation*}
		\Phi^X: \R \times H^2(\R^d,\C)\to H^2(\R^d,\C), 
	\end{equation*}
	whenever $\sigma$ satisfies \eqref{eq:nlscondpowernl} in the defocusing case or $1\le \sigma<1+\frac{4}{d}$ in the focusing case. This, however, is not enough for our purposes, since it only guarantees the existence of the derivative of $t\to u(t)$ as a function in $L^2(\R^d,\C)$, and not as a function in $\Ban=H^1(\R^d,\C)$. In other words, we cannot take $\Dcal=H^2(\R^d,\C)$ if we wish to satisfy~\eqref{eq:diffeq}. To obtain sufficient propagation of regularity, having in mind Remark~\ref{strongdiff.rem}, we state the following results. 
	
	 In dimension $d=1$ both in the defocusing case, for $3\le \sigma <+\infty$, and in the focusing case, for $3\le \sigma<5$, 
		\begin{equation*}
		\Phi^X: \R \times H^3(\R,\C)\to H^3(\R,\C). 
	\end{equation*}
	Hence, in these cases, using the notation introduced in Section \ref{ss:dynamicsbanach}, $E=H^1(\R,\C)$ and the domain $\mathcal D$ of the vector field $X$ can be chosen to be the Sobolev space $H^3(\R,\C)$.  
	
	In dimension $d=2,3$ and in the defocusing case, the global flow $\Phi^X$ can be defined on $E=H^1(\R^d,\C)$ for all $3\le \sigma <1+\frac{4}{d-2}$. As before, the domain $\mathcal D$ of the vector field $X$ can be chosen to be the Sobolev space $H^3(\R^d,\C)$. 
	
	It follows in particular from what precedes that the cubic Schr\"odinger equation $(\sigma=3)$ fits in the framework of the previous section provided either $d=1$ (with $\lambda$ arbitrary) or $\lambda<0$ and $d=2,3$. 

We now turn to the study of the symmetries of the nonlinear Schr\"odinger equation~\eqref{eq:nlspurepower}. Let $G=\mathrm{SO}(d)\times \R^d \times \R$ and define its action on $E=H^1(\R^d,\C)$ via
		\begin{equation}\label{eq:nlsinvariance}
		\forall u\in H^1(\R^d),\quad \left(\Phi_{R,a,\gamma}(u)\right)(x)=e^{i\gamma} u(R^{-1}(x-a)).
	\end{equation}
	Here the group law of $G$ is
	\begin{equation*}
		(R_1,a_1,\gamma_1)(R_2,a_2,\gamma_2)=(R_1R_2,a_1+R_1a_2,\gamma_1+\gamma_2)
	\end{equation*}
	for all $R_1, R_2\in \mathrm{SO}(d)$, $a_1,a_2\in \R^d$ and $\gamma_1, \gamma_2\in \R$. We claim that $G$ is an invariance group (see Definition \ref{def:symgroup}) for the dynamics $\Phi_t^X$. Indeed, let $u(t,x)=(\Phi_t^X(u))(x)$ a solution to the nonlinear Schr\"odinger equation \eqref{eq:nlspurepower} and consider $\left((\Phi_{R,a,\gamma}\circ \Phi_t^X)(u)\right)(x)=e^{i\gamma}u(t,Rx-a)$. A straightforward calculation shows that $e^{i\gamma}u(t,R^{-1}(x-a))$ is again a solution to equation \eqref{eq:nlspurepower}. More precisely,
	\begin{align*}
		&i\partial_t (e^{i\gamma}u(t,R^{-1}(x-a)))+\Delta (e^{i\gamma}u(t,R^{-1}(x-a)))\\ 
		&+\lambda|e^{i\gamma}u(t,R^{-1}(x-a))|^{\sigma-1}(e^{i\gamma}u(t,R^{-1}(x-a)))\\
		&=e^{i\gamma}\left(i(\partial_t u)(t,R^{-1}(x-a))+(\Delta u)(t,R^{-1}(x-a)) +(\lambda|u|^{\sigma-1} u)(t,R^{-1}(x-a))\right)\\
		&=0
	\end{align*}
	where we use the fact that the Laplace operator is invariant under space rotations, space translations and phase rotations. As a consequence, $$\left((\Phi_{R,a,\gamma}\circ \Phi_t^X)(u)\right)(x)= \left((\Phi_t^X\circ \Phi_{R,a,\gamma})(u)\right)(x)$$ and $G$ is an invariance group for the dynamics $\Phi_t^X$. Moreover, we can easily prove that $H\circ \Phi_{R,a,\gamma}=H$. Indeed, using the definition of $H$ given in \eqref{eq:nlshamiltonianpower}, we have
	\begin{align*}
		H\circ \Phi_{R,a,\gamma}(u)&=\frac{1}{2}\int_{\R^d}|\nabla u|^2(R^{-1}(x-a))\diff x-\frac{\lambda}{\sigma+1}\int_{\R^d}|u|^{\sigma+1}(R^{-1}(x-a))\diff x\\
		&= H(u).
	\end{align*}
	We will see later (in Section \ref{ss:symconstmotion}) why this is important.


	Now, let us give some examples of $G$-relative equilibria of the nonlinear Schr\"odinger equation \eqref{eq:nlspurepower}. First, consider the simplest case where $d=1$ and $\sigma=3$. The invariance group $G$ reduces to $\R\times \R$ and the nonlinear Schr\"odinger equation
becomes
	\begin{equation}
		\label{eq:nlscubrelativeeq}
		i\partial_t u(t,x) +\partial_{xx}^2 u(t,x) +\lambda |u(t,x)|^2u(t,x)=0.
	\end{equation}
	In the focusing case ($\lambda>0$), there exists a two-parameters family of functions, the so-called {\em bright solitons}\index{solitons},
	\begin{equation*}
		u_{\alpha,c}(t,x)=\alpha \sqrt{\frac{2}{\lambda}}\sech(\alpha(x-ct))e^{i\left(\frac{c}{2} x+\left(\alpha^2-\frac{c^2}{4}\right)t\right)}
	\end{equation*}
	that are solutions to \eqref{eq:nlscubrelativeeq} for all $(\alpha,c)\in \R\times \R$, with initial conditions
		\begin{equation}\label{eq:soliton}
		u_{\alpha,c}(x)=u_{\alpha,c}(0,x)=\alpha \sqrt{\frac{2}{\lambda}}\sech(\alpha x)e^{i\left(\frac{c}{2} x\right)}\in E=H^1(\R).
	\end{equation}
	For each $(\alpha,c)\in \R\times \R$, $u_{\alpha,c}(x)$ is a $G$-relative equilibrium of \eqref{eq:nlscubrelativeeq}. Indeed, the $G$-orbit of $u_{\alpha,c}(x)$ is given by 
	\begin{equation}\label{eq:solitonorbit}
		\Ocal_{u_{\alpha,c}}=\left\{e^{i\gamma}u_{\alpha,c}(x-a), (a,\gamma)\in \R\times \R\right\}.
	\end{equation}
	Hence, it is clear that for all $t\in \R$, $u_{\alpha,c}(t,x)\in \Ocal_{u_{\alpha,c}}$ and, by Definition \ref{def:orbfixedpoint}, we can conclude that $u_{\alpha,c}(x)$ is a $G$-relative equilibrium of \eqref{eq:nlscubrelativeeq}.


	More generally, standing and travelling waves are examples of $G$-relative equilibria of the nonlinear Schr\"odinger equation \eqref{eq:nlspurepower}. More precisely, standing waves\index{standing waves} are solutions to \eqref{eq:nlspurepower} of the form
	\begin{equation}
		\label{eq:stationarystatesNLS}
		u_{\mathrm{S}}(t,x)=e^{i\en t}w_{\mathrm{S}}(x)
	\end{equation}
	with $\en\in \R$. For this to be the case, the profile $w_{\mathrm{S}}$ has to be a solution of the \emph{stationary equation} 
	\begin{equation*}
		\Delta w+\lambda |w|^{\sigma-1} w=\xi w.
	\end{equation*}
	Bright solitons with $c=0$ are examples of such standing waves, with $d=1, \sigma=3$. Standing waves of the one-dimensional Schr\"odinger equation with a spatially inhomogeneous nonlinearity, as well as their orbital stability, will be studied in Section~\ref{curves.sec}. Travelling waves are solutions to \eqref{eq:nlspurepower} of the form
	\begin{equation}
		\label{eq:travelingwavesNLS}
		u_{{\mathrm{TW}}}(t,x)=e^{i\xi t}w_{{\mathrm{TW}}}(x-ct)
	\end{equation}
	with $\xi \in \R$ and $c\in \R^d$. Now, the profile $w_{{\mathrm{TW}}}$ has to be a solution of
	\begin{equation*}
		\Delta w+\lambda |w|^{\sigma-1} w=\xi w+ i c\cdot \nabla w.
	\end{equation*}
	Bright solitons with $c\neq0$ are examples of such travelling waves, with $d=1, \sigma=3$.
	
	The $G$-orbit of the initial condition $w_{\mathrm{S}}(x)$ is given by 
	\begin{equation}
		\Ocal_{w_{\mathrm{S}}}=\left\{e^{i\gamma}w_{\mathrm{S}}(R^{-1}(x-a)), (R,a,\gamma)\in G\right\}
	\end{equation}
	and it is clear that $u_{\mathrm{S}}(t,x)\in \Ocal_{w_{\mathrm{S}}}$ for all $t\in \R$. The same holds true for $u_{\mathrm{TW}}$ with $w_{\mathrm{S}}$ replaced by $w_{\mathrm{TW}}$.
\end{example}




Another, closely related example of an infinite dimensional dynamical system is the cubic Schr\"odinger equation
\begin{equation}
	\label{eq:nlsper}
	\left\{
	\begin{aligned}
		&i\partial_t u(t,x) +\partial_{xx}^2 u(t,x) \pm|u(t,x)|^2 u(t,x)=0\\
		&u(0,x)=u_0(x)
	\end{aligned}
	\right.
\end{equation}
in the space periodic setting $\T=\R/(2\pi\Z)$ (the one dimensional torus). In \cite{bourgain1993}, the following theorem is proven.
\begin{theorem}[\cite{bourgain1993}]
	\label{thm:nlsperglobalflowHs} The Cauchy problem \eqref{eq:nlsper} is globally well-posed for data $u_0\in H^{s}(\T,\C)$, $s\ge 0$ and the solution $u\in  C^0(\R,H^s(\T))$. Moreover, if $u$, $v$ are the solutions corresponding to data $u_0,v_0\in H^s(\T,\C)$, there is the regularity estimate
	\begin{equation}
		\label{eq:nlsperglobalestimate}
		\|u(t)-v(t)\|_{H^s}\le C^{|t|}\|u_0-v_0\|_{H^s}
	\end{equation}
	where $C$ depends on the $L^2$-size of the data, i.e. $C=C(\|u_0\|_{L^2},\|v_0\|_{L^2})$.
\end{theorem}
This ensures the existence of a global flow 
	\begin{equation*}
		\Phi^X: \R \times H^s(\T,\C)\to H^s(\T,\C). 
	\end{equation*}
for all $s\ge 1$. Hence, we can choose $E=H^1(\T,\C)$ and $\Dcal=H^3(\T,\C)$ to ensure the conditions of Section~\ref{ss:dynamicsbanach} are satisfied.


As before, by using the invariance of Equation \eqref{eq:nlsper} under space translations and phase rotations, we can show that the dynamics defined by $\Phi^X_t$ are invariant under the action of the group $G= \R \times \R$ given by
\begin{equation}\label{eq:nlsinvarianceper}
	(\Phi_{a,\gamma}(u))(x)=e^{i\gamma} u(x-a).
\end{equation}
As an example of $G$-relative equilibria, we can consider the two-parameter family of plane waves\index{plane waves}
\begin{equation}
 	\label{eq:planewavesreleq}
 	u_{\alpha,k}(t,x)=\alpha e^{-ikx} e^{i\en t}
\end{equation} 
with $\alpha\in \R$ and $k\in \Z$ and $\xi=-k^2\pm |\alpha|^2$. The $G$-orbit of the initial condition $u_{\alpha,k}(x)=\alpha e^{-ikx}$ is given by 
	\begin{equation*}
		\Ocal_{u_{\alpha,k}}=\left\{\alpha e^{i\gamma}e^{-ik(x-a)}, (a,\gamma)\in G\right\}.
	\end{equation*}
	As before, it is clear that $u_{\alpha,k}(t,x)\in \Ocal_{u_{\alpha,k}}$ for all $t\in \R$. We will study the orbital stability of these relative equilibria in Section~\ref{s:nlsetorus1d}.

Remark that plane waves are the simplest elements of a family of solutions of the NLS equation of the form 
\begin{equation*}
u_{p,c}(t,x)=e^{i\xi t}e^{-ipx}U(x-ct), \ (t,x)\in\R\times\R
\end{equation*}
with $\xi,p,c\in\R$ and $U:\R\to\C$ a periodic function. This kind of solutions are called \emph{quasi-periodic travelling waves} and their orbital stability has been studied in \cite{galhar07b,galhar07a}.


\subsection{The Manakov equation}\label{ss:manakov}

The Manakov equation\index{equation!Manakov} \cite{manakov74,gazeau12} is a system of two coupled nonlinear Schr\"odinger equations which describe the evolution of nonlinear electric fields in optical fibers with birefringence, defined by 
\begin{equation}
	\label{eq:nlsmanakov}
	\left\{
	\begin{aligned}
		&i\partial_t u(t,x) +\Delta u(t,x) +\lambda|u(t,x)|^2 u(t,x)=0\\
		&u(0,x)=u_0(x)
	\end{aligned}
	\right.
\end{equation}
with $u(t,x)=\begin{pmatrix}u_1(t,x)\\u_2(t,x)\end{pmatrix}:\R\times\R\to \C^2$, $|u(t,x)|^2=(|u_1(t,x)|^2+|u_2(t,x)|^2)$ and $\lambda\in \RR$.

With the same arguments as those used for the nonlinear Schr\"odinger equation \eqref{eq:nlspurepower}, one can easily show that the flow is globally defined in $H^1(\R,\C^2)$, i.e.
\begin{equation}
	\Phi^{X}:\R\times H^{1}(\R,\C^2)\to H^1(\R,\C^2)
\end{equation}
both in the focusing ($\lambda>0$) and in the defocusing case ($\lambda<0$). Moreover, thanks to the propagation of regularity, the flow preserves $H^3(\R,\C^2)$ i.e.
\begin{equation}
	\Phi^{X}:\R\times H^{3}(\R,\C^2)\to H^3(\R,\C^2)
\end{equation}
as before. Hence, using the notation of Section \ref{ss:dynamicsbanach}, one can choose $E=H^1(\R,\C^2)$ and the domain $\mathcal D=H^3(\R,\C^2)$. 

Now, let $(a, S)\in G=\R \times \mathrm{U}(2)$ act on $E=H^1(\R,\C^2)$ via 
	\begin{equation}\label{eq:manakovinvariance}
		\Phi_{a,S}(u)=S u(x-a).
	\end{equation}
	Here the group law of $G$ is
	$
		(a_1,S_1)(a_2,S_2)=(a_1+a_2,S_1S_2)
	$
	for all $a_1,a_2\in \R^d$ and $S_1, S_2\in \mathrm{U}(2)$. 
	A straightforward calculation proves that $G$ is an invariance group for the dynamics $\Phi_t^X$. 

In the focusing case ($\lambda>0$), there exists a family of solitons,
	\begin{equation*}
		u_{\nu}(t,x)=\alpha \sqrt{\frac{2}{\lambda}}\sech(\alpha(x-ct))e^{i\left(\frac{c}{2} x+\left(\alpha^2-\frac{c^2}{4}\right)t\right)}\begin{pmatrix}\cos\theta e^{i\gamma_1}\\ \sin\theta e^{i\gamma_2} \end{pmatrix}
	\end{equation*}
	that are solutions to \eqref{eq:nlsmanakov} for all $\nu=(\alpha,c,\theta,\gamma_1,\gamma_2)\in \R^5$, with initial condition 
		\begin{equation*}
		u_{\nu}(x)=u_{\nu}(0,x)=\alpha \sqrt{\frac{2}{\lambda}}\sech(\alpha x)e^{i\left(\frac{c}{2} x\right)}\begin{pmatrix}\cos\theta e^{i\gamma_1}\\ \sin\theta e^{i\gamma_2} \end{pmatrix}\in E=H^1(\R,\C^2).
	\end{equation*}
	For each $\nu\in \R^5$, $u_{\nu}(x)$ is a $G$-relative equilibrium of \eqref{eq:nlsmanakov}. Indeed, the $G$-orbit of $u_{\nu}(x)$ is given by 
	\begin{equation*}
		\Ocal_{u_{\nu}}=\left\{Su_{\nu}(x-a), (a,S)\in \R\times \mathrm{U}(2)\right\}.
	\end{equation*}
	Hence, it is clear that for all $t\in \R$, $u_{\nu}(t,x)\in \Ocal_{u_{\nu}}$ and, by Definition \ref{def:orbfixedpoint}, we can conclude that $u_{\nu}(x)$ is a $G$-relative equilibrium of \eqref{eq:nlscubrelativeeq}.	


\subsection{The nonlinear wave equation}\label{ss:nlw}

Let us consider the nonlinear wave\index{equation!nonlinear wave} equation 

\begin{equation}
	\label{eq:nlwpurepower}
	\left\{
	\begin{aligned}
		&\partial^2_{tt} u(t,x) -\Delta u(t,x) +\lambda |u(t,x)|^{\sigma-1} u(t,x)=0\\
		&u(0,x)=u_0(x), \partial_t u(0,x)=u_1(x)
	\end{aligned}
	\right.
\end{equation}
with $u(t,x):\R\times\R^d\to \R$, and, for simplicity, let us take $d=1,2,3$. Moreover, we restrict our attention to the defocusing case, that in our notation corresponds to $\lambda>0$ (because of the minus sign in front of the Laplacian), and to the so-called algebraic nonlinearities, which means $\sigma\in \N$ is odd. As a consequence the function $f(u)=|u|^{\sigma-1} u$ is smooth.

Let $H$ defined by 
\begin{equation}\label{eq:nlwhamiltonian}
	H(u,\partial_t u)=\frac 12 \int_{\R^d} |\nabla u|^2\diff x+\frac 12 \int_{\R^d} |\partial_t u|^2\diff x+\frac{\lambda}{\sigma+1}\int_{\R^d}|u|^{\sigma+1}\diff x
\end{equation}
be the Hamiltonian of the system. As for the Schr\"odinger equation, we will explain in Section \ref{s:hamdyninfinite} how the nonlinear wave equation defines an infinite dimensional Hamiltonian dynamical system.

In the defocusing case and whenever $1\le\sigma<+\infty$ for $d=1$ or $1\le\sigma<1+\tfrac{4}{d-2}$ for $d=2,3$, we can define a global flow on $H^1(\R^d,\R)\times L^2(\R^d,\R)$, i.e. 
\begin{equation}
	\label{eq:flownlw}
	\begin{aligned}
		\Phi^X:\R\times (H^1(\R^d,\R)\times L^2(\R^d,\R))&\to H^1(\R^d,\R)\times L^2(\R^d,\R)\\
		 (t,u(0),\partial_t u(0))&\to (t,u(t),\partial_t u(t))
	\end{aligned}
\end{equation}
with $u\in C(\R,H^1(\R^d))\cap C^1(\R,L^2(\R^d))$ the unique solution to \eqref{eq:nlwpurepower}. Moreover the Hamiltonian  energy \eqref{eq:nlwhamiltonian} is conserved along the flow, i.e.
\begin{equation*}
	H(u(0),\partial_t u(0))=H(u(t),\partial_t u(t))
\end{equation*}
for all $t\in \R$ (see \cite{tao_book} and references therein). Furthermore, it follows from the integral form of \eqref{eq:nlwpurepower} (see \cite[Ex.~2.18 and 2.22]{tao_book}) 
that $u\in C^2(\R,H^{-1}(\R^d))$.

In the algebraic case, thanks to the persistence of regularity, the flow can be restricted to $H^s(\R^d,\R)\times H^{s-1}(\R^d,\R)$,
\begin{equation*}
		\Phi^X:\R\times (H^s(\R^d,\R)\times H^{s-1}(\R^d,\R))\to H^s(\R^d,\R)\times H^{s-1}(\R^d,\R)
\end{equation*}
for all $s>\tfrac{d}{2}$. Hence, using the notation of Section \ref{ss:dynamicsbanach}, $E=H^1(\R^d,\R)\times L^2(\R^d,\R)$ and the domain $\mathcal D$ of the vector field $X$ can be chosen to be the Sobolev space $H^2(\R^d,\R)\times H^1(\R^d,\R)$. 

As for the nonlinear Schr\"odinger equation, by using the invariance of Equation \eqref{eq:nlwpurepower} under space rotations,  space translations and phase rotations, we can show that the dynamics defined by $\Phi^X_t$ are invariant under the action of the group $G=\mathrm{SO}(d)\times \R^d$ on $E=H^1(\R^d,\R)\times L^2(\R^d,\R)$ defined by 
\begin{equation*}
	\Phi_{R,a}(u,\partial_t u)=(u(R^{-1}(x-a)),\partial_t u(R^{-1}(x-a))).
\end{equation*}
Moreover, $H\circ \Phi_{R,a,\gamma}=H$ and we will explain in Section \ref{ss:symconstmotion} the consequences of this fact.

\subsection{Generalized symmetries}\label{ss:gensym}

The nonlinear Schr\"odinger equation is often said to be invariant under Galilei transformations. This invariance is however of a slightly different nature than the one defined in Definition~\ref{def:symgroup}, as we now explain\footnote{We will, in this section, make free use of the material of Appendices~\ref{s:liegroups} and~\ref{s:hammech}.}. 

Recall that Newtonian mechanics is known to be invariant under coordinate changes between inertial frames. These include space and time translations, rotations, and changes to a moving frame, often referred to as Galilei boosts. All together, they form a group, the Galilei group $G_{\mathrm{Gal}}$, which is a Lie group that can be defined formally as
$$
\Gal=\mathrm{SO}(d)\times\R^d\times\R^d\times\R
$$
with composition law
$$
(R', v', a', t')(R, v, a, t)=(R'R, R'v+v', R'a+a'+v't, t+t').
$$
It acts naturally on space-time $(x,t)\in\R^d\times\R$ as follows:
$$
(R', v', a', t')(x,t)=(R'x+a'+v't, t'+t).
$$
Of course, the physical case corresponds to $d=3$. 

The statement that Newton's equations are invariant under boosts means for example that, if $t\to(q_1(t), q_2(t))$ is the solution of Newton's equations of motion for two particles moving in a spherically symmetric interaction potential $V$ 
$$
m_1\ddot q_1(t)=-\nabla_{q_1}V(\|q_1(t)-q_2(t)\|), \quad m_2\ddot q_2(t)=-\nabla_{q_2}V(\|q_1(t)-q_2(t)\|),
$$
with initial conditions
$$
q_1(0)=q_1,\ q_2(0)=q_2,\ \dot q_1(0)=\frac{p_1}{m_1},\ \dot q_2(0)=\frac{p_2}{m_2},
$$
then, for all $v\in\R^3$, $t\to (q_1(t)+vt, q_2(t)+vt)$ is also such a solution, with initial conditions
$$
q_1(0)=q_1,\ q_2(0)=q_2,\ \dot q_1(0)=\frac{p_1}{m_1}+v,\ \dot q_2(0)=\frac{p_2}{m_2}+v.
$$
In a Hamiltonian description\footnote{See Appendix~\ref{s:hammech}.}, the above equations of motion are associated to the Hamiltonian
$$
H(q,p)=\frac{p_1^2}{2m_1}+\frac{p_2^2}{2m_2}+V(\|q_1-q_2\|),
$$
which generates a flow $\Phi_t^H$ that is clearly invariant under space translations and rotations. The situation for Galilei boosts, however,  is different. Indeed, in this context they act on the phase space $E=\R^6\times \R^6$ with symplectic transformations, as follows:
$$
\forall v\in\R^3, \quad \Phi^K_v(q,p)=(q, p_1-m_1v, p_2-m_2v).
$$
Here $K=m_1q_1+m_2 q_2$ and $\Phi_v^K$ is a shorthand notation for 
$$
\Phi_v^K=\Phi_{v_1}^{K_1}\circ\Phi_{v_2}^{K_2}\circ\dots\circ\Phi_{v_n}^{K_n,},
$$
where each $\Phi_{v_i}^{K_i}$ is the hamiltonian flow of one component of $K$. 
But those do NOT commute with the dynamical flow $\Phi_t^H$. Indeed, one easily checks that
\begin{equation}\label{eq:galileiinv}
\Phi^K_{v}\Phi_t^H\Phi^K_{-v}=\Phi^P_{vt}\Phi_t^H,
\end{equation}
where $P=p_1+p_2$ is the total momentum of the system, which generates translations: 
$\Phi_a^P(q_1, q_2, p_1, p_2)=(q_1+a, q_2+a, p_1, p_2)$. 
In that sense, the three dimensional commutative group of Galilei boosts is NOT an invariance group for the dynamical system according to Definition~\ref{def:symgroup}. To remedy this situation, one can proceed as follows. Define, on $\Ban=\R^6\times\R^6$, for each $g=(R,v, a, t) \in \Gal$, the symplectic transformation
$$
\Phi_g=\Phi^P_a\Phi_t^H\Phi^K_v\Phi_R,
$$
where $\Phi_R$ is defined as in~\eqref{eq:so3}. It is then easily checked using~\eqref{eq:galileiinv} that the $\Phi_g$ define an action of $\Gal$ on $\Ban$. It is clearly globally Hamiltonian (Definition~\ref{def:globhamaction})\footnote{It is however not Ad$^*$-equivariant.}. It follows that the Galilei boosts are generalized symmetries for the dynamical system $\Phi_t^H$, in the following sense:
\begin{definition}\label{def:generalizedsym} Let $G$ be  a Lie group, and $\Phi$ an action of $G$ on a Banach space $\Ban$. Let $\Phi_t^X$ a dynamical system on $\Ban$. We say $G$ is a generalized symmetry group\index{generalized symmetry group} for $\Phi_t^X$ provided there exists $\xi\in \frak g$ so that 
$\Phi_t^X=\Phi_{\exp(t\xi)}$.
\end{definition} 
For our purposes, an important difference between symmetries and generalized symmetries in Hamiltonian systems is that the latter do NOT give rise to constants of the motion. To illustrate this, remark that, although the Galilei boosts are generated by $K(q,p)=m_1q_1+m_2q_2$, it is clear that $K$ is not a constant of the motion of $H$:
\begin{equation}\label{eq:centermass}
\{K, H\}=P,
\end{equation}
where $P=p_1+p_2$ is the total momentum of the two-particle system. This is not a surprise: $K=MR$, where $R$ is the center of mass of the two-particle system and $M=m_1+m_2$ its mass. And of course, the center of mass moves: in fact,~\eqref{eq:centermass} implies it moves at constant velocity. 

A similar situation occurs with the nonlinear Schr\"odinger equation. If $u(t,x)$ is  a solution of~\eqref{eq:nlsgen} with a power law nonlinearity, then so is, for every $v\in\R^d$,
\begin{equation}\label{eq:boost}
\tilde u(t,x)=\exp\big(-{\txt\frac{i}{2} (v\cdot x}+\txt\frac{v^2}{2}t)\big)\, u(t, x+vt),
\end{equation}
as is readily checked. The function $\tilde u$ can be interpreted as the wave function in the moving frame, as can be seen from the shift $x\to x+vt$ in position and from the factor $\exp(-i\frac{v}{2}\cdot x)$, which corresponds to a translation by $\frac12 v$ in momentum, in the usual quantum mechanical interpretation of the Schr\"odinger equation. Adopting the framework of Section~\ref{s:dynsysexamples}, one observes that the maps
$$
 \widehat\Psi_v\,u(x)=\exp\left(-{\txt\frac{i}{2} (v\cdot x})\right)\, u( x),
 $$
 defined for all $v\in\R^d$ on $\Ban=H^1(\R^d)$ are not symmetries for the Schr\"odinger flow $\Phi^{X}$ defined in~\eqref{eq:schrodflow} but that
\begin{equation}\label{eq:boostinvariance}
\widehat\Psi_{v}\Phi_t^{X}\widehat\Psi_{-v}=\Phi_{I, vt, -\tfrac{v^2}{4}t}\Phi_t^{X},
\end{equation}
where $\Phi_{I, vt,  -\tfrac{v^2}{4} t}$ is defined in~\eqref{eq:nlsinvariance}. This commutation relation is very similar to~\eqref{eq:galileiinv}, except for the extra phase $\exp(-i\tfrac{v^2}{4}t)$. We note in passing that  the boosts $\hat\Psi_v$ are unitary on $L^2$, but do not preserve the $H^1$ norm. They are nevertheless bounded operators on $\Ban=H^1(\R^d)$.

As in classical mechanics, one can put together the above transformations with the representation of the Euclidean group in~\eqref{eq:nlsinvariance} to form a (projective) representation of the Galilei group showing that the Galilei boosts are generalized symmetries of the nonlinear Schr\"odinger equation with a power law nonlinearity. We will not work this out in detail here, but note for further use that
\begin{equation}\label{eq:projective}
\Phi_{R, a, \gamma} \widehat\Psi_v=\exp(i\tfrac{v\cdot a}{2})\widehat \Psi_{Rv}\Phi_{R, a, \gamma}.
\end{equation}
In particular $  \Phi_{I, a, 0} \widehat\Psi_v=\exp(i\tfrac{v\cdot a}{2})\widehat \Psi_{v}\Phi_{I, a, 0}$ so that, in this setting, the boosts $\hat \Psi_v$  commute with translations only ``up to a global phase'' $\exp(i\tfrac{v\cdot a}{2})$, in the usual terminology of quantum mechanics. In contrast, in classical mechanics, $\Phi^K_v$ and $\Phi_a^P$ clearly commute. 

Generalized symmetries do not provide constants of the motion via Noether's Theorem, and hence cannot quite play the same role as symmetries in the study of relative equilibria. We will now show how one may nevertheless use~\eqref{eq:boostinvariance} in the analysis of the stability of the relative equilibria of the (non)linear Schr\"odinger equation. 

 We first remark that the $u_{\alpha, c}$, defined in~\eqref{eq:soliton}, satisfy $u_{\alpha, c}=\widehat\Psi_{-c}u_{\alpha, 0}$. We will show that, thanks to~\eqref{eq:boostinvariance}, if $u_{\alpha, 0}$ is orbitally stable, then so is $u_{\alpha, c}$, for any $c\in\R$. We will only sketch the argument, leaving the details to the reader. Note first that $u_{\alpha, 0}$ is orbitally stable, if and only if, for all $\epsilon>0$, there exists $\delta>0$ so that, for all $w\in \Ban$ with $\rd(w, u_{\alpha, 0})\leq \delta$, there exists, for all $t\in\R$, $a(t)\in\R, \gamma(t)\in\R$ so that
$$
\|\Delta_t\|\leq \epsilon,\quad\mathrm{where}\quad \Delta_t:=\Phi_t^Xw-\Phi_{I, a(t), \gamma(t)}u_{\alpha,0}.
$$
Now suppose $u\in\Ban$ is sufficiently close to $u_{\alpha, c}$, for some $c\in\R$. Then, since $\widehat\Psi_c$ is a bounded operator, $\widehat\Psi_{c}u=w$ is close to $u_{\alpha, 0}$. Then, using~\eqref{eq:boostinvariance} and~\eqref{eq:projective}, one finds
\begin{align*}
\Phi_t^Xu&=\Phi_t^X\widehat\Psi_{-c}w\\
&=\widehat\Psi_{-c}\Phi_{I,ct,-\frac{c^2}{4}t}\Phi_t^Xw\\
&=\widehat\Psi_{-c}\Phi_{I,ct,-\frac{c^2}{4}t}\Phi_{I, a(t), \gamma(t)}u_{\alpha,0}+\widehat\Psi_{-c}\Phi_{I,ct,-\frac{c^2}{4}t}\Delta_t\\
&=\widehat\Psi_{-c}\Phi_{I,ct+a(t),-\frac{c^2}{4}t+\gamma(t)}u_{\alpha, 0}+\widehat\Psi_{-c}\Phi_{I,ct,-\frac{c^2}{4}t}\Delta_t\\
&=\Phi_{I,ct+a(t),-\tfrac{c^2}{4}t+\gamma(t)+\tfrac{c(ct+a(t))} {2}}\widehat\Psi_{-c}u_{\alpha, 0}+\widehat\Psi_{-c}\Phi_{I,ct,-\frac{c^2}{4}t}\Delta_t.
\end{align*}
Since $u_{\alpha, c}=\widehat\Psi_{-c}u_{\alpha, 0}$, and since $\widehat \Psi_{-c}$ is bounded, it is now clear that $\Phi_t^Xu$ is at all times close to $\Ocal_{u_{\alpha, c}}$, defined in~\eqref{eq:solitonorbit}. 

The above argument shows, more generally, that the relative equilibria of the homogeneous NLS for $G=\mathrm{SO}(d)\times\R^d\times\R$ (see~\eqref{eq:nlsinvariance}) come in families $\widehat \Psi_{-c}u_0=u_{c}$, indexed by $c\in\R^d$. Moreover, if $u_0$ is spherically symmetric and orbitally stable, then all $u_c$ are orbitally stable.

\section{Orbital stability: a general definition}\label{s:orbstab}
We can now formulate the general definition of orbital stability that we shall study. In fact, several definitions appear naturally:
\begin{definition} \label{def:orbstabgen}  Let $\Phi_t^X$ be a dynamical system on a Banach space $\Ban$ and let $G$ be a symmetry group for $\Phi_t^X$. 
\begin{enumerate}[label=({\roman*})]
\item Let $u\in \Ban$ and let $\Ocal_u$ be the corresponding $G$-orbit . We say $u\in \Ban$ is orbitally stable\index{orbitally stable} if 
$$
\forall \epsilon>0, \exists \delta>0, \forall v\in \Ban,\  \left(\rd(v,u)\leq \delta\Rightarrow \forall t\in\R, \ \inf_{t'\in\R}\rd(v(t), \Ocal_{u(t')})\leq \epsilon\right).
$$
\item Let $\Ocal$ be a $G$-orbit in $\Ban$. We say $\Ocal$ is stable if each $u\in \Ocal$ is orbitally stable in the sense of (i) above. 
\item Let $\Ocal$ be a $G$-orbit in $\Ban$. We say $\Ocal$ is uniformly stable if it is stable and $\delta$ in (i) does not depend on $u\in\Ocal$. In other words, if $\forall \epsilon>0$, there exists $\delta>0$ so that, $\forall u\in\Ocal$, $\forall v\in \Ban$, 
\begin{equation}\label{eq:orbstabunif}
\rd(v,u)\leq \delta\Rightarrow \forall t\in\R, \ \inf_{t'\in\R}\rd(v(t), \Ocal_{u(t')})\leq \epsilon.
\end{equation}
\item We say $\Ocal\in \Ban_G$ is  Hausdorff orbitally stable if
$\Ocal$ satisfies: $\forall \epsilon>0$, there exists $\delta >0$ so that, $\forall \Ocal'\in E_G$
\begin{equation}
\Delta(\Ocal, \Ocal')\leq \delta\Rightarrow \forall t\in\R,\
\inf_{t'}\Delta(\Ocal'(t), \Ocal(t') )\leq \epsilon.
\end{equation}
\end{enumerate}
\end{definition}
The four definitions are subtly different. 

Definition (i) requires that the dynamical orbit issued from the nearby initial condition $v$ remains close to the orbit $\{\Phi_{t'}^X\Phi_g(u)\mid  t'\in\R, g\in G\}$ of the larger group $\R\times G$. It is therefore a generalization of definition~\eqref{eq:orbstab0}, which corresponds to the case $G=\{e\}$. This notion of orbital stability therefore depends on the choice of the group $G$ and it is clear that, the larger $G$, the weaker it is. As we will see in the examples of Section~\ref{s:spherpotstab} and Section~\ref{ss:hampde}, there are cases where definition~\eqref{eq:orbstab0} is not satisfied for some $u\in E$, but where the above definition holds for a suitable choice of $G$. As we will also see, the choice of $G$ may depend on the point $u\in E$ considered and it is in particular not always necessary to use the largest symmetry group $G$ available for $\Phi_t^X$ to obtain orbital stability. 

The stability of the orbit $\Ocal$ as defined in part (ii) simply requires the orbital stability of each point $u\in\Ocal$, as defined in (i). Note that $\delta$ depends on $u$ here. In part~(iii) of the definition, uniformity is required. 

Part (iv) requires that if two $G$-orbits $\Ocal, \Ocal'\subset \Ban$ are initially close (in the sense of the Hausdorff metric) then, for all $t$, $\Ocal'(t)$ is close to $\Ocal(t')$ for some value of $t'$. It is the natural transcription of the definition of orbital stability in~\eqref{eq:orbstab0} from the original dynamical system on $E$  to the reduced dynamics on $\Ban_G$.  

Parts (i), (ii) and (iii) are the most telling/interesting, since they give a statement directly on the phase space $E$, using the original distance $\rd$, rather than in the more abstract quotient space $\Ban_G$. They do moreover not use the somewhat unpleasant Hausdorff metric. In applications, one really wants to prove (i), (ii) or~(iii). 

As shown in the lemma below,  the four definitions in Definition~\ref{def:orbstabgen} are equivalent when the group action is isometric. For many applications in infinite dimensional systems in particular, this is the case.

\begin{lemma} \label{lem:isometric} Let $\Phi_t^X$ be a dynamical system on $E$ and let $G$ be a symmetry group for $\Phi_t^X$, acting isometrically. Let $u\in E$. Then the following statements are equivalent.
\begin{enumerate}[label=({\roman*})]
\item $u\in E$ is orbitally stable. 
\item Each $v\in \Ocal_u$ is orbitally stable. 
\item $\Ocal_u$ is uniformly stable.
\item $\Ocal_u$ is Hausdorff orbitally stable.
\end{enumerate}
\end{lemma}
In practice, one often proves (i) for a suitably chosen $u$ on the orbit. This then automatically yields (iii). The statement in terms of the reduced dynamics in (iv) is intellectually satisfying but rarely encountered, it seems. 
\begin{proof}

We prove $(i)\Leftrightarrow(ii)$ and $(i)\Rightarrow(iii)\Rightarrow(iv)\Rightarrow(i)$.

$(i)\Rightarrow(iii)$ and $(i)\Rightarrow(ii)$: Let $v\in\Ocal_u$ and $v'\in E$, $\rd(v', v)\leq \delta$. Then there exists $g\in G$ so that $v=\Phi_g(u)$. Define $u'=\Phi_g^{-1}(v')$. Then, by the isometry of $\Phi_g$, $\rd(u',u)\leq \delta$ and hence, by hypothesis, for all $t$, there exists $t'$ so that $\rd(u'(t), \Ocal_{u(t')})<\epsilon.$ Hence 
$$
\rd(v'(t), \Ocal_{v(t')})=\rd(\Phi_g(u'(t)), \Ocal_{u(t')})=\rd(u'(t),\Ocal_{u(t')})\leq \epsilon.
$$
This proves $(iii)$ and, in particular, $(ii)$. Since it is clear that $(ii)\Rightarrow(i)$, we obtain  $(i)\Leftrightarrow(ii)$.

$(iii)\Rightarrow(iv)$: Suppose $\Ocal_u$ is uniformly stable. Let $\Ocal'$ be such that $\Delta(\Ocal_u, \Ocal')<\delta$. Let $u'\in \Ocal'$ with $\rd(u,u')\leq \delta$. Then~\eqref{eq:orbstabunif}, together with Proposition~\ref{prop:isometry} $(ii)$, imply $\Delta(\Ocal'(t),\Ocal_{u(t')})\leq \epsilon$.

$(iv)\Rightarrow(i)$: Suppose $\Ocal_u$ is orbitally stable. Let $u'\in E$ so that $\rd(u,u')\leq \delta$. Let $\Ocal'=\Ocal_{u'}$. Then, by Proposition~\ref{prop:isometry}~$(iii)$, $\Delta(\Ocal,\Ocal')\leq\delta$. Hence, for all $t$, $\inf_{t'}\Delta(\Ocal'(t), \Ocal(t'))\leq\epsilon$. Proposition~\ref{prop:isometry}~$(ii)$ then implies~$(i)$. 
\end{proof}
In many applications, especially in infinite dimensional problems, the $\Phi_g$ are both linear and norm-preserving: several examples were given in Section~\ref{s:dynamics}. In that case the action is of course isometric. In addition, all group orbits are then bounded. Note nevertheless that, if the $\Phi_g$ are norm-preserving, but not linear, the action is no longer isometric, while the group orbits are still bounded.  Finally, isometric actions may have unbounded group orbits: think for example of translations on $\Ban=\R^{2n}$.

\section{Orbital stability in spherical potentials}\label{s:spherpotstab}
Before presenting the general Lyapunov approach to the proof of orbital stability in Section~\ref{s:orbstabproof}, we show here the orbital stability of the relative equilibria in spherical potentials that we identified in Section~\ref{s:spherpot}. This simple example is instructive for several reasons. First, it permits one  to appreciate  the group theoretic and symplectic mechanisms  underlying the construction of a suitable candidate Lyapunov function. Second, it nicely illustrates the various methods available to use this Lyapunov function in order to prove orbital stability via an appropriate ``coercivity estimate'' generalizing~\eqref{eq:localmin}. We will present three such methods below. 
\subsection{Fixed points}\label{s:fixedpoints}
The proof of the uniform orbital stability of $\Ocal_{\rho_*,0,0}$ in~\eqref{eq:fixorbit} is straightforward, and can be done with $H$ itself as the Lyapunov function, in close analogy with the proof sketched in the introduction. 
\begin{proposition} \label{lem:sphericalfixedpoints} Let $V\in C^2(\R^3)$ be a spherical potential and $H(u)=\frac12 p^2+V(q)$ the corresponding Hamiltonian. Let $\rho_*>0$ with $V'(\rho_*)=0$, $V''(\rho_*)>0$. Let $\Ocal_{\rho_*,0,0}=\{(q,p)\in\R^6\mid \|q\|=\rho_*, p=0\}$ be the corresponding SO$(3)$ orbit. Then $\Ocal_{\rho_*, 0,0}$ is uniformly orbitally stable.
\end{proposition}
This result is intuitively clear. Under the assumptions stated, the Hamiltonian reaches a local minimum at each of the fixed points of the dynamics that make up the sphere $\Ocal_{\rho_*, 0,0}$, and it increases quadratically in directions perpendicular to that sphere. Any nearby initial condition must therefore give rise to an orbit that stays close to the sphere: the potential acts locally as a potential well trapping the particle close to $\Ocal_{\rho_*, 0,0}$. 
\begin{proof}
We know from Section~\ref{s:spherpot} that the Hamiltonian $H$ in~\eqref{eq:spherpotham} is an SO$(3)$-invariant constant of the motion, and that  $D_uH=0$ for all $u\in \Ocal_{\rho_*,0,0}$, so that each such point is  a fixed point of the dynamics. We will write
$
H_*=H(u), \forall u\in\Ocal_{\rho_*,0,0}.
$
Moreover, for all $u=(q,0)\in\Ocal_{\rho_*,0,0}$
$$
D^2_{u}  H=
\begin{pmatrix}
V''(\rho_*){\hat {q}}_i {\hat {q}}_j&0\\
0& \mathrm{I}_3\\
\end{pmatrix}.
$$
Note that the Hessian is not positive definite. In fact, it vanishes on $w=(a, 0)$, for $a\cdot q=0$, which is the two-dimensional tangent space $T_u\Ocal_{\rho_*,0,0}$ to the orbit.  We can therefore not expect to obtain a coercive estimate as in~\eqref{eq:localmin}. On the other hand,  since $V''(\rho_*)>0$,  $D^2_uH$ \emph{is} positive definite on the four-dimensional orthogonal complement to the tangent space, given by
\begin{equation}
\left(T_u\Ocal_{\rho_*,0,0}\right)^\perp=\{(\alpha \hat q, b)\in\R^6\mid \alpha\in\R, b\in\R^3\}.
\end{equation}
As a result, we can still show that there exist constants $c_*,\eta_*>0$ with the property that
\begin{equation}\label{eq:spherpotcoercive1}
\forall u'\in\Ban,  \left(\rd(u',\Ocal_{\rho_*,0,0})\leq \eta_*\Rightarrow  H(u')-H_*\geq c_*\rd(u',\Ocal_{\rho_*,0,0})^2\right),
\end{equation}
and this will suffice for the proof of orbital stability. To show~\eqref{eq:spherpotcoercive1}, note first that setting $u'=(q',p')$ and taking $\eta_*<\rho_*/2$, one has $q'\not=0$. Consider then $u=(\rho_*\hat{q'}, 0)\in\Ocal_{\rho_*,0,0}$ and remark that 
$
\rd(u',\Ocal_{\rho_*,0,0})=\|u'-u\|.
$
Now compute
\begin{eqnarray*}
H(u')-H_*&=&H(u')-H(u)=D^2_vH(u'-u, u'-u)+\mathrm{o}(\|u'-u\|^2)\\
&\geq&\min\{1, V''(\rho_*)\}\rd(u',\Ocal_{\rho_*,0,0})^2+\mathrm{o}(\rd(u',\Ocal_{\rho_*,0,0})^2).
\end{eqnarray*}
One can then conclude~\eqref{eq:spherpotcoercive1} holds by using that the term in $\mathrm{o}(\rd(u',\Ocal_{\rho_*,0,0})^2)$ is uniformly small in $u\in \Ocal_{\rho_*,0,0}$ since $H$ is SO$(3)$-invariant. 
We now prove that $\Ocal_{\rho_*,0,0}$ is uniformly orbitally stable. Since the action of SO$(3)$ is isometric, Lemma~\ref{lem:isometric} shows it is enough to prove all $u\in \Ocal_{\rho_*,0,0}$ are orbitally stable. Suppose that this is not true. Then there exists $u\in\Ocal_{\rho_*,0,0}$ and $\epsilon>0$, and for each $n\in\N_*$, $u'_n\in\Ban$, $t_n\in\R$ so that $\rd(u'_n,u)\leq \frac1n$ and $\rd(u'_n(t_n), \Ocal_{\rho_*, 0,0})=\epsilon_0$.  Since we can choose $\epsilon<\eta_*$, we can apply~\eqref{eq:spherpotcoercive1} to write
\begin{eqnarray*}
H(u'_n)-H(u)=H(u'_n(t_n))-H_*\geq c_*\rd(u'_n(t_n), \Ocal_{\rho_*,0,0})^2=c_*\epsilon^2.
\end{eqnarray*}
Taking $n\to+\infty$ leads to the desired contradiction. 
\end{proof}


\subsection{Circular orbits}\label{s:circular}
Proving an appropriate notion of  stability for the initial conditions in~\eqref{eq:circincond} giving rise to circular orbits of the dynamics turns out to be slightly less straightforward. Intuitively, as explained already in the introduction, one expects that, under a suitable condition on the potential, an initial condition close to a circular orbit will generate a dynamical orbit that stays close to this orbit. As a result, orbital stability is satisfied in the sense of~\eqref{eq:orbstab0}. The following proposition gives a precise statement of this phenomenon. 
\begin{proposition} \label{lem:sphercircular} Let $V\in C^2(\R^3)$ be a spherical potential and $H(u)=\frac12 p^2+V(q)$ the corresponding Hamiltonian. Let $\rho_*, \sigma_*>0$ with $V'(\rho_*)\rho_*=\sigma_*^2$. Consider $u_{\mu_*}=(q_*,p_*)=(\rho_*\hat q_*, \sigma_*\hat p_*)$, with $\hat q_*\cdot\hat p_*=0$. Then $u_{\mu_*}$ is a relative equilibrium for the group SO$(2)$ of rotations about $\mu_*=q_*\wedge p_*$. If in addition, 
\begin{equation}\label{eq:circstable}
V''(\rho_*)>-3\sigma_*^2\rho_*^{-2},
\end{equation}
 $u_{\mu_*}$ is orbitally stable in the sense of definition~\eqref{eq:orbstab0} and of Definition~\ref{def:orbstabgen}~(i). In addition, $u_{\mu_*}$ is a local minimum of $H_{\mu_*}$, the restriction of $H$ to the level surface $\Sigma_{\mu_*}$, defined in~\eqref{eq:sphericalsigmamu}.
\end{proposition}
Note that the two definitions of orbital stability mentioned coincide in this particular case. Also, since the action of the rotation group is isometric, the result implies uniform orbital stability as well. Below, we will give three different arguments to prove the proposition, each of which can and has been used to treat various infinite dimensional problems. 

The origin of the condition $V''(\rho_*)>-3\sigma_*^2\rho_*^{-2}$ can be understood as follows. In standard mechanics textbooks such as~\cite{gold}, motion in a spherical potential is treated by fixing the angular momentum $q\wedge p=\mu_*$, and then using for $q,p$ polar coordinates $(r,\theta, p_r, p_\theta)$ in the plane perpendicular to the angular momentum. The Hamiltonian then reads, in these coordinates,
$$
H(r,\theta, p_r, p_\theta)=\frac{p_r^2}{2}+\frac{p_\theta^2}{2r^2}+V(r).
$$
The equation of motions are 
$$
\dot r=p_r,\quad \dot\theta=\frac{p_\theta}{r^2}\quad \dot p_r=\frac{p_\theta^2}{r^3}-V'(r),\quad \dot p_\theta=0
$$
and $|\mu_*|=p_\theta$.
It follows that the radial motion is decoupled from the angular one, since $\ddot r=-V'_{\mu_*}(r)$ with $V_{\mu_*}(r)=V(r)+\frac{\mu_*^2}{2r^2}$. It is then clear that the circular orbits correspond to the critical points $r=\rho_*$ of the effective potential $V_{\mu_*}$ which are fixed points of the radial dynamics. By an argument as in the introduction, those are stable if the critical point is a local minimum of  
$$
H_{\mu_*}(r, p_r)=\frac{p_r^2}{2}+V_{\mu_*}(r),
$$
and so in particular if $V_{\mu_*}''(\rho_*)>0$, which is precisely condition~\eqref{eq:circstable}. Note however that the preceding argument does not prove orbital stability of the circular orbits: it does not allow to consider initial conditions $u\in\R^6$ with $\mu\not=\mu_*$. This is actually the tricky part of the proof of the proposition. 
\begin{proof}
To mimic the previous proof, we would like to find a constant of the motion $\Lcal$ which is SO$(2)$ invariant and so that $D\Lcal$ vanishes on the orbit under consideration. We cannot use $H$ for this, since clearly $D_{u_{\mu_*}}H\not=0$, as we are not dealing with a fixed point of the dynamics. On the other hand, as we pointed out after the definition of relative equilibrium, when $u_{\mu_*}$ is a relative equilibrium, then there exists an element $\xi$ of the Lie-algebra of the invariance group so that $X_H(u_{\mu_*})=X_\xi(u_{\mu_*})$ or, equivalently, so that
\begin{equation*}
D_{u_{\mu_*}}(H-F_\xi)=0. 
\end{equation*}
In the present case, $F_\xi$ is defined in~\eqref{eq:so3mommap}, the invariance group is a one-dimensional rotation group and the statement becomes: there exists $\eta\in\R$ so that
\begin{equation}\label{eq:Lagrangecircular}
D_{u_{\mu_*}}(H-\eta\mu_*\cdot L)=0,
\end{equation}
since, as we saw in Section~\ref{s:spherpot}, $\mu_*\cdot L$ generates rotations about the $\mu_*$-axis. So here $\xi=\eta\mu_*$. Since, for all $u\in\R^6$
$$
D_uH =(V'(\|q\|)\hat q, p),\qquad D_u(\mu_*\cdot L)=(p\wedge \mu_*, \mu_*\wedge q),
$$
one easily checks that~\eqref{eq:Lagrangecircular} is satisfied iff $\eta=\rho^{-2}_*$. This suggests to define
$$
\Lcal(u)=H(u)-\rho_*^{-2}\mu_*\cdot L(u)
$$
and to try using it as a Lyapunov function. 
$\Lcal$ is often referred to as the ``augmented Hamiltonian''. Note that the theory of Lagrange multipliers implies that~\eqref{eq:Lagrangecircular} is equivalent to the statement that the restriction $H_{\mu_*}$ of $H$ to $\Sigma_{\mu_*}$ has $u_{\mu_*}$ as a critical point. Hence the circular orbits can be characterized as the critical points of  $H_{\mu_*}$. This is a general feature of relative equilibria of Hamiltonian systems with symmetry, as shown in Theorem~\ref{thm:relequicritical}.

The main ingredient of the proof is the following statement:
\begin{equation}\label{eq:sphericalhessian}
\exists c>0, \forall v\in\Ocal_{u_{\mu_*}}, \forall w\in \left(T_{v}\Ocal_{u_{\mu_*}}\right)^\perp\cap T_{v}\Sigma_{\mu_*},\quad
D_{v}^2\Lcal(w,w)\geq c\|w\|^2.
\end{equation}
This is a lower bound on the Hessian of $\Lcal$ restricted to the two-dimensional subspace of $\R^6$ spanned by the vectors tangent to $\Sigma_{\mu_*}$ (see~\eqref{eq:sphericalsigmamu}) and perpendicular to the dynamical orbit $\Ocal_{u_{\mu_*}}\subset \Sigma_{\mu_*}$. It will allow us to show the following lower bound on the variation of the Lyapunov function, which is to be compared to~\eqref{eq:localmin}:
\begin{align}\label{eq:coercivecircular}
\exists \delta>0,\ &c>0, \forall u'\in \Sigma_{\mu_*}, \nonumber\\
&\left(\rd(u', \Ocal_{u_{\mu_*}})\leq \delta\Rightarrow  \Lcal(u')-\Lcal(u_{\mu_*})\geq c \rd^2(u', \Ocal_{u_{\mu_*}})\right).
\end{align}
Note that this immediately implies that $H_{\mu_*}$ attains a local minimum on $\Ocal_{u_{\mu_*}}$.

To show~\eqref{eq:sphericalhessian}, note that the three vectors 
\begin{equation}\label{eq:sphericalbasis1}
e_1=\begin{pmatrix} p\\-\left(\frac{\sigma^*}{\rho^*}\right)^2q\end{pmatrix},\quad
e_2=\begin{pmatrix} q\\ -p\end{pmatrix},\quad
e_3=\begin{pmatrix}p\\ q\end{pmatrix},\quad
\end{equation}
form an orthogonal basis of $T_v\Sigma_{\mu_*}$, for each point $v=(q,p)\in\Ocal_{u_*}$; $e_1$ is easily seen to be tangent to $\Ocal_{u_*}$, so that $e_2$ and $e_3$ span $(T_v\Ocal_{u_*})^\perp\cap T_v\Sigma_{\mu_*}$.  A simple but tedious computation then shows that the matrices of $D_v^2(\mu_*\cdot L)$ and of $D^2_vH$ in this basis are
$$
D^2_v(\mu_*\cdot L)=
\begin{pmatrix}
2\sigma_*^4&0&\mu_*^2\left[\left(\frac{\sigma_*}{\rho_*}\right)^2-1\right]\\
0&-2\mu_*^2&0\\
\mu_*^2\left[\left(\frac{\sigma_*}{\rho_*}\right)^2-1\right]&0 &-2\mu_*^2
\end{pmatrix}
$$
and
$$
D_v^2H=
\begin{pmatrix}
V'(\rho_*)\rho_*^{-1}\sigma_*^2&0&(V'(\rho_*)\rho_*^{-1}-1)\sigma_*^2\\
0&V''(\rho_*)\rho_*^2+\sigma_*^2&0\\
(V'(\rho_*)\rho_*^{-1}-1)\sigma_*^2&0&V'(\rho_*)\rho_*^{-1}\sigma_*^2+\rho_*^2
\end{pmatrix}
$$
The estimate~\eqref{eq:sphericalhessian} now follows immediately from the hypothesis that $V''(\rho_*)\rho_*^2+3\sigma_*^2>0$. 

We now turn to the proof of~\eqref{eq:coercivecircular}. Let $u'\in\Sigma_{\mu_*}$. Then there exists $v'\in\Ocal_{u_*}$ so that $\rd(v',\Ocal_{u_*})=\|u'-v'\|$ and as a result, one has that $u'-v'\in \left(T_{v'}\Ocal_{u_*}\right)^\perp$. We can write
$$
u'=u'-v'+v'=v'+(u'-v')_\parallel +(u'-v')_\perp.
$$
Here $(u'-v')_\perp$ is perpendicular to $T_{v'}\Sigma_{\mu_*}$, and $(u'-v')_\parallel$ belongs to $T_{v'}\Sigma_{\mu_*}$ and is perpendicular to $T_{v'}\Ocal_{u_*}$ since $u'-v'$ is. 
Now remark that, since $D_{v'}L((u'-v')_\parallel)=0$, and since $u',v'\in\Sigma_{\mu_*}$,
\begin{equation}\label{eq:perpcontrol}
0=L(u')-L(v')=D_{v'}L((u'-v')_\perp)+\mathrm{O}(\|u'-v'\|^2).
\end{equation}
It is easily checked that, for each $v'\in\Ocal_{u_{\mu_*}}$, the restriction of $D_{v'}L$ to $(T_{v'}\Sigma_{\mu_*})^\perp$ is an isomorphism. It follows that there exists a constant $C$ so that
\begin{equation}\label{eq:estimate1}
\|(u'-v')_\perp\|\leq C\|(u'-v')\|^2.
\end{equation}
Note that this constant is independent of $v'\in\Ocal_{\mu_*}$ since, for all $R\in \mathrm{SO}(3)$, and for all $u\in \R^6$,
$$
\Phi_R\circ D_u L\circ \Phi_{R^{-1}}= D_{\Phi_Ru} L,
$$
where $\Phi_R$, defined in~\eqref{eq:so3}, is an isometry.
Returning to~\eqref{eq:perpcontrol}, and using this last remark, we conclude there exists a constant $c_0$ so that, for $\| u'-v'\|$ small enough, one has
\begin{eqnarray}\label{eq:estimate2}
\|(u'-v')_\parallel\|\geq \|u'-v'\|-\|(u'-v')_\perp\|\geq c_0\|u'-v'\|.
\end{eqnarray}
 We can now conclude the proof of~\eqref{eq:coercivecircular} as follows, using \eqref{eq:estimate1}, \eqref{eq:estimate2} and~\eqref{eq:sphericalhessian}:
 \begin{eqnarray*}
\Lcal(u')-\Lcal(u_{\mu_*})&=&\Lcal(u')-\Lcal(v')\\
&=&D_{v'}\Lcal(u'-v')+\frac12 D_{v'}^2\Lcal(u'-v', u'-v')+\mathrm{o}(\|u'-v'\|^2)\\
&=&\frac12 D_{v'}^2\Lcal((u'-v')_\parallel,(u'-v')_\parallel)+ \mathrm{O}(\|u'-v'\|^3) 
+\mathrm{o}(\|u'-v'\|^2)\\
&=&\frac12 D_{v'}^2\Lcal((u'-v')_\parallel,(u'-v')_\parallel)+\mathrm{o}(\|u'-v'\|^2)\\
&\geq& \frac12 c\|(u'-v')_\parallel\|^2+\mathrm{o}(\|u'-v'\|^2)\\
&\geq &\tilde{c}\|u'-v'\|^2=\tilde{c}\rd^2(u',\Ocal_{u_{\mu_*}}).
\end{eqnarray*}
Remark that as before, the constant $c$ is independent of $v'\in \Ocal_{\mu_*}$.  This shows~\eqref{eq:coercivecircular}. Note that we used the boundedness of $D^2_{v'}\Lcal$, uniformly in $v'\in\Ocal_{u_{\mu_*}}$.

We can now prove orbital stability, namely:
\begin{equation}\label{eq:spherpotstable}
\forall \epsilon>0, \exists\delta>0, \forall u'\in\R^6,\quad\left(\rd(u',\Ocal_{u_{\mu_*}})\leq \delta\Rightarrow \forall t\in\R, \rd(u'(t), \Ocal_{u_{\mu_*}})\leq \epsilon\right).
\end{equation} 
For that purpose, we propose three different arguments.

\medskip
\noindent{\bf First argument.} We proceed by contradiction, as before. Suppose there exists $\epsilon_0>0$ and for each $n\in\N$, $u'_n\in\R^6$ and $t_n\in\R$  such that $\rd(u'_n,u_{\mu_*})\leq \frac1n$ and $\rd(u'_n(t_n), \Ocal_{u_{\mu_*}})=\epsilon_0$. We can suppose, without loss of generality, that $2\epsilon_0<\delta$, where $\delta$ is given in~\eqref{eq:coercivecircular}. We know that $\Lcal(u'_n(t_n))=\Lcal(u'_n)$, since $\Lcal$ is a constant of the motion. Hence 
$$
\lim_{n\to+\infty}\Lcal(u'_n(t_n))=\Lcal(u_{\mu_*})=\mu_*.
$$
Since the orbit $\Ocal_{u_{\mu_*}}$ is bounded, and since $\rd(u'_n(t_n),\Ocal_{u_{\mu_*}})=\epsilon_0$, it follows that the sequence $u'_n(t_n)$ is bounded; we can therefore conclude that $\lim_{n\to+\infty}\rd(u'_n(t_n),\Sigma_{\mu_*})=0$. (In other words $\Lcal$ satisfies Hypothesis F, see Lemma~\ref{lem:hypF}.) As a consequence, there exist $w_n\in\Sigma_{\mu_*}$ so that $\|w_n-u'_n(t_n)\|\to 0$. We can now conclude. Since, for $n$ large enough, $\frac{\epsilon_0}2\leq \rd(w_n, \Ocal_{u_{\mu_*}})\leq \frac32\epsilon_0$, we have
\begin{align*}
\Lcal(u'_n)-\Lcal(u_{\mu_*})&=\Lcal(u'_n(t_n))-\Lcal(u_{\mu_*})\\
&=\Lcal(u'_n(t_n))-\Lcal(w_n)+\Lcal(w_n)-\Lcal(u_{\mu_*})\\
&\geq \Lcal(u'_n(t_n))-\Lcal(w_n)+c\rd^2(w_n,\Ocal_{u_{\mu_*}}).
\end{align*}
The sequences $u'_n(t_n)$ and $w_n$ are bounded. This, combined with the uniform continuity of $\Lcal$ on bounded sets, leads again to a contradiction upon taking $n\to+\infty$.

\medskip
\noindent{\bf Second argument.} The second proof uses the fact that the relative equilibrium $u_{\mu_*}$, which gives rise to a circular orbit, belongs to a continuous family $\mu\to u_{\mu}$ of such equilibria, defined on a neighbourhood $I\subset \R^3$ of $\mu_*$. We will only sketch the argument, the general case is treated in Theorem~\ref{thm:lyapmethodrestrictedmod}.  One first observes that, for $\mu$ belonging to a suitably small neighbourhood of $\mu_*$, both~\eqref{eq:sphericalhessian} and~\eqref{eq:coercivecircular} hold, with $\mu_*$ replaced by $\mu$, and with $\mu$-independent $c$ and $\delta$. This allows one to prove that the equilibria $u_\mu$ are orbitally stable with respect to perturbations of the initial condition \emph{within} $\Sigma_\mu$, that is:
\begin{equation}\label{eq:spherpotrestrictedstable}
\forall \epsilon>0, \exists\delta>0, \forall u'\in\Sigma_{\mu},\quad\left(\rd(u',\Ocal_{u_\mu})\leq \delta\Rightarrow \forall t\in\R, \rd(u'(t), \Ocal_{u_\mu})\leq \epsilon\right).
\end{equation} 
Indeed, suppose that this is not true. Then there exists  $\epsilon_0>0$, and for each $n\in\N^*$, $u'_n\in\Sigma_{\mu}$, $t_n\in\R$ so that $\rd(u'_n,u_\mu)\leq \frac1n$ and $\rd(u'_n(t_n), \Ocal_{u_\mu})=\epsilon_0$.  Since we can choose $\epsilon_0<\delta$, we can apply~\eqref{eq:coercivecircular}  to write
\begin{eqnarray*}
\Lcal(u'_n)-\Lcal(u_\mu)=\Lcal(u'_n(t_n))-\Lcal(u_\mu)\geq c\rd(u'_n(t_n), \Ocal_{u_\mu})^2=c_*\epsilon_0^2.
\end{eqnarray*}
Taking $n\to+\infty$ leads to the desired contradiction. It remains to prove~\eqref{eq:spherpotrestrictedstable} with ``$\forall u'\in\Sigma_{\mu}$'' replaced by ``$\forall u\in\R^6$.'' 
For that purpose, note that, if $u'\in\R^6$ is close to $u_{\mu_*}$, then $\mu=L(u')$ is close to $\mu_*$ and hence $u_\mu$ close to $u_{\mu_*}$. So $u'$ is close to $u_\mu$. Hence $u'(t)$ remains close at all times to $\Ocal_{u_\mu}$ by~\eqref{eq:spherpotrestrictedstable}. Now, since $\Ocal_\mu$ is close to $\Ocal_{\mu_*}$, the result follows. 

\medskip
\noindent{\bf Third argument.} If~\eqref{eq:sphericalhessian} had been valid for all $w\in (T_v\Ocal_{u_{\mu_*}})^\perp$, the first argument above  would have been slightly easier, since we could then have mimicked the proof of Proposition~\ref{lem:sphericalfixedpoints} directly. As it stands, we were able to first show~\eqref{eq:coercivecircular}, which is valid only for $v'\in \Sigma_{\mu_*}$ and which shows $\Lcal$, restricted to $\Sigma_{\mu_*}$,  attains a local minimum on the orbit. This immediately implies an orbital stability result for perturbations $u'$ of the initial condition $u_{\mu_*}$ that stay within $\Sigma_{\mu_*}$, as is readily seen.  But to obtain a stability result for arbitrary perturbations $u'\in\R^6$ of the initial condition $u_{\mu_*}$, we had to work a little harder and invoke Hypothesis F (see Section~\ref{ss:coercivitystability1}), which may fail in infinite dimensional problems, as we will see. It turns out that~\eqref{eq:sphericalhessian} is \emph{not} valid\footnote{This can be seen from a straightforward computation, which is most readily made in the basis $e_i$ introduced in~\eqref{eq:sphericalbasis1} and \eqref{eq:sphericalbasis2}.} for all $w\in (T_v\Ocal_{u_{\mu_*}})^\perp$. However, it is possible to adjust the Lyapunov function $\Lcal$ so that this \emph{is} the case.   Consider, for all $K>0$, 
\begin{equation}
\Lcal_K(u)=\Lcal(u)+K(L(u)-\mu_*)^2.
\end{equation}
Note that the additional term vanishes on $\Sigma_{\mu_*}$, where $\Lcal_K$  reaches an absolute minimum. We now show
\begin{equation}\label{eq:LyapunovKestimate}
\exists \hat c>0, K>0, \forall v\in\Ocal_{u_{\mu_*}}, \forall w\in \left(T_{v}\Ocal_{u_{\mu_*}}\right)^\perp,\quad
D_{v}^2\Lcal_K(w,w)\geq \hat c\|w\|^2.
\end{equation}
For that purpose, introduce, for each $v=(q,p)\in \Ocal_{u_{\mu_*}}$,
\begin{equation}\label{eq:sphericalbasis2}
e_4=
\begin{pmatrix}
\hat q\wedge \hat p\\ 0
\end{pmatrix},\quad 
e_5=
\begin{pmatrix}
0\\\hat q\wedge \hat p
\end{pmatrix}, \quad
e_6=\frac1{\sqrt{\rho_*^2+\sigma_*^2}}
\begin{pmatrix}
\sigma_* \hat q\\ \rho_*\hat p
\end{pmatrix},
\end{equation}
which, together with $e_1, e_2, e_3$ in~\eqref{eq:sphericalbasis1} form an orthonormal basis of $\R^6$. Clearly, $D_v(L-\mu_*)^2(w)=0$, for all $v\in\Ocal_{u_{\mu_*}}$ and for all $w\in\R^6$. Moreover, if $\eta_1, \eta_2, \eta_3\in\R^3$ form an orthonormal basis, then 
$$
D^2_v(L-\mu_*)^2(w,w)=2\sum_{i=1}^3 \left[D_v(\eta_i\cdot L)(w)\right]^2,
$$
with
$$
D_v(\eta_i\cdot L)(w)=w_1\cdot (p\wedge \eta_i)+ w_2\cdot(\eta_i\wedge q),\quad w=(w_1, w_2)\in\R^6.
$$
Now, writing $w=\sum_{j=2}^6 \alpha_j e_j\in (T_v\Ocal_{u_{\mu_*}})^\perp$ and using $\eta_1=\hat q, \eta_2=\hat p, \eta_3=\hat q\wedge \hat p$, we find
\begin{align}
D^2_v(L-\mu_*)^2(w,w)&=2\left[\alpha_4^2\sigma_*^2 +\alpha_5^2\rho_*^2 +\alpha_6^2(\rho_*^2+\sigma_*^2)\right]\nonumber\\
&\geq 2\min\{\sigma_*^2, \rho_*^2\}\left[\alpha_4^2+\alpha_5^2+\alpha_6^2\right].
\end{align}
We can now conclude the proof of~\eqref{eq:LyapunovKestimate} as follows. We write $w=w_A+w_B$ with $w_A=\alpha_2 e_2+\alpha_3 e_3$ and $w_B=\alpha_4 e_4 +\alpha_5 e_5 +\alpha_6 e_6$. Then there exists a constant $C>0$, independent of $v\in \Ocal_{u_{\mu_*}}$, so that
\begin{align*}
D_v^2\Lcal_K(w,w)&\geq D_v^2\Lcal(w,w)+2K\min\{\sigma_*^2, \rho_*^2\}\|w_B\|^2\\
&\geq D_v^2\Lcal(w_A,w_A)+2K\min\{\sigma_*^2, \rho_*^2\}\|w_B\|^2-C\left[\|w_A\|\|w_B\|+\|w_B\|^2\right].
\end{align*}
Using~\eqref{eq:sphericalhessian}, one finds that, for all $m>0$,
$$
D_v^2\Lcal_K(w,w)\geq \Big(c-\frac{Cm^2}{2}\Big)\|w_A\|^2 +\Big(2K\min\{\sigma_*^2, \rho_*^2\}-C-\frac{C}{2m^2}\Big)\|w_B\|^2,
$$
where we have applied Young's inequality to the term $\|w_A\|\|w_B\|$.
Choosing $m$ small enough and $K$ large enough, one finds~\eqref{eq:LyapunovKestimate}. We can now prove the following statement, which is to be compared to~\eqref{eq:coercivecircular}: $\exists \delta, c>0$ so that, for all $u'\in\R^6$, 
\begin{equation}\label{eq:coercivecircular3}
\rd(u', \Ocal_{u_{\mu_*}})\leq \delta \Rightarrow \Lcal_K(u')-\Lcal_K(u_{\mu_*})\geq c^2\rd^2(u', \Ocal_{u_{\mu_*}}).
\end{equation}
Indeed, for all $u'\in \R^6$, there exists $v'\in \Ocal_{u_{\mu_*}}$ so that $u'-v'\in (T_{v'}\Ocal_{u_{\mu_*}})^\perp$. Hence
$$
\Lcal_K(u')-\Lcal(u_{\mu_*})=\Lcal(u')-\Lcal_K(v')\geq \frac{\hat c}{2}\|u'-v'\|^2 +\mathrm{O}(\|u'-v'\|^3).
$$
This implies~\eqref{eq:coercivecircular3}, from which orbital stability follows by the now familiar argument.
\end{proof}

We point out that the core ingredient of all three arguments in the proof is estimate~\eqref{eq:sphericalhessian}. Its proof constitutes the only truly model-dependent part of the proofs of orbital stability via the energy-momentum method. This will become clear in Section~\ref{s:orbstabproof} where we will show how a suitably adapted version of this estimate implies orbital stability in a general infinite dimensional setting as well (Theorem~\ref{thm:lyapmethodrestricted}, Theorem~\ref{thm:lyapmethodrestrictedmod}, Theorem~\ref{thm:lyapmethodrestrictedLK}). 

As a second remark, note that~\eqref{eq:coercivecircular} allows one to prove immediately the orbital stability for perturbations of the initial condition that preserve the angular momentum. The three strategies of the proof above therefore concern three different methods for extending this result to arbitrary perturbations of the initial condition. The same structure of the proof will be apparent in the general situation treated in Section~\ref{s:orbstabproof}. 

The first argument in the above proof is the one used in~\cite{gssI} and~\cite{gssII}. It has the disadvantage of using Hypothesis F, which, while obvious in finite dimensions, may not hold in infinite dimensional systems, notably when the group $G_{\mu_*}$ is not one-dimensional (as in~\cite{gssII}).  We will illustrate this phenomenon in Section~\ref{s:orbstabproof}. It has the advantage -- when Hypothesis~F does work -- of not using the fact that the relative equilibrium under consideration belongs to a continuous family.

The second argument seems to go back to Benjamin (see Section~\ref{history.sec}) and is used for example in~\cite{weinstein86}, and in~\cite{galhar07a, galhar07b}. For this argument the existence of a continuous family of relative equilibria is needed but not Hypothesis~F. 

The third argument is commonly used in the literature on finite dimensional Hamiltonian systems~\cite{pat}, and appears also in~\cite{stuart08} in the infinite dimensional case. It is not universally useable, since it depends on the existence of a $G_\mu$-invariant Euclidean structure on the dual of the Lie-algebra of $G$, as we will see in Section~\ref{s:orbstabproof}.


\section{Hamiltonian dynamics in infinite dimension}\label{s:hamdyninfinite}
The modern formulation of Hamiltonian dynamics has been adapted to the framework of infinite dimensional Banach manifolds in~\cite{chm, marsdenratiu}. This approach is not well suited for our purposes for two reasons. First, we are interested in flows defined by the solutions to (nonlinear) partial differential equations that are defined on Banach (or even Hilbert) spaces, for which a general Banach manifold formulation is overly complex.
In addition, the notions of ``Hamiltonian vector field'' and ``Hamiltonian flow'' introduced in~\cite{chm} seem too general for the purpose of studying stability questions. We therefore present a simpler and more restricted framework that is well adapted to  the analysis of the stability questions that are our main focus, including for nonlinear Schr\"odinger and wave equations.

Our main goal in this section is thus to give a workable and not too complex definition of ``Hamiltonian dynamical system'' or of ``Hamiltonian flow'' in the infinite dimensional Banach space setting (Section~\ref{ss:ham flows}). The formalism allows us to easily obtain general results on the link between symmetries and conserved quantities for such systems, as in the finite dimensional case (Section~\ref{ss:symconstmotion}). This link is indeed an essential ingredient for the identification of relative equilibria and the construction of coercive Lyapunov functions in Hamiltonian systems with symmetry, as we shall explain in Section~\ref{s:identifyreleq}. Several examples of Hamiltonian PDE's that fit in our framework are given in Section~\ref{ss:hampde}.  
 Although this section is self-contained, the reader unfamiliar with finite dimensional Hamiltonian dynamical systems and their symmetries may find it useful to consult Appendix~\ref{s:hammech} for a concise and self-contained treatment of this case. We will make regular use of the notation and concepts introduced there. 
\subsection{Symplectors, symplectic Banach triples, symplectic transformations, Hamiltonian vector fields}\label{ss:symplectors}
We first generalize the notion of symplectic form to the infinite dimensional setting and introduce the equivalent notion of \emph{symplector} (Definition~\ref{def:symplector}). It turns out that, in the infinite dimensional setting, it is convenient to treat the latter as the central object of the theory, rather than the symplectic form itself, as is customary in finite dimensions. As we will see, the two approaches are perfectly equivalent.  

We need some preliminary terminology. Let $\Ban$ be a Banach space and $B:\Ban\times \Ban\to\R$ a bilinear continuous form. We can then define, in the usual manner, for all $u\in \Ban$, $\mathcal{J}_Bu\in \Ban^*$ via
$$
\mathcal{J}_Bu(v)=B(u,v).
$$
It follows easily that
$
\mathcal{J}_B: u\in \Ban\to \mathcal{J}_Bu\in \Ban^*
$
is linear and continuous, with $\|{\mathcal J}_B\|=\|B\|$. We will write $\Rcal_{\hskip-1pt\Jcal_B}=\mathrm{Ran}\Jcal_B$.  Conversely, given a continuous linear map $\Jcal: \Ban\to \Ban^*$, one can construct $B_{\!\Jcal}(u,v)=(\Jcal u)(v)$. 
We introduce the following terminology:
\begin{definition} A bilinear continuous form $B$ is non-degenerate (or weakly non-degenerate) if $\Jcal_B$ is injective. It is strongly non-degenerate if $\Jcal_B$ is both injective and surjective.
Similarly, a linear map $\Jcal:E\to E^*$ is said to be (weakly) non-degenerate if it is injective, and strongly non-degenerate if it is a bijection. 
\end{definition}

\begin{definition}  \label{def:symplector} 
We now introduce the notion of symplector.\footnote{This object does not seem to have been blessed with a name in the literature, so we took the liberty to baptize it.}
\begin{enumerate}[label=({\roman*})]
\item A  symplector\index{symplector} or weak symplector\index{weak symplector}  is a continuous linear map $\Jcal: \Ban\to \Ban^*$ that is injective and anti-symmetric, in the sense that
$$
(\Jcal u)(v)=-(\Jcal v)(u).
$$ 
If in addition $\Jcal$ is surjective, we say it is a strong symplector.
\item A (strong) symplectic form $\omega$ is a (strongly) non-degenerate bilinear continuous form that is anti-symmetric. 
\item When $\Jcal$ is a (strong) symplector, we will say $(\Ban,\Jcal)$ is a (strong) symplectic vector space, or simply that $\Ban$ is a (strong) symplectic vector space, when there is no ambiguity about the choice of $\Jcal$. 
\end{enumerate}
\end{definition}
There clearly is a one-to-one correspondence between (strong) symplectors and (strong) symplectic forms. Note that the definition implies that
\begin{equation}
\forall \alpha, \beta\in \Rcal_{\!\Jcal}, \quad \alpha(\Jcal^{-1}\beta)=-\beta(\Jcal^{-1}\alpha).
\end{equation}

The following examples of (strong) symplectors cover all applications we have in mind in these notes. Let $\Kcal$ be a real Hilbert space and set $\Ban=\Kcal\times \Kcal$. Then
$$
\Jcal: (q,p)\in \Ban\to (-p,q)\in \Ban^*
$$ 
is clearly a strong symplector. Here we wrote $u=(q,p)\in \Kcal\times\Kcal$ and used the Riesz identification of $\Ban$ with $\Ban^*$. The corresponding strong symplectic form is
$$
\omega_{\!\Jcal}(u, u')= q\cdot p'-q'\cdot p,
$$
where $\cdot$ denotes the inner product on $\Kcal$. The analogy with~\eqref{eq:basicsymplectic} is self-evident: there $\Kcal=\R^n$, where $\R^n$ is equipped with its standard Euclidean structure. Note that if $Q$ is a bounded self-adjoint operator on $\Kcal$ with Ker$Q=\{0\}$, then 
$$
\Jcal: (q,p)\in \Ban\to (-Qp, Qq)\in \Ban^*
$$
is also a symplector with 
$$
\omega_{\!\Jcal}(u, u')=q\cdot Qp'- p\cdot Qq'.
$$
We will need the following straightforward generalization of the above construction.
Let $K^2$ be a positive (possibly and typically unbounded) self-adjoint operator on $\Kcal$, with domain $\Dcal(K)$.    Introduce, for all $s\in \R$, $\Kcal_s=[\Dcal(\langle K\rangle^s)]$, where $\langle K\rangle=\sqrt{1+K^2}$ and where $\langle K\rangle^s$ is defined by the functional calculus of self-adjoint operators. Here $[\Dcal(\langle K\rangle^s)]$ denotes the closure of $\Dcal(\langle K\rangle^s)$ in the topology induced by the Hilbert norm 
$$
\|u\|_s:=\| \langle K\rangle^s u\|.
$$
Note that, since $\langle K\rangle^s:(\Dcal(\langle K\rangle^s), \|\cdot\|_s)\to (\Dcal(\langle K\rangle^{-s}), \|\cdot \|)$ is an isometric bijection, it extends to a unitary map from $\Kcal_s$ to $\Kcal$ for which we still write $\langle K\rangle^s$.  With these conventions, we can then make the usual identification between $\Kcal_s^*$ and $\Kcal_{-s}$: $\forall v\in \Kcal_{-s}$, we define
$$
u\in \Kcal_s\to v\cdot u\in \R,
$$
by setting $v\cdot u:=\langle K\rangle^{-s} v\cdot \langle K\rangle^s u$. Note that
$$
\forall s,s'\in\R, \quad s\leq s' \Rightarrow \Kcal_{s'}\subset \Kcal_s.
$$
It is easy to see using the spectral theorem that this is an inclusion as sets, and we will therefore not introduce explicit identification operators to represent such inclusions which are moreover continuous for the respective Hilbert space topologies.  The typical example of this construction to keep in mind is $K^2=-\Delta$ on $\Kcal=L^2(\R^d)$. We then have $\Kcal_s=H^s(\R^d)$, the usual Sobolev spaces.

For $s=(s_1, s_2)\in\R^2$, we define $\Ban_s=\Kcal_{s_1}\times\Kcal_{s_2}$. Defining a partial order relation by $s\preceq s'$ iff $s_1\leq s_1'$ and $s_2\leq s_2'$, we have 
$$
\forall s, s'\in\R^2,\quad s\preceq s' \Rightarrow \Ban_{s'}\subset \Ban_s.
$$
Setting $\bar s=(s_2, s_1)$ we then define
\begin{equation}\label{eq:symplectors}
\Jcal_s: u=(q,p)\in \Ban_s\to (-p, q)\in \Ban_{\bar s}.
\end{equation}
The following lemma is now immediate.
\begin{lemma}\label{lem:symplectors1} $\Jcal_s$ is a weak symplector if and only if $s_1\geq  -s_2$. In that case
$$
\Jcal_s: u=(q,p)\in \Ban_s\to (-p,q)\in \Ban_{\bar s}\subset \Ban_{-s}=\Ban^*_s.
$$
We have $\Rcal_s:=\Rcal_{\!\Jcal_s}=\Ban_{\overline s}$. And $\Jcal_s^{-1}={\Jcal_{-s}}_{|\Ban_{\overline s}}$. If $K^2$ is unbounded, $\Jcal_s$  is a strong symplector if and only if $s_1=-s_2$.
\end{lemma}
Typical examples of this construction are the use of $E=\Ban_{(1/2, -1/2)}$ or of $E=\Ban_{(1, 0)}$ with $\Kcal=L^2(\R^d)$ and $K^2=-\Delta$ to study the wave equation. 
For the Schr\"odinger equation, $E=\Ban_{(1, 1)}$ is a natural choice. We refer to Section~\ref{ss:hampde} for the details of these examples. Note that of these three examples, only the first corresponds to a strong symplector and hence to a strong symplectic form. It is therefore clear that the use of weak symplectors is unavoidable in applications to PDE's.

We end our discussion of symplectors with a simple lemma that  collects some of their essential properties. 
\begin{lemma}\label{lem:symplectors2} Let $E$ be a Banach space and $\Jcal:E\to E^*$ be a bounded linear map. Then the following holds:
\begin{enumerate}[label=({\roman*})]
\item If $\Jcal$ is a strong symplector, then $\Jcal^{-1}$ is bounded.
\item If $\Jcal$ is injective and (anti-)symmetric, and if $E$ is reflexive, then $\Rcal_{\!\Jcal}$ is dense in $\Ban^*$.
\item Suppose $\Jcal$ is injective and (anti-)symmetric, and that its inverse is bounded on $\Rcal_{\!\Jcal}$. Suppose $E$ is reflexive. Then $\Rcal_{\!\Jcal}=E^*$.
\end{enumerate}
\end{lemma}
\begin{proof} 
\begin{enumerate}[label=(\roman*)]
\item This is a consequence of the open mapping theorem. 
\item Suppose $v\in \Ban$ satisfies $\Jcal u(v)=0$ for all $u\in \Ban$. Then $\Jcal v(u)=0$ for all $u\in \Ban$, by (anti-)symmetry. Hence $\Jcal v=0$ and hence, since $\Jcal$ is injective, $v=0$. Since $E$ is reflexive, this means that, if $v\in E^{**}$ vanishes on $\Rcal_\Jcal\subset E^*$, then $v=0$.  This implies $\Rcal_\Jcal$ is dense (Hahn-Banach). 
\item Since the inverse is bounded, $\Rcal_\Jcal$ is closed. The result then follows from (ii).
\end{enumerate}
\end{proof}
If $\Ban$ is not reflexive, a symplector may not have a dense range, as the following example\footnote{Due to S. Keraani.} shows. Let 
$$
\Ban=\{u\in L^1(\R, \rd x)\mid \int_\R u(x) \rd x=0\}\subset L^1(\R)
$$
and define 
$$
\Jcal u(x)=\int_{-\infty}^x u(y)\rd y\in L^\infty(\R)\subset E^*.
$$
This is clearly bounded, injective and antisymmetric. But it is clear that 
$$
\|\Jcal u-1\|_\infty\geq 1,
$$ 
for all $u\in \Ban$. So the range is not dense in $L^\infty(\R)$ and a fortiori not dense in $E^*$. 

We are now ready to define what we mean by a symplectic transformation and by a Hamiltonian vector field. First we recall a very  basic definition: when $F:\Ban_1\to \Ban_2$ is a function between two Banach spaces $\Ban_1$ and $\Ban_2$, and when $u\in \Ban_1$, one says that $F$ is (Fr\'echet) differentiable at $u$ if there exists $D_u F\in\Lcal(\Ban_1, \Ban_2)$ so that
$$
\lim_{w\to0}\frac{\|F(u+w)-F(u)-D_uF(w)\|_{\Ban_2}}{\|w\|_{\Ban_1}}=0.
$$
Also, one says that $F:\Ban_1\to\Ban_2$ is differentiable on some subset of $E_1$ if for all $u$ in that subset, $F$ is differentiable in the above sense.  

In particular, if $\Ban_1=\Ban, \Ban_2=\R$, and if $F$ is differentiable at $u\in \Ban$, we have $D_uF\in \Ban^*$. And if $\Dcal$ is a domain in $\Ban$, saying that  $F:\Ban\to\R$ \emph{is differentiable on} $\Dcal$ means that $F$ is differentiable at each $u\in\Dom$. In that case, one can define
$$
u\in\Dcal\subset \Ban\to D_uF\in \Ban^*.
$$
As a last comment, we stress that, in these definitions, the only topology used is the one on $\Ban$. This is important to keep in mind in the applications, where the domain $\Dcal$ often carries a natural topology, stronger than the one induced by the norm on $\Ban$, and for which $\Dcal$ is closed. One can think of $\Ban=H^1(\R)$ and $\Dcal=H^3(\R)$. Such a topology is NOT used in the above statements, nor in the following general definition. We refer to the examples treated in Sections~\ref{ss:linflows} and~\ref{ss:hampde} for several illustrations of this last comment.  
\begin{definition}\label{def:Jdif}  Let $\Ban$ be a Banach space, $\Dcal$ a domain in $\Ban$ (See Section~\ref{ss:dynamicsbanach}) and $\Jcal$ a symplector.
\begin{enumerate}[label=({\roman*})]
\item  We will refer to $\left(\Ban, \Dcal, \Jcal\right)$ as a symplectic Banach triple\index{symplectic Banach triple}.
\item Let $\left(\Ban, \Dcal, \Jcal\right)$ be a symplectic Banach triple and $\Phi\in C^0(E,E)\cap C^1(\Dcal, E)$. We say $\Phi$ is a symplectic transformation\index{symplectic transformation} if 
\begin{equation}\label{eq:sympltransfbis}
\forall u\in \Dcal, \forall v, w\in \Ban, (\Jcal D_u\Phi (v))(D_u\Phi(w)) = (\Jcal v)(w).
\end{equation}
\item We say that a function $F:\Ban\to \R$ has a $\Jcal$-compatible derivative\index{$\Jcal$-compatible derivative} if $F$ is differentiable on $\Dcal$ and if, for all $u\in\Dom$, $D_uF\in\Rcal_{\!\Jcal}$. In that case we write $F\in\JDif$.
\item For each $F\in\JDif$,   the Hamiltonian vector field\index{Hamiltonian vector field} $X_F:\Dcal\subset \Ban\to \Ban$ associated to $F$ is defined by 
\begin{equation}\label{eq:HamfieldJbis}
X_F(u)=\Jcal^{-1}D_uF,\quad \forall u\in\Dom.
\end{equation}
\end{enumerate}
\end{definition}
The analogy between~\eqref{eq:sympltransfbis} and~\eqref{eq:sympltransf} as well as between~\eqref{eq:HamfieldJbis} and~\eqref{eq:HamfieldJ} is evident.
   Note however that, when dealing with \emph{weak} symplectors, as is often the case in applications, the vector field $X_F$ does not inherit the continuity or smoothness properties that $F$ may enjoy. In particular, even if 
$$
D_\cdot F:\Dcal\subset \Ban\to \Ban^*
$$ 
is continuous, the same may not hold for $X_F$. We shall for that reason avoid making use of the vector fields $X_F$ where possible and state all our hypotheses in terms of $F$ directly. We finally point out that, here and in what follows, and unless otherwise specified, all functions we consider are globally defined\footnote{This is a difference with~\cite{chm}, as we will explain in some detail in Section~\ref{ss:linflows}.} on $\Ban$.  


\subsection{Hamiltonian flows and constants of the motion}\label{ss:ham flows}\index{constants of the motion}
\begin{definition} \label{def:hamflow} Let $\left(\Ban, \Dcal, \Jcal\right)$ be a symplectic Banach triple. Let $F\in\JDif$. A Hamiltonian flow\index{Hamiltonian flow} for $F$ is a separately continuous map $\Phi^F:\R\times \Ban\to\Ban$ with the following properties:
\begin{enumerate}[label=({\roman*})]
\item For all $t,s\in\R$, $\Phi_{t+s}^F=\Phi_t^F\circ\Phi_s^F,\, \Phi_0^F=\mathrm{Id}$;
\item For all $t\in\R$, $\Phi_t^F(\Dom)=\Dom$;
\item For all $u\in\Dom$, the curve $t\in\R\to u(t):=\Phi_t^F(u)\in\Dcal\subset \Ban$ is differentiable and is the unique solution of 
\begin{equation}\label{eq:hameqmotionJbis}
\Jcal\dot u(t)=D_{u(t)}F,\quad u(0)=u.
\end{equation}
\end{enumerate}
\end{definition}

Local Hamiltonian flows are defined in the usual way. 
We refer to~\eqref{eq:hameqmotionJbis} as the Hamiltonian differential equation associated to $F$ (Compare to~\eqref{eq:hameqmotionJ} and~\eqref{eq:hameqmotion}) and to its solutions as Hamiltonian flow lines. Note that in this setting separate continuity implies continuity (See~\cite{chm}, Section~3.2). We refer to Section~\ref{ss:hampde} for examples  of PDE's generating Hamiltonian flows. 

To compare this definition to the ones of~\cite{gssI, gssII, stuart08}, we first observe that~\eqref{eq:hameqmotionJbis} implies that, for all $u\in \Dom$,
\begin{equation}\label{eq:hameqmotionJbisweak}
\forall\alpha\in \Rcal_{\!\Jcal},\quad -\frac{\rd}{\rd t} \alpha(u(t)) = D_{u(t)}F(\Jcal^{-1}\alpha),
\end{equation}
which is a weak form of~\eqref{eq:hameqmotionJbis}. With this in mind, one could think of changing Definition~\ref{def:hamflow} by replacing (iii)  by the following alternative statement\footnote{Note that for this formulation one needs $F\in \Dif(E,\R)$, but it is not necessary that it has a $\Jcal$-compatible derivative.}:\\

\noindent (iii') For all $u\in\Ban$, the curve $t\in\R\to u(t):=\Phi_t^F(u)\in \Ban$ belongs to $C(\R, \Ban)$ and~\eqref{eq:hameqmotionJbisweak} holds.\\

\noindent This has the advantage of eliminating the introduction of the domain $\Dcal$ (and therefore of condition~(ii)) and is precisely the definition of ``solution'' to~\eqref{eq:hameqmotionJbis} used in~\cite{gssI, gssII}.  In~\cite{stuart08}, $E$ is a Hilbert space and still a different formulation is adopted. Basically, the domain $\Dcal$ is not introduced, the equation~\eqref{eq:hameqmotionJbis} is interpreted as an equation in $E^*$ and the time derivative is understood as a strong derivative for $E^*$-valued functions. Those alternative formulations do not allow for a direct proof of the kind of natural ``conservation theorems'' such as Theorem~\ref{thm:conservation} below, that are typical for Hamiltonian systems and that we need for the stability analysis. As a result, the conclusions of such conservation theorems are added as assumptions in the general setup of the cited works. It turns out that, in examples, the proof of such assumptions requires a stronger notion of ``solution'' than the ones used in~\cite{gssI, gssII, stuart08}, so we found it more efficient to adopt from the start the stronger notion of Hamiltonian flow found in Definition~\ref{def:hamflow}. 

Let us finally point out that the formulation adopted in~\cite{stuart08} puts further restrictions on $\Jcal$, ruling out for example the treatment of the wave equation as a Hamiltonian system as in Section~\ref{ss:hampde}. Also, only one-dimensional invariance groups are considered there, and restrictions on their action rule out, for example, the consideration of the translation group as a symmetry group for the nonlinear homogeneous Schr\"odinger equation. The formalism does therefore not apply to the study of the orbital stability of the bright solitons in~\eqref{eq:soliton}. On the other hand, it can and has been used to study the orbital stability of standing waves of the inhomogeneous nonlinear Schr\"odinger equation. We refer to Section~\ref{curves.sec} for more details.

\begin{definition}\label{def:poissonbracketJbis}  Let $F,G\in\JDif$. Then the Poisson bracket\index{Poisson bracket} of $F$ and $G$ is defined by
\begin{equation}\label{eq:poissonbrackJbis}
\{F, G\}(u)=D_uF(\Jcal^{-1}D_uG),\qquad \forall u\in\Dom.
\end{equation}
\end{definition}
Equation~\eqref{eq:poissonbrackJbis} is the obvious transcription of~\eqref{eq:poissonbrackJ} to the infinite dimensional setting.
We now have the following  crucial result, which is a simple form of Noether's Theorem\index{Noether's Theorem} in the Hamiltonian setting.  A more complete form follows below (Theorem~\ref{thm:nother3}).
\begin{theorem}\label{thm:conservation}
Let $(\Ban, \Dom, \Jcal)$ be a symplectic Banach triple. Let $H,F\in C(\Ban, \R)$  and suppose they have a $\Jcal$-compatible derivative, \emph{i.e.} $H,F\in \JDif$. Suppose there exist Hamiltonian flows $\Phi_t^H, \Phi_t^F$ for $H$ and $F$.  Then:
\begin{enumerate}[label=({\roman*})]
\item For all $u\in\Dom$, and for all $t\in\R$,
\begin{equation}\label{eq:FG}
\frac{\rd}{\rd t} H(\Phi_t^F(u))=\{H,F\}(\Phi_t^F(u)).
\end{equation}
\item The following three statements are equivalent:
\begin{enumerate}[label=({\alph*})]
\item For all $u\in\Dom$, $\{F,H\}(u)=0$. 
\item For all $u\in \Ban$, and for all $t\in\R$,
\begin{equation}\label{eq:constants}
(H\circ\Phi_t^F)(u)=H(u).
\end{equation}
\item For all $u\in \Ban$, and for all $t\in\R$,
\begin{equation}\label{eq:constants1}
(F\circ\Phi_t^H)(u)=F(u).
\end{equation}
\end{enumerate}
\end{enumerate}
\end{theorem}
In this result, the roles of $H$ and $F$ are interchangeable. But in practice, one of the flows, say $\Phi_t^F$, is simple, explicitly known, and often linear, whereas $\Phi_t^H$ is obtained by integrating a possibly nonlinear PDE of some complexity, such as the nonlinear Schr\"odinger or wave equations. It is then often very easy to check by a direct computation that $H\circ \Phi_t^F$ is constant in time for all $u\in\Ban$: one says that $H$ is invariant under the flow $\Phi_t^F$, or that the $\Phi_t^F$ are symmetries of $H$. The important conclusion of the theorem is that this implies that $F$ is a constant of the motion for $\Phi_t^H$. This is a strong statement, since in applications, the flow $\Phi_t^H$ is complex and poorly known. So being able to assert that it leaves the level surfaces of $F$ invariant is a non-trivial piece of information. Several examples are given in Section~\ref{ss:hampde}. 

\begin{proof} 
\begin{enumerate}
\item[(i)] Let $u\in\Dom$. Then $t\in\R\to H(\Phi_t^F(u))\in \R$ is differentiable and the chain rule applies: writing $u(t)=\Phi_t^F(u)$, we have
$$
\frac{\rd}{\rd t} H(\Phi_t^F(u))=D_{\Phi_t^F(u)}H(\dot u(t)),
$$
which yields the first equality in~\eqref{eq:FG} since $\Jcal\dot u(t)=D_{u(t)}F$.
\item[(ii)] That~\eqref{eq:constants} or~\eqref{eq:constants1} imply $\{H,F\}(u)=0$ for $u\in\Dom$ is immediate from (i). Conversely, it follows from (i) and the fact that $\{H,F\}(u)=0$, for all $u\in\Dcal$, that $(H\circ\Phi_t^F)(u)=H(u)$. Since $\Dom$ is dense in $\Ban$, $H\in C(\Ban,\R)$ and $\Phi_t^F\in C(\Ban,\Ban)$, (b) now follows for all $u\in \Ban$. Similarly for (c). 
\end{enumerate}
\end{proof}
It should be noted that condition~(ii) of Definition~\ref{def:hamflow} is crucial here. We are assuming there is a common invariant domain for both flows. To obtain conservation theorems of the above type without such an assumption requires other technical conditions~\cite{chm}.

We end with some  technical remarks. First, it follows from Theorem~\ref{thm:symplecticflow} in the Appendix, that Hamiltonian flows $\Phi_t^F$ are symplectic as soon as $F\in C^2(E,E)$ and $\Phi_t^F\in C^2(E, E)$. But these two assumptions (especially the latter) are generally too strong to be of use in infinite dimensional dynamical systems generated by PDE's, except possibly when they are linear. Of course, one can conceive of weaker conditions  that imply the result. For efforts in that direction, we refer to~\cite{chm}. In other words, proving that Hamiltonian flows, as defined above, are symplectic, can be painful. A second, related issue is the following. In finite dimensional systems, we know that, if $\{ F_1, F_2\}=0$, with $F_1, F_2\in C^2(E)$, then the corresponding Hamiltonian flows commute: see~\eqref{eq:compoisson} and Lemma~\ref{lem:commutator}. This is a very useful fact: indeed, computing a Poisson bracket is a routine matter of taking derivatives, and the information obtained about the flows is very strong. Again, this is not immediate in infinite dimensional systems under reasonable conditions. For our purposes, and in particular for the proof of Theorem~\ref{thm:nother3}, the following analog of Lemma~\ref{lem:composeF} will suffice. 
\begin{lemma}\label{lem:composeFbis}
Let $(\Ban, \Dcal, \Jcal)$ be a symplectic Banach triple. Let $\Phi$ be a $C^1$-diffeo\-morphism on $\Ban$ and suppose that $\Phi(\Dcal)=\Dcal$ and that $\Phi$ is symplectic. Let $F\in  \JDif$ and let $X_F$ be its Hamiltonian vector field. (See Definition~\ref{def:Jdif}~(iv)). Then, $F\circ \Phi\in \JDif$ and, for all $u\in \Dcal$
\begin{equation}\label{eq:hampushbis}
D_u\Phi(X_{F\circ\Phi}(u))=X_F(\Phi(u)).
\end{equation}
Moreover, for all $t\in\R$,
\begin{equation}\label{eq:hamintertwinebis}
\Phi\circ\Phi_t^{F\circ\Phi}\circ\Phi^{-1}=\Phi_t^{F}.
\end{equation}
In particular, if $F\circ\Phi=F$, then $\Phi$ commutes with $\Phi_t^F$, for all $t\in\R$. Finally, if $F\in C^1(\Ban, \R)$ and if $\Phi$ commutes with $\Phi_t^F$, for all $t\in\R$, then there exists $c\in\R$ so that $F\circ\Phi=F+c$. 
\end{lemma}
\begin{proof} The proof is very close to the one of Lemma~\ref{lem:composeF}. It gives a good illustration of the technical difficulties associated with the domain $\Dcal$. Since $F\in\JDif$ and since $\Phi\in C^1(E, E)$ and leaves $\Dcal$ invariant, one can compute, for all $u\in\Dcal$ and $v\in E$,
$$
D_u(F\circ \Phi)(v)=D_{\Phi (u)}FD_u\Phi(v)=\left[\Jcal X_F(\Phi(u))\right]D_u\Phi(v)=-\left[\Jcal D_u\Phi(v)\right](X_F(\Phi(u))).
$$
Since $\Phi$ is symplectic, this yields
$$
D_u(F\circ \Phi)(v)=-\left[\Jcal v\right]([D_u\Phi]^{-1}(X_F(\Phi(u)))=\left[\Jcal[D_u\Phi]^{-1}(X_F(\Phi(u))\right](v).
$$
This shows $D_u(F\circ \Phi)\in\Rcal_\Jcal$ and that $X_{F\circ \Phi}(u)=[D_u\Phi]^{-1}(X_F(\Phi(u))$, for all $u\in\Dcal$.
Finally, considering for each $u\in\Dcal$ the strongly differentiable curve $t\in\R\to \Phi^{-1}\circ \Phi_t^F\circ \Phi\in E$, one checks readily that it is the flowline  of $X_{F\circ \Phi}$ with initial condition $u$, which concludes the proof. \end{proof}
 The point here is that we suppose $\Phi$ to be a symplectic transformation. As we just saw, that is a strong assumption. In practice, to avoid the difficulties just mentioned,  we will always \emph{assume} that the symmetry group of the system under consideration acts with symplectic transformations. Since the latter are often linear, that they are symplectic can then be checked through a direct computation. 
We finally point out that,  if one wanted to exploit the presence of a formal constant of the motion with a nonlinear flow, such as in completely integrable systems, it could in general be difficult to prove it acts symplectically and commutes with the dynamics. This, in turn, makes it difficult to exploit such formal constants of the motion in the stability analysis that is our main interest here.

\subsection{Symmetries and Noether's Theorem}\label{ss:symconstmotion}
When dealing with a symplectic Banach triple, the appropriate type of group action to consider is the following. 
\begin{definition}\label{def:globhamactioninf} Let $(\Ban, \Dcal, \Jcal)$ be a symplectic Banach triple. Let $G$ be a Lie group and $\Phi: (g,x)\in G\times \Ban\to \Phi_g(x)\in \Ban,$ an action of $G$ on $\Ban$. We will say $\Phi$ is a globally Hamiltonian action\index{globally Hamiltonian action} if the following conditions are satisfied:
\begin{enumerate}[label=({\roman*})]
\item For all $g\in G$, $\Phi_g\in C^1(E,E)$ is symplectic.
\item For all $g\in G$, $\Phi_g(\Dcal)=\Dcal$.
\item For all $\xi\in\frak{g}$, there exists $F_\xi\in C^1(\Ban, \R)\cap\JDif$ such that $\Phi_{\exp(t\xi)}=\Phi_t^{F_\xi}$, and the map $\xi\to F_\xi$ is linear. 
\end{enumerate}
\end{definition}
This definition reduces to Definition~\ref{def:globhamaction} in the Appendix, for finite dimensional spaces $\Ban$: in that case $\Dcal=E$ and the restriction that $F\in\JDif$ is superfluous. We can now state the version of Noether's Theorem\index{Noether's Theorem} that we need. It links the invariance group of Hamiltonian dynamics to constants of the motion and is to be compared to the finite dimensional version given in the appendix (Theorem~\ref{thm:nother2}). As in~\eqref{eq:basisdecompmoment}, we will identify $
 \frak g$ and $\frak g^*$ with $\R^m$ and view $F$ as a map $F:E\to\R^m$ (See~\eqref{eq:momentmap1}). This allows us to write
$$
F_\xi=\xi\cdot F,
$$
where $\cdot$ refers to the canonical inner product on $\R^m$. 

\begin{theorem} \label{thm:nother3} Let $(\Ban, \Dcal, \Jcal)$ be a symplectic Banach triple.  Let $G$ be a Lie group and $\Phi$ a globally Hamiltonian action of $G$ on $\Ban$. Let $H\in C^1(\Ban,\R)\cap \JDif$ and let $\Phi_t^H$ be the corresponding Hamiltonian flow. Suppose that
\begin{equation}\label{eq:haminvariantbis}
\forall g\in G, \quad H\circ \Phi_g=H.
\end{equation}
Then:
\begin{enumerate}[label=({\roman*})]
\item For all $\xi\in\frak g$, $\{H, F_\xi\}=0.$
\item For all $t\in\R$, $F_\xi\circ \Phi_t^H=F_\xi$. 
\item $G$ is an invariance group\footnote{See Definition~\ref{def:symgroup}} for $\Phi_t^H$. 
\end{enumerate}
\end{theorem}
This is an immediate consequence of Theorem~\ref{thm:conservation} and Lemma~\ref{lem:composeFbis}. 
In the applications, the result is used as follows. The action $\Phi$ of $G$ is simple and well known. It is then easy to check~\eqref{eq:haminvariantbis} directly. One then concludes that (ii) and (iii) hold, which are the important pieces of information for the further analysis. In particular, the level surfaces $\Sigma_\mu$, defined in~\eqref{eq:levelsurface} are invariant under the dynamics $\Phi_t^H$.  Examples are given in the next section. The result in~\cite{chm} that is closest in spirit to our Theorem~\ref{thm:nother3} is Theorem~2 of Section~6.2. 
\begin{remark} For the statements of this section, we could have taken $H, F\in C(E,\R)$ rather than $H, F\in C^1(E,\R)$, but in applications, it is more convenient to take them to be $C^1$, as we will see in the next section. 
\end{remark}

\subsection{Linear symplectic flows}\label{ss:linflows}
Since invariance groups often act linearly on the symplectic Banach space $(\Ban, \Jcal)$, and since the nonlinear dynamical flows studied often are perturbations of linear ones, it is important to have a good understanding of linear symplectic flows. Their study also sheds some light on the various technical difficulties mentioned above, and in particular on the role of the domain $\Dcal$, the definition of Hamiltonian flow we adopted, etc. 

Proposition~\ref{prop:linsympl} below (which corresponds to Theorem 2 in Section 2.3 of~\cite{chm}) characterizes all strongly continuous linear symplectic one-parameter groups on a symplectic Banach space in terms of their generators. We adopt the following notation. Given a strongly continuous group of linear transformations on $\Ban$, we denote its generator by $A$, with domain $\Dcal(A)$. By the Hille-Yosida theorem, we then know that $t\in\R\to u(t)=\Phi_tu\in \Ban$ satisfies
\begin{equation}\label{eq:HillYos}
\dot u(t)=Y_A(u(t)), 
\end{equation}
provided $u\in \Dcal(A)$, where we introduced the vector field 
$$
Y_A: u\in \Dcal(A)\subset \Ban \to Au\in \Ban.  
$$
Note that $Y_A$ is not continuous if $A$ is an unbounded operator. Clearly,  the $\Phi_t$ form a dynamical system as defined in Section~\ref{s:dynamics}. We introduce the function
$$
H_A:u\in\Dcal(A)\to H_A(u)=\tfrac12\omega_{\!\Jcal}(Au,u)\in\R.
$$
Observe that $H_A$ admits directional (or G\^ateaux) derivatives $\delta_uH_A(v)$, for all $u,v\in\Dcal(A)$: 
\begin{eqnarray*}
\delta_uH_A(v)&=&\lim_{t\to0}\frac1{t}\left(H_A(u+tv)-H_A(u)\right)\\
&=&\frac12\left(\omega_{\!\Jcal}(Av, u)+\omega_{\!\Jcal}(Au, v)\right).
\end{eqnarray*}
Nevertheless, if $A$ is an unbounded operator, $H_A$ is not continuous since, for all $u, w\in \Dcal(A)$
$$
H_A(u+w)-H_A(u)=\omega_\Jcal(Au, w) + \omega_\Jcal(Aw, u) +\omega_\Jcal(Aw, w)
$$
and the last term in particular does not necessarily converge to $0$ as $w\to0$ in the topology of $\Ban$. It follows that, a fortiori, $H_A$ is not Fr\'echet differentiable. 
\begin{proposition}\label{prop:linsympl}
Let $(\Ban, \Jcal)$ be a symplectic vector space. Let $\Phi_t$ be a strongly continuous one-parameter group of bounded linear operators on $\Ban$. Let $(A, \Dcal(A))$ be the generator of $\Phi_t$. Then the following are equivalent.
\begin{enumerate}[label=({\roman*})]
\item The $\Phi_t$ are symplectic, \emph{i.e.} $\omega_\Jcal( \Phi_tu, \Phi_t v)=\omega_\Jcal( u, v)$ for all $u,v\in \Ban$;
\item  For all $u,v\in\Dcal(A)$, 
$$
\omega_\Jcal(Au, v)=-\omega_\Jcal(u, Av);
$$
\item For all $u\in\Dcal(A)$, one has
\begin{equation}\label{eq:Hamfieldlinear}
\Jcal Y_A(u)=\delta_uH_A\in \Ban^*.
\end{equation}
\end{enumerate}
In this case, $\delta_uH_A(v)=\omega_\Jcal( Au, v)$, $H_A(\Phi_tu)=H_A(u)$ for all $u\in \Dcal(A)$ and for all $t\in\R$.
\end{proposition}
\begin{proof} 
The three equivalences are obvious. To prove $H_A$ is a constant of the motion, it suffices to remember that the Hille-Yosida theorem implies $A\Phi_t u=\Phi_t Au$ provided $u\in\Dcal(A)$. 
\end{proof}
In other words, when the $\Phi_t$ are symplectic, the equation of motion~\eqref{eq:HillYos} can be rewritten
\begin{equation}\label{eq:hameqweak}
\Jcal \dot u(t)=\delta_{u(t)}H_A,
\end{equation}
which is to be compared to~\eqref{eq:hameqmotionJbis}. Clearly, the symplectic linear flows considered here are NOT Hamiltonian in the sense of Definition~\ref{def:hamflow}.  Still,~\eqref{eq:hameqweak} gives meaning to the idea that in infinite dimension as well, linear strongly continuous symplectic flows are of ``Hamiltonian nature,'' with a quadratic Hamiltonian. 
Moreover, the Hamiltonian $H_A$ is a constant of the motion for the flow $\Phi_t$. But note that, whereas in~\eqref{eq:constants}, the conservation of energy holds for all $u\in\Ban$, this makes no sense here, since $H_A$ is only defined on $\Dcal(A)$. 

Generally, because of the appearance of the G\^ateaux derivative rather than a Fr\'echet differential in the right hand side, it turns out that the above formulation is inadequate for various reasons. For example, the absence of a chain rule for G\^ateaux derivatives prevents one from computing derivatives such as $\frac{\rd}{\rd t}H_A(u(t))$ directly to prove $H_A$ is constant along the motion. In fact, in the proof above, this result is proven using the Hille-Yosida theorem, and without computing a derivative at all. This approach cannot work for  nonlinear flows of course. Similar problems arise when dealing with other constants of the motion than the Hamiltonian himself, even in the linear case, due to various  domain questions and the complications in defining commutators. Finally, for our purposes, we need to restrict the motion to the level sets of the constants of the motion, and to use their manifold structure. This requires sufficient smoothness, a property not guaranteed at all by G\^ateaux differentiability alone. Again, as pointed out before, an approach to the resolution of these technical difficulties other than the one chosen here can be found in~\cite{chm}.

In applications to PDE's, the function spaces that occur naturally are often complex Hilbert spaces. To make the link with Hamiltonian dynamics, one then proceeds as follows. Let $\Hcal$ be a complex Hilbert space and let us write $\langle \cdot, \cdot \rangle$ for its inner product. 
First, it is clear that $\Hcal$ is a real Hilbert space for the real inner product defined by Re$\langle \cdot, \cdot\rangle$, which induces the same topology on $\Hcal$ as the original inner product since both inner products have the same associated norm. Let us write $\Ban$ for this real Hilbert space. We now identify $\Ban^*$ with $\Ban$ using the corresponding Riesz isomorphism. Note that this is not the same as identifying $\Hcal^*$ with $\Hcal$ through the Riesz isomorphism associated to $\langle\cdot, \cdot\rangle$ and that there is no natural identification between $\Hcal^*$ and $\Ban^*$ as sets:  each non-zero element of $\Hcal^*$ necessarily takes complex values, whereas the elements of $\Ban^*$ take real values only. 

On the real Hilbert space $\Ban$, one checks readily that 
$$
\omega(u, v)=\mathrm{Im}\langle u, v\rangle\in \R
$$
defines a strong symplectic form. Note in particular that $\omega$ is real bilinear, but not complex bilinear. To identify the corresponding symplector $\Jcal: \Ban\to\Ban$ in a convenient manner\footnote{We identified $E^*$ with $E$, so the symplector can be seen as a map from $E$ to $E$.}, one proceeds as follows:
$$
\omega(u, v)=\mathrm{Re}\,\langle iu, v\rangle
$$
so that $\Jcal u=iu$. The reader should not let itself be confused by the fact that we write $iu$, while considering $u$ as an element of the real vector space $\Ban$. The way to see this is as follows: the real vector space $\Ban$ is, as a set, identical to $\Hcal$. And on $\Hcal$, multiplication by $i$ is well defined and actually an isometric complex linear map. So multiplication by $i$ is well defined on $\Ban$ as an isometric real linear map. 

To sum up, we showed how to associate to a complex Hilbert space $(\Hcal, \langle\cdot, \cdot\rangle)$ a real Hilbert space $(\Ban, \langle\cdot, \cdot\rangle_E)$ with symplectic structure
$$
\omega(u,v)=\langle \Jcal u, v\rangle_E, \quad \Jcal u=iu.
$$ 
Now let us return to the linear symplectic flows. Suppose $B$ is a self-adjoint operator on $\Hcal$, with domain $\Dcal(B)$. Then $U_t=\exp(-iBt)$ is a strongly continuous one-parameter group of unitaries\footnote{By Stone's theorem, every strongly continuous one parameter group of unitaries is of this form.} . The corresponding Hille-Yosida generator is $A=-iB$, with $\Dcal(A)=\Dcal(B)$. Clearly, each $U_t$ is a symplectic transformation on $\Ban$ with the symplectic form $\omega$. We are therefore in the setting of Proposition~\ref{prop:linsympl} and 
\begin{equation}\label{eq:expvalue}
H_A(u)=\frac12\langle u, Bu\rangle.
\end{equation}
It turns out that in the applications we have in mind, the one parameter subgroups of the symmetry group $G$ act with such unitary groups on the relevant Hilbert space $\Hcal$. But within this framework, as we pointed out above, the $U_t$ are NOT Hamiltonian flows. To remedy this situation, one can, and we will, proceed along the following lines. First remark that the function $H_A$ above is $C^1$ if we view it as a function on the Banach space $\Ban_B$ obtained by considering on $\Dcal(|B|^{1/2})$ the graph norm. And that the flow $U_t$ is strongly differentiable on $\Dcal:=\Dcal(|B|^{3/2})$, viewed as a subset of $\Ban_B$. So now we are in the setting of Definition~\ref{def:hamflow}, and $U_t$ is a Hamiltonian flow on $\Ban_B$, on which $\Jcal$ still defines a weak symplector.  
The trouble with this reformulation so far is that now the Banach space $\Ban_B$ and the domain $\Dcal$ depend on $B$. If the symmetry group is multi-dimensional, it will have several generators, and we need a common domain and Banach space on which to realize them all as Hamiltonian flows. We will see several examples where this formalism is implemented.

In practice, very often, $\Hcal=\Kcal^\C=\Kcal\oplus i\Kcal$, where $\Kcal$ is a real Hilbert space. One has $u=q+ip\in\Hcal$ with $q,p\in\Kcal$. Then, clearly $\Ban=\Kcal\times\Kcal$ with its natural Hilbert space structure. Moreover, identifying $u\in\Hcal$ with $(q,p)\in \Kcal\times\Kcal$, clearly $\Jcal (q,p)=(-p,q)$ and we are back to the examples of symplectors given in Section~\ref{ss:symplectors}.


\subsection{Hamiltonian PDE's: examples}\label{ss:hampde}

In this section we give some examples of PDE's generating Hamiltonian flows in the sense of Definition \ref{def:hamflow}.

Let $\Ban=H^1(\R^d,\C)$, $\Dom=H^3(\R^d,\C)$ and consider the nonlinear Schr\"odinger equation 
	\begin{equation}
	\label{eq:nlspurepowerham}
	\left\{
	\begin{aligned}
		&i\partial_t u(t,x) +\Delta u(t,x) +\lambda |u(t,x)|^{\sigma-1} u(t,x)=0\\
		&u(0,x)=u_0(x)
	\end{aligned}
	\right.
	\end{equation}
introduced in Section \ref{s:dynsysexamples}, defined on $\R^d$, $d=1,2,3$. For $d=1$, suppose that $3\le \sigma<+\infty$ in the defocusing case and $3\le \sigma<5$ in the focusing case. In dimension $d=2,3$, consider only the defocusing case and assume $3\le \sigma<1+\frac{4}{d-2}$. Let $\Phi_t^X:\Ban\to\Ban$ be the global flow defined in \eqref{eq:schrodflow}. Recall that the existence of $\Phi_t^X$ is ensured by Theorem \ref{thm:nlsglobalflowH1} and, thanks to Theorem \ref{thm:nlsregularityH2Hm}, $\Phi_t^X(\Dom)=\Dom$ for all $t\in \R$.

Our purpose is to show that Equation \eqref{eq:nlspurepowerham} is the Hamiltonian differential equation associated to the function $H$ defined by \eqref{eq:nlshamiltonianpower} and $\Phi_t^X=\Phi^H_t$ for all $t\in \R$.

As explained in the end of Section \ref{ss:linflows}, we usually identify $u=q+ip \in H^s(\R^d,\C)$ with $(q,p)\in H^s(\R^d,\R)\times H^s(\R^d,\R)$ for all $s\in \R$. Hence, let $(\Ban, \Dom, \Jcal)$ be the symplectic Banach triple given by 
\begin{align*}
 	&\Ban=H^1(\R^d,\R)\times H^1(\R^d,\R),\\
 	&\Dom=H^3(\R^d,\R)\times H^3(\R^d,\R),\\
 	&\Jcal (q,p)=(-p, q),\ \forall (q,p)\in E.
\end{align*} 
Clearly $\Jcal u=iu$ and $\Rcal_\Jcal=\Ban\subset \Ban^*$. Now consider 
\begin{equation*}
H(q,p)=\frac12\left(\|\nabla q\|_{L^2}^2+\|\nabla p\|_{L^2}^2\right)-\frac{\lambda}{\sigma+1}\int_{\R^d}(|q|^2+|p|^2)^{\frac{\sigma+1}{2}}, 
\end{equation*}
and remark that if we write $u=q+ip$ with $(q,p)\in \Ban$, $H(u)=H(q,p)$ is exactly the energy defined in \eqref{eq:nlshamiltonianpower}. 
A straightforward calculation, using the Sobolev embedding theorem, shows that $H\in C^2(\Ban,\R)$. In particular,
\begin{equation*}\label{eq:diffnls}
	D_{(q,p)}H=(-\Delta q, -\Delta p)-\lambda (|q|^2+|p|^2)^{\tfrac{\sigma-1}{2}}(q,p) \in \Ban^*
\end{equation*}
which can be written as
\begin{equation*}\label{eq:diffnlsbis}
	D_{u}H=-\Delta u-\lambda |u|^{\sigma-1}u
\end{equation*}
in terms of $u=q+ip$. Next, using the fact that the Sobolev space $H^3(\R^d)$ is an algebra for $d=1,2,3$, we have $DH(\Dcal)\subset \Rcal_\Jcal$ so that $H$ has a $\Jcal$-compatible derivative on $\Dom$.

Moreover, the curve $(q(t),p(t))=\Phi^X_t(q,p)$ is the unique solution to 
\begin{equation*}
	\Jcal (\dot q(t),\dot p(t))= (-\Delta q, -\Delta p)-\lambda (|q|^2+|p|^2)^{\tfrac{\sigma-1}{2}}(q,p)=D_{(q(t),p(t))}H
\end{equation*}
that is Equation \eqref{eq:hameqmotionJbis}. As a consequence, $\Phi^X$ is a Hamiltonian flow for $H$ in the sense of Definition \ref{def:hamflow}, $\Phi_t^X=\Phi_t^H$ and the nonlinear Schr\"odinger equation \eqref{eq:nlspurepowerham} is a Hamiltonian differential equation.


In Section \ref{s:dynsysexamples}, we prove directly from the equation that $G=\mathrm{SO}(d)\times \R^d\times \R$ with the action defined by \eqref{eq:nlsinvariance} is an invariance group for the dynamics. In general, the action of this group is not globally Hamiltonian. Nevertheless, let us consider the subgroup $\tilde G =\R^d\times \R$ and the restricted action
\begin{align}
	\Phi: \ & \tilde G\times \Ban\to \Ban \notag \\
	&(a,\gamma,u)\to \Phi_{a,\gamma}(u)=e^{i\gamma}u(x-a).  \label{restraction}
\end{align}
For all $g\in \tilde G$, $\Phi_g\in C^1(\Ban,\Ban)$ is symplectic, $\Phi_g(\Dom)=\Dom$ and for all \[
\xi=(\xi_1,\ldots,\xi_d,\xi_{d+1})\in \frak{g}, 
\]
point (iii) of Definition \ref{def:globhamactioninf} is satisfied by taking $F_{\xi_j}=\xi_jF_j$ with
\begin{align}
	\label{eq:defconstantmotionSchr}
	&F_j(u)=-\frac{i}{2}\int_{\R^d}\bar u(x)\partial_{x_j}u(x)\diff x\  \forall j=1,\dots,d,\\
	&F_{d+1}(u)=-\frac{1}{2}\int_{\R^d} \bar u(x) u(x) \diff x.
\end{align}

As a consequence the action $\Phi$ of $\tilde G$  on $\Ban$ is globally Hamiltonian. Moreover, in Section \ref{s:dynsysexamples}, we showed that $H\circ \Phi_g=H$, hence we may apply Theorem \ref{thm:nother3} and conclude that $F_{\xi_j}\circ \Phi_t^H= F_{\xi_j}$ that means that each $F_{j}$ is a constant of the motion.

Finally we show that the action $\Phi: (R,u)\in G\times \Ban\to \Phi_R(u)=u(R^{-1}x)\in \Ban$ of $G=\mathrm{SO}(d)$ on $\Ban$ is not globally Hamiltonian. For simplicity, let us consider $d=2$ and let us identify a matrix $\xi \in \mathrm{so}(2)$ with $\xi \in \R$ 
$$\xi=\begin{pmatrix}
	0 & \xi\\
	-\xi & 0
\end{pmatrix}.$$
Then for each $\xi\in \R$, $\Phi_{\exp(t\xi)}=\Phi_t^{F_\xi}$ with $F_{\xi}=\xi F$ and 
\begin{equation*}
	F(u)=-\frac{i}{2}\int_{\R^d}(x_1\partial_{x_2}-x_2\partial_{x_1})u(x)\bar u(x)\diff x.
\end{equation*}
The issue is that $F$ is not even well-defined on the Banach space $H^1(\R^2)$! 

Finally, let us remark that if we choose $\Dom= H^2(\R^d)\times H^2(\R^d)$, then $DH(\Dcal)\subset L^2(\R^d)\times L^2(\R^d)\not\subset \Rcal_\Jcal=H^1(\R^d)\times H^1(\R^d)$ and $H$ does not have a $\Jcal$-compatible derivative for this new choice of $\Dom$. In the same way, if we take $\Ban=L^2(\R^d) \times L^2(\R^d)$ and $\Dcal=H^1(\R^d)\times H^1(\R^d)$, the same function $H$ is not even continuous.

We point out that the Manakov equation can be treated similarly. In that case, in addition to the momentum, there are four constants of the motion associated to the $U(2)$ symmetry. 

Next, let $(\Ban, \Dom, \Jcal)$ be the symplectic Banach triple given
\begin{align*}
 	&\Ban=H^1(\R^d,\R)\times L^2(\R^d,\R),\\
 	&\Dom=H^2(\R^d,\R)\times H^1(\R^d,\R),\\
 	&\Jcal (q,p)=(-p, q),\ \forall (q,p)\in E.
\end{align*} 
and consider the nonlinear wave equation 
\begin{equation}
	\label{eq:nlwpurepowerham}
	\left\{
	\begin{aligned}
		&\partial^2_{tt} u(t,x) -\Delta u(t,x) +\lambda |u(t,x)|^{\sigma-1} u(t,x)=0\\
		&u(0,x)=u_0(x), \partial_t u(0,x)=u_1(x)
	\end{aligned}
	\right.
\end{equation}
introduced in Section \ref{ss:nlw}, defined on $\R^d$, $d=1,2,3$. Suppose $\lambda>0$ and $\sigma$ an odd integer such that $3\le \sigma<+\infty$ in dimension $d=1$ and $3\le \sigma<1+\frac{4}{d-2}$ for $d=2,3$. Let $\Phi_t^X:\Ban\to\Ban$ the global flow defined in \eqref{eq:flownlw}. Thanks to the persistence of regularity, we have $\Phi_t^X(\Dom)=\Dom$  for all $t\in\R$ (see Section \ref{ss:nlw}).

As before, our purpose is to show that Equation \eqref{eq:nlwpurepowerham} is the Hamiltonian differential equation associated to the function $H$ defined by \eqref{eq:nlwhamiltonian} and $\Phi_t^X=\Phi_t^H$ for all $t\in \R$.

First of all, note that $\Rcal_\Jcal=L^2(\R^d)\times H^1(\R^d)\subset \Ban^*=H^{-1}(\R^d)\times L^2(\R^d)$. Next, consider 
\begin{equation*}
H(q,p)=\frac12\left(\|\nabla q\|_{L^2}^2+\|p\|_{L^2}^2\right)+\frac{\lambda}{\sigma+1}\int_{\R^d}(|q|)^{{\sigma+1}}, 
\end{equation*}
and remark that if we write $q=u$ and $p=\partial _t u$ with $(q,p)\in \Ban$, $H(u)=H(q,p)$ is exactly the energy defined in \eqref{eq:nlwhamiltonian}. As for the nonlinear Schr\"odinger equation, a straightforward calculation, using the Sobolev embedding theorem, shows that $H\in C^2(\Ban,\R)$. 
In particular,
\begin{equation*}\label{eq:diffnlw}
	D_{(q,p)}H=(-\Delta q+\lambda|q|^{\sigma-1}q , p) \in \Ban^*.
\end{equation*}
Next, using the fact that the Sobolev space $H^2(\R^d)$ is an algebra for $d=1,2,3$, we have $DH(\Dcal)\subset \Rcal_\Jcal$ so that $H$ has a $\Jcal$-compatible derivative on $\Dom$.

Moreover, the curve $(u(t),\partial_t u(t))=\Phi^X_t(u(0),\partial_t u(0))$ is the unique solution to \eqref{eq:nlwpurepowerham}. As a consequence, using $u=q$ and $\partial _t u=p$, we have that $(q(t),p(t))=\Phi^X_t(q,p)$ is the unique solution to
\begin{equation*}
	\Jcal (\dot q(t),\dot p(t))= (-\Delta q+\lambda|q|^{\sigma-1}q ,p)=D_{(q(t),p(t))}H,
\end{equation*}
that is, Equation \eqref{eq:hameqmotionJbis}. 
Finally, if $(q,p)\in \Dom$, the curve $t\to \Phi^H_t(q,p)\in C(\R,\Dom)\cap C^1(\R,E)$. As a consequence, $\Phi^X$ is a Hamiltonian flow for $H$ in the sense of Definition \ref{def:hamflow}, $\Phi_t^X=\Phi_t^H$ and the nonlinear wave equation \eqref{eq:nlwpurepowerham} is a Hamiltonian differential equation.


\section{Identifying relative equilibria}\label{s:identifyreleq}
We now dispose of the necessary tools that will allow us to characterize the relative equilibria of Hamiltonian systems with symmetry and that will yield the candidate Lyapunov function\index{Lyapunov function} to study their stability.  Before stating the main result (Theorem~\ref{thm:relequicritical}), we recall some of the terminology used below, but refer to the appendices for details. First, for $\mu\in\frak g^*$, we have (see~\eqref{eq:stabilizer}),
$$
G_\mu=\{g\in G\mid \mathrm{Ad}_g^* \mu=\mu\};
$$ 
 $\frak g, \frak g_\mu$ are the Lie algebras of $G$ and $G_\mu$ respectively, and $\frak g^*, \frak g^*_\mu$ their duals. We always identify $\frak g^*$ with $\R^m$ (see~\eqref{eq:gstaridentif}). Hence, if $\Phi$ is a globally Hamiltonian action, we think of its momentum map as a map $F:\Ban\to \R^m$ and define, for all $\mu\in\R^m$,
$$
\Sigma_\mu=\{ u\in \Ban\mid F(u)=\mu\}.
$$
We then know  from Proposition~\ref{eq:reduction} that $G_\mu=G_{\Sigma_\mu}$ provided the momentum map is Ad$^*$-equivariant.

\begin{theorem}\label{thm:relequicritical}
Let $(\Ban, \Dcal, \Jcal)$ be a symplectic Banach triple. Let $H\in C^1(\Ban, \R)\cap \JDif$ and suppose $H$ has a Hamiltonian flow $\Phi_t^H$. Let furthermore $G$ be a Lie group, and $\Phi$ a globally Hamiltonian action on $\Ban$ with Ad$^*$-equivariant momentum map $F$. Suppose that,
\begin{equation}
\forall g\in G,\quad H\circ \Phi_g=H.\label{eq:HGinv}
\end{equation}
\begin{enumerate}[label=({\roman*})]
\item Then $G$ is an invariance group for $\Phi_t^H$.
\item Let $u\in \Ban$ and let $\mu= F(u)\in\R^m\simeq\frak g^*$. Consider the following statements: 
\begin{enumerate}[label=({\arabic*})]
\item  $u$ is a relative $G$-equilibrium.
\item $u$ is a relative $G_\mu$-equilibrium.
\item There exists $\xi\in \frak{g}_\mu$ so that, for all $t\in\R$, 
\begin{equation}
\Phi_t^H(u)=\Phi_{\exp(t\xi)}(u).
\end{equation} 
\item  There exists $\xi\in \frak g_\mu$ so that 
\begin{equation}\label{eq:findrelequi}
D_uH-\xi\cdot D_u F=0.
\end{equation}
\item There exists $\xi\in \frak g$ so that 
\begin{equation}\label{eq:findrelequibis}
D_uH-\xi\cdot D_u F=0.
\end{equation}
\end{enumerate}
Then  (1) $\Leftrightarrow$ (2) $\Leftarrow$ (3).\\
 If $u\in\Dcal$, then (1) $\Leftrightarrow$ (2) $\Leftarrow$ (3) $\Leftrightarrow$ (4) $\Leftrightarrow$ (5).\\
  If in addition, $\mu$ is a regular value of $F$ (See Definition~\ref{def:levelsurfregular}), then\\
   (1) $\Leftrightarrow$ (2) $\Leftarrow$ (3) $\Leftrightarrow$ (4) $\Leftrightarrow$ (5) $\Leftrightarrow$ (6), where (6) is the statement:
\begin{itemize}
\item[(6)] $u$ is a critical point of $H_\mu$ on $\Sigma_\mu$, where $H_\mu=\restr{H}{\Sigma_\mu}.$
\end{itemize}
In addition, $\xi$ is then unique.
\end{enumerate}
\end{theorem}

That (1) is equivalent to (2)  is a particular feature of Hamiltonian systems. In fact, its statement makes no sense outside of the Hamiltonian setting.  It implies that, if $u$ is a $G$-relative equilibrium, it is automatically a relative equilibrium for the \emph{smaller} group $G_\mu$. So the relevant invariance group depends on the point $u$ through the value $\mu=F(u)$ of the constants of the motion at $u$. This is important since, as we will see in Section~\ref{s:orbstabproof}, one then ends up showing $u$ is $G_\mu$-orbitally stable, which is a stronger result than $G$-orbital stability. We already saw examples of this mechanism in Section~\ref{s:spherpotstab}. The proof of the equivalence between (1) and (2), although very simple, uses the subtle relations between constants of the motion and symmetries for Hamiltonian systems explained in the previous section. 

 For our purposes, the most interesting information obtained in this result is the observation that if $u\in\Dcal$  satisfies~\eqref{eq:findrelequi}, sometimes referred to in the PDE literature as ``the stationary equation''\index{stationary equation}, then it is a relative equilibrium.  And that, if $\mu$ is a regular value of $F$, those solutions are precisely the critical values of $H_\mu$.  This means that, given a Hamiltonian system with symmetries,  one can find relative equilibria by looking for critical points of the Hamiltonian $H$ restricted to the surfaces $\Sigma_\mu$. In practice, this can be done concretely by solving~\eqref{eq:findrelequibis}, which in applications to Hamiltonian PDE's often takes the form of a stationary PDE in which $\xi$ is treated as a (vector  valued) parameter. Examples are given in the following sections. See also Section~\ref{s:spherpotstab} for examples in finite dimension. 
 
 One immediately suspects that the Lagrange theory of multipliers for the study of constrained extrema should be of relevance here. This is indeed the case: introducing, on $E$, the Lagrange function 
\begin{equation}\label{eq:lagfunction}
\forall v\in\Ban,\quad \Lcal(v)=H(v)-\xi\cdot F(v),
\end{equation}
one sees that \eqref{eq:findrelequibis} expresses the vanishing of its first variation at $u$: $D_u\Lcal=0$. Here, $\xi\in\frak g\simeq \R^m$ plays the role of a Lagrange multiplier.
From the experience gained with the examples given so far, one suspects that, to show $u$ is a stable relative equilibrium, one could try proceeding in two steps. First, show $u$ is not just a critical point, but actually a local minimum of $H_\mu$ by studying the second variation of the Lagrange function $\Lcal$
on $\Sigma_\mu$. Next, use the Lagrange function as Lyapunov function in the proof of stability. Indeed, $u\in\Sigma_\mu$ is a local minimum of $H_\mu$ if and only if 
$$
\exists\rho>0, \forall v\in\Sigma_\mu, \quad \mathrm{d}(v,u)\leq\rho \Rightarrow 
H_\mu(v)-H_\mu(u)\geq 0,
$$
which is equivalent to 
$$
\exists\rho>0, \forall v\in\Sigma_\mu, \quad \mathrm{d}(v,u)\leq\rho \Rightarrow 
\Lcal(v)-\Lcal(u)\geq 0,
$$
since $F$ is constant on $\Sigma_\mu$. This is clearly the strategy used in the proofs of Section~\ref{s:spherpotstab}. We will see in Section~\ref{s:orbstabproof} how to implement it in a general setting and give examples from the nonlinear Schr\"odinger equation in Sections~\ref{s:nlsetorus1d} and~\ref{curves.sec}. This is the approach that goes by the name of \emph{energy-momentum method}.
\begin{proof} 
\begin{itemize}
\item[(i)] This is an immediate consequence of Theorem~\ref{thm:nother3}~(iii).
\item[(ii)]  (1) $\Leftrightarrow$ (2). If $u$ is a relative $G$-equilibrium, then there exists, for each $t\in\R$, $g(t)\in G$ so that $\Phi_t^H(u)=\Phi_{g(t)}(u)$. Since $u\in \Sigma_\mu$, so is $\Phi_t^H(u)$, since $F$ is a constant of the motion for $H$, by Theorem~\ref{thm:nother3}~(ii). Hence 
$$
\mu=F(u)=F(\Phi_t^H(u))=F(\Phi_{g(t)}(u))=\mathrm{Ad}^*_{g(t)}\mu.
$$
It follows that $g(t)\in G_\mu$, which concludes the argument. The reverse implication is obvious. \\
(3) $\Rightarrow$ (2). Obvious from the definition.\\
Now suppose $u\in \Dcal$. \\
(3) $\Leftrightarrow$ (4). Suppose (3) holds. Since $u\in\Dcal$, this implies that $\Jcal^{-1}D_uH=\Jcal^{-1}D_u(\xi\cdot F)$, which implies~(4). Now suppose (4) holds.  Since $u\in\Dcal$ and since $H\circ \Phi_t^H=H$ and $F_\xi\circ \Phi_t^H=F_\xi$ by Theorem~\ref{thm:nother3}~(ii), we have, for all $t\in\R$,
$$
D_{\Phi_t^H(u)}HD_u\Phi^H_t=D_uH, \quad D_{\Phi_t^H(u)}(\xi\cdot F)D_u\Phi^H_t=D_u(\xi\cdot F).
$$
Writing $u(t)=\Phi_t^H(u)$, this yields
$
D_{u(t)}H=D_{u(t)}(\xi\cdot F)
$
so that 
$\Jcal\dot u(t)=D_{u(t)}(\xi\cdot F)$, which shows $t\to u(t)$ is a flow line of the Hamiltonian $\xi\cdot F$, with initial condition $u$. Since the latter is unique, we find $u(t)=\Phi^{\xi\cdot F}_t(u)$, which concludes the argument since $\Phi_t^{\xi\cdot F}=\Phi_{\exp(t\xi)}$ (See Definition~\ref{def:globhamactioninf}~(iii)).\\
(4) $\Leftrightarrow$ (5). We only need to establish that (5) implies (4). As above, (5) implies $u(t)=\Phi_{\exp(t\xi)}$. Hence
$$
\mathrm{Ad}^*_{\exp(t\xi)}\mu=\mathrm{Ad}^*_{\exp(t\xi)}F(u)=(F\circ \Phi_{\exp(t\xi)})(u)=F(u(t))=F(u)=\mu,
$$
since $F_i\circ \Phi_t^H=F_i$. Hence $\xi\in\frak g_\mu$. 

 Now suppose in addition $\mu$ is a regular value of $F$. \\
 (4) $\Leftrightarrow$ (6).  We remark that, since $\mu$ is a regular value of $F$, $\Sigma_\mu$ is a co-dimension $m$ submanifold of $\Ban$ and (see~\eqref{eq:deftangentspace})
$$
T_u\Sigma_\mu=\{v\in\Ban \mid D_uF(v)=0\}.
$$
Hence clearly (4) implies (6). Conversely, suppose $D_uH$ vanishes on $T_u\Sigma_\mu$. Since $\mu$ is a regular value of $F$, we know that $D_uF$ is onto $\R^m$. Let $W$ be a subspace of $\Ban$ complementary to $T_u\Sigma$, so that $\Ban=T_u\Sigma\oplus W$. It follows dim$W=m$ and that the $m$ one-forms $D_uF_i\in W^*$, $i=1,\dots m$ form a basis of $W^*$.  Consequently, the restriction of $D_u H$ to $W$ can be written uniquely as $D_uH=\sum_{i=1}^m \xi_i D_uF_i=D_u(\xi\cdot F)$. Since both sides vanish on $T_u\Sigma_\mu$, (4) follows. 
\end{itemize}
\end{proof}

We conclude this section with two  technical remarks that can be skipped in a first reading.  
\begin{remark}\label{rem:criticalpoints}
We have seen that (3) implies (2). Under suitable technical conditions, the reverse is also true. This can be understood as follows. If $u\in\Dcal$ is a $G_\mu$-relative equilibrium 
then, for all $t\in\R$, there exists $g(t)\in G_\mu$ so that $u(t)=\Phi_t^Hu=\Phi_{g(t)}u$. So the curve 
$$
t\in \R\to \Phi_t^H(u)\in G_\mu u:=\{\Phi_g(u)\mid g\in G_\mu\}\subset \Ban
$$ 
is a smooth curve on the group orbit $G_\mu u$. Under appropriate topological conditionson $G_\mu$ and $G_u$ (defined in~\eqref{eq:isotropyu}), and if the action $\Phi$ of the group $G_\mu$ is sufficiently smooth\footnote{See for example Section~4 of \cite{am}, and in particular Corollary~4.1.22.} , this orbit is an immersed submanifold of $\Sigma_\mu$ that can be identified with the homogeneous space $G_\mu/G_u$, and its tangent space at $u$ is therefore 
$$
T_u(G_\mu u)=\{ X_{F_\xi}(u)\mid \xi \in \frak g_\mu\}.
$$
We recall that $X_{F_\xi}$ is the Hamiltonian vector field associated to the function $F_\xi=\xi\cdot F$. Since $X_H(u)=\frac{\rd}{\rd t} \Phi_t^H(u)_{\mid t=0}\in T_u(G_\mu u)$, it follows that there exists $\xi\in \frak g_\mu$ so that
$$
X_H(u) =X_{\xi\cdot F}(u),
$$
which is equivalent to~\eqref{eq:findrelequi} and therefore implies~(3). We refer to~\cite{am, ma} for the detailed argument, in the finite dimensional setting.  We shall not have a need for the implication $(2)\Rightarrow (3)$, but will point out that, ``morally'', there is a one-one relationship between the critical points of $H_\mu$ and the relative equilibria of the Hamiltonian flow~$\Phi_t^H$. 
\end{remark}
\begin{remark}\label{rem:muregular}
What is the role of the condition that $\mu$ be a regular value of $F$? This has several consequences. First, it guarantees that $\Sigma_\mu$ is a co-dimension $m$ submanifold of $\Ban$ and that $T_u\Sigma_\mu=\mathrm{Ker} D_uF$. This is convenient in the further stability analysis, as we will see. Second, if $u\in\Dcal$ and Rank\,$D_uF=m$, then $\xi\in\R^m\simeq \frak g\to \Phi^{\xi\cdot F}_1(u)\in \Ocal_u=Gu\subset\Ban$ is a local immersion and the action is locally free, meaning that the isotropy group $G_u$ of $u$ is discrete. Hence  $\xi\in\frak g_\mu\to \Phi^{\xi\cdot F}_1(u)\in \Ocal_u\cap \Sigma_\mu=G_\mu u \subset\Ban$ is also a local immersion. This observation will be used in Lemma~\ref{lem:stuartlemham} in the next section. If $\mu$ is not regular, various additional technical difficulties arise in the stability analysis of the next section, even in finite dimensional settings,   
where they have been studied in~\cite{lersi, monrod}. 
As an example of such a singular value $\mu$, consider the action of SO$(3)$ on $\R^6$ introduced in Section~\ref{s:spherpot}, on the level set $L(u)=\mu=0$. The corresponding isotropy group $G_\mu$ is SO$(3)$ itself in that case. Its action is not locally free, since $G_u$, for $u=(q,p)$, with $q$ and $p$ parallel, is the copy of SO$(2)$ given by the rotations about the common axis of $q$ and $p$. We will see another example of such a situation when treating the nonlinear Schr\"odinger
equation on the torus in Section~\ref{s:nlsetorus1d}. In both these cases, the ensuing complication is easily dealt with on an ad hoc basis. 
\end{remark}

\section{Orbital stability: an abstract proof}\label{s:orbstabproof}
\subsection{Introduction: strategy}\label{ss:strategy} 
We have seen that in many situations the relative equilibria of Hamiltonian systems with symmetry are precisely the critical points of the restriction $H_\mu$ of the Hamiltonian $H$ to a level surface $\Sigma_\mu$, for some $\mu\in\frak g^*$, of the constants of the motion $F$ associated to the symmetry group via the Noether Theorem. This at once explains why they tend to come in families $u_\mu$, indexed by $\mu$ in some open subset of $\frak g^*\simeq \R^m$. Indeed, considering equation~\eqref{eq:findrelequibis}, it is natural to think of it as an equation in which both $\xi$ and $u$ are unknown. And so, under suitable circumstances, one can hope to find a family of solutions $u_\xi$ of~\eqref{eq:findrelequibis} by letting $\xi$ run through some neighbourhood inside $\frak g$. Typically, as $\xi$ changes, so does $\mu_\xi=F(u_\xi)\in \frak g^*$. Depending on the situation, it may be more convenient to label the solutions by $\mu_\xi$ than by $\xi\in\frak g$. In these notes, we use mostly $\mu$ as a parameter, except in Section~\ref{curves.sec} where $\xi$ is used. The question of the existence of such families of relative equilibria -- a problem related to bifurcation theory -- is studied, in the finite dimensional setting, in~\cite{mon} and~\cite{lersi}. We already saw several examples of this phenomenon and more will be provided in Sections~\ref{s:nlsetorus1d} and~\ref{curves.sec}. 

It remains to see how one can prove the orbital stability of those relative equilibria. The basic intuition is that -- modulo technical problems -- they should be stable if they are not just critical points, but actually local minima of $H_\mu$. To understand the origin of this intuition, recall that, if $u_\mu\in \Sigma_\mu$ is a relative equilibrium of the Hamiltonian dynamics $\Phi_t^H$, then the orbit $G_\mu u_\mu=\{\Phi_g(u_\mu)\mid g\in G_\mu\}$ of $G_\mu$, viewed as an element of the orbit space $\Sigma_\mu/G_\mu$, is a fixed point of the reduced dynamics. And, since $H_\mu$ is invariant under the action of $G_\mu$, it can be viewed as a function on this orbit space. If $H_\mu$ has a local minimum at $u_\mu$, it thus has a local minimum at the orbit $G_\mu u_\mu \in \Sigma_\mu/G_\mu$.  Finally, since $H_\mu$ is a constant of the motion for the reduced dynamics, we are precisely in the situation described in the introduction: $G_\mu u_\mu$ is a fixed point of the reduced dynamics, and $H_\mu$ is a constant of the motion for which $G_\mu u_\mu$ is a minimum.  We can therefore hope to use the Lyapunov method to prove the stability of $G_\mu u_\mu$. To do so, it would suffice to prove a coercive estimate of the type~\eqref{eq:localmin} for $H_\mu$ on $\Sigma_\mu/G_\mu$. 

There are two obvious problems one has to face when trying to implement this strategy. 
First, even if one executes this program, one will have proven only  that $u_\mu$ is orbitally stable with respect to perturbations $v$ of $u_\mu$ with $v\in\Sigma_\mu$. But one would like to prove this is true for arbitrary perturbations $v\in E$. Second, it is difficult to work on the abstract quotient space $\Sigma_\mu/G_\mu$, which, even in finite dimensional systems, but particularly in infinite dimensional ones, may not have a nice topological or differentiable structure, so that analytical tools to prove estimates are not readily available. To deal with both these problems, the idea is to use the theory of constraint minimization and Lagrange multipliers. This has the obvious advantage that one can work in the ambient space $\Ban$, which has the added redeeming feature of being linear. As already outlined in the dicussion following
Theorem~\ref{thm:relequicritical}, it turns out that it is the Lagrange function 
$$
\Lcal_\mu=H-\xi_\mu\cdot F
$$ 
associated to the relative equilibrium $u_\mu$ (see~\eqref{eq:lagfunction}) that plays the role of Lyapunov function in the proofs. In practice, one uses a Taylor expansion to second order of $\Lcal_\mu$ about points on the orbit $G_\mu u_\mu$ and one controls the second derivative of $\Lcal_\mu$ to prove it is a minimum; this in turn gives the necessary coercivity to conclude stability.  The reader will have noticed that the above strategy was worked out in all detail in the simple example of motion in a spherical potential presented in Section~\ref{s:spherpotstab}.

In this section, we will provide  a detailed implementation of the above strategy in the following general setup. We refer to Section~\ref{s:dynamics} for the definitions of the objects used below. 

\vskip0.3cm
\noindent{\bf \textsc{Hypothesis} A}
\begin{enumerate}[label=({\roman*})]
\item $E$ is a Banach space and $\Dcal$ a domain in $\Ban$. 
\item $\Phi_t^X$ is a dynamical system on $\Ban$ with a vector field $X:\Dcal\to E$.
\item $F\in C^2(E,\R^m)$ is a vector of constants of the motion for $\Phi_t^X$ with level surfaces $\Sigma_\mu, \mu\in\R^m$, as in~\eqref{eq:levelsurface}.
\item $\Phi_t^X$ admits an invariance group $G$, with an action $\Phi$ of $G$ on $E$.
\end{enumerate}

Recall that if $\mu$ is a regular value for $F$ then $\Sigma_\mu$ is a co-dimension $m$ submanifold of $\Ban$. In this setting, we consider relative equilibria of the following type. 

\medskip
Let $\mu\in \R^m$. 

\vskip0.2cm
\noindent{\bf \textsc{Hypothesis} B$\mu$} 
\begin{enumerate}[label=({\roman*})]
\item There exists  $u_{\mu}\in \Sigma_{\mu}$ which is a relative equilibrium of the dynamics for the group $G_{\Smu}=\{g\in G\mid \Phi_g\Sigma_\mu=\Sigma_\mu\}$. 
\item  There exists ${\mathcal L_{\mu}}\in C(E,\R)$ which is a $G_{\Sigma_\mu}$- invariant constant of the motion. 
\item There exist $\eta>0, c>0$ so that
\begin{equation}\label{eq:coerciverestricted}
\forall u\in\Ocal_{u_{\mu}}, \forall u'\in \Sigma_{\mu}, \quad \rd(u,u')\leq \eta\Rightarrow \Lcal_{\mu}(u')-\Lcal_{\mu}(u)\geq c\rd^2(u', \Ocal_{u_{\mu}})
\end{equation}
where 
\begin{equation}\label{eq:muorbit}\Ocal_{u_\mu}=\Phi_{G_{\Sigma_\mu}}(u_\mu)=\{\Phi_g(u_\mu)\mid g\in G_{\Sigma_\mu}\}.
\end{equation}
\end{enumerate}
\vskip0.3cm

\noindent Under the above conditions, we say $\Lcal_{\mu}$ is a coercive Lyapunov function\index{Lyapunov function!coercive} on $\Ocal_{u_{\mu}}$ \emph{along} $\Sigma_{\mu}$. If the $G_{\Sigma_\mu}$-action is isometric then it is enough to check~\eqref{eq:coerciverestricted} holds at one single point $u\in\Ocal_{u_\mu}$. It will then hold everywhere, with the same $\eta, c$, as a result of the $G_{\Sigma_\mu}$-invariance of $\Lcal_\mu$. Isometric actions are common in applications and this is one of the places where they provide a simplification. 
For what follows, the power $2$ in the right hand side of~\eqref{eq:coerciverestricted} is of no consequence. One can generalize the definition by replacing the right hand side in~\eqref{eq:coercive} by $f(\rd(u', \Ocal_{u_\mu}))$, for some function $f:\R^+\to \R^+$, $f(0)=0$, $f(d)>0$ if $d>0$. In practice, as we will see below, one gets the lower bound in~\eqref{eq:coerciverestricted} from a Taylor expansion of $\Lcal$, so that the square appears naturally.  We point out that conditions~(ii) and~(iii) in Hypothesis~B$\mu$ imply~(i). Indeed, if $u\in\Ocal_{u_{\mu}}$ and $u'=u(t')$ for small enough $t'$, then~(ii) and~(iii) imply that
$$
0=\Lcal_{\mu}(u(t'))-\Lcal_{\mu}(u)\geq c\rd^2(u(t'), \Ocal_{u_{\mu}}),
$$
so that $u(t')\in\Ocal_{u_\mu}$. Hence the flow $\Phi_t^X$ leaves $\Ocal_{u_\mu}$ invariant and consequently each $u\in\Ocal_{u_\mu}$ is a $G_{\Sigma_\mu}$ relative equilibrium. We have however found it convenient to keep this redundancy in the statement of the hypothesis.

We point out that  Hypotheses A and B$\mu$ are formulated without imposing the dynamical system to be Hamiltonian. Nor do they impose any link between the symmetry group $G$, the constants of the motion $F$ and the Lyapunov function $\Lcal_{\mu}$. 
The first goal of this section is to formulate and prove very general abstract theorems  establishing orbital stability under the above general  assumptions and some extra technical conditions. The first such result, Theorem~\ref{thm:lyapmethod}, is a general version of Proposition~\ref{lem:sphericalfixedpoints}: it imposes a strong coercivity condition, but is nevertheless sometimes of use, as we will see in Section~\ref{s:nlsetorus1d}.  Theorem~\ref{thm:lyapmethodrestricted} and Theorem~\ref{thm:lyapmethodrestrictedmod}  correspond essentially to the first two arguments proposed in the proof of Proposition~\ref{lem:sphercircular}. The proofs of these results are quite simple, as we shall see. These three results show that the essential ingredient in the proof of orbital stability is the coercivity condition in Hypothesis B$\mu$~(iii). 

It therefore remains to understand how to find a Lyapunov function satisfying in particular Hypothesis~B$\mu$~(iii). It is at this point that the Hamiltonian nature of the dynamical system plays an important role. We already saw in Section~\ref{s:identifyreleq} that a candidate Lyapunov function arises naturally in that context. 
We will furthermore show in Proposition~\ref{thm:hessianestimate}  how to obtain the coercivity condition Hypothesis~B$\mu$~(iii) from a lower bound  on the Hessian of the Lyapunov function, in the case of Hamiltonian systems with symmetry. Combining this with Theorem~\ref{thm:lyapmethodrestricted} and Theorem~\ref{thm:lyapmethodrestrictedmod} then yields a complete proof of orbital stability. 

We will end this section with Theorem~\ref{thm:lyapmethodrestrictedLK} which provides a slightly different proof of orbital stability of relative equilibria in Hamiltonian systems, and which is a generalization of the third argument proposed in the proof of Proposition~\ref{lem:sphercircular}. The argument uses Proposition~\ref{thm:hessianestimate} again, but combines it with the construction of an ``augmented'' Lyapunov function.

In applications of the theory developed in this section, the work is therefore reduced to  solving~\eqref{eq:findrelequibis} to identify the relative equilibria, and to proving a suitable  lower bound on the Hessian of the corresponding Lyapunov function. This usually involves non-trivial (spectral) analysis, as one may expect.   A first illustrative example - the orbital stability of plane waves for the nonlinear Schr\"odinger equation on the torus -- is presented in Section~\ref{s:nlsetorus1d}. A widely applicable technique for obtaining the appropriate lower bound on the Hessian is described in~\cite{gssI, gssII}. It is illustrated in Section~\ref{curves.sec} for standing wave solutions of the inhomogeneous nonlinear Schr\"odinger equation in one dimension.

In conclusion, the theorems of this section isolate the ``soft analysis'' part of the proof of orbital stability of relative equilibria  from the more concrete and model dependent  estimates needed to prove coercivity. 

\begin{remark}
We point out that the domain $\Dcal$ of the dynamical system $\Phi_t^X$ appears in Hypothesis~A~(i) and~(ii). As already seen before, it is used in these notes when the system is Hamiltonian to identify the appropriate constants of the motion via Noether's theorem, to construct the Lyapunov function $\Lcal$, and to identify the relative equilibria of the system. If this can be accomplished by some other means, $\Dcal$ is not needed. In fact, for the results of Sections~\ref{ss:simplecase}-\ref{ss:coercivitystability1}-\ref{ss:coercivityhessian} the hypotheses involving $\Dcal$ are not used. For the results of Section~\ref{ss:coercivitystability2}, and notably for Theorem~\ref{thm:lyapmethodrestrictedLK}, they are on the contrary essential.
\end{remark}

\subsection{A simple case}\label{ss:simplecase}
Before turning to the general results, we first formulate and prove a simple orbital stability result, under a stronger coercivity condition than~\eqref{eq:coerciverestricted}. 
\begin{theorem}\label{thm:lyapmethod}
Let Hypotheses A and B$\mu_*$~(i)--(ii) for some $\mu_*\in \R^m$ be satisfied. Let $\Ocal_{u_{\mu_*}}$ be as in \eqref{eq:muorbit}. Suppose there exist $ \eta>0, c>0$ so that 
\begin{equation}\label{eq:coercive}
\forall u\in\Ocal_{u_{\mu_*}}, \forall v\in \Ban, \quad \rd(v,u)\leq \eta\Rightarrow \Lcal_{\mu_*}(v)-\Lcal_{\mu_*}(u)\geq c\rd^2(v, \Ocal_{u_{\mu_*}}).
\end{equation}
Then, all $u\in\Ocal_{u_{\mu_*}}$ are orbitally stable $G_{\Sigma_{\mu_*}}$-relative equilibria. 
\end{theorem}
\noindent We refer to Definition~\ref{def:orbstabgen} for the definition of orbital stability. Observe that in~\eqref{eq:coercive} the coercivity estimate is imposed for all perturbations $v$ in $\Ban$, rather than only in $\Sigma_{\mu_*}$, as in~\eqref{eq:coerciverestricted}.  So here we are assuming that the Lyapunov function reaches a local minimum on $\Ocal_{u_{\mu_*}}$, when viewed as a function on $\Ban$, rather than only as a function on $\Sigma_{\mu_*}$. This  therefore constitutes a strengthening of Hypothesis~B$\mu_*$(iii).The theorem can be used to prove orbital stability in some cases: for the fixed points in the spherical potentials treated in Section~\ref{s:fixedpoints}, for example, this is how we proceeded. Similarly, to establish the stability of the plane waves for the nonlinear defocusing Schr\"odinger equation on a one-dimensional torus, this theorem will also be sufficient, as we will see in Section~\ref{s:nlsetorus1d}. But we have already noticed in Section~\ref{s:spherpotstab} that the coercivity imposed in~\eqref{eq:coercive} may be too strong a condition: we saw it is not satisfied for the natural choice of Lyapunov function for the circular orbits of Section~\ref{s:circular}, for example.  It is too strong also in many situations involving the stability of solitons or standing waves. An example is treated in Section~\ref{curves.sec}.

The proof is very simple, and based on the usual argument by contradiction. 
\begin{proof} Suppose there exists a point $u\in\Ocal_{u_{\mu_*}}$ that is not orbitally stable. 
Then there exists $\epsilon_0>0$ and for all $n\in\N^*$, there exists $v_n\in E$ so that $\rd(v_n, u)\leq \frac1n$ and $\exists t_n\in\R$ so that $\rd(v_n(t_n), \Ocal_{u_{\mu_*}})=\epsilon_0$. 
We can suppose $\epsilon_0<\eta$. Then there exists $\tilde v_n\in \Ocal_{u_{\mu_*}}$ so that $\rd(v_n(t_n), \tilde v_n)\leq \eta$ and hence, since $\Lcal_{\mu_*}$ is both a constant of the motion and constant on $\Ocal_{u_{\mu_*}}$, 
\begin{align*}
\Lcal_{\mu_*}(v_n)-\Lcal_{\mu_*}(u)&=\Lcal_{\mu_*}(v_n(t_n))-\Lcal_{\mu_*}(\tilde v_n)\geq c\rd^2(v_n(t_n), \Ocal_{u_{\mu_*}})=c\epsilon_0^2.
\end{align*}
Since $\Lcal_{\mu_*}$ is continuous, the left hand side tends to zero when $n\to+\infty$, which is a contradiction. 
\end{proof}
\subsection{Coercivity implies stability I}\label{ss:coercivitystability1}
We now turn to the task of showing that Hypotheses~A and~B$\mu_*$ imply the 
$G_{\Sigma_{\mu_*}}$-orbital stability of $u_{\mu_*}$. 
For our first result, we need the following hypothesis.

\vskip0.3cm
\noindent{\bf \textsc{Hypothesis} F} Let $F:\Ban\to\R^m$. Let $\mu\in\R^m$. We say $F$ satisfies Hypothesis F at $\mu$ if, for any bounded sequence $u_n$ in $\Ban$,
\begin{equation}\label{eq:hypothesisF}
\lim_n F(u_n)=\mu \Rightarrow \rd(u_n,\Sigma_\mu)\to 0.
\end{equation}

\vskip0.3cm
\noindent The following lemma gives sufficient conditions for this to be satisfied.
\begin{lemma}\label{lem:hypF}
\begin{enumerate}[label=(\alph*)]
\item Suppose $\mathrm{dim}\, \Ban<+\infty$. Let $F\in C(\Ban,\R^m)$. Then $F$ satisfies Hypothesis~F for all $\mu\in\R^m$. 
\item Suppose $F\in {C}(E,\R^m)$ and that there exists $C>0$ so that $\{u\in E\mid F(u)^2\leq C^2\}$ is compact. Let $\mu\in\R^m$ with $\mu^2<C^2$. Then $F$ satisfies Hypothesis~$F$ at $\mu$. 
\item Let $F: \Ban\to\R$. Suppose that there exists $k\in\R^*$ so that,  $\forall u\in\Dcal$, for all $\lambda\in\R^*$, $F(\lambda u)=\lambda^kF(u)$.  Suppose $\mu\not=0$.  Then $F$ satisfies Hypothesis~F at $\mu$. 
\end{enumerate}
\end{lemma}
\begin{proof}
\begin{enumerate}[label=(\alph*)]
\item Suppose there exists $\epsilon_0>0$ and a  bounded sequence $u_n$ so that $F(u_n)\to \mu$, but $\rd(u_n,\Sigma_\mu)\geq \epsilon_0$. Then the boundedness of the sequence implies the existence of a convergent subsequence $u_{n_k}\to v\in E$. By continuity of $F$, it follows that $F(v)=\mu$ so that $v\in\Sigma_\mu$. So $\rd(u_{n_k}, \Sigma_\mu)\to 0$. This is a contradiction. 
\item The proof is similar to the one in (a).
\item Let $(u_n)_n$ be a bounded sequence satisfying $F(u_n)\to \mu\not=0$. Then, for large enough $n$ one has $\mu/F(u_n)>0$ and we can define $v_n=\left(\frac{\mu}{F(u_n)}\right)^{1/k}u_n$. Then $F(v_n)=\mu$. Clearly $\|v_n-u_n\|\to 0$ so that $\rd(u_n, \Sigma_\mu)\to 0$.
\end{enumerate}
\end{proof}
\begin{remark}\label{rem:hypF} (i) The boundedness of the sequence is important, even in finite dimension. Indeed, consider on $\R^2$ the function $F(x,y)=\frac{y^2}{1+x^4}$, $\mu=0$ and remark that $F(x,x)\to 0$ as $x\to+\infty$. \\
(ii) Condition (c) can be used for constants of the motion arising from linear actions of one-parameter groups on a Hilbert space, as described in Section~\ref{ss:linflows}, and which have a quadratic hamiltonian of the type
$$
F(u)=\frac12\langle u, Bu\rangle,
$$
such as in~\eqref{eq:defconstantmotionSchr}. An example of such application will be given in the proof of Proposition~\ref{propstability}, at the end of Section~\ref{s:nlsetorus1d}.\\
(iii) The condition $\mu\not=0$ is essential in part (c) of the Lemma. Indeed, consider $E=H^1(\R^d)$ and $F(u)=\|u\|^2_{L^2}$. Let $\mu=0$. Then $\Sigma_{\mu}=\{0\}$. But $F(u_n)\to0$ does not imply $u_n\to0$ in $H^1(\R^d)$. \\
(iv) Condition (c) is no longer sufficient to ensure $F$ satisfies Hypothesis~F when $F:\Ban\to\R^m$, with $m\geq 2$. To see this, we consider an example relevant to the treatment of the Manakov equation. Let $\Ban=H^1(\R, \C^2)$ and consider $F_1(u)=\|v\|^2_{L^2}, F_2(u)=\|w\|^2_{L^2}$, where we wrote $u=(v,w)\in\Ban$. Note that those are the two constants of the motion associated to the diagonal part of the $U(2)$ action on $\Ban$ (See Section~\ref{ss:manakov}). We choose $\mu=(1,0)\not=0\in\R^2$. Then $\Sigma_\mu=\{u\in\Ban\mid w=0, \|v\|^2_{L^2}=1\}$. Now let $a,b\in C_0^{\infty}(\R)$, such that $\|a\|^2_{L^2}=1=\|b\|^2_{L^2}$ and consider $u_n(x)=(a(x), \frac1{\sqrt n}b(n(x-n)))=:(v_n, w_n)\in \Ban$. Note that this sequence is bounded. Moreover, clearly, $\lim_{n\to+\infty} F(u_n)=\mu$. Now, for $u=(v,0)\in\Sigma_\mu$, one has
\begin{align*}
\|u_n-u\|^2 &=\|a-v\|_{H^1(\R,\C)}^2+\|w_n\|_{H^1(\R,\C)}^2\\
		&\geq \|w_n\|_{H^1(\R,\C)}^2 \geq \frac{n^2}{n}\int_\R |b'(n(x-n))|^2\rd x=\|b'\|^2_{L^2}.
\end{align*}
It follows that $\rd(u_n,\Sigma_\mu)=\inf_{u\in\Sigma_\mu}\|u_n-u\|\geq \|b'\|_{L^2}$, so that Hypothesis~F is clearly not satisfied in this situation.
\end{remark}

\begin{theorem}\label{thm:lyapmethodrestricted}
Suppose Hypotheses~A and B$\mu_*$ (Section~\ref{ss:strategy}) are satisfied for some $\mu_*\in \R^m$. Then
\begin{equation}\label{eq:orbstabrestricted}
\forall u\in\Ocal_{u_{\mu_*}}, \forall \epsilon>0, \exists \delta>0, \ (\forall u'\in \Sigma_{\mu_*}, \rd(u', u)\leq \delta\Rightarrow \sup_{t\in\R} \rd(u'(t),\Ocal_{u_{\mu_*}})\leq \epsilon).
\end{equation}
If in addition, 
\begin{enumerate}[label=({\roman*})]
\item $\Lcal_{\mu_*}$ is uniformly continuous on bounded sets,
\item $\Ocal_{u_{\mu_*}}$ is bounded,
\item $F:\Ban\to\R^m$  satisfies Hypothesis F,
\end{enumerate}
 then all $u\in\Ocal_{u_{\mu_*}}$ are orbitally stable $G_{\Sigma_{\mu_*}}$-relative equilibria.
\end{theorem}
We point out that~\eqref{eq:orbstabrestricted} is already an orbital stability result for all $u\in \Ocal_{u_{\mu_*}}=G_{\Sigma_{\mu_*}}u$, but only with respect to perturbations of the initial condition $u$ \emph{inside} $\Sigma_{\mu_*}$. The theorem asserts that, with the extra conditions (i)--(ii)--(iii), orbital stability with respect to all perturbations within $\Ban$ is obtained. 
It is the observation that coercivity \emph{along} $\Sigma_{\mu_*}$ (Hypothesis~B$\mu$~(iii)) suffices to establish orbital stability that explains, \emph{in fine}, the advantage of Theorem~\ref{thm:lyapmethodrestricted} over Theorem~\ref{thm:lyapmethod}.   This is already illustrated in Section~\ref{s:circular} on a simple example. Note furthermore that conditions~(i) and~(iii) of the theorem are automatically satisfied in finite dimension. The boundedness of 
$\Ocal_{u_{\mu_*}}$ (condition~(ii)) is guaranteed for example when the group is compact, or when $E$ is a Hilbert space and the group acts with unitary transformations, which is often the case in infinite dimensional systems. 

The argument in the proof of Theorem~\ref{thm:lyapmethodrestricted} is extracted from the proof of Theorem~5.3 in~\cite{gssI} and is used in~\cite{gssII} as well.  We point out however, that conditions~(i) and~(iii) are not made explicit there. The first one is usually easy to check in examples, where the Lyapunov function tends at any rate to be uniformly Lipschitz on bounded sets. For the second one, we gave some sufficient conditions in Lemma~\ref{lem:hypF}. But, as pointed out in Remark~\ref{rem:hypF}, it may fail, in particular in the very general setting of~\cite{gssI, gssII}. In that case, a different argument is needed; we will provide two below.

\begin{proof}
We will prove~\eqref{eq:orbstabrestricted} by contradiction, yet again. Let us therefore suppose there exists $u\in\Ocal_{u_{\mu_*}}$ and $\epsilon_0>0$ so that for all $n\in\N_*$, there exists $u_n\in \Sigma_{\mu_*}$ so that 
$$
\rd(u_n, u)\leq \frac1n,\quad\mathrm{and}\quad \exists\, \tilde t_n\in\R \ \mathrm{so\ that}\ \rd(u_n(\tilde t_n),\Ocal_{u_{\mu_*}})>\epsilon_0.
$$
We can choose, without loss of generality, $\epsilon_0<\eta$, where $\eta$ is defined in~\eqref{eq:coerciverestricted} and choose $t_n$ the smallest value of $t$ so that
$$
\rd(u_n, u)\leq \frac1n,\quad\mathrm{and}\quad \ \rd(u_n(t_n),\Ocal_{u_{\mu_*}})=\epsilon_0<\eta.
$$
 Consequently, there exists $y_n\in\Ocal_{u_{\mu_*}}$ so that $\rd(u_n(t_n), y_n)<\eta$. Note that $u_n(t_n)\in\Sigma_{\mu_*}$, since $\Sigma_{\mu_*}$ is invariant under the dynamical flow. Then, since $\Lcal_{\mu_*}$ is a constant of the motion, and since it is constant and coercive on $\Ocal_{u_{\mu_*}}$ along $\Sigma_{\mu_*}$,
\begin{align*}
\Lcal_{\mu_*}(u_n)-\Lcal_{\mu_*}(u)&=\Lcal_{\mu_*}(u_n(t_n))-\Lcal_{\mu_*}(u)\\
&=\Lcal_{\mu_*}(u_n(t_n))-\Lcal_{\mu_*}(y_n)\geq c\rd^2(u_n(t_n),\Ocal_{u_{\mu_*}})=c\epsilon_0^2.
\end{align*}
Since $\Lcal_{\mu_*}$ is continuous,  one obtains a contradiction by taking $n\to+\infty$. This shows~\eqref{eq:orbstabrestricted}. 

To prove the last statement, suppose $\Ocal_{u_{\mu_*}}$ is bounded and $\Lcal_{\mu_*}$ uniformly continuous on bounded sets. We need to show that
\begin{equation}\label{eq:III}
\forall u\in \Ocal_{u_{\mu_*}}, \forall \epsilon>0, \exists \delta>0, \  (\forall u'\in E, \rd(u', u)\leq \delta\Rightarrow \sup_{t\in\R} \rd(u'(t),\Ocal_{u_{\mu_*}})\leq \epsilon).
\end{equation}
We proceed again by contradiction. Suppose there exists $u\in\Ocal_{u_{\mu_*}}$ and $0<\epsilon_0<\eta$ so that, for all $n\in\N$, there exists $u_n\in E$, 
$$
\rd(u_n, u)\leq \frac1n,\quad\mathrm{and}\quad \exists t_n\in\R \ \mathrm{so\ that}\ \rd(u_n( t_n),\Ocal_{u_{\mu_*}})=\epsilon_0<\eta.
$$
Note that, this time, $u_n\in E$ and $u_n(t_n)\in E$, not in $\Sigma_{\mu_*}$. So we can't use the coercivity of $\Lcal_{\mu_*}$ along $\Sigma_{\mu_*}$ directly. We do know, however, that $F(u_n(t_n))=F(u_n)$, since $F$ is a constant of the motion. Hence 
$$
\lim_{n\to+\infty}F(u_n(t_n))=\mu_*.
$$
Since the orbit $\Ocal_{u_{\mu_*}}$ is bounded, and since $\rd(u_n(t_n),\Ocal_{u_{\mu_*}})=\epsilon_0$, it follows that the sequence $u_n(t_n)$ is bounded. Hypothesis F then implies there exist $z_n\in\Sigma_{\mu_*}$ so that $\|u_n(t_n)-z_n\|\to 0$.

We can now conclude. Since, for $n$ large enough, $\frac{\epsilon_0}2\leq \rd(z_n, \Ocal_{u_{\mu_*}})\leq \eta$, we have
\begin{align*}
\Lcal_{\mu_*}(u_n)-\Lcal(u)&=\Lcal_{\mu_*}(u_n(t_n))-\Lcal_{\mu_*}(u)\\
&=\Lcal_{\mu_*}(u_n(t_n))-\Lcal_{\mu_*}(z_n)+\Lcal_{\mu_*}(z_n)-\Lcal_{\mu_*}(u)\\
&\geq \Lcal_{\mu_*}(u_n(t_n))-\Lcal_{\mu_*}(z_n)+c\rd^2(z_n,\Ocal_{u_{\mu_*}}).
\end{align*}
Since the orbit $\Ocal_{u_{\mu_*}}$ is bounded, the sequences $u_n(t_n)$ and $z_n$ are bounded. This, combined with the uniform continuity of $\Lcal_{\mu_*}$ on bounded sets, leads again to a contradiction upon taking $n\to+\infty$. 
\end{proof} 

We now give a third proof of orbital stability starting from a coercive Lyapunov function, along the lines of the second argument in the proof of Proposition~\ref{lem:sphercircular}. The point here is that we exploit the fact that the relative equilibria $u_\mu$ often come in families. 
\begin{theorem}\label{thm:lyapmethodrestrictedmod}
Suppose the following.
\begin{enumerate}[label=({\roman*})]
\item Hypothesis~A holds.
\item There exists a continuous map $\mu\in U\subset\R^m\to u_\mu\in \Sigma_\mu\subset \Ban$  so that Hypothesis~B$\mu$ is satisfied for all $\mu\in U$, with $\eta$ and $c$ in~\eqref{eq:coerciverestricted} independent of $\mu$.
\item $\sup_{\mu\in U}\| u_\mu\|<+\infty$. 
\item There exists $C>0$ so that
\begin{equation}\label{eq:boundrestrictedmodbis}
\forall \mu\in U,   \forall u'\in \Sigma_\mu, \quad  \|u'- u_\mu\|\leq \eta\Rightarrow \Lcal_\mu(u')-\Lcal_\mu(u_\mu)\leq C\|u'-u_\mu\|.
\end{equation}
\item $\forall g\in G$, $\Phi_g$ is an isometry on $\Ban$: $\forall u, u'\in \Ban$, $\rd(\Phi_g(u), \Phi_g(u'))=\rd(u, u')$. 
\end{enumerate}
Then, any $u\in \Ocal_{u_\mu}$ is an orbitally stable $G_{\Sigma_\mu}$-relative equilbrium of the flow $\Phi_t^H$.
\end{theorem}
Condition~(iii) is not very restrictive. It is sufficient to take $U$ bounded, for example. Condition~(iv) follows if we know that $D_{u}\Lcal_{\mu}$ is bounded for $u$ in bounded sets. This is a reasonable condition. Condition~(v) is commonly satisfied in PDE systems, but is quite restrictive, as we already explained. It implies we can use Proposition~\ref{prop:isometry} and Lemma~\ref{lem:isometric}.
\begin{proof}
Let $\mu_*\in U$. As a result of Lemma~\ref{lem:isometric},  it is enough to show the orbital stability of $u_{\mu_*}$. So we need to show that, for all $\epsilon>0$, there exists $\delta>0$ so that, for all $u'\in E$, one has
\begin{equation}\label{eq:toprove}
\|u'-u_{\mu_*}\|\leq \delta \Rightarrow \forall t\in\R, \rd(u'(t), \Ocal_{u_{\mu_*}})\leq \epsilon.
\end{equation}
For that purpose, we need three preliminary estimates. We first show that $\forall \epsilon>0$, there exists $\hat\delta>0$ so that, for all $\mu\in U$, for all $u'\in \Sigma_\mu$,
\begin{equation}\label{eq:stepone}
\|u'-u_{\mu}\|\leq \hat \delta \Rightarrow \forall t\in\R, \rd(u'(t), \Ocal_{u_{\mu}})\leq \epsilon/2.
\end{equation}
In other words, we first show that the $u_\mu$ are all orbitally stable for perturbations \emph{within} $\Sigma_\mu$. The method of proof -- by contradiction --  is the same as several times before, but we need to make sure to obtain the necessary uniformity in $\mu$. If the above is not true, then there exists $\epsilon_0>0$ so that for all $n\in\N^*$ there exist $\mu_n\in U$ and $u_n\in\Sigma_{\mu_n}$, $t_n\in\R$, so that
$$
\|u_n-u_{\mu_n}\|\leq \frac1n, \quad \rd(u_n(t_n), \Ocal_{u_{\mu_n}})=\frac{\epsilon_0}2<\eta.
$$
Here $\eta$ is given in Hypothesis~B$\mu$~(iii) and we recall that it is independent of $\mu_n$. 
Hence
$$
\Lcal_{\mu_n}(u_n)-\Lcal_{\mu_n}(u_{\mu_n})=\Lcal_{\mu_n}(u_n(t_n))-\Lcal_{\mu_n}(u_{\mu_n})\geq c\rd^2(u_n(t_n), \Ocal_{\mu_n})=c\frac{\epsilon_0^2}{4}.
$$
Now, since the $u_{\mu_n}$ form a bounded set by hypothesis~(iii) of the theorem, the same is true for the $u_n$. Hence, it follows from hypothesis~(iv) of the theorem that
$$
\Lcal_{\mu_n}(u_n)-\Lcal_{\mu_n}(u_{\mu_n})\leq C\|u_n- u_{\mu_n}\|,
$$
where $C$ does not depend on $n$. Hence 
$
C\|u_n- u_{\mu_n}\|\geq c\frac{\epsilon_0^2}{4},
$
so that, taking $n\to+\infty$, we obtain a contradiction. This proves~\eqref{eq:stepone}. 

As a second step, we show the following estimate. Let $\mu_*\in U$. Then, for all $\epsilon>0$, there exists $\hat\rho>0$ so that,  
\begin{equation}\label{eq:step2}
\forall\mu\in U,\quad\left( \|\mu-\mu_*\|\leq \hat\rho\Rightarrow\forall v\in \Ocal_{u_\mu},  \rd(v, \Ocal_{u_{\mu_*}})\leq \frac\epsilon2\right).
\end{equation}
To see, this, note that hypothesis~(i) of the theorem implies that there exists $\hat\rho>0$ so that $\|\mu-\mu_*\|\leq \hat\rho$ implies $\|u_\mu-u_{\mu_*}\|\leq \epsilon/2$. Hence $\rd(u_\mu, \Ocal_{u_{\mu_*}})\leq \epsilon/2$. The result then  follows from Proposition~\ref{prop:isometry}, since we suppose the action $\Phi$ of $G$ is isometric.

The third ingredient for the proof of~\eqref{eq:toprove} is the following:
\begin{align}\label{eq:step3}
\forall \hat\delta>0, \forall \hat\rho>0, \exists\delta>0, \forall u'\in E,&\nonumber \\
 &\hskip-2cm \left(\|u'-u_{\mu_*}\|\leq \delta\Rightarrow \|\mu'-{\mu_*}\|\leq \hat\rho, \|u'-u_{\mu'}\|\leq \hat \delta\right),
\end{align}
where $\mu'=F(u')$.  This follows immediately from the continuity of $F$ and of $\mu\to u_\mu$ at $\mu_*$. 

We can now conclude. Let $\mu_*\in U$ and $\epsilon>0$. Choose $\hat \delta$ as in~\eqref{eq:stepone}, $\hat\rho$ as in~\eqref{eq:step2} and $\delta$ as in~\eqref{eq:step3}. Then, by~\eqref{eq:stepone} and~\eqref{eq:step3}, we find that
$$
\forall u'\in E, \left(\|u'-u_{\mu_*}\|\leq \delta\Rightarrow \forall t\in \R, \rd(u'(t), \Ocal_{u_{\mu'}})< \frac\epsilon2\right). 
$$
Hence, for all $t\in\R$, there exists $v(t)\in\Ocal_{u_{\mu'}}$, so that $\rd(u'(t), v(t))< \epsilon/2$. Next, from~\eqref{eq:step3} and~\eqref{eq:step2} , there exists $w(t)\in \Ocal_{u_{\mu_*}}$ so that $\rd(v(t), w(t))<\frac\epsilon2$. Hence $\rd(u'(t), \Ocal_{u_{\mu_*}})<\epsilon$. This proves~\eqref{eq:toprove}.
\end{proof}

\subsection{Sufficient condition for coercivity}\label{ss:coercivityhessian}

We now turn to the task of showing how one can obtain the coercivity Hypothesis B$\mu$~(iii)  from an estimate on the Hessian of $\Lcal_\mu$ (Proposition \ref{thm:hessianestimate}). We work in the following setting.

As before, let $\Ban$ be a Banach space, $G$ a Lie group and $\Phi$ a $G$-action on $E$. Let $F\in C^2(E,\R^m)$. We recall that, for $\mu\in\R^m$,
$$
\Sigma_\mu=\{u\in E\mid F(u)=\mu\},
$$
and that $G_{\Sigma_\mu}$ is the subgroup of $G$ leaving $\Sigma_\mu$ invariant. We now introduce one extra ingredient to the theory. Let $\langle\cdot,\cdot\rangle$ be a scalar product on $\Ban$, which is continuous in the sense that 
$$
\forall v, w\in \Ban, \quad |\langle v, w\rangle|\leq \|v\| \|w\|,
$$
where we recall that $\|\cdot\|$ is our notation for the Banach norm on $\Ban$. This inner product induces a metric on $E$, that we shall denote by
\begin{equation}\label{eq:scalarmetric}
\rd_{\mathrm s}(v,w)=\langle v-w, v-w\rangle.
\end{equation}
Clearly $\rd_{\mathrm s}(v,w)\leq \rd(v,w)$. We introduce this inner product since we need a notion of orthogonality for the statement of the main result of this section, Proposition~\ref{thm:hessianestimate}: see in particular~\eqref{eq:hessiancondition} and~\eqref{eq:perp}. 

We point out that we are not supposing $\Ban$ is a Hilbert space for  this inner product, and that the only topology we will be using in what follows is the one induced by the Banach norm on $\Ban$. In addition, even if $\Ban$ \emph{is} in fact a Hilbert space, the inner product $\langle\cdot, \cdot\rangle$ above is not necessarily the Hilbert space inner product. 
As an example, if $\Ban=H^1(\R^d,\C)$ and depending on the problem considered, one may want to use either the $L^2$ inner product or the $H^1$ inner product: in Section \ref{s:nlsetorus1d} the first choice is made and in Section \ref{curves.sec} the second one. In the formalism developed in~\cite{gssI, gssII, stuart08}, $\Ban$ is always supposed to be a Hilbert space, and only the Hilbert space inner product is used in the analysis of the Hessian. But the introduction of a second inner product is a regularly used device in the literature on orbital stability for the Schr\"odinger in particular.  Our approach  here gives a systematic treatment in the general setting presented above.

Let $\mu\in \R^m$ and $u_{\mu}\in\Sigma_{\mu}$. We need the following hypothesis on the group action and on the
function $F$. \vskip0.2cm
\noindent{\bf \textsc{Hypothesis} C$\mu$} 
\begin{enumerate}[label=({\roman*})]
\item $\Phi_g$ is linear and preserves both the structure $\scalh$ and the norm $\|\cdot\|$ for all $g\in G$;
\item $\mathrm{Ad}^*_g\in \mathrm{O}(m)$ for all $ g\in G_{\Sigma_\mu}$;
\item $\mu$ is a regular value of $F$;
\item $u_{\mu}$ is a $C^1$-vector for $\Phi$ and the map
\begin{equation}\label{eq:localdiffeo}
\xi\in \frak g_{\Sigma_\mu}\to \Phi_{\exp(\xi)}u_\mu\in \Ban
\end{equation}
is  one to one in a neighbourhood of  $\xi=0$.
\end{enumerate}
\vskip0.3cm
Note that both Hypothesis~C$\mu$ above and Proposition~\ref{thm:hessianestimate} below involve $G$ and its action on $E$, as well as $F$, but not the dynamics $\Phi^X_t$ itself.

\begin{remark}
(i) The meaning of condition~(ii) of Hypothesis~C$\mu$ is explained in Remark \ref{rem:euclidianstructure}. \\
(ii) We say $u\in\Ban$ is a $C^1$-vector for the action $\Phi$ if the map
$
g\in G\to \Phi_g(u)\in\Ban
$
is $C^1$. Now, if $u'\in\Ocal_u=\Phi_G(u)$, then $u'$ is also a $C^1$-vector.  Indeed, there exists $g'\in G$ so that $\Phi_{g'}u=u'$ and, since $g\to gg'$ is smooth, it follows that $g\to \Phi_{gg'}u$ is $C^1$. 

To state the result, we need the following notation. Let $\tilde G$ be a subgroup of $G$; we can then define, for all $u'\in \Ocal_u=\Phi_{\tilde G}(u)$, 
\begin{equation}\label{eq:tangentorbit}
T_{u'}\Ocal_u:=\{w\in \Ban\mid \exists \xi\in\frak g,  w=X_\xi(u')\},
\end{equation}
where we recall from~\eqref{eq:generator} that 
$$
X_\xi(u)=\frac{\rd}{\rd t} \Phi_{\exp(t\xi)}(u)_{\mid t=0}.
$$
\end{remark}

\begin{proposition}\label{thm:hessianestimate}  Let $\Ban$ be a Banach space and $\scalh$ be a continuous scalar product on $\Ban$. Let $G$ be a Lie group and $\Phi$ a $G$-action on $E$. Let $F\in C^2(E,\R^m)$. Let $\mu_*\in\R^m$ and $u_{\mu_*}\in \Sigma_{\mu_*}$. Let $\Lcal_{\mu_*}\in C^2(E,\R)$ be a $G_{\Sigma_{\mu_*}}$-invariant function. Suppose Hypothesis~C$\mu_*$ holds
and that, for all $u\in \Ocal_{u_{\mu_*}}$ (defined in \eqref{eq:muorbit}),
\begin{equation}
	\label{eq:hypX}
	\forall j=1,\ldots,m\ \exists \nabla F_j(u)\in \Ban \text{ such that }  D_{u}F_j(w)=\scal{ \nabla F_j(u)}{w}\ \forall w\in \Ban.
\end{equation}
Suppose $\Lcal_{\mu_*}$ satisfies the following conditons:
\begin{enumerate}[label=({\alph*})]
\item $D_{u}\Lcal_{\mu_*}(w)=0$ for all $u\in \Ocal_{u_{\mu_*}}$ and $w\in \Ban$;
\item there exists $C>0$ so that 
\begin{equation*}
\forall u \in \Ocal_{u_{\mu_*}}, \forall w\in \Ban,\  D^2_u\Lcal_{\mu_*}(w,w)\leq C\|w\|^2;
\end{equation*}
\item there exists $c>0$ so that 
\begin{equation}\label{eq:hessiancondition}
\forall u \in \Ocal_{u_{\mu_*}}, \forall w\in T_{u}\Sigma_{\mu_*}\cap (T_{u}\Ocal_{u_{\mu_*}})^\perp, \ 
D^2_{u}\Lcal_{\mu_*}(w,w)\geq c\|w\|^2
\end{equation}
where
\begin{equation}\label{eq:perp}
\left(T_w\Ocal_u\right)^\perp=\{z\in \Ban\mid \scal{z}{y}=0, \forall y\in T_w\Ocal_u\}.
\end{equation}
\end{enumerate}
Then Hypothesis~B$\mu_*$ (iii) holds.\\
\end{proposition}
Condition~\eqref{eq:hypX} is automatically satisfied when $E$ is a Hilbert space and $\langle\cdot, \cdot\rangle$ the Hilbert space inner product. But not in general. For example, let $\Ban=H^1(\R, \C)$ and let $\scal{u}{v}=\mathrm{Re}\int_\R \bar u(x) v(x) \rd x$. Now, if $F_1(u)= \frac1{2i}\int \overline u(x)\partial_xu(x)\rd x$, \eqref{eq:hypX} is satisfied if $u\in H^2(\R, \C)$ but not for arbitrary $u\in\Ban$.

For the proof of this proposition, we need some simple technical results.

First, let $V$ be a bounded open neighbourhood  of $e$ in a subgroup $\tilde G$ of $G$ with the property that, for all $g\in \tilde G$, $gVg^{-1}=V$. Let us introduce
$$
R_V(u)=\min\{\rd_{\mathrm s}(\Phi_g(u), u) \mid g\in \partial V\}.
$$
It then follows that, for all $u'\in \Ocal_u$, $R_V(u')=R_V(u)$. Indeed, there exists $g'\in \tilde G$ so that $\Phi_{g'}(u)=u'$. Hence
\begin{align*}
R_V(u')&=\min\{ \rd_{\mathrm s}(\Phi_{gg'}(u), \Phi_{g'}(u)) \mid g\in \partial V\}\\
&=\min\{ \rd_{\mathrm s}(\Phi_{{g'}^{-1}gg'}(u), u) \mid g\in \partial V\}=R_V(u),
\end{align*}
since ${{g'}^{-1}}\partial V{g'}=\partial V$. \\
We can now formulate the following simple but crucial technical result, which is a multi-dimensional version of Lemma~2.1 in~\cite{stuart08}.  
\begin{lemma}\label{lem:stuartlem}
Let $\Ban$ be a Banach space and $\scalh$ be a continuous scalar product on $\Ban$. Let $\tilde G$ be a Lie subgroup of $G$ and $\Phi$ a linear $G$-action on $E$ which preserves the inner product $\scalh$. Suppose $u\in \Ban$ is a $C^1$-vector for $\Phi$ and let $V$ be a bounded open neighbourhood of $e\in \tilde G$ which is conjugation invariant (\emph{i.e.} $gVg^{-1}=V$, for all $g\in \tilde G$). Suppose $R_V(u)>0$. Then, for all $v\in \Ban$, 
\begin{equation}
\rd(v, \Ocal_u)<\frac13 R_V(u)\Rightarrow \exists w\in\Ocal_u=\Phi_{\tilde G}(u), w-v\in \left(T_w\Ocal_u\right)^\perp.
\end{equation}
\end{lemma}
The lemma states that if $v$ is not too far from the orbit $\Ocal_u$, then there exists a point $w$ on the orbit so that the segment from $v$ to $w$ is orthogonal to the orbit at $w$. This point does {\it not} necessarily realize the distance between $v$ and the orbit, which can vanish.

\begin{proof}
Let $v\in\Ban$ and $\rd(v, \Ocal_u)<\frac13 R_V(u)$. Then there exists $u'\in \Ocal_u$ so that $\rd(v,u')\le \frac13 R_V(u)=\frac13 R_V(u')$ and hence $\rd_{\mathrm s}(v,u')\le\frac13 R_V(u')$. Now consider 
$$
g\in \overline V\to \rd^2_{\mathrm s}(v, \Phi_gu')\in \R^+.
$$
Since $\overline V$ is compact, this function reaches a minimum at some point $\tilde g\in \overline V$. We set $w=\Phi_{\tilde g}u'\in\Ocal_u$ so that 
$\rd_{\mathrm s}(v, w)\leq \rd_{\mathrm s}(v, u')\le \frac13 R_V(u')$. We now show that $\tilde g$ cannot belong to $\partial V$. Indeed, if $\tilde g$ were on the boundary of $V$, then, by the definition of $R_V(u')$, $\rd_{\mathrm s}(w, u')\geq R_V(u')$. But then 
$$
\rd_{\mathrm s}(w, v)\geq -\rd_{\mathrm s}(u', v)+\rd_{\mathrm s}(u', w)\ge R_V(u')-\frac13 R_V(u')=\frac 23 R_V(u').
$$
which is a contradiction because $\rd_{\mathrm s}(v, w)\le \frac13 R_V(u')$. So $\tilde g$ belongs to $V$. Now choose $\xi\in\frak g$ and consider
$$
t\in\R\to \rd^2_{\mathrm s}(v, \Phi_{\exp(t\xi)\tilde g}(u'))\in \R^+,
$$
which now reaches a local minimum at $t=0$ since for small $t$, $\exp(t\xi)\tilde g$ belongs to $V$. Hence its derivative vanishes. So
\begin{align*}
0=\frac{\rd}{\rd t} \rd^2_{\mathrm s}(v, \Phi_{\exp(t\xi)\tilde g}(u'))_{\mid t=0}&=
\frac{\rd}{\rd t} \scal{v-\Phi_{\exp(t\xi)}(w)}{v-\Phi_{\exp(t\xi)}(w)}_{\mid t=0}\\
&=-2\langle X_\xi(w),v-w\rangle,
\end{align*}
which proves the result in view of \eqref{eq:tangentorbit}.
\end{proof}

In the proof of Proposition \ref{thm:hessianestimate}, we will need to apply the previous lemma to the group $G_{\Sigma_{\mu_*}}$ for some $\mu_*\in \R^m$ and $u_{\mu_*}\in \Sigma_{\mu_*}$. The following lemma gives hypotheses for this to be possible. It appears in various guises in the literature, and can be referred to as a ``modulation'' argument.

\begin{lemma}\label{lem:stuartlemham}  
Let $\Ban$ be a Banach space and $\scalh$ be a continuous scalar product on $\Ban$. Let $G$ be a Lie group and $\Phi$ a $G$-action on $E$. Let $F\in C^2(E,\R^m)$. Let $\mu_*\in\R^m$ and $u_{\mu_*}\in \Sigma_{\mu_*}$. Suppose Hypothesis~C$\mu_*$ holds.
Then, there exists $R>0$ such that, for all $v\in \Ban$, 
\begin{equation}
\rd(v, \Ocal_{u_{\mu_*}})<R\Rightarrow \exists w\in\Ocal_{u_{\mu_*}}, w-v\in \left(T_w\Ocal_{u_{\mu_*}}\right)^\perp
\end{equation}
where $\Ocal_{u_{\mu_*}}=\Phi_{G_{\Sigma_{\mu_*}}}u_{\mu_*}$.
\end{lemma}

\begin{proof}
Thanks to Lemma \ref{lem:stuartlem}, it is enough to prove that there exists $V$ a bounded open neighbourhood of $e\in G_{\Sigma_{\mu_*}}$, which is conjugation invariant (\emph{i.e.} ${g}Vg^{-1}=V$, for all $g\in G_{\Sigma_{\mu_*}}$) and  such that $R_V(u_{\mu_*})>0$.

First of all, we recall that the exponential map $$\exp: \xi\in \frak{g}_{\mu_*}\to \exp(\xi)\in G_{\Sigma_{\mu_*}}$$ is a local diffeomorphism from some neighbourhood of $0\in  \frak{g}_{\mu_*}$ to a neighbourhood of $e\in G_{\Sigma_{\mu_*}}$. In other words, there exists $\delta>0$ such that 
$$\exp: \xi\in B_{\delta}(0)\subset \frak{g}_{\mu_*}\to \exp(\xi)\in G_{\Sigma_{\mu_*}}$$ 
is a local diffeomorphism onto a bounded open neighbourhood $V:= \exp(B_{\delta}(0))$  of $e$ in $G_{\mu_*}$. In particular, note that $\partial V= \exp(\partial B_{\delta}(0))$.

Since, thanks to Hypothesis~C$\mu_*$(ii), $ B_{\delta}(0)$ is $\mathrm{Ad}_g$-invariant for all $g\in G_{\Sigma_{\mu_*}}$, $V$ is conjugation invariant. Indeed, for all $\xi\in B_{\delta}(0)$ and all $g\in G_{\Sigma_{\mu_*}}$, 
we have that $g \exp(\xi)g^{-1}=\exp(\mathrm{Ad}_g \xi)\in V$.

Hence, it only remains to show that $R_V(u_{\mu_*})>0$, which is equivalent to
$G_{u_{\mu_*}}\cap \partial V=\emptyset$. Thanks to Hypothesis~C$\mu_*$(iv),
there exists $\delta_0>0$ such that 
\begin{equation*}
	\xi\in B_{\delta_0}(0)\to \Phi_{\exp(\xi)}u_{\mu_*}\in E
\end{equation*} 
is one to one. As a conclusion, choosing $\delta<\delta_0$, we have $\partial V \subset \exp(B_{\delta_0}(0))$ which implies $\Phi_{\exp(\xi)}u_{\mu_*}\neq u_{\mu_*}$ for all $\exp(\xi)\in \partial V$. Hence, for all $\exp(\xi)\in \partial V$, $\exp(\xi)\notin  G_{u_{\mu_*}}$.
\end{proof}

We can then conclude this section with the proof of Proposition  \ref{thm:hessianestimate}.

\begin{proof}[Proof of Proposition \ref{thm:hessianestimate}]
Recall that we have to prove there exist $\eta>0, \tilde c>0$ so that
\begin{equation*}
\forall u\in\Ocal_{u_{\mu_*}}, \forall u'\in \Sigma_{\mu_*}, \quad \rd(u,u')\leq \eta\Rightarrow \Lcal_{\mu_*}(u')-\Lcal_{\mu_*}(u)\geq \tilde c\rd^2(u', \Ocal_{u_{\mu_*}}).
\end{equation*}

Let $u'\in\Sigma_{\mu_*}$, $\rd(u', \Ocal_{u_{\mu_*}})<R$. Thanks to Lemma~\ref{lem:stuartlemham}, there exists $v'\in \Ocal_{u_{\mu_*}}$ such that $u'-v'\in \left(T_{v'}\Ocal_{u_{\mu_*}}\right)^\perp$.  

Next, let $W_{v'}$ be the subspace of $\Ban$ spanned by $\{\nabla F_j(v')\}_{j=1\ldots,m}$. It follows from \eqref{eq:deftangentspace} and hypothesis~\eqref{eq:hypX} that $T_{v'}\Sigma_{\mu_*}=(W_{v'})^\perp$. As a consequence, we can write $\Ban=T_{v'}\Sigma_{\mu_*}\oplus W_{v'}$. Indeed, since $W_{v'}$ has finite dimension, it admits an orthonormal basis $\{e_1,\ldots,e_m\}$ w.r.t. $\scalh$. Hence, all $w\in \Ban$ can be written as
\begin{equation*}
w=\Big(w-\sum_{j=1}^m\scal{w}{e_j}e_j\Big)+\sum_{j=1}^m\scal{w}{e_j}e_j.
\end{equation*}
Clearly $w-\sum_{j=1}^m\scal{w}{e_j}e_j\in (W_{v'})^\perp=T_{v'}\Sigma_{\mu_*}$, $\sum_{j=1}^m\scal{w}{e_j}e_j\in W_{v'}$ and $W_{v'}\cap(W_{v'})^\perp=\{0\}$.
Then,  
$$
u'-v'=(u'-v')_1 +(u'-v')_2
$$
where $(u'-v')_1\in T_{v'}\Sigma_{\mu_*}$ and $(u'-v')_2\in W_{v'}$. Moreover, since $u'-v'\in \left(T_{v'}\Ocal_{u_{\mu_*}}\right)^\perp$, we can easily show $(u'-v')_1\in T_{v'}\Sigma_{\mu_*}\cap\left(T_{v'}\Ocal_{u_{\mu_*}}\right)^\perp$ and $(u'-v')_2\in W_{v'}\cap \left(T_{v'}\Ocal_{u_{\mu_*}}\right)^\perp$. Now, Lemma~\ref{eq:basicestimate} ensures the existence of constants $c_1,c_0$ so that, for $\| u'-v'\|$ small enough, one has
\begin{eqnarray}\label{eq:estimate2gen}
\|(u'-v')_1\| \geq c_0\|u'-v'\|\text{ and } \|(u'-v')_2\| \leq c_1\|u'-v'\|^2.
\end{eqnarray}
Since the action $\Phi_g$ is linear and preserves both $\scalh$ and $\|\cdot\|$, the decomposition above is group invariant and the constant $c_0$ and $c_1$ do not depend on $v'$. 

 We can now conclude the proof as follows, using respectively conditions~(a),~(b) and~(c), and~\eqref{eq:estimate2gen}:
 \begin{eqnarray*}
\Lcal_{\mu_*}(u')-\Lcal_{\mu_*}(u_{\mu_*})&=&\Lcal_{\mu_*}(u')-\Lcal_{\mu_*}(v')\\
&=&D_{v'}\Lcal_{\mu_*}(u'-v')+\frac12 D_{v'}^2\Lcal_{\mu_*}(u'-v', u'-v')+\mathrm{o}(\|u'-v'\|^2)\\
&=&\frac12 D_{v'}^2\Lcal_{\mu_*}((u'-v')_1,(u'-v')_1)+ \mathrm{O}(\|u'-v'\|^3) 
+\mathrm{o}(\|u'-v'\|^2)\\
&=&\frac12 D_{v'}^2\Lcal_{\mu_*}((u'-v')_1,(u'-v')_1)+\mathrm{o}(\|u'-v'\|^2)\\
&\geq& \frac{c}{2}\|(u'-v')_1\|^2+\mathrm{o}(\|u'-v'\|^2)\\
&\geq &\tilde c\|u'-v'\|^2\ge\tilde c\rd^2(u',\Ocal_{u_{\mu_*}}).
\end{eqnarray*}
Remark that as before, the constant $\tilde c$ is independent of $v'\in \Ocal_{\mu_*}$.
\end{proof}

\subsection{Coercivity implies stability II}\label{ss:coercivitystability2}

We can now state and prove a fourth theorem yielding orbital stability under slightly different technical assumptions. We will work in the Hamiltonian setting and in particular use the characterization of relative equilibria given by Theorem~\ref{thm:relequicritical}. Recall that in this context, for each $\mu\in \frak{g}^*\simeq \R^m$, $G_{\Sigma_\mu}=G_{\mu}$ (Proposition~\ref{eq:reduction}).

\begin{theorem}\label{thm:lyapmethodrestrictedLK} Let $\Ban$ be a Banach space and $\scalh$ be a continuous scalar product on $\Ban$, $\Dom$ a domain in $\Ban$ and $\Jcal$ a symplector. Let $H\in C^2(\Ban,\R)\cap \JDif$.
Let $G$ be  a Lie group, and $\Phi$ a globally Hamiltonian $G$-action on $\Ban$ with Ad$^*$-equivariant momentum map $F$.  Let $\mu_*\in \R^m\simeq \frak{g}^*$ and $u_{\mu_*}\in \Dom\cap \Sigma_{\mu_*}$. Suppose that Hypothesis~C$\mu_*$(i)--(iii) is satisfied, and $H\circ \Phi_g=H$ for all $g\in G$.
Let $\Lcal_{\mu_*}=H-\xi_{\mu_*}\cdot F$ with $\xi_{\mu_*}\in \mathfrak{g}_{\mu_*}$ given by Theorem~\ref{thm:relequicritical} and assume $D_{u_{\mu_*}}\Lcal_{\mu_*}=0$. 
Suppose in addition that 
\begin{equation}
	\label{eq:hypXbis}
	\forall j=1,\ldots,m\ \exists \nabla F_j(u_{\mu_*})\in \Ban \text{ such that }  D_{u_{\mu_*}}F_j(w)=\scal{ \nabla F_j(u_{\mu_*})}{w}\ \forall w\in \Ban.
\end{equation}
and
\begin{enumerate}[label=({\alph*})]
	\item $G_{\mu_*}$ is commutative;
	\item there exists $C>0$ so that 
	\begin{equation*}
	\forall w\in \Ban,\  D^2_{u_{\mu_*}}\Lcal_{\mu_*}(w,w)\leq C\|w\|^2;
	\end{equation*}
	\item there exists $c>0$ so that 
	\begin{equation*}
	\forall w\in T_{u_{\mu_*}}\Sigma_{\mu_*}\cap (T_{u_{\mu_*}}\Ocal_{u_{\mu_*}})^\perp, \ 
	D^2_{u_{\mu_*}}\Lcal_{\mu_*}(w,w)\geq c\|w\|^2.
	\end{equation*}
\end{enumerate}
Then all $u\in\Ocal_{u_{\mu_*}}$ are orbitally stable $G_{\mu_*}$-relative equilibria. 
\end{theorem}

Hypothesis~(a) in Theorem~\ref{thm:lyapmethodrestrictedLK} is not very restrictive (see \cite{dufver1969}). 

\begin{proof}
Let $K>0$ and define 
$$
\Lcal_K(u)=\Lcal_{\mu_*}(u)+K(F(u)-\mu_*)^2.
$$
Here $(F(u)-\mu_*)^2=(F(u)-\mu_*)\cdot(F(u)-\mu_*)$ where,  $\cdot$ is the 
$G_{\mu_*}$-invariant inner product described in Remark~\ref{rem:euclidianstructure}. It follows that $\Lcal_K$ is a $G_{\mu_*}$-invariant constant of the motion. Indeed, for all $g\in G_{\mu_*}$ and for all $u\in \Ban$,
\begin{align*}
	\Lcal_K&(\Phi_g u )=H(\Phi_g u)-\xi_{\mu_*}\cdot F(\Phi_g u)+K(F(\Phi_g u)-\mu_*)^2&&\\
	&=H(u)-\xi_{\mu_*}\cdot \mathrm{Ad}^*_g F(u)+K(\mathrm{Ad}^*_g F(u)-\mathrm{Ad}^*_g \mu_*)^2&&\\
	&=H(u)-\mathrm{Ad}_g\xi_{\mu_*}\cdot F(u)+K(F(u)-\mu_*)^2&&\text{as } \mathrm{Ad}^*_g\in \mathrm{O}(m)\\
	&=H(u)-\xi_{\mu_*}\cdot F(u)+K(F(u)-\mu_*)^2&&\text{as }G_{\mu_*} \text{ is commutative}\\
	&=\Lcal_K(u ).
\end{align*}
The main idea is to prove that the hypotheses of Proposition~\ref{thm:hessianestimate} are satisfied by $\Lcal_K$ and then use its proof to conclude that all $u\in\Ocal_{u_{\mu_*}}$ are orbitally stable $G_{\mu_*}$-relative equilibria. 

First, note that in this setting Hypothesis~C$\mu_*$(iv) follows from Remark~\ref{rem:muregular}.

Next, we claim that $D_u \Lcal_{K}(w)=0$ for all $u\in \Ocal_{u_{\mu_*}}$ and for all $w\in \Ban$. Indeed, it is clear that $D_u (F(u)-\mu_*)^2=2 (F(u)-\mu_*)\cdot D_u F=0$ for all $u\in \Ocal_{u_{\mu_*}}$ and, thanks to the fact that $D_{u_{\mu_*}} \Lcal_{K_{\mu_*}}(w)=0$, we  obtain $D_{u_{\mu_*}} \Lcal_{K}(w)=0$ for all $w\in \Ban$. Next, let $u\in \Ocal_{u_{\mu_*}}$ and $g\in G_{\mu_*}$ such that $u=\Phi_g (u_{\mu_*})$, then 
\begin{align*}
	D_u \Lcal_{K}(w)=[D_{\Phi_g (u_{\mu_*})} \Lcal_{K}\circ \Phi_{g^{-1}}](w)=[D_{u_{\mu_*}} \Lcal_{K}\circ D_{\Phi_g (u_{\mu_*})}\Phi_{g^{-1}}](w)=0.
\end{align*}

Using the fact that $\Phi_g$ is linear and preserves both $\scalh$ and $\|\cdot\|$, we can easily show, as a consequence of hypothesis~(c), that 
\begin{equation}
	\label{eq:estimhessianorbitgen}
	\forall u \in \Ocal_{u_{\mu_*}},\ D^2_{u}\Lcal_{\mu_*}(w,w)\geq c\|w\|^2,
\end{equation}
for all $w\in T_{u}\Sigma_{\mu_*}\cap (T_{u}\Ocal_{u_{\mu_*}})^\perp$. Indeed, for all $u \in \Ocal_{u_{\mu_*}}$ and $w\in T_{u}\Sigma_{\mu_*}\cap (T_{u}\Ocal_{u_{\mu_*}})^\perp$, 
\begin{align*}
	D^2_{u}\Lcal_{\mu_*}(w,w)&= D^2_{\Phi_g u_{\mu_*}}(\Lcal_{\mu_*}\circ \Phi_{g^{-1}})(w,w)
	\\
	&=D^2_{u_{\mu_*}}\Lcal_{\mu_*}(D_{u}\Phi_{g^{-1}}w, D_{u}\Phi_{g^{-1}} w)+D_{u_{\mu_*}}\Lcal_{\mu_*}(D^2_u \Phi_{g^{-1}}(w,w))\\
	&= D^2_{u_{\mu_*}}\Lcal_{\mu_*}(\Phi_{g^{-1}}w, \Phi_{g^{-1}} w)
	\geq c\| \Phi_{g^{-1}}w\|^2= c\| w\|^2
\end{align*}
because $\Phi_{g^{-1}}w\in T_{u_{\mu_*}}\Sigma_{\mu_*}\cap (T_{u_{\mu_*}}\Ocal_{u_{\mu_*}})^\perp$.

Similarly, using hypothesis~(b), we prove that 
\begin{equation*}
	D^2_{u}\Lcal_{\mu_*}(w,w)\le C \|w\|^2
\end{equation*}
for all $u\in\Ocal_{u_{\mu_*}}$ and $w\in E$.

Next, by a straightforward calculation, we obtain for all $u\in \Ocal_{u_{\mu_*}}$ and $w\in E$, $D^2_{u}(F-\mu_*)^2(w,w)=2 D_u F(w)\cdot D_u F(w)$, and
\begin{align}\label{eq:adstaragain}
	D_u F (w)&=[D_{\Phi_g u_{\mu_*}}F\circ \Phi_g\circ \Phi_{g^{-1}}](w)=[D_{u_{\mu_*}}F\circ \Phi_g](D_u \Phi_{g^{-1}}w)\nonumber\\
	&=[D_{u_{\mu_*}}\mathrm{Ad}^*_g \circ F](\Phi_{g^{-1}}w)=\mathrm{Ad}^*_g(D_{u_{\mu_*}}F(\Phi_{g^{-1}}w)).
\end{align}
As a consequence, since $\mathrm{Ad}^*_g\in \mathrm{O}(m)$, 
\begin{equation}
	\label{eq:defhessianKF}
	D^2_{u}(F-\mu_*)^2(w,w)=2 D_{u_{\mu_*}}F(\Phi_{g^{-1}}w)\cdot D_{u_{\mu_*}}F(\Phi_{g^{-1}}w).
\end{equation}
It is then clear that $D^2_{u}(F-\mu_*)^2(w,w)\le C_{\mu_*} \|w\|^2$ for all $u\in  \Ocal_{u_{\mu_*}}$ and $w\in E$, and hypothesis~(b) of Proposition~\ref{thm:hessianestimate} is satisfied by $\Lcal_K$. In addition~\eqref{eq:adstaragain}  together with the fact that the $\Phi_g$ preserve the inner product $\langle\cdot, \cdot\rangle$ shows that~\eqref{eq:hypXbis} implies~\eqref{eq:hypX}.

Now let $w\in (T_{u}\Ocal_{u_{\mu_*}})^\perp$ and write $w=w_1+w_2$ with $w_1\in T_{u}\Sigma_{\mu_*}\cap (T_{u}\Ocal_{u_{\mu_*}})^\perp$ and $w_2\in W_{u}\cap (T_{u}\Ocal_{u_{\mu_*}})^\perp$. Then 
\begin{align*}
	D^2_{u} \Lcal_K(w,w) &= D^2_{u} \Lcal_{\mu_*}(w,w)+2 K D_{u_{\mu_*}}F(\Phi_{g^{-1}}w_2)\cdot D_{u_{\mu_*}}F(\Phi_{g^{-1}}w_2)\\
	\ge&\, D^2_{u} \Lcal_{\mu_*}(w_1,w_1)-C(\|w_1\|\|w_2\|+\|w_2\|^2)\\
	&+2 K D_{u_{\mu_*}}F(\Phi_{g^{-1}}w_2)\cdot D_{u_{\mu_*}}F(\Phi_{g^{-1}}w_2)\\
	\ge&\, c\|w_1\|^2-C(\|w_1\|\|w_2\|+\|w_2\|^2)+Kc_{\mu_*} \|w_2\|^2,
\end{align*}
where in the last line we use the fact that $\mathrm{dim}W_{u_{\mu_*}}=m$ and ${D_{u_{\mu_*}}F}_{\mid W_{u_{\mu_*}}}: W_{u_{\mu_*}}\to \R^m$ is an isomorphism. Finally, thanks to Young's inequality, there exists $\varepsilon>0$ so that 
\begin{align*}
	D^2_{u} \Lcal_K(w,w)
	\ge&\, \left(c-\frac{C\varepsilon}{2}\right)\|w_1\|^2+\left(Kc_{\mu_*}-C-\frac{C}{2\varepsilon}\right) \|w_2\|^2\ge \tilde c \|w\|^2
\end{align*}
with $\tilde c>0$ provided that $K>0$ is chosen large enough. As a consequence, using the same arguments as in the proof of Proposition~\ref{thm:hessianestimate}, we conclude that there exist $\eta>0, c>0$ so that
\begin{equation*}
\forall u\in\Ocal_{u_{\mu_*}}, \forall v\in \Ban, \quad \rd(u,v)\leq \eta\Rightarrow \Lcal_{K}(v)-\Lcal_{K}(u)\geq c\rd^2(v, \Ocal_{u_{\mu}})
\end{equation*}
which implies, thanks to Theorem~\ref{thm:lyapmethod}, that all $u\in\Ocal_{u_{\mu_*}}$ are orbitally stable $G_{\mu_*}$-relative equilibria.
\end{proof} 

\section{Plane wave stability on the torus for NLS}\label{s:nlsetorus1d}

In this section we will illustrate the general theory described above on a simple example, that is the orbital stability of plane waves\index{plane waves!stability} of the cubic focusing and defocusing nonlinear Schr\"odinger equation on the one-dimensional torus. More precisely, let us consider the cubic Schr\"odinger equation  
\begin{equation}
 	\label{nlscubic}
 	i{\partial_t}u(t,x)+\beta{\partial_{xx}^2}u(t,x)+\lambda\vb u(t,x)\vb^2u(t,x)=0
\end{equation} 
in the space periodic setting $\TT_L$, the one-dimensional torus of length $L>0$, and with $u(t,x)\in\CC$. The constants $\beta$ and $\lambda$ are parameters of the model; $\beta \lambda <0$ corresponds to the defocusing case and $\beta \lambda >0$ to the focusing one. In what follows, we fix $\beta>0$. 

Using the same arguments as in Section \ref{ss:hampde}, we can show that Equation \eqref{nlscubic} is the Hamiltonian differential equation associated to the function $H$ defined by 
\begin{equation}
	\label{defenergy}
	H(u)=\frac{1}{2}\left(\beta\int_0^L\vb \partial_x u(x)\vb^2\diff x-\frac{\lambda}{2}\int_0^L\vb u(x)\vb^4\diff x\right).
\end{equation}
As before the symplectic Banach triple is given by $(\Ban,\Dom,\Jcal)$ with $\Ban=H^{1}(\T_L,\C)$, $\Dom=H^{3}(\T_L,\C)$, both viewed as real Hilbert spaces, and $\Jcal u=iu$ (see Section \ref{ss:linflows} to understand how a complex Hilbert space can be viewed as a real Hilbert space with symplectic structure). We recall that the scalar product on 
$\Ban=H^1(\TT_L,\CC)$ is  
\begin{equation}
	\label{defscalarprodX}
	(u,v)_\Ban=\mathrm{Re}\int_0^L(\partial_x  u(x)\partial_x \bar v(x)+u(x)\bar  v(x))\diff x\quad u,v\in \Ban,
\end{equation}
and the dual space $\Ban^*$ can be identified with $H^{-1}(\TT_L,\CC)$ through the pairing
\begin{equation}
	\label{defpairingX}
	\langle u,v \rangle=\mathrm{Re}\int_0^L u(x) \bar v(x)\diff x,\quad u\in \Ban^*,\ v\in \Ban.
\end{equation}
Moreover, since the action $\Phi$ of the group $G=\R\times\R$ defined by $\Phi_{a,\gamma}(u)=e^{i\gamma}u(x-a)$ is globally Hamiltonian (see Section \ref{ss:hampde}) and $H\circ \Phi_g=H$ (see Section \ref{s:dynsysexamples}), the quantities
\begin{align}
	\label{defmomentum}
	F_1(u)&=-\frac{i}{2}\int_0^L\bar u(x)\partial_x u(x)\diff x,\\
	\label{defcharge}
	F_2(u)&=-\frac{1}{2}\int_0^L\vb u(x)\vb^2\diff x=-\frac{1}{2}\langle u,u \rangle
\end{align}
are constants of the motion.

As  pointed out in Section \ref{s:dynsysexamples}, the two-parameter family of plane waves
\begin{equation}
 	\label{planewaves}
 	u_{\alpha,k}(t,x)=\alpha e^{-ikx}e^{i\en t}
\end{equation} 
with $\en\in \RR$, $k\in \frac{2\pi}{L}\ZZ$ and $\alpha \in \R$ are $G$-relative equilibria of \eqref{nlscubic} whenever $\en, k$ and $\alpha$ satisfy the dispersion relation
\begin{equation}
	\label{disprel}
	\en +\beta k^2=\lambda\vb\alpha\vb^2.
\end{equation}
In the notation of the previous sections, $u_{\alpha,k}=u_{\muak}$ with $\mu_{\alpha,k} \in \R^2$ given by 
$$
\mu_{\alpha,k}=\begin{pmatrix}F_1(u_{\alpha,k})\\ F_2(u_{\alpha,k})\end{pmatrix}=-\frac{\alpha^2}{2}L\begin{pmatrix}k \\ 1\end{pmatrix}. 
$$
Remark that in this case $\muak$ is not a regular value of $F=(F_1,F_2)$, as is readily checked (see Definition~\ref{def:levelsurfregular}).

The $G$-orbit of the initial condition $u_{\mu_{\alpha,k}}(x)=\alpha e^{-ikx}$ is given by 
\begin{equation}
	\label{eq:Gorbitplanewaves}
	\Ocal_{u_\mu{_{\alpha,k}}}=\big\{\alpha e^{i\gamma}e^{-ik(x-a)}, (a,\gamma)\in G\big\}.
\end{equation}
Our goal is to investigate the orbital stability of these particular solutions by applying the general arguments presented above. Our main result is the following theorem showing the orbital stability of plane waves in the defocusing case ($\lambda<0$) as well as in the focusing case provided  $0<2\lambda \vb \alpha\vb^2<\beta\left(\frac{2\pi}{L}\right)^2$.

\begin{theorem}
	\label{thmstability}
	If $\beta\left(\frac{2\pi}{L}\right)^2-2\lambda \vb \alpha\vb^2>0$, then all $u \in \Ocal_{u_{\muak}}$ are orbitally stable relative equilibria. 
	\end{theorem}

Furthermore, in the case $\beta\left(\frac{2\pi}{L}\right)^2-2\lambda \vb \alpha\vb^2<0$, we can investigate the linear stability of the plane waves and we obtain the following theorem.

\begin{theorem}
	\label{thmspectralinstability} Let the plane wave $u_{\alpha,k}(t,x)=\alpha e^{i(\xi t -kx)}$ 
	be a solution to~\eqref{nlscubic} and $\beta\left(\frac{2\pi}{L}\right)^2-2\lambda \vb \alpha\vb^2<0$. Then the spectrum of the linearization of \eqref{nlscubic} around $u_{\alpha,k}$ in $L^2(\TT_L)$ has eigenvalues with strictly positive real part. Consequently, this wave is spectrally unstable in $L^2(\TT_L)$.
\end{theorem}

This second result follows from a rather straightforward computation that we do not reproduce here.

As discussed in the introduction, the nonlinear (in)stability of plane waves for the cubic focusing and defocusing nonlinear Schr\" odinger equation in a one-dimensional space is a result  known to the experts in the field (see the introduction of  \cite{galhar07a,galhar07b}, for example). We did not however find a complete proof of it in the literature, so we furnish one here as an illustration of the general theory presented in the previous sections.

In \cite{zhidkov01}, a related but slightly different analysis is proposed. The cubic nonlinear Schr\"odinger equation is defined on the entire line $\RR$ and not on the one-dimensional torus $\TT_L$. Using the  Galilean invariance of the equation (see Section~\ref{ss:gensym}), the stability of any plane wave is equivalent to that of $u(t,x)=\alpha e^{i\lambda |\alpha|^2t}$. The main result on stability of plane waves of  \cite{zhidkov01} is given in Theorem III.3.1. It states that, in the defocusing case ($\lambda<0$), the plane wave $u(t,x)=\alpha e^{i\lambda |\alpha|^2t}$ is orbitally stable under small perturbations in $H^1(\RR)$. 

Our approach is different: we focus on the Schr\"odinger equation on a one-dimensional torus. Our functions live on a torus and the perturbations too. In other words, our definition of stability is with respect to perturbations within $H^1(\TT_L)=H^1_\mathrm{per}([0,L])$. Moreover in Zhidkov's book nothing is said about the (in)stability of plane waves in the focusing case, a situation we cover partially.

Finally, the analysis of orbital stability of plane waves of the cubic nonlinear Schr\"odinger equation on a torus of dimension $1<d\le 3$ is more involved and it will be done in a forthcoming paper together with the periodic Manakov equation \cite{debrot14}.

\subsection{Orbital stability}

To study the stability of $u_{\muak}(x)$, it is useful to write the solutions of \eqref{nlscubic} in the form
\begin{equation}
	\label{solpartnlscubic}
	u(t,x)=e^{-ikx}U(t,x)
\end{equation}
where $U(t,x)$ is a function which satisfies the evolution equation
\begin{equation}
	\label{nlscubicpartform}
	i{\partial_t}U+\beta{\partial^2_{xx}}U-2i\beta k {\partial_x}U+\lambda\vb U\vb^2U-\beta k^2 U=0.
\end{equation}
Equation \eqref{nlscubicpartform} is the Hamiltonian differential equation associated to the function $\tilde H$ defined by 
\begin{equation}
	\label{defenergybis}
	\tilde H(U)= H(U)-2\beta k F_1(U)-\beta k^2 F_2(U).
\end{equation}
As before, the action $\Phi$ of the group $G=\R\times\R$ defined by $\Phi_{a,\gamma}(u)=e^{i\gamma}u(x-a)$ is globally Hamiltonian, $\tilde H\circ \Phi_g=\tilde H$ and the quantities $F_1,F_2$ defined by \eqref{defmomentum} and \eqref{defcharge} are constants of the motion. 

If $\xi,k$ and $\alpha$ satisfy the dispersion relation \eqref{disprel}, $U_{\mu_\alpha}(t,x)=\alpha e^{i\xi t}$ is a solution to \eqref{nlscubicpartform}.
Moreover, $U_{\mu_\alpha}(x)=U_{\mu_\alpha}(0,x)=\alpha$ is a one-parameter family of $G$-relative equilibria and our goal is to study their stability. Here $\mu_\alpha=-\frac{\alpha^2}{2}L\begin{pmatrix}0\\1\end{pmatrix}$ and, as above, $\mu_\alpha$ is not a regular value of $F=(F_1,F_2)$. 

Recall that the $G$-orbit of $U_{\mua}(x)=\alpha$ is
\begin{equation}
\label{eq:Gorbitplanewave}
\Ocal_{U_{\mua}}=\left\{e^{i\gamma}\alpha,\gamma\in [0,2\pi)\right\}.
\end{equation} 
and, by definition, $U\in \Ocal_{U_{\mua}}$ is orbitally stable if 
$$
\forall \epsilon, \exists \delta, \forall W\in \Ban,\  \left(\rd(W,U)\leq \delta\Rightarrow \forall t\in\R, \ \rd(W(t,\cdot), \Ocal_{U_{\mua}})\leq \epsilon\right)
$$
(see Definition~\ref{def:orbstabgen}).

\begin{proposition}
	\label{propstability}
	Let $\beta\left(\frac{2\pi}{L}\right)^2-2\lambda \vb \alpha\vb^2>0$. Then every $U\in \Ocal_{U_{\mua}}$ is orbitally stable.
\end{proposition}

Our stability result in Theorem \ref{thmstability} is an immediate consequence of the previous statement since the change of variables $u\to U$ is bounded in $\Ban$.

Now, to prove this proposition, we would like to apply the general results given in the previous section and more precisely Theorem~\ref{thm:lyapmethod} or Theorem~\ref{thm:lyapmethodrestricted}. The idea is to construct a Lyapunov function $\mathcal L_{\mua}$ which is a group invariant constant of the motion and such that  $D\Lcal_{\mua}$ vanishes on $\Ocal_{U_{\mua}}$. Since $U_{\mua}$ is a $G$-relative equilibrium, Theorem \ref{thm:relequicritical} ensures that it satisfies 
$$
D_{U_{\mua}}\tilde H - \tilde\xi \cdot D_{U_{\mua}} F =0
$$ 
for some $\tilde\xi \in \R^2$. As a consequence, $\tilde H - \tilde\xi \cdot F$ is a good candidate to be a Lyapunov function. Nevertheless, since $D_{U_{\mu_\alpha}}F_1=0$, $\mua$ is not a regular value of $F$, and the choice of $\tilde\xi \in \R^2$ is not unique. Hence, working in the spirit of Section~\ref{ss:strategy}, we will consider only $F_2$ as constant of motion and we define
\begin{equation}
	\label{eqsigmamuplanewave}
	\Sigma_{\alpha}=\left\{W\in \Ban \mid F_2(W)=-\frac{\alpha^2}{2}L\right\}.
\end{equation}
With this definition, $\Sigma_{\alpha}$ is a co-dimension $1$ submanifold of $\Ban$.

Moreover, we need $\Lcal_{\mua}$ to be coercive on $\Ocal_{U_{\mua}}$, which means here that there exist $\delta>0$ and $c>0$, depending only on $\beta,L,\lambda$ and $\vb \alpha\vb^2$, such that, for all $W\in \Ban$ (as in \eqref{eq:coercive}) or $W\in \Sigma_{\alpha}$ (as in \eqref{eq:coerciverestricted}),
 \begin{equation}\label{eq:coerlyapplanewave}
 	\rd(W,\Ocal_{U_{\mua}})\le \delta \Rightarrow \mathcal L_{\mua}(W)- \mathcal L_{\mua}(U_{\mu_\alpha})\ge c\rd(W,\Ocal_{U_{\mua}})^2.
 \end{equation}
A convenient choice for $\mathcal L_{\mua}$ turn out to be
\begin{equation}
	\label{defF}
	\mathcal L_{\mua}(U)=H(U)-(\en +\beta k^2)F_2(U),
\end{equation}
which corresponds to $\tilde \xi= \begin{pmatrix}
	-2\beta k\\
	\en
\end{pmatrix}$.
By construction, $D_{U}\Lcal_{\mua}$ vanishes for $U\in \Ocal_{U_{\mua}}$. Indeed, since $D_U\Lcal_{\mua} \in \Ban^*$, $D_U\Lcal_{\mua} (V)=\langle D_U\mathcal L_{\mua},V\rangle$ with
\begin{equation}
	\label{eqfirstderF}
	D_U\Lcal_{\mua}= -\beta \partial_{xx}^2 U -\lambda \vb U\vb^2 U + (\xi +\beta k^2)U\in H^{-1}(\TT_L,\C),
\end{equation}
so clearly $D_U\Lcal_{\mua}=0$ if $U\in \Ocal_{U_{\mua}}$.
Furthermore, the bilinear form $D^2_U\Lcal_{\mua}:\Ban\times \Ban \to \RR$ is given by $D^2_U\Lcal(V,V)=\langle \nabla^2\Lcal_{\mua}(U)V,V\rangle$ with
\begin{equation}
	\label{eqHessianF}
	\nabla^2 \Lcal_{\mua}(U)V=-\beta \partial_{xx}^2 V - \lambda \vb U\vb^2 V -\lambda (\vb U \vb ^2 V + \bar V U^2) + (\en +\beta k^2)V\in H^{-1}(\TT_L;\C);
\end{equation}
in particular, for all $U\in \Ban$, $\nabla^2\mathcal L_{\mua}(U)$ is a bounded linear operator from $\Ban$ to $\Ban^*$ and the expression above makes sense.

Now to prove \eqref{eq:coerlyapplanewave}, the main ingredient is the property:
\begin{equation*}
	\exists c>0, \forall V\in \left(T_{U_{\mua}}\Ocal_{U_{\mua}}\right)^\perp,\
D_{U_{\mua}}^2\Lcal_{\mua}(V,V)\geq c\|V\|^2,
\end{equation*}
or 
\begin{equation*}
	\exists c>0, \forall V\in T_{U_{\mua}}\Sigma_{\alpha}\cap \left(T_{U_{\mua}}\Ocal_{U_{\mua}}\right)^\perp,\
D_{U_{\mua}}^2\Lcal_{\mua}(V,V)\geq c\|V\|^2,
\end{equation*}
where
\begin{align*}
&T_{U_{\mua}}\Sigma_{\alpha}=\{W\in \Ban, \langle \alpha, W\rangle=0\},\\
&\left(T_{U_{\mua}}\Ocal_{U_{\mua}}\right)^\perp=\{W\in \Ban, \langle i, W\rangle=0\}.
\end{align*}
This is proven in the following proposition, from which coercivity is deduced in Proposition~\ref{positivityF}.

\begin{proposition}
	\label{spectrumhessianF} 
	Let $ \beta \left(\frac{2\pi}{L}\right)^2-2\lambda  \alpha^2>0$ and $\alpha\neq 0$.
	\begin{enumerate}[label=({\alph*})]
	\item If $\lambda <0$ then
		\begin{equation}
		 	\label{eqcoercivitydefoc}
		 	D_{U_{\mua}}^2\Lcal_{\mua}(V,V)=\langle \nabla^2\mathcal L_{\mua}(U_{\mua})V,V\rangle \ge {c_\lambda} \norm V\norm^2
		\end{equation}
		for all $V\in \left(T_{U_{\mua}}\Ocal_{U_{\mua}}\right)^\perp$ and $c_\lambda=\min\left\{\frac{\beta \left(\frac{2\pi}{L}\right)^2}{1+\left(\frac{2\pi}{L}\right)^2}, -2\lambda  \alpha^2\right\}$.
		
		\item If $0< 2\lambda \alpha^2<\beta\left(\frac{2\pi}{L}\right)^2$ then,  
		\begin{equation}
		 	\label{eqcoercivityfoc}
		 	D_{U_{\mua}}^2\Lcal_{\mua}(V,V)=\langle \nabla^2\mathcal L_{\mua}(U_{\mua})V,V\rangle \ge c_\lambda \norm V\norm^2
		\end{equation}
		for all $V\in T_{U_{\mua}}\Sigma_{\alpha}\cap \left(T_{U_{\mua}}\Ocal_{U_{\mua}}\right)^\perp$ and $c_\lambda=\frac{ \beta \left(\frac{2\pi}{L}\right)^2-2\lambda  \alpha^2}{1+\left(\frac{2\pi}{L}\right)^2}$.
		\end{enumerate}
\end{proposition}


\begin{proof}
Let $V=v_1+iv_2=\vctl{v_1}{v_2}\in \Ban$. A straightforward calculation gives

\begin{align*}
	D_{U_{\mua}}^2\Lcal_{\mua}(V,V)&=\mathrm{Re}\int_{0}^L\left(-\beta \partial_{xx}^2 V -\lambda \alpha^2 ( V + \bar V ) \right)\bar V\\
	&=\int_{0}^L\beta(\vb \nabla v_1\vb^2 +\vb\nabla v_2\vb^2)-2\lambda\alpha^2\vb v_1\vb^2.
\end{align*}
Now, since $v_1$ and $v_2$ are real functions on the torus, we can write them in Fourier representation, namely,
	\begin{eqnarray*}
	  	v_1(x)=\frac{a_0(v_1)}{2}+\sum_{n=1}^{\infty}a_n(v_1)\cos\left(\frac{2\pi}{L}nx\right)+b_n(v_1)\sin\left(\frac{2\pi}{L}nx\right),\\
	  	v_2(x)=\frac{a_0(v_2)}{2}+\sum_{n=1}^{\infty}a_n(v_2)\cos\left(\frac{2\pi}{L}nx\right)+b_n(v_2)\sin\left(\frac{2\pi}{L}nx\right),
	\end{eqnarray*}
and recall that
\begin{align*}
\norm V\norm^2=&\frac{L}{2}\left(\frac{a^2_0(v_1)}{2}+\sum_{n=1}^{\infty}\left(\left(\frac{2\pi}{L}n\right)^2+1\right)(a_n^2(v_1)+b_n^2(v_1))\right)\\
	&+\frac{L}{2}\left(\frac{a^2_0(v_2)}{2}+\sum_{n=1}^{\infty}\left(\left(\frac{2\pi}{L}n\right)^2+1\right)(a_n^2(v_2)+b_n^2(v_2))\right).
\end{align*} 
Next, 
\begin{align*}
	D_{U_{\mua}}^2\Lcal_{\mua}(V,V)=&\frac{L}{2}\left(-2\lambda \alpha^2\frac{a^2_0(v_1)}{2}+\sum_{n=1}^{\infty}\left(\beta\left(\frac{2\pi}{L}n\right)^2-2\lambda \alpha^2\right)(a_n^2(v_1)+b_n^2(v_1))\right)\\
	&+\frac{L}{2}\left(\sum_{n=1}^{\infty}\beta\left(\frac{2\pi}{L}n\right)^2(a_n^2(v_2)+b_n^2(v_2))\right).
\end{align*}
\begin{enumerate}[label=({\alph*})]
	\item If $\lambda<0$, it is clear that $D_{U_{\mua}}^2\Lcal_{\mua}(V,V)\geq 0$ for all $V\in \Ban$. Moreover, if $V\in \left(T_{U_{\mua}}\Ocal_{U_{\mua}}\right)^\perp$, then $\langle i,V\rangle=0$, that is  $a_0(v_2)=0$. Hence, the coercivity property of $D_{U_{\mua}}^2\Lcal_{\mua}(\cdot,\cdot)$ on $\left(T_{U_{\mua}}\Ocal_{U_{\mua}}\right)^\perp$ follows easily.
	\item Now, let $0< 2\lambda \alpha^2<\beta\left(\frac{2\pi}{L}\right)^2$ and $V\in T_{U_{\mua}}\Sigma_{\alpha}\cap \left(T_{U_{\mua}}\Ocal_{U_{\mua}}\right)^\perp$. As a consequence, $\langle i,V\rangle=0=\langle \alpha,V\rangle$ which implies $a_0(v_1)=0=a_0(v_2)$. As before, the coercivity property of $D_{U_{\mua}}^2\Lcal_{\mua}(\cdot,\cdot)$ on $T_{U_{\mua}}\Sigma_{\alpha}\cap \left(T_{U_{\mua}}\Ocal_{U_{\mua}}\right)^\perp$ follows.
\end{enumerate}

\end{proof}


The following lemma gives a representation of the elements of $\Ban$ which are close to the $G$-orbit $\Ocal_{U_{\mua}}$. It is used in the proof of Proposition~\ref{positivityF} and is a special case of Lemma~\ref{lem:stuartlem}. We give a direct proof in the current simple setting. 

\begin{lemma}
	\label{lemmodulation}
	There exists $\delta>0$ such that any $W\in \Ban$ with $\rd(W,\Ocal_{U_{\mua}})\le\delta$ can be represented as
	\begin{equation}
		\label{eqdecomposition}
		e^{i\gamma}W=U_{\mua}+V
	\end{equation}
	with $\gamma=\gamma(W)\in[0,2\pi)$ and $V\in \left(T_{U_{\mua}}\Ocal_{U_{\mua}}\right)^\perp$. Moreover, there exists a positive constant $C$ such that
	\begin{equation}
		\label{estimatey}
		\rd(W,\Ocal_{U_{\mua}})\le \norm V\norm\le C\rd(W,\Ocal_{U_{\mua}}).
	\end{equation}
\end{lemma}

\begin{proof}
	Let $W\in \Ban$ such that $\rd (W,\Ocal_{U_{\mua}})<\delta$ with $\delta>0$ sufficiently small. Hence there exists $\tilde \gamma$, which depends on $W$, such that 
	\begin{displaymath}
		\norm e^{i\tilde \gamma}W-\alpha\norm\le 2 \inf_{\lambda\in[0,2\pi)}\norm W-e^{i\lambda}\alpha\norm\le 2\delta 
	\end{displaymath}

	Next, consider the functional
	\begin{align*}
		\mathcal F: \Ban\times \RR&\rightarrow \RR\\
		(v,\phi) &\to \langle e^{i\phi}v,i\rangle=-\mathrm{Re}\int_{0}^L i e^{i\phi} v(x)\diff x.
	\end{align*}
	Since $\mathcal F(\alpha,0)=0$ and $\partial_{\phi}\mathcal F(\alpha,0)=\alpha L\neq 0$, by means of the implicit function theorem, we can conclude that there exists $\Lambda : \mathcal V\rightarrow (-\varepsilon,\varepsilon)$ with $\mathcal V$ a neighbourhood of $\alpha$ in $\Ban$ and $\varepsilon>0$ sufficiently small, such that if $v\in \mathcal V$ then there exists a unique $\phi=\Lambda(v)\in (-\varepsilon,\varepsilon)$ for which we have $\langle e^{i\phi}v,i\rangle=0$.

	As a consequence, since $\norm e^{i\tilde \gamma}W-\alpha\norm<2\delta$, if we choose $\delta>0$ sufficiently small then there exists $\phi\in \RR$ such that $\langle e^{i(\phi+\tilde \gamma)}W,i\rangle=0$. By taking $\gamma=\tilde \gamma +\phi$ modulo $2\pi$, we obtain \eqref{eqdecomposition}. Indeed, $\Ban=T_{U_{\mua}}\Ocal_{U_{\mua}}\oplus\left(T_{U_{\mua}}\Ocal_{U_{\mua}}\right)^\perp$ and
	$T_{U_{\mua}}\Ocal_{U_{\mua}}=\mathrm{span}_{\R}\left\{i\right\}$. Hence, 
	\begin{displaymath}
		e^{i\gamma}W-\alpha=ai+V
	\end{displaymath}
	with $a\in \RR$ and $V\in \left(T_{U_{\mua}}\Ocal_{U_{\mua}}\right)^\perp$. As a consequence,
	\begin{displaymath}
		0=\langle e^{i\gamma}W-\alpha,i\rangle=a\langle i,i\rangle
	\end{displaymath}
	and $a$ has to be equal to $0$.

	Estimate \eqref{estimatey} follows directly from the definition of $V$.
\end{proof}

Finally, the following proposition, in the same spirit of Proposition~\ref{thm:hessianestimate}, proves the coercivity of $\Lcal_{\mua}$ on $\Ocal_{U_{\alpha}}$.

\begin{proposition}
	\label{positivityF} Let $\beta\left(\frac{2\pi}{L}\right)^2 -2\lambda \vb \alpha \vb^2>0$, $\alpha\neq 0$ and $\mathcal L_{\mua}$ be defined as in \eqref{defF}. 
	\begin{enumerate}[label=({\alph*})]
	\item If $\lambda <0$, let $c_\lambda=\min\left\{\frac{\beta \left(\frac{2\pi}{L}\right)^2}{1+\left(\frac{2\pi}{L}\right)^2}, -2\lambda  \alpha^2\right\}$.
	Then there exists $\delta>0$ such that 
	\begin{equation}
		\label{eqposF}
		\mathcal L_{\mua}(W)-\mathcal L_{\mua}(U_{\mua})\ge \frac{c_\lambda}{4}\rd(W,\Ocal_{U_{\mua}})^2
	\end{equation}
	for all $W\in \Ban$, such that $\rd(W, \Ocal_{U_{\mua}})\le\delta$.
	\item If $0< 2\lambda \alpha^2<\beta\left(\frac{2\pi}{L}\right)^2$, let $c_\lambda=\frac{ \beta \left(\frac{2\pi}{L}\right)^2-2\lambda  \alpha^2}{1+\left(\frac{2\pi}{L}\right)^2}$.
	Then there exists $\delta>0$ such that 
	\begin{equation}
		\label{eqposFbis}
		\mathcal L_{\mua}(W)-\mathcal L_{\mua}(U_{\mua})\ge \frac{c_\lambda}{16}\rd(W,\Ocal_{U_{\mua}})^2
	\end{equation}
	for all $W\in \Sigma_{\alpha}$, such that $\rd(W, \Ocal_{U_{\mua}})\le\delta$.
	\end{enumerate}
\end{proposition}

\begin{proof}
	Let $W\in \Ban$ such that $\rd (W,\Ocal_{U_{\mua}})\le \delta$ with $\delta>0$ sufficiently small. By Lemma \ref{lemmodulation}, there exists $\gamma\in [0,2\pi]$ such that $e^{i\gamma}W-U_{\mua}=V$ with $V\in \left(T_{U_{\mua}}\Ocal_{U_{\mua}}\right)^\perp$ and $\norm V \norm\le C\rd(W,\Ocal_{U_{\mua}})$. As a consequence, since $D_{U_{\mu_\alpha}}\Lcal_{\mua}=0$, 
	\begin{align*}
		\mathcal L_{\mua}(W)-\mathcal L_{\mua}(U_{\mua})=& \mathcal L_{\mua}(e^{i \gamma}W)-\mathcal L_{\mua}(U_{\mua})\\
		=& \frac{1}{2}D^2_{U_{\mu_\alpha}}\Lcal_{\mua}(e^{i\gamma}W-U_{\mua},e^{i\gamma}W-U_{\mua}) 
		+\mathrm{o}(\norm e^{i\gamma}W-U_{\mua}\norm^2) \\
		=&\frac{1}{2}D^2_{U_{\mu_\alpha}}\Lcal_{\mua}(V,V)+\mathrm{o}(\norm V\norm^2).
	\end{align*}
	If $\lambda<0$, we can apply \eqref{eqcoercivitydefoc}, and for all $W\in \Ban$ with $\rd(W,\Ocal_{U_{\mua}})$ small, we obtain
	\begin{displaymath}
		\mathcal L_{\mua}(W)-\mathcal L_{\mua}(U_{\mu_\alpha})\ge \frac{c_\lambda}{4}\rd(W,\Ocal_{U_{\mua}})^2.
	\end{displaymath}
	If $0< 2\lambda \alpha^2<\beta\left(\frac{2\pi}{L}\right)^2$, we proceed as follows. Let $W\in \Sigma_{\alpha}$ such that $\rd (W,\Ocal_{U_{\mua}})\le \delta$ with $\delta>0$ sufficiently small. As before, thanks to Lemma \ref{lemmodulation}, there exists $\gamma\in [0,2\pi]$ such that $e^{i\gamma}W-U_{\mua}=V$ with $V\in \left(T_{U_{\mua}}\Ocal_{U_{\mua}}\right)^\perp$ and $\norm V \norm\le C\rd(W,\Ocal_{U_{\mua}})$.
	Next, it is clear that $\Ban=T_{U_{\mua}}\Sigma_{\alpha}\oplus\mathrm{span}\left\{U_{\mua}\right\}$. Hence, $V=V_1+V_2$ with $V_1\in T_{U_{\mua}}\Sigma_{\alpha}\cap \left(T_{U_{\mua}}\Ocal_{U_{\mua}}\right)^\perp$ and $V_2\in \mathrm{span}\left\{U_{\mua}\right\}\cap \left(T_{U_{\mua}}\Ocal_{U_{\mua}}\right)^\perp$. Moreover, using the same arguments as in Lemma~\ref{eq:basicestimate}, for $\| V\|$ small enough, one has
\begin{eqnarray*}
 \|V_2\| \leq \frac{1}{2\sqrt{\alpha^2L}}\|V\|^2\text{ and }\|V_1\| \geq \frac{1}{2}\|V\|.
\end{eqnarray*}
As a consequence, 
	\begin{align*}
		\mathcal L_{\mua}(W)-\mathcal L_{\mua}(U_{\mua})=&\frac{1}{2}D^2_{U_{\mu_\alpha}}\Lcal_{\mua}(V,V)+\mathrm{o}(\norm V\norm^2)
		=\frac{1}{2}D^2_{U_{\mu_\alpha}}\Lcal_{\mua}(V_1,V_1)+\mathrm{o}(\norm V\norm^2)\\
		&\geq \frac{c_\lambda}{2}\norm V_1\norm^2+\mathrm{o}(\norm V\norm^2)\geq \frac{c_\lambda}{8}\norm V\norm^2+\mathrm{o}(\norm V\norm^2).
	\end{align*}
Finally, if $0< 2\lambda \alpha^2<\beta\left(\frac{2\pi}{L}\right)^2$, and for all $W\in \Sigma_{\alpha}$ with $\rd(W,\Ocal_{U_{\mua}})$ small, we obtain
	\begin{displaymath}
		\mathcal L_{\mua}(W)-\mathcal L_{\mua}(\alpha)\ge \frac{c_\lambda}{16}\rd(W,\Ocal_{U_{\mua}})^2.
	\end{displaymath}
\end{proof}

Now, whenever $\lambda \vb \alpha\vb^2<0$, a straightforward application of the proof of Theorem \ref{thm:lyapmethod} with $\mathcal L_{\mua}$ as Lyapunov function allows us to conclude that $\Ocal_{U_{\mua}}$ is orbitally stable under small perturbations in $\Ban$.

To conclude in the case $0< 2\lambda \alpha^2<\beta\left(\frac{2\pi}{L}\right)^2$, we can apply Theorem~\ref{thm:lyapmethodrestricted}. Indeed, Hypotheses~A and B$\mu_{\alpha}$ (Section~\ref{ss:strategy}) are fulfilled and the function $F_2$ satisfies Hypothesis~F thanks to Lemma~\ref{lem:hypF}.


\section{Orbital stability for inhomogeneous NLS}\label{curves.sec}

This section is concerned with an NLS equation of the form
\begin{equation}\label{nls}
i\partial_t u+\Delta u+f(x,|u|^2)u=0, \quad u=u(t,x):\real\times\rn\to\C.
\end{equation}
We consider standing wave\index{standing waves} solutions $u(t,x)=e^{i\lam t}w(x)$, 
where $w:\rn\to\R$ is localized\footnote{Note that we focus here on situations where the wave profile
$w(x)$ is real-valued.}
-- typically $w\in H^1(\rn)$ and $w(x)\to0$ exponentially
as $|x|\to\infty$. Such a solution exists if and only if
\begin{equation}\label{stat}
\Delta w -\lam w+f(x,w^2)w=0,
\end{equation}
which is precisely the ``stationary equation'' \eqref{eq:findrelequi}.
Note that the notation for the nonlinearity in \eqref{nls} is slightly different than in 
Section~\ref{s:dynsysexamples}, and automatically ensures that \eqref{eq:nlsextnl} holds,
for all $u\in \C\setminus\{0\}$.

The existence of solutions of \eqref{stat} can be obtained under various hypotheses 
on $f$, the easiest case being the pure power nonlinearity, $f(x,w^2)= |w|^{\sigma-1}, \ \sigma>1$.
Note that, unlike in the case of periodic boundary conditions studied in the previous section,
it is crucial here that the nonlinearity be focusing for standing waves to exist. The stationary
equation \eqref{stat} has no solutions if, for instance, $f(x,w^2)= - |w|^{\sigma-1}$.
In the sequel, we will indeed suppose that the nonlinearity is focusing, which in the context
of \eqref{nls} means that $f(x,s)$ is positive and increasing in $s>0$.

The purpose of this section is to further illustrate 
the general stability theory developed in Section~\ref{s:orbstabproof}.
Orbital stability results for standing waves of \eqref{nls} have been obtained in \cite{dcds,jde,ans,eect}
and will be summarized here.
The stability analysis in these papers benefits from having solution curves
$\lam\to w_\lam$. In the setting of Section~\ref{s:orbstabproof}, they
can be seen as an application of Theorem~\ref{thm:lyapmethodrestrictedmod}.
The approach used in \cite{dcds,jde,ans,eect} was to apply the celebrated
Theorem~2 of Grillakis, Shatah, Strauss \cite{gssI}. This result essentially relies on the
set of spectral conditions (S1)--(S3), formulated below in the context of \eqref{nls}, 
together with a convexity condition, which here takes the form \eqref{slope}.
In the framework developed in these notes, the role of Theorem~2 of \cite{gssI} can be interpreted
as follows. It will be shown in Proposition~\ref{slope-coerc.thm} that the conditions (S1)--(S3) and 
\eqref{slope} ensure that the coercivity property \eqref{eq:hessiancondition} required by 
Proposition~\ref{thm:hessianestimate} is satisfied at the relative equilibrium $w_\lam$.
Theorem~\ref{thm:lyapmethodrestrictedmod} 
can then be applied. As already mentioned in the
introduction to Section~\ref{s:orbstabproof}, and explained in more detail
after the proof of Proposition~\ref{slope-coerc.thm}, the relative equilibria of \eqref{nls} can be 
parametrized equivalently by the parameter $\xi$ appearing in \eqref{stat}, or by the corresponding
value $\mu=\frac12\Vert w_\xi\Vert_{L^2}^2$ of the constant of the motion.
It turns out that using $\xi$ is more convenient here.
Note that, since this constant of the motion satisfies Hypothesis~F, one could also apply 
Theorem~\ref{thm:lyapmethodrestricted} instead of Theorem~\ref{thm:lyapmethodrestrictedmod}.

The notion of orbital stability we shall be concerned with here is that corresponding to the group action 
\eqref{restraction} of Section~\ref{ss:hampde}. Note however that the explicit spatial dependence
in \eqref{nls} breaks the invariance under translations, and one rather needs to consider the
restricted action $\Phi_\gamma$ on the phase space $E=H^1(\rn,\complex)$,
\begin{equation}
\Phi_{\gamma}(u)=e^{i\gamma}u(x), \quad u\in E, \ \gamma\in\real.  \label{morerestraction}
\end{equation}
The standing waves
corresponding to solutions $w_\lam$ of the stationary equation \eqref{stat}
are then relative equilibria for the dynamics of \eqref{nls}, with respect to the action $\Phi_\gamma$. 

\begin{remark}\label{transl.rem}
\rm
If $f$ does not depend on $x$ then the full group action \eqref{restraction}
is to be considered, and the standing waves of \eqref{nls} are in general not orbitally stable
in the sense of \eqref{morerestraction}. Orbital stability in the sense of 
the full group action \eqref{restraction} was proved by Cazenave and Lions \cite{cazlions}
by variational arguments.
\end{remark}

We will only consider here situations where the coefficient $f$ explicitly
depends on the space variable
$x\in\rn$ -- \eqref{nls} is then often referred to as an {\em inhomogeneous NLS} --, 
and decays as $|x|\to\infty$, in a sense that will be made more precise below. 
We shall also suppose
that $f(x,w^2)\sim V(x)|w|^{\sigma-1}$ as $w\to0$. Conditions relating the function $V$
and the power $\sigma>1$ will be given for stability of standing waves to hold. In particular
our assumptions will imply $\sigma<1+\frac{4}{d-2}$, so that local existence in $H^1(\rn)$ for the
Cauchy problem associated with \eqref{nls} is ensured by the results of Section~\ref{s:dynsysexamples}.
Two cases will be considered: 

\medskip
\noindent
{\bf (PT)} the power-type nonlinearity $f(x,w^2)=V(x)|w|^{\sigma-1}$;

\medskip
\noindent
{\bf (AL)} the asymptotically linear case $f(x,w^2) \to V(x)$ as $|w|\to\infty$ \\
({\em e.g.} with $f(x,w^2)=V(x)\frac{|w|^{\sigma-1}}{1+|w|^{\sigma-1}}$).

\medskip
We will give a short account of the main arguments used in \cite{dcds,jde,ans,eect} 
to establish the stability of standing waves along a global solution curve.
We will also briefly sketch the bifurcation analysis yielding a smooth branch of non-trivial solutions
of \eqref{stat} emerging from the trivial solution $w=0$. This part of the argument is crucial since,
in the approach originally developed in \cite{dcds}, the spectral properties and the condition
\eqref{slope} required to obtain the coercivity of an appropriate Lyapunov functional are derived by continuation from the limit $w_\lam\to0$. It is worth emphasizing 
here that the verification of these hypotheses is precisely that part of the stability analysis which 
strongly relies on the model considered.
Once the required coercivity properties are established, the orbital stability can be deduced
from the abstract results of Section~\ref{s:orbstabproof}. 


\subsection{Hamiltonian setting}

Similarly to Section~\ref{s:nlsetorus1d}, we work here with
\[
E=H^1(\rn,\complex), \quad (u,v)_E=\re \intrn \nabla u(x) \cdot \nabla\bar v(x) + u(x) \bar v(x) \diff x.
\]
The Hamiltonian and the charge are respectively defined by $H,Q:E\to\real$,
\begin{equation}\label{energie}
H(u)=\frac12\intrn |\nabla u|^{2}\diff x-\frac12\intrn \int_0^{|u|^2}f(x,s)\diff s\diff x,
\quad Q(u)=\frac{1}{2}\intrn |u|^2\diff x, \quad  u \in E.
\end{equation}
In the notation of Section~\ref{ss:hampde},
$Q(u)\equiv -F_{d+1}(u)$, but we will keep the customary notation $Q$ here.
Under our assumptions, $H,Q\in C^2(\Ban,\real)$.

Now \eqref{nls} can precisely be written in the form
\begin{equation}\label{hamiltonian}
\calJ\dot{u}_t=D_{u_t}H
\end{equation}
considered in Section~\ref{s:hamdyninfinite},
with $E=H^1(\rn,\complex)\simeq H^1(\rn,\real)\times H^1(\rn,\real)$
and
$$
\calJ=
\begin{pmatrix}
0 & -I\\
I & 0
\end{pmatrix}
$$
with $I:H^1\hookrightarrow H^{-1}$ the (dense) injection. 
That is, $\calJ(q,p)=(-p,q)\in E^*$, for all $(q,p)\in E$, as in Section~\ref{ss:hampde}.
Note that we use the identification
$$
H^1(\rn,\real)\subset L^2(\rn,\real) = L^2(\rn,\real)^*\subset H^{-1}(\rn,\real).
$$

In this setting a solution of \eqref{nls} is a function 
$u\in C^1((-T_\mathrm{min},T_\mathrm{max}),E)$, for some 
$T_\mathrm{min},T_\mathrm{max}>0$
(depending on $u(0)$),
satisfying \eqref{hamiltonian} for all $t\in(-T_\mathrm{min},T_\mathrm{max})$.
Standing waves are particular solutions of the form $u(t)=\Phi_{(\lam t)}w$, $w\in E$,
and the stationary equation \eqref{stat} now reads
\begin{equation}\label{stat3}
D_w H+\lam D_w Q=0.
\end{equation}
Hence, the discussion in Sections~\ref{s:identifyreleq} and \ref{s:orbstabproof} indicates that
\begin{equation}\label{GSSlyap}
\calL_\lam=H+\lam Q 
\end{equation}
is the natural candidate for the Lyapunov function. 
Furthermore, the invariance of $H$ and $Q$ under the action of $\Phi_\gamma$
implies that
\begin{equation}\label{invar}
D_{\Phi_\gamma(w)}H+\lam D_{\Phi_\gamma(w)}Q=0, \quad \gamma\in\R.
\end{equation}
Finally, note that
the isometric action \eqref{morerestraction} can equivalently be expressed as 
$$
\Phi_\gamma 
\left(
\begin{array}{c}
\re u\\
\im u\\
\end{array}
\right)
=
\begin{pmatrix}
\cos \gamma & -\sin \gamma\\
\sin \gamma & \cos \gamma
\end{pmatrix}
\left(
\begin{array}{c}
\re u\\
\im u\\
\end{array}
\right), \quad u\in E, \ \gamma\in\real.
$$


\subsection{Bifurcation results}\label{bifcurves.sec}

In this section we present bifurcation\index{bifurcation} results ensuring the existence of smooth curves of solutions
of \eqref{stat}. From a bifurcation-theoretic viewpoint the peculiarity of these results is that, in both
the (PT) and (AL) cases, bifurcation occurs from the essential spectrum of
the linearization of \eqref{stat}, namely
$$
\Delta w = \lam w,
$$
this linear problem set on $\rn$ having no eigenvalues.

We start with the power-type case (PT), that is, we first consider the problem
\begin{equation}\label{ptstat}
\Delta w(x) + V(x)|w(x)|^{\sigma-1}w(x) = \lam w(x), \quad w \in H^1(\rn,\real),
\end{equation}
where $d\geq1$ and $V:\rn\to\real$ satisfies:

\medskip
\noindent
(V1) $V\in C^1(\real^d)$;

\smallskip
\noindent
(V2) there exists $b\in(0,2)$ ($b\in(0,1)$ if $d=1$) such that 
$$1<\sigma<\textstyle\frac{4-2b}{d-2} \quad \text{if} \ d\geq3, 
\quad 1<\sigma<\infty \quad \text{for} \ d=1,2,$$
$$\lim_{|x|\to\infty}|x|^{b}V(x)=1 \quad \text{and} \quad 
\lim_{|x|\to\infty}|x|^{b}[x\cdot\nabla V(x)+bV(x)]=0;$$

\noindent
(V3) $V$ is radial with $V(r)>0$ and $V'(r)<0$ for $r>0$;

\smallskip
\noindent
(V4) $\disp r\frac{V'(r)}{V(r)}$ is decreasing in $r>0$ (and so $\to -b$ by (V2)).

\medskip
\noindent
Note that $V(x)=(1+|x|^2)^{-b/2}$ satisfies all of the above assumptions.

\begin{theorem}\label{ptglobal.thm}
Suppose that the hypotheses (V1) to (V4) hold. Then there exists a curve 
$w\in C^1\big((0,\infty),H^1(\rn)\big)$ such that, for all $\lam\in(0,\infty)$,
$w_\lam\equiv w(\lam)$ is the unique positive radial solution of \eqref{ptstat}, 
$w_\lam\in C^2(\real^d)\cap L^\infty(\real^d)$, and $w_\lam$ is strictly radially
decreasing, with $w_\lam(x), |\nabla w_\lam(x)|\to0$ exponentially as $|x|\to\infty$. 
Furthermore, the asymptotic behaviour of the curve reads
$$
\lim_{\lam\to0}\Vert w_\lam\Vert_{H^1}=
\left\{ \begin{array}{ll}
0 		& \ \text{if} \quad 1<\sigma<1+\frac{4-2b}{d}, \\
\infty 	& \ \text{if} \quad 1+\frac{4-2b}{d}<\sigma<1+\frac{4-2b}{d-2},
\end{array}\right.
$$
and
$$
\lim_{\lam\to\infty}\Vert w_\lam\Vert_{H^1}=\infty 
\textstyle \quad\text{for all}\quad 1<\sigma<1+\frac{4-2b}{d-2}.
$$
\end{theorem}

This theorem has been proved in \cite{calcvar} by a combination of variational and analytical 
arguments.
It provides a global continuation, in the radial case, of the local curve of solutions of \eqref{ptstat}  
obtained in \cite{dcds} (parametrized by $\lam\in(0,\lam_0)$, with $\lam_0>0$ small)
under the much weaker assumptions (V1) and (V2). Note in particular that (V2) only requires
the problem to be focusing at infinity, no further sign restrictions being imposed on $V$.
The orbital stability
of the solutions $w_\lam, \ \lam\in(0,\lam_0)$, is also discussed in \cite{dcds}, and it 
is found that they are stable provided
\begin{equation}\label{stablexp}
1<\sigma<1+\txt\frac{4-2b}{d},
\end{equation}
and unstable if $1+\frac{4-2b}{d}<\sigma<1+\frac{4-2b}{d-2}$. 

\begin{remark}
\rm
In fact, more information about the asymptotic behaviour as $\lam\to0$ is obtained in 
\cite{dcds}. In particular,
\begin{equation*}
\lim_{\lam\to0}\Vert w_\lam\Vert_{L^2}=
\left\{ \begin{array}{ll}
0 		& \ \text{if} \quad 1<\sigma<1+\frac{4-2b}{d}, \\
\infty 	& \ \text{if} \quad 1+\frac{4-2b}{d}<\sigma<1+\frac{4-2b}{d-2},
\end{array}\right.
\end{equation*}
whereas
$$
\lim_{\lam\to0}\Vert \nabla w_\lam\Vert_{L^2}=0 
\quad\text{for all} \  1<\sigma<1+\txt\frac{4-2b}{d-2}.
$$
\end{remark}

We now state a global bifurcation result similar to Theorem~\ref{ptglobal.thm},
for \eqref{stat} in dimension $d=1$, in the asymptotically linear case (AL). That is, we consider
\begin{equation}\label{1dstat}
w''(x)+f(x,w(x)^2)w(x)=\lam w(x), \quad w \in H^1(\real,\real),
\end{equation}
where, to fix the ideas\footnote{More general 
assumptions on the coefficient $f$ in (AL) can be given, under which
the bifurcation and stability results presented here still hold, see \cite{eect}.}, 
we let
\begin{equation}\label{alnonlin}
f(x,w^2)=V(x)\frac{|w|^{\sigma-1}}{1+|w|^{\sigma-1}}.
\end{equation}

In the asymptotically linear case, one cannot expect to find positive solutions of 
\eqref{1dstat}--\eqref{alnonlin} for large values of $\lam>0$. 
Heuristically, letting $u\to\infty$ in \eqref{1dstat}--\eqref{alnonlin} leads to
the so-called asymptotic linearization 
\begin{equation}\label{aslin}
w''(x)+V(x)w(x)=\lam w(x),
\end{equation}
having a ray of positive eigenfunctions $\{\mu w_\infty: \mu>0\}$ corresponding to 
a principal eigenvalue $\ly>0$.
This has been put on rigorous grounds in \cite{na}, where it is shown that 
positive even solutions of \eqref{1dstat}--\eqref{alnonlin} only exist for $\lam<\ly$,
and satisfy $\Vert w_\lam\Vert_{H^1}\to\infty$ as $\lam\to\ly$.

\begin{theorem}\label{alglobal.thm}
Suppose (V1) to (V3) and $1<\sigma<5-2b$. 
Then there exists a curve $w\in C^1\big((0,\ly),H^1(\real)\big)$ such that, for all $\lam\in(0,\ly)$, 
$w_\lam$ is the unique positive even solution of \eqref{1dstat}--\eqref{alnonlin}, 
$w_\lam\in C^2(\real)\cap H^2(\real)$ with $w_\lam'(x)<0$ for $x>0$, 
and $w_\lam(x),w_\lam(x)'\to0$ exponentially as $|x|\to\infty$.
Furthermore, there holds
$$
\lim_{\lam\to0}\Ve w_\lam\Ve_{H^1(\real)}=0 \quad\text{and}\quad
\lim_{\lam\to\ly}\Ve w_\lam\Ve_{H^1(\real)}=\infty.
$$
\end{theorem}

\begin{remark}
\rm
The reader might wonder
why (V4) is not needed for Theorem~\ref{alglobal.thm}. It turns out that this assumption
is essential in the proof of Theorem~\ref{ptglobal.thm}, where it ensures uniqueness of positive 
radial solutions of \eqref{ptstat}, for any fixed $\lam>0$. In the one-dimensional problem 
\eqref{1dstat}--\eqref{alnonlin}, uniqueness can be proved without invoking
(V4).\footnote{Note that 
the main reason for restricting the discussion to $d=1$ in Theorem~\ref{alglobal.thm} 
is the lack of uniqueness results in higher dimensions for the (AL) case.}
However, we will see in the next section that this hypothesis is crucial to the stability analysis, 
in both the (PT) and (AL) cases.
\end{remark}

\begin{remark}
\rm
Thanks to the form of the nonlinearity in \eqref{alnonlin} the global branch of 
Theorem~\ref{alglobal.thm}, bifurcating from the trivial solution $u=0$
at $\lam=0$, is obtained by perturbation from the (PT) nonlinearity dealt with
in Theorem~\ref{ptglobal.thm}. In fact, the case where asymptotic bifurcation occurs
at $\lam=0$, corresponding in dimension $d=1$ to $5-2b<\sigma<\infty$, could also be
extended to the (AL) case, where instability could be inferred, in the limit $\lam\to0$.
We refrain from going in this direction here since we were only able so far to extend the discussion
to a global branch in the stable case. We shall therefore assume \eqref{stablexp} from now
on, both for (PT) and (AL).
\end{remark}


\subsection{Stability}\label{stablecurves.sec}

In dimension $d=1$, assuming that $1<\sigma<5-2b$, the global curves
of standing wave solutions given by Theorems~\ref{ptglobal.thm} and 
\ref{alglobal.thm} are stable. 
This has been proved in \cite{ans} for the (PT) case and in \cite{eect}
for the (AL) case. The proofs rely on the theory of orbital stability in \cite{gssI} and we will now outline
the main arguments.

We shall start by convincing the reader that, in the context of \eqref{nls}, one cannot hope for
stability in the usual sense \eqref{eq:stab0}.
Indeed, suppose $\lam_n\to\lam$ and consider 
$$
u_\lam(t,x)=e^{i\lam t}w_\lam(x)\quad\text{and}\quad 
u_n(t,x)=e^{i\lam_n t}w_{\lam_n}(x).
$$
Then 
$$
\forall\delta>0 \ \exists N_\delta\in\nat,
\ n\geq N_\delta \Rightarrow \Ve u_n(0,\cdot)-u_\lam(0,\cdot)\Ve_{H^1}
=\Ve w_{\lam_n}-w_\lam\Ve_{H^1}\leq\delta.
$$
However,
$$
\Ve u_n(t,\cdot)-u_\lam(t,\cdot)\Ve_{H^1}\geq 
\big| |e^{i\lam t}-e^{i\lam_n t}|\Ve w_\lam\Ve_{H^1} 
- \Ve w_{\lam_n}-w_\lam\Ve_{H^1} \big|
$$
$$
\Rightarrow\sup_{t\geq0}\Ve u_n(t)-u_\lam(t)\Ve_{H^1}\geq 
2\Ve w_\lam\Ve_{H^1}-\delta, \ n\geq N_\delta.
$$
Therefore, for $n$ large enough, the initial datum $u_n(0)$ may be chosen $\delta$-close
to $u_\lam(0)$, $u_n(t)$ will nevertheless drift at least $2\Ve w_\lam\Ve_{H^1}-\delta$
far away from $u_\lam(t)$.

\begin{theorem}\label{globalstab.thm}
Suppose that $d=1$ and the hypotheses (V1) to (V4) are satisfied. Then
the standing waves $u_\lam(t,x)=e^{i\lam t}w_\lam(x)$ of \eqref{nls} given by
either Theorem~\ref{ptglobal.thm} or Theorem~\ref{alglobal.thm} are orbitally stable.
\end{theorem}

The proofs of Theorem~\ref{globalstab.thm} given in \cite{ans,eect} used Theorem~2 of
\cite{gssI}, and so relied upon verifying Assumptions~1--3 of \cite{gssI}, as well as the condition
\begin{equation}\label{slope}
\Vert w_\lam\Vert_{L^2} \ \text{is strictly increasing in} \ \lam>0.
\end{equation}
The latter is often referred to as {\em the slope condition}\index{slope condition} or the {\em Vakhitov-Kolokolov
condition}\index{Vakhitov-Kolokolov condition|seeonly{slope condition}}. It seems to have indeed first appeared in the paper \cite{vk} of Vakhitov and 
Kolokolov (1968), in the context of nonlinear optical waveguides.\footnote{The mathematical 
theory of NLS has been intimately connected to nonlinear optics from its early days. 
See \cite{ans} for additional references on this.}

Assumption~1 of \cite{gssI} is about the well-posedness of the Cauchy problem for \eqref{nls}
which, under our hypotheses, follows from Section~\ref{s:dynsysexamples}.
Assumption~2 pertains to the existence of
smooth solution curves and is ensured by Theorem~\ref{ptglobal.thm}/\ref{alglobal.thm}.
It is this property which allows us to apply Theorem~\ref{thm:lyapmethodrestrictedmod}
of Section~\ref{s:orbstabproof}. 

We will see that Assumption~3 of \cite{gssI}, together with the slope condition \eqref{slope},
ensure the required 
coercivity property  of the Lyapunov function $\calL_\lam$ introduced in \eqref{GSSlyap}.
In order to formulate Assumption~3 in the present context, consider 
the bounded linear operator $D^2_{w_\lam}\calL_\lam:E\to E^*$,
\begin{equation}\label{quadraticform}
D^2_{w_\lam}\calL_\lam=D^2_{w_\lam}H+\lam D^2_{w_\lam}Q, \quad \lam >0.
\end{equation}
We define the \emph{spectrum} of $D^2_{w_\lam}\calL_\lam$ as the following subset of $\real$:
\begin{equation}\label{spectrumofH}
\sigma(D^2_{w_\lam}\calL_\lam)=
\bigl\{\lambda\in\real:D^2_{w_\lam}\calL_\lam-\lambda\widetilde{R}:E\to E^* \ 
\text{is not an isomorphism}\bigr\},
\end{equation}
where $\widetilde{R}=\text{diag}(R,R)$
and $R=-\frac{\dif^2}{\dif x^2}+1:H^1(\real,\real)\to H^{-1}(\real,\real)$ 
is the Riesz isomorphism. Under the 
hypotheses of Theorem~\ref{ptglobal.thm}/\ref{alglobal.thm}, 
$\widetilde{R}^{-1}D^2_{w_\lam}\calL_\lam:E\to E$ is a bounded self-adjoint
Schr\"odinger operator, and its spectrum coincides with $\sigma(D^2_{w_\lam}\calL_\lam)$. 
The motivation for this definition of the spectrum of $D^2_{w_\lam}\calL_\lam$
will be discussed in Remark~\ref{essential.rem}.

A straightforward calculation shows that $D^2_{w_\lam}\calL_\lam$ is explicitly given by
\begin{equation}\label{matrixS}
D^2_{w_\lam}\calL_\lam=
\begin{pmatrix}
-\disp\frac{\dif^2}{\dif x^2}+\lam-[f(x,w_\lam^2)+2\partial_2f(x,w_\lam^2)w_\lam^2] & 0 \\
0 & -\disp\frac{\dif^2}{\dif x^2}+\lam-f(x,w_\lam^2)
\end{pmatrix},
\end{equation}
and the spectral conditions formulated in Assumption~3 of \cite{gssI} are:

\medskip
\noindent
(S1) $\exists\alpha_\lam\in\R$ such that $\sigma(D^2_{w_\lam}\calL_\lam)\cap(-\infty,0)=\{-\alpha_\lam^2\}$ and $\ker(D^2_{w_\lam}\calL_\lam+\alpha_\lam^2\widetilde{R})$ is one-dimensional;

\noindent
(S2) $\ker D^2_{w_\lam}\calL_\lam=\mathrm{span}\{iw_\lam\}$;

\noindent
(S3) $\sigma(D^2_{w_\lam}\calL_\lam)\setminus\{-\alpha_\lam^2,0\}$ is bounded away from zero.

\medskip
\noindent 
The fact that $iw_\lam\in \ker D^2_{w_\lam}\calL_\lam$
directly follows by differentiating \eqref{invar} with respect to $\gamma$ at $\gamma=0$.
So (S2) really only states that $\ker D^2_{w_\lam}\calL_\lam$ is one-dimensional.

\medskip
We now explain how hypotheses (S1)--(S3), together
with \eqref{slope}, imply the coercivity property \eqref{eq:hessiancondition} in 
Proposition~\ref{thm:hessianestimate}.
In order to explicitly write down condition \eqref{eq:hessiancondition}, 
let us first observe that we parametrized the standing waves by the
``frequency'' $\lam$ here, whereas in Section~\ref{s:orbstabproof} the relative equilibria are rather
labelled using the value $\mu$ of the
constraint. In the present context, $\mu=\mu(\lam)=Q(w_\lam)$, and we only
deal with situations where $\mu$ is a smooth, strictly increasing function of $\lam$, so 
both parametrizations are equivalent.
Now the level surface 
$$\Sigma_{Q(w_\lam)}=\{u\in\Ban \mid Q(u)=Q(w_\lam)\}$$ 
and, given a standing wave $u_\lam(t)=\Phi_{(\lam t)}w_\lam$
we have, for any $u=e^{-i\gamma(u)}w_\lam\in\Ocal_{u_\lam}$,
$$
T_{u}\Sigma_{Q(w_\lam)}=\{v \in\Ban \mid \la e^{-i\gamma(u)}D_{w_\lam} Q, v\ra=0\}.
$$ 
On the other hand,
$T_u\Ocal_{u_\lam}=\mathrm{span}\{e^{-i\gamma(u)}iw_\lam\}$, so that
$$
T_{u}\Sigma_{Q(w_\lam)}\cap (T_{u}\Ocal_{u_\lam})^\perp=
\{v\in\Ban \mid \la e^{-i\gamma(u)}D_{w_\lam} Q, v\ra=(e^{-i\gamma(u)}iw,v)_{\Ban}=0\}.
$$

Next, differentiating 
\[
D_{w_\lam}H+\lam D_{w_\lam}Q=0
\]
with respect to $\lam$ yields
\begin{equation}\label{SQ}
D^2_{w_\lam}\calL_\lam \chi_\lam=-D_{w_\lam}Q, \quad\text{where}\quad 
\chi_\lam:=\frac{\dif w_\lam}{\dif\lam},
\end{equation}
so that 
\begin{equation}\label{ffi'neg}
\la D^2_{w_\lam}\calL_\lam \chi_\lam, \chi_\lam \ra = - \la D_{w_\lam}Q, \chi_\lam \ra
=-\frac{\dif}{\dif\lam} Q(w_\lam)<0
\end{equation}
by \eqref{slope}. 

\begin{proposition}\label{slope-coerc.thm}
Suppose that (S1) to (S3) hold, as well as \eqref{slope}. Then there exists $c>0$ such that
$$
\forall u\in\Ocal_{u_\lam}, \forall v \in T_{u}\Sigma_{Q(w_\lam)}\cap (T_{u}\Ocal_{u_\lam})^\perp,
\ D^2_u\calL_\lam(v,v)\geq c\Vert v\Vert_\Ban^2.
$$
\end{proposition}

\begin{proof}
Let $u=e^{-i\gamma(u)}w_\lam\in\Ocal_{u_\lam}$. First remark that, by the invariance
of $\calL$ on the orbit $\{\Phi_\gamma w_\lam \mid \gamma\in\real\}$, we have
\begin{equation*}
D^2_{u}\calL_\lam=D^2_{\Phi_{\gamma(u)}w_\lam}\calL_\lam
=D^2_{w_\lam}(\calL_\lam\circ\Phi_{-\gamma(u)})=D^2_{w_\lam}\calL_\lam.
\end{equation*}
Therefore, we need only prove the result at $u=w_\lam$, {\em i.e.} that there exists $c>0$ such that
\begin{equation*}
\forall v\in\Ban, \
\la D_{w_\lam} Q, v\ra=(iw_\lam,v)_{\Ban}=0 
\Rightarrow D^2_{w_\lam}\calL_\lam(v,v)\geq c\Vert v\Vert_\Ban^2.
\end{equation*}
Introducing the bounded self-adjoint operator 
$S_\lam:=\wt{R}^{-1}D^2_{w_\lam}\calL_\lam:\Ban\to \Ban$,
this is equivalent to
\begin{equation*}
\forall v\in\Ban, \
(S_\lam\chi_\lam, v)_\Ban=(iw_\lam,v)_{\Ban}=0 
\Rightarrow (S_\lam v,v)_\Ban\geq c\Vert v\Vert_\Ban^2.
\end{equation*}
Now by \eqref{ffi'neg} we see that $(S_\lam\chi_\lam, \chi_\lam)_\Ban<0$, and
the result readily follows from Lemma~5.3 in \cite{stuart08}.
\end{proof}



\medskip
The verification of properties (S1)--(S3) and of the slope condition \eqref{slope}
in \cite{ans,eect} is intimately connected with the behaviour  as $\lam\to0$ 
of the solutions given by Theorem~\ref{ptglobal.thm}/\ref{alglobal.thm}. The main idea
is to show that the required properties hold true for a limiting problem obtained by
letting $\lam\to0$ in the stationary equation \eqref{stat} (in suitably rescaled variables),
and then to deduce them for the original problem by perturbation and continuation along
the global curve given by Theorem~\ref{ptglobal.thm}/\ref{alglobal.thm}. In other words,
it is first shown that (S1)--(S3) and \eqref{slope} hold for small values of $\lam>0$,
and then that these properties cannot change along the global curve.
It is worth noting here that, in both Theorem~\ref{ptglobal.thm} and Theorem~\ref{alglobal.thm},
it can be shown that $\Vert w_\lam\Vert_{L^\infty}\to0$ as $\lam\to0$ 
(see Section~\ref{local.sec} below). Therefore, case (AL) can be seen as a perturbation of (PT), 
in the limit of small $\lam$, and the stability properties of standing waves are the same in both 
cases for small $\lam>0$.

The remainder of this section is devoted to the proof of Theorem~\ref{globalstab.thm}.
We will sketch the arguments yielding the local stability results
close to $\lam=0$, and the continuation procedure extending these to the whole curves
of solutions in Theorem~\ref{ptglobal.thm} and Theorem~\ref{alglobal.thm}. For the local results,
we shall only consider case (PT), the details of the perturbation argument one has to go through
to deal with (AL) being cumbersome and not very enligthening (see \cite{jde} for more details). 
We will however present the
global continuation procedure for both cases in a unified manner. For this we will use the
general notation of \eqref{nls}--\eqref{stat} rather than the particular form of $f$ in each case, 
and we will merely write $\lam>0$ throughout, of course really meaning $0<\lam<\ly$ in the
(AL) case.


\subsubsection{Local stability by bifurcation}\label{local.sec}

We consider here \eqref{stat} in dimension $d=1$, and with $f(x,s^2)=V(x)|s|^{\sigma-1}$.
The scaling
\begin{equation}\label{chvar}
\lam=k^2, \quad
u(x)=k^\frac{2-b}{\sigma-1} v(y), \quad y:=kx, \quad k>0,
\end{equation}
yields
\begin{equation}\label{aux}
v'' - v + k^{-b}V(y/k)|v|^{\sigma-1}v = 0, \quad k>0.
\end{equation}
Then, by (V2),
$$
\lim_{k\to0}k^{-b}V(y/k)=|y|^{-b}|y/k|^{b}V(y/k)=|y|^{-b} 
\quad \forall y\neq0,
$$
which suggests considering the limit problem
\begin{equation}\label{aux0}
v'' - v + |y|^{-b}|v|^{\sigma-1}v = 0.
\end{equation}
It turns out \cite{ans} that \eqref{aux0} has a unique positive radial solution $v_0\in H^1(\real)$.
This solution can be shown to have a variational characterization, from which it bears the name
{\em ground state} of \eqref{aux0}.

The advantage of the scaling is that, in the new variables $(k,v)$, one can now 
obtain solutions by perturbation of \eqref{aux0}, which is non-degenerate. More precisely,
one can apply a version of the implicit function theorem to the function
$F:\real\times H^1(\real)\to H^{-1}(\real)$ defined by
$$
F(k,v)=
\left\{ \begin{array}{ll}
v'' - v + |k|^{-b}V(y/|k|)|v|^{\sigma-1}v, 	& \  k\neq0, \\
v'' - v + |y|^{-b}|v|^{\sigma-1}v, 		& \ k=0,
\end{array}\right.
$$
at the point $(k,v)=(0,v_0)\in \real\times H^1(\real)$, where
$D_2 F(0,v_0):H^1(\real)\to H^{-1}(\real)$ is an isomorphism (see \cite[Proposition~2.1]{ans}). 
This provides a small $k_0>0$ and a local $C^1$ curve of solutions 
$\{(k,v_k): |k|<k_0\}\subset \real\times H^1(\real)$ of $F(k,v)=0$. 
The local bifurcation in Theorem~\ref{ptglobal.thm} can then be obtained by going back to the
original variables using \eqref{chvar}, which yields a local $C^1$ curve of solutions 
$$\big\{(\lam,w_\lam): 0<\lam<k_0^2\big\}\subset \real\times H^1(\real)$$ 
of \eqref{stat}.
The various solution norms in the two sets of variables are related by
$$
\Vert w_\lam\Vert_{L^{2}}^{2}=
\lam^{\alpha-1}\Vert v_{\lam^{1/2}}\Vert_{L^{2}}^{2}, \quad
\Vert\nabla w_\lam\Vert_{L^{2}}^{2}=
\lam^{\alpha}\Vert\nabla v_{\lam^{1/2}}\Vert_{L^{2}}^{2},
$$
$$
\Vert w_\lam\Vert_{L^\infty}=\lam^{\frac{2-b}{2(\sigma-1)}}\Vert v_{\lam^{1/2}}\Vert_{L^\infty},
\quad \text{where} \  \alpha=\txt\frac{4-2b+(\sigma-1)}{2(\sigma-1)}.
$$
The behaviour of $w_\lam$ as $\lam\to0$ follows readily from these relations and the fact
that $v_k\to v_0$ both in $H^1(\real)$ and in $L^\infty(\real)$ (see \cite[Proposition~3.1]{ans}).


\medskip
\noindent
{\bf The slope condition.}\index{slope condition}
Let us now explain how the slope condition \eqref{slope} can be derived from this analysis, 
for small $\lam>0$. We show that $\frac{\dif}{\dif\lam}\Vert w_\lam\Vert_{L^2}^2>0$
for $\lam>0$ small enough. Observe that
$$
\frac{\dif}{\dif\lam}\Vert w_\lam\Vert_{L^2}^2
=\frac{1}{2k}\frac{\dif}{\dif k}\Vert w_{k^2}\Vert_{L^2}^2
=\frac{1}{2k}\frac{\dif}{\dif k}\{k^\beta\Vert v_k\Vert_{L^2}^2\}
$$
where 
\begin{equation}\label{signofalphaminusone}
\beta=\frac{4-2b-(\sigma-1)}{\sigma-1}=2(\alpha-1).
\end{equation}
Now
\begin{align*}
\frac{\dif}{\dif k}\{k^\beta\Vert v_k\Vert_{L^2}^2\}
    &=\beta k^{\beta-1}\Vert v_{k}\Vert_{L^{2}}^{2}
   +k^{\beta}2\bigl \langle v_{k},\frac{\dif}{\dif k}v_{k}\bigr\rangle_{L^{2}} \\
   & \\
    &=k^{\beta-1}\bigl\{\beta\Vert v_{k}\Vert_{L^{2}}^{2}
    +2k\bigl\langle v_{k},\frac{\dif}{\dif k}v_{k}\bigr\rangle_{L^{2}}\bigr\}.
\end{align*}
Since $\Vert v_{k}\Vert_{L^{2}}^{2}\to\Vert v_0\Vert_{L^{2}}^{2}>0$ as $k\to0$,
we have that
\begin{equation}\label{localsign}
\sgn\{\disp\frac{\dif}{\dif\lam}\Vert w_\lam\Vert_{L^2}^2\}=
\sgn\{\alpha-1\} \ \text{for} \ \lam=k^2 \ \text{small},
\end{equation}
provided
\begin{equation}\label{localpos}
k\bigl\langle v_{k},\frac{\dif}{\dif k}v_{k}\bigr\rangle_{L^{2}} \to 0 \
\text{as} \ k\to0.
\end{equation}
On the other hand,
\begin{multline*}
F(k,v_k)=0 \Rightarrow
D_{k}F(k,v_{k})+D_{v}F(k,v_{k})\frac{\dif}{\dif k}v_{k}=0 \\
\Rightarrow k\frac{\dif}{\dif k}v_{k}
    =-D_{v}F(k,v_{k})^{-1}kD_{k}F(k,v_{k})
    =-D_{v}F(k,v_{k})^{-1}k^{-b}W(y/k)v_{k}^{\sigma},
\end{multline*}
where $W(x):=x\cdot V'(x)+bV(x)$ appears in hypothesis (V2).
Then, using (V2), it is not difficult to show that
$$
k^{-b}W(y/k)v_{k}^{\sigma}\to 0 \ \text{in} \ H^{-1} \ \text{as} \ k\to0.
$$
Finally, it follows from the open mapping theorem that
$$
D_{v}F(k,v_{k})^{-1} \rightarrow D_{v}F(0,v_0)^{-1} \ \text{in} \ B(H^{-1},H^1)
 \ \text{as} \ k\to0,
$$
and we conclude that
$k\frac{\dif}{\dif k}v_{k} \to 0 \ \text{in} \ H^1 \ \text{as} \ k\to0,$
from which \eqref{localpos} follows. Recalling our assumption that $1<\sigma<5-2b$,
the slope condition \eqref{slope} now readily follows from \eqref{signofalphaminusone} 
and \eqref{localsign}.


\medskip
\noindent
{\bf The spectral assumptions.}
Regarding the verification of (S1)--(S3), we shall not give as much detail as for the slope
condition. That the solutions $w_\lam$ indeed give rise to a Hessian
$D^2_{w_\lam}\calL_\lam:E\to E^*$ with the appropriate spectral structure also follows from the properties
of the limit problem \eqref{aux0} through the perturbation procedure outlined above. 
The crucial point is 
the variational characterization of the ground state $v_0$, which can be shown to minimize
the functional
$$
\wt\calL_0(v)=\frac12\intr (v')^2+v^2 \diff x - \frac{1}{\sigma+1}\intr |x|^{-b}|v|^{\sigma+1}\diff x
$$
on an appropriate codimension 1 submanifold $N$ of $H^1(\real)$. 
Note that the direct method of the calculus of variations cannot be applied
to the functional $\wt\calL_0$ since it is not coercive. In fact it turns out that $v_0$ is a saddle-point
of $\wt\calL_0$. More precisely, $v_0$ is a critical point of $\wt\calL_0$ ({\em i.e.}~$D_{v_0}\wt\calL_0=0$), 
and the quadratic form
$D^2_{v_0}\wt\calL_0:H^1\times H^1\to\real$ is positive definite tangentially to $N$, 
and negative along the ray spanned by $v_0$, transverse to $N$.
This information -- together with some Schr\"odinger operator theory --
precisely implies that $D^2_{v_0}\wt\calL_0$ enjoys the properties (S1)--(S3).
Furthermore, if $w_\lam$ and $v_k$ are related by the change of variables \eqref{chvar},
a straightforward calculation shows that
$$
\calL_\lam(w_\lam)=k^{\frac{3-2b+\sigma}{(\sigma-1)}}\wt\calL_0(v_k),
$$
where $\calL_\lam$ is the Lyapunov function defined in \eqref{GSSlyap}.
However,
it is by no means trivial to verify that the spectral properties of $D^2_{v_0}\wt\calL_0$ 
are carried through to $D^2_{w_\lam}\calL_\lam$, for $\lam>0$ small, 
in the perturbation procedure. This was shown in \cite{dcds} in arbitrary dimension.

Note that, if the solutions $w_\lam$ are themselves saddle-points of $\calL_\lam$, the
perturbation procedure can be dispensed of, and the spectral properties of the Hessian
$D^2_{w_\lam}\calL_\lam$
derived directly from this variational characterization. This is in fact the case for the solutions
obtained in Theorem~\ref{ptglobal.thm}, but it is not known in the (AL) case, where the variational
structure is much less transparent.

\begin{remark}\label{essential.rem}
\rm
When verifying assumptions (S1)--(S3)
in the context of \eqref{nls}--\eqref{stat} (which are set on the whole of $\rn$) one has to deal
with the continuous spectrum of $D^2_{w_\lam}\calL_\lam$ in addition to the negative eigenvalue lying at
the bottom of the spectrum. The standard approach to tackle this is {\it via} the
theory of Schr\"odinger operators applied 
to the self-adjoint operator $\widetilde{R}^{-1}D^2_{w_\lam}\calL_\lam:E\to E$.
This motivates the definition of $\sigma(D^2_{w_\lam}\calL_\lam)$ given in \eqref{spectrumofH}. On the other hand,
the problem considered in Section~\ref{s:nlsetorus1d} (set on a compact manifold) only gives rise
to discrete spectrum in the linearization, and so can be handled with 
a more elementary spectral analysis, not requiring
to introduce the Riesz isomorphism $\widetilde{R}:E\to E^*$ explicitly.
\end{remark}

\subsubsection{Global continuation}

In this section we show how both the slope condition \eqref{slope} and the spectral
properties (S1)--(S3) extend from the previous local analysis to the global curve given by either
Theorem~\ref{ptglobal.thm} or Theorem~\ref{alglobal.thm}. We will handle the two
cases in a unified approach, using the general notation $f(x,w^2)w$ for the nonlinearity. 
As earlier, we will often merely write $\lam>0$, really meaning $\lam\in(0,\infty)$ in the
(PT) case and $\lam\in(0,\ly)$ in the (AL) case. Again, we only consider here the case $d=1$. 


\medskip
\noindent
{\bf The slope condition.}\index{slope condition}
From the previous analysis, \eqref{slope} holds for $\lam>0$ small enough. 
Hence we need only verify that
$$
\frac{\dif}{\dif\lam}\int_\real w_\lam^2\diff x \neq0
\quad \forall\lam>0.
$$
First notice that, since the solutions $w_\lam$ are even,
$$
\frac{\dif}{\dif\lam}\int_\real w_\lam^2\diff x
=2\int_\real w_\lam\frac{\dif}{\dif\lam}w_\lam\diff x
=4\int_0^\infty w_\lam \chi_\lam,
$$
where $\chi_\lam=\frac{\dif w_\lam}{\dif\lam}$ satisfies
$$
\chi_\lam''+ \{f(x,w_\lam^2)+2\partial_2f(x,w_\lam^2)w_\lam^2\}\,\chi_\lam 
= \lam \chi_\lam + w_\lam.
$$
To simplify the notation,
we will drop the subscript $\lam$ in the remainder of the argument.
It can be shown \cite{ans,eect} that
\begin{equation}\label{intid}
\int_0^\infty \big\{2f(x,w^2)+x\partial_1f(x,w^2)-\partial_2f(x,w^2)w^2\big\}w\chi\diff x
=2\lam\int_0^\infty w\chi\diff x
\end{equation}
and that there exists $x_0>0$ such that 
$$
\chi>0 \  \text{on} \  (0,x_0), \quad \chi(x_0)=0, \quad \chi<0 \ \text{for} \ x>x_0.
$$
Supposing by contradiction that $\int_0^\infty w\,\chi\diff x=0$,
we can write \eqref{intid} as
$$
\int_0^\infty \Big\{\frac{2f(x,w^2)+x\partial_1 f(x,w^2)}{\partial_2f(x,w^2)w^2} - 1\Big\}
\partial_2f(x,w^2)w^3\chi \diff x=0.
$$
Denoting by $\zeta(x)$ the function in the curly brackets, this
becomes
$$
\int_0^\infty \zeta(x) \partial_2f(x,w^2)w^3\chi \diff x=0.
$$
Now using the unique zero $x_0$ of $\chi$, we can rewrite this identity as
$$
\int_0^\infty\{\zeta(x)-\zeta(x_0)\}\partial_2f(x,w^2)w^3\chi \diff x
+\zeta(x_0)\int_0^\infty\partial_2f(x,w^2)w^3\chi \diff x=0.
$$
Moreover, multiplying the equation for $w$ by $\chi$, the equation for $\chi$ by $w$,
subtracting and integrating, yields
$$
\int_0^\infty w^2 \diff x=2\int_0^\infty \partial_2f(x,w^2)w^3\chi \diff x,
$$
and so
\begin{equation}\label{intid2}
\int_0^\infty \partial_2f(x,w^2)w^3 \{\zeta(x)-\zeta(x_0)\}\chi \diff x
+\frac{\zeta(x_0)}{2}\int_0^\infty w^2 \diff x=0.
\end{equation}
Now,
$$
\partial_2f(x,w^2)w^3=
\begin{cases}
\frac{\sigma-1}{2}V(x)w^\sigma & \text{in the (PT) case},\\
\frac{\sigma-1}{2}V(x)\frac{w^\sigma}{(1+w^{\sigma-1})^2} & \text{in the (AL) case},
\end{cases}
$$
hence $\partial_2f(x,w^2)w^3>0$ on $(0,\infty)$ in any case.
On the other hand,
$$
\zeta(x)=
\begin{cases}
\frac{2}{\sigma-1}[x\frac{V'(x)}{V(x)}+\frac{5-\sigma}{2}] & \text{(PT)}\\
\frac{2}{\sigma-1}[x\frac{V'(x)}{V(x)}+\frac{5-\sigma}{2}]
+ \frac{2}{\sigma-1}[x\frac{V'(x)}{V(x)}+2]w^{\sigma-1} & \text{(AL)}
\end{cases}
$$
and we claim that $\zeta$ is positive and decreasing in any case,
which immediately leads to a contradiction with \eqref{intid2}.
To conclude, the claim follows from our hypotheses since
$$ 
x \to x\frac{V'(x)}{V(x)} \ \text{decreasing}, 
\quad x\frac{V'(x)}{V(x)}\geq -b \quad \text{and} \quad \sigma<5-2b
$$
$$
\Rightarrow x\frac{V'(x)}{V(x)}+\frac{5-\sigma}{2}>0 \ \, \text{and} \ \text{decreasing}
$$
(note that hypothesis (V4) is crucial here).
Furthermore,
$$
w>0 \ \, \text{and} \ \text{decreasing} \ \Rightarrow
\Big[\underbrace{x\frac{V'(x)}{V(x)}+2}_{\geq -b+2>0}\Big]w^{\sigma-1}>0 
\ \, \text{and} \ \text{decreasing},
$$
so that $\zeta$ is indeed positive and decreasing in any case.


\medskip
\noindent
{\bf The spectral conditions.}
The spectral conditions (S1)--(S3) can be reformulated in terms of the self-adjoint operators
$L_\lam^+, L_\lam^-: H^2(\real)\subset L^2(\real) \to L^2(\real)$ defined by
\begin{align*}
L_\lam^+ v  &= -v''+ \lam v-[f(x,w_\lam^2)+2\partial_2f(x,w_\lam^2)w_\lam^2]v, \\
L_\lam^- v  &= -v''+ \lam v-f(x,w_\lam^2)v.
\end{align*}
Then (S1)--(S3) are equivalent to
\begin{itemize}
\item[(C1)]\quad
$\inf\sigma_\textnormal{ess}(L_\lam^+)>0$, \quad $M(L_\lam^+)=1$, \quad
$\ker L_\lam^+ =\{0\}$,
\item[(C2)]\quad 
$\inf\sigma_\textnormal{ess}(L_\lam^-)>0$, \quad $0=\inf\sigma(L_\lam^-)$, \quad 
$\ker L_\lam^-=\textnormal{vect}\{w_\lam\}$,
\end{itemize}
where $\sigma_\textnormal{ess}(A)$ denotes the {\em essential spectrum}
of a self-adjoint operator $A$, and
$M(A)$ its {\em Morse index}, {\em i.e.} the dimension of the larger
subspace where $A$ is negative definite.

A first step toward verifying that (C1) and (C2) hold for all $\lam>0$ is to show that
all eigenvalues of $L_\lam^+, L_\lam^-$ are simple, which follows by standard
ODE arguments. Then, since 
$$
\lim_{|x|\to\infty}f(x,w_\lam(x)^2)=
\lim_{|x|\to\infty}2\partial_2f(x,w_\lam(x)^2)w_\lam(x)^2=0,
$$
it follows from the spectral theory of Schr\"odinger operators (see {\em e.g.} \cite{trieste}) that
$$
\inf\sigma_\textnormal{ess}(L_\lam^+)=\inf\sigma_\textnormal{ess}(L_\lam^-)=\lam>0.
$$
Furthermore, applying ODE comparison arguments to the equations $L_\lam^+v=0$  
and \eqref{stat}, it can be seen that $\ker L_\lam^+ =\{0\}$.
On the other hand, since $w_\lam>0$ is a solution of \eqref{stat}, it follows again from
standard spectral theory that
$$\ker L_\lam^-=\textnormal{span}\{w_\lam\} \ \text{and} \
0=\inf\sigma(L_\lam^-).$$

It remains to show that $L_\lam^+$ has exactly one negative eigenvalue.
As discussed earlier, the local bifurcation analysis close to $\lam=0$ 
shows that $M(L_\lam^+)=1$ for $\lam>0$ small enough.
By perturbation theory, the eigenvalues of $L_\lam^+$ depend
continuously on $\lam>0$.
Since $\ker L_\lam^+ =\{0\}$ for all $\lam>0$, 
the eigenvalues cannot cross zero as $\lam$ varies. Therefore,
$M(L_\lam^+)=1$ for all $\lam>0$, which completes the proof of conditions
(C1) and (C2).

\section{A brief history of orbital stability}\label{history.sec}

The stability theory of infinite dimensional nonlinear evolution equations has been the object of intense study 
in the past four decades. It originated in the mathematical analysis of nonlinear waves propagating in 
dispersive media, such as waves on a water surface, or electromagnetic waves in dielectric media. Giving
an exhaustive review of the subject would take us far outside the scope of these notes. We shall only aim
to guide the reader through a choice of references which appear important to us, providing possible 
directions for further investigation of the literature on orbital stability.

Let us first remark that the notion of {\em orbital stability} defined in \eqref{eq:orbstab0} is a classical one in the study
of periodic solutions of finite dimensional dynamical systems, which originated in the pioneering works of Floquet \cite{floquet}, 
Poincar\'e \cite{poincare} and Lyapunov \cite{lyap}. 
The rigorous mathematical analysis of orbital stability for {\em nonlinear dispersive PDE's} 
has been initiated in 1972 by Benjamin \cite{ben72}, 
who considered solitary waves of the Korteweg--de Vries (KdV) equation. 
This equation was first written down by Boussinesq in 1877 \cite{bouss} and then rediscovered independently
by Korteweg and de Vries in 1895 \cite{KdV}, as a model for water wave motions. It describes long waves 
in shallow water ({\em i.e.} with water depth small compared to wavelength) propagating in one space 
direction.\footnote{The KdV equation also appears in other physical contexts \cite{zk}.}
The terminology of ``orbital stability'' is not employed by Benjamin, who rather speaks of the stability of the shape of the
solitary waves: ``A device entailing the definition of a certain quotient space is used to discriminate the stability of
solitary waves in respect of shape -- which is a more reasonable property to investigate than absolute stability.'' 
(\cite[p.~155]{ben72}). The quotient referred to by Benjamin is with respect to space translations in $\R$, which
is a group of symmetry for the KdV equation. Benjamin's proof of stability makes use of a Lyapunov functional
constructed by means of the constants of motion, i.e. the energy-momentum method studied in these notes. 
It is worth observing here that,
before proving stability for arbitrary perturbations of the initial data, he starts by proving stability for perturbations having
same $L^2$ norm as the solitary wave, and then uses the fact that solitary waves come as continuous families
parametrized by the wave speed. This idea was later used by Weinstein \cite{weinstein86} for general
NLS equations and a generalized KdV equation. We use it to prove our Theorem~\ref{thm:lyapmethodrestrictedmod}.
Benjamin motivates his approach heuristically by discussing some early remarks of Boussinesq \cite{bouss} 
suggesting the use of a Lyapunov function to prove stability.

An abundant literature on the stability theory of solitary waves for 
equations modelling water waves has followed Benjamin's paper. Just to mention a few, the interested reader
may consult the following papers and references therein: 
\cite{bona,weinstein86,bss,constrauss_2000,constmol,egw,geyer} for waves in shallow water, including the KdV and 
Camassa-Holm equations; \cite{buffoni,constrauss_2007,bgsw} for the full water wave problem, 
governed by the Euler equation. 

A couple of years after Benjamin's seminal work, Bona \cite{bona} made a substantial contribution to the theory, 
by grounding it into the Sobolev space setting. Indeed,
in the absence of a general well-posedness theory, Benjamin had assumed that solutions were global in time and smooth.
Bona proved global well-posedness in appropriate Sobolev spaces and rephrased 
Benjamin's arguments in this natural framework. This 
was an important step for subsequent work on stability for nonlinear dispersive equations. 

Two remarkable contributions to the stability theory of KdV-like equations were given about a decade later 
by Weinstein \cite{weinstein86} and by Bona, Souganidis and Strauss \cite{bss}, who applied the energy-momentum method
to generalized versions of the KdV equation. Weinstein \cite{weinstein86} also proves the orbital stability of standing
wave solutions to a general class of nonlinear Schr\"odinger equation. His proof, based on the energy-momentum method,
provides the first alternative, in the NLS context, to the proof of orbital stability given a few years earlier
by Cazenave and Lions \cite{cazlions} for the NLS with a power-law nonlinearity (see also \cite{caz}), 
which is purely variational, based on Lions' concentration-compactness principle \cite{lions}. 

In the same spirit, taking advantage of general existence results for nonlinear waves that were obtained in the early
1980's (see {\em e.g.} \cite{strauss,berlions}), an important body of work including
\cite{shatah,shrauss,jones,joloney,grill88,grill90} made use of linear stability analysis 
and the energy-momentum method to study stability properties of standing/solitary waves for
Hamiltonian systems including the NLS and nonlinear Klein-Gordon equations. 
This line of research culminated in the general theory of orbital stability of
Grillakis, Shatah and Strauss \cite{gssI,gssII}, who derived sufficient and necessary conditions for the stability of
standing/solitary waves of infinite-dimensional Hamiltonian systems with symmetry, 
{\em via} a combination of spectral properties and a general convexity condition. In the NLS
context, this convexity condition takes the form of the condition \eqref{slope} of Section~\ref{curves.sec}.
This stability condition seems to have first appeared in 1968 in a paper of Vakhitov and 
Kolokolov \cite{vk}, where stability of trapped modes in a cylindrical nonlinear optical waveguide is discussed
by formal arguments. In fact, the NLS equation is a standard model for slowly modulated waves in nonlinear media, 
for instance in nonlinear optics, see \cite{sulsul,maimistov}.

Following the seminal contributions of the 1980's, the amount of work on stability for the NLS and other nonlinear
dispersive equations has increased tremendously. Important results have been obtained for instance in 
\cite{ohta,cp,fukohta,fibichwang,hajstu,debfuk,fuk,jjlecoz,galhar07a,galhar07b,dcds,maeda08,lecozfuk,lecoz,jjcolin,lmr,maeda12,eect,an13,an14}, and many other references can be found in these papers.

In addition to orbital stability, the stronger property of asymptotic (orbital) stability\footnote{This notion is well known
in the finite dimension context, see {\em e.g.} \cite{coddlev}.} has also been investigated, see {\em e.g.}
\cite{pw,mm01,mmt02,mm05,mm08} for KdV and \cite{swI,swII,cucc_survey,mmt06,kz,cucc,cuccpel} for NLS.
Roughly speaking, a relative equilibria $U$ is (orbitally) {\em asymptotically stable} if 
it is orbitally stable and any solution starting close to its orbit eventually resolves into a ``modulation'' of the
original wave $U$ and a purely dispersive part, solution of the linear version of the governing equation. An
important related conjecture, known as the {\em soliton resolution conjecture} stipulates that, generically,
any reasonable initial data should give rise to a solution which eventually resolves into a sum of solitary waves (solitons)
and a purely dispersive part (radiation). More details and references on these topics can be found in \cite{soffer,tao}.
Let us just conclude by remarking that the term ``soliton'' (which was coined in \cite{zk}) comes 
from the literature on integrable systems, originating in \cite{fpu,zk,ggkm,lax,zs,manakov74}. 
Loosely speaking, solitons\index{solitons} are (stable) solitary waves of integrable systems, 
that can be obtained by exact solution methods,\footnote{These methods are somewhat reminiscent of the Fourier
transform approach to solve linear PDE's, though the formulas are much more involved for nonlinear waves.}
such as the {\em inverse scattering transform} \cite{lax}. 
However, the term soliton is now used in a more flexible manner
throughout the nonlinear dispersive PDE's community, whenever referring to a persistent localized wave resulting
from a balance of dispersion and nonlinear effects. 
The inverse scattering transform provides detailed information about the asymptotic behaviour
({\em e.g.} soliton resolution) of general solutions {\em in the integrable cases} 
-- see \cite{tao,klein_saut} and references therein for recent accounts 
comparing the inverse scattering to other PDE methods.

Further discussion and more references about nonlinear dispersive PDE's 
can be found in the monographs \cite{ac,cazenave2003,tao_book,pava}.


\newpage

\vglue 0.5cm
\appendix

\setcounter{section}{0}
\renewcommand{\thesection}{A.\arabic{section}}
 
\centerline{\bf\Large Appendix}
\medskip
The goal of this appendix is to present those very  basic notions from differential geometry, Lie group theory and Hamiltonian mechanics that are indispensable to follow the treatment of the main text and that are not necessarily familiar to all. The only prerequisites for this part are a good grasp of differential calculus on finite dimensional normed vector spaces not going much beyond  a fluent mastery of the chain rule for differentiation and an intuitive grasp of what  a submanifold of such spaces is. 
\section{Differential geometry: the basics}\label{s:diffgeom}
We first recall some elementary notions of differential geometry and dynamical systems on a  normed vector space $E$. For the general theory on differentiable manifolds, one may for example consult~\cite{am, ma, spi}.

By a vector field on $\Ban$ we will mean a smooth map $X:\Ban\to \Ban$. Given $u\in \Ban$, one should think of $X(u)$ as a ``tangent vector to $\Ban$ at $u$''.  With this idea in mind, a vector field naturally determines a differential equation
$$
\dot u(t)=X(u(t)), \quad u_0=u,
$$
the solutions of which induce a flow on $\Ban$ defined as $\Phi^X_t(u)=u(t)$. For ease of discussion, we will suppose throughout the appendix that all solutions are global and hence all flows complete.  Most results carry over even if the flow exists only locally in time. 

The diffeomorphisms\footnote{We mean $\Phi\in C^1(E,E)$ with a $C^1(E,E)$ inverse.}  $\Phi$ of $\Ban$ act naturally on vector fields as follows. First note that, when $\Phi$ is a diffeomorphism, and $\gamma:t\in (a,b)\to E$ a curve with $\gamma(0)=u, \dot\gamma(0)=v$, then we can consider the curve $\tilde\gamma:t\in (a, b) \to E$ defined by $\tilde \gamma(t)=\Phi(\gamma(t))$. This is the curve $\gamma$, ``pushed forward'' by $\Phi$: we invite the reader to draw a picture. This new curve satisfies $\tilde\gamma(0)=\Phi(u)$, so it passes through $\Phi(u)$. What is its tangent vector at that point? The chain rule yields immediately
$$
\dot{\tilde\gamma}(0)=D_u\Phi(v),
$$
where $D_y\Phi$ is our notation for the Fr\'echet derivative of $\Phi$ at $y\in \Ban$, which is a continuous linear map from $\Ban$ to $\Ban$. This equality gives a geometric interpretation to the purely analytical object $D_u\Phi(v)$: it is the tangent vector at $\Phi(u)$ to the curve $\tilde\gamma$ at $t=0$. With this in mind, given a vector field $X$, we can now define a new vector field $\Phi_*X$, the \emph{push forward} of the vector field $X$ by the diffeomorphism $\Phi$, as follows:
$$
\Phi_*X(\Phi(u)):=D_u\Phi(X(u)).
$$
Note that, with the above interpretation of the ``push forward'' of a vector at $u$, $D_u\Phi(X(u))$ is a vector ``at $\Phi(u)$'', which explains why $\Phi(u)$ appears in the argument in the left hand side. Of course, we can write
\begin{equation}\label{eq:pushforward}
\Phi_*X(u)=D_{\Phi^{-1}(u)}\Phi(X(\Phi^{-1}(u))).
\end{equation}
We will make little use of this notation from differential geometry, preferring to write out the explicit expression $D_u\Phi(X(u))$ whenever needed. 

Diffeomorphisms also act naturally on flows, as follows. Given a diffeomorphism $\Phi:\Ban\to \Ban$, one has, for all $u\in \Ban$,
$$
\frac{\rd}{\rd t} (\Phi\circ\Phi_t^X)(u)= D_{\Phi_t(u)}\Phi(X(\Phi_t(u))).
$$
From this and~\eqref{eq:pushforward}, one concludes
\begin{eqnarray*}
\frac{\rd}{\rd t} (\Phi\circ\Phi_t^X\circ \Phi^{-1})(u)&=& D_{\Phi_t^X(\Phi^{-1}(u))}\Phi(X(\Phi_t^X(\Phi^{-1}(u))))\\
&=&\Phi_*X(\Phi\circ\Phi_t^X\circ\Phi^{-1}(u)).
\end{eqnarray*}
In other words, the flow $\Phi\circ\Phi_t^X\circ\Phi^{-1}$ is generated by the pushed forward vector field $\Phi_*X$. 

It follows from the above and an application of the chain rule that, if $X,Y$ are two vector fields on $\Ban$, then, for all $u\in \Ban$,
\begin{eqnarray}\label{eq:commutator}
\frac{\partial^2}{\partial s\partial t} \Phi_s^Y\circ\Phi_t^X\circ \Phi_{-s}^Y(u)_{\mid s=0=t}&=&\frac{\rd}{\rd s}D_{\Phi_{-s}^Y(u)}\Phi_s^Y(X(\Phi_{-s}^Y(u)))_{s=0}\nonumber\\
&=&\frac{\rd}{\rd s}X(\Phi_{-s}^Y(u))_{s=0} + \frac{\rd}{\rd s} D_x\Phi_s^Y(X(u))_{s=0}\nonumber\\
&=&[X,Y](u),
\end{eqnarray}
where the commutator $[X,Y]$ of two vector fields is defined as follows:
$$
[X,Y](u)=D_u Y(X(u))-D_u X(Y(u)).
$$
This definition is justified by the following observation. Given a vector field $X$ and a $C^1$ function $F:\Ban\to \R$, one can define a differential operator
\begin{equation}\label{eq:derivationX}
\widehat X(F)(u)=D_uF(X(u)),
\end{equation}
which is -- geometrically -- nothing but the directional derivative of $F$ at $u$ in the direction $X(u)$. A simple computation shows readily that
\begin{equation}\label{eq:commutator3}
[\widehat X, \widehat Y]=\widehat{[X,Y]}.
\end{equation}
The following is then  well known:
\begin{lemma}\label{lem:commutator} The following are equivalent:
\begin{enumerate}[label=({\roman*})]
\item For all $s,t\in\R$, $\Phi_t^X\circ\Phi_s^Y=\Phi_s^Y\circ\Phi_t^X$;
\item $ [X,Y]=0$.
\end{enumerate}
\end{lemma}
\begin{proof} That (i) implies (ii) follows immediately from the preceding computation. The proof of the converse is slightly more involved, for a simple argument we refer to~\cite{spi}. 
\end{proof}
 
\begin{remark}
Note that, if $X(u)=Au, Y(u)=Bu$, where $A,B:\Ban\to \Ban$ are linear, then, with our convention, 
$
[X,Y](u)=-[A,B]u.
$
Here $[A,B]=AB-BA$ is the standard commutator of linear maps.
\end{remark}

\begin{definition}\label{def:levelsurfregular}
Let $F\in C^k(\Ban,\R^m)$ for some $k\ge 1$. For each $\mu\in\R^m$ we define a level set\index{level set} of $F$ by
\begin{equation}\label{eq:sigmamu}
\Sigma_\mu=\{u\in \Ban\mid F(u)=\mu\}.
\end{equation}
We will say $u\in \Ban$ is a regular point\index{regular point} of $F$ if $D_uF:\Ban\to\R^m$ is surjective. We will say $\mu$ is a regular value\index{regular value} of $F$, if $\Sigma_\mu\not=\emptyset$ and all $u\in\Sigma_\mu$ are regular points of $F$. 
\end{definition}
If $\mu$ is a regular value of $F$, then $\Sigma_\mu$ is a co-dimension $m$ submanifold of $E$ \cite[Theorem 6.3.34]{baoebepre99}. In that case, the tangent space to $\Sigma_\mu$ at $u$ is defined as follows:
\begin{equation}
	\label{eq:deftangentspace}
	T_u\Sigma_\mu=\left\{w\in E \mid D_u F(w)=0\right\}=\mathrm{Ker}(D_uF).
\end{equation}
We point out that if $r=\mathrm{Rank}(D_uF)$ is constant on $\Sigma_\mu$, then $\Sigma_\mu$ is a co-dimension $r$ submanifold. We will need the following simple result in Section~\ref{ss:coercivityhessian}.

\begin{lemma}\label{eq:basicestimate}
Let $F\in C^k(\Ban,\R^m)$ for some $k\geq 2$. Let $\mu\in\R^m$ be a regular value of $F$. Let $u\in\Sigma_\mu$ and let $W_u$ be a subspace of $\Ban$ so that $\Ban=T_u\Sigma_\mu\oplus W_u$. Then, for all $v\in\Sigma_\mu$,
$$
\|(v-u)_2\|\le \mathrm{O}(\|v-u\|^2),
$$
and there exist $\delta, C>0$ such that
$$
\|v-u\|\le \delta \Rightarrow \|(v-u)_1\|\geq C\|v-u\|,
$$
where $(v-u)=(v-u)_1+(v-u)_2\in T_u\Sigma_\mu\oplus W_u$. 
\end{lemma}
Note that both $\delta$ and $C$ depend on $u$ and on the decomposition of $\Ban$ chosen. 

\begin{proof}
Write $u-v=w_1+w_2$, with $w_1\in T_u\Sigma_\mu$ and $w_2\in W_u$. Then, using that $D_uF(w_1)=0$,
we have 
$$
0=F(v)-F(u)=D_uF(w_2)+\mathrm{O}(\|v-u\|^2).
$$
Now, since $D_uF$ is a diffeomorphism from $W_u$ to $\R^m$, there exists $c>0$ so that
$$
\|D_uF(w_2)\|\geq c \|w_2||, \quad\mathrm{hence}\quad \mathrm{O}(\|v-u\|^2)\geq c \|w_2\|.
$$
Finally
$$
\|w_1\|=\|u-v-w_2\|\geq \|u-v\|-\|w_2\|\geq \|u-v\|-\mathrm{O}(\|v-u\|^2),
$$
from which the result follows.
\end{proof}

\section{Lie algebras, Lie groups and their actions}\label{s:liegroups}
In general, a Lie algebra\index{Lie!algebra} is a vector space $V$ equipped with a bilinear composition law $(u,v)\in V\times V\to [u,v]\in V$, called a Lie bracket\index{Lie!bracket},  which is anti-symmetric and satisfies the Jacobi identity, meaning that for all $u,v,w\in V$:
\begin{equation}\label{eq:jacobi}
[[u,v], w]+[[v,w], u]+[[ w, u],v]=0.
\end{equation}
The basic example of this structure is given by spaces of matrices or, more generally, of linear operators on vector spaces, where the Lie bracket is given by the usual commutator. Two other examples play an important role in these notes, namely the space of vector fields on a normed vector space with the commutator defined in~\eqref{eq:commutator} and the space of all smooth functions on a symplectic vector space, where the Lie bracket is given by the Poisson bracket, as explained in Section~\ref{s:hammech} below. The validity of the Jacobi identity follows in all these examples from a direct computation, whereas the bilinearity and the anti-symmetry are obvious. Lie algebras are intimately linked to Lie groups, as the terminology strongly suggests, and as we now further explain.

In general, a Lie group\index{Lie!group} is a group equipped with a compatible manifold structure. For our purposes, it is however enough to define a Lie group $G$ to be a subgroup of GL$(\R^N)$, such that $G$ is also a submanifold of $\R^{N^2}$ (\emph{i.e.} for our purposes, typically the level surface of a vector-valued function). As such, GL$(\R^N)$ itself, which is an open subset of $\R^{N^2}$, is a Lie group. So are the rotation group
$$
\mathrm{SO}(N)=\{R\in \mathrm{GL}(N,\R)\mid R^TR=\mathrm{I}_N\}
$$
and the symplectic group
\begin{equation}\label{eq:spgroup}
\mathrm{Sp}(2N)=\{S\in \mathrm{GL}({2N},\R)\mid S^TJS=J\},\quad\mathrm{with} \quad J=
\begin{pmatrix}
0&\mathrm{I}_N\\-\mathrm{I}_N&0
\end{pmatrix}.
\end{equation}
A simple verification shows that Sp$(2)=\mathrm{SL}(2,\R)$, the space of two by two matrices of determinant one. The dimension of a Lie group is by definition its dimension as a manifold. For SO$(N)$, it is $N(N-1)/2$, and for Sp$(2N)$, it is $N(2N+1)$, as is readily checked. The group $\R^n$ is also a Lie group in this sense. Indeed, putting $N=n+1$, and defining, for each $a\in\R^n$,
$$
A(a)=
\begin{pmatrix}
\mathrm{I}_n &a\\
0&1
\end{pmatrix}
$$
one readily sees that $A(a)A(b)=A(a+b)$, so that one can view $\R^n$ as a subgroup of GL$(n+1,\R)$. 

We recall that, in general, an action of a group $G$ on a set $\Sigma$ is a map $\Phi:(g,x)\in G\times \Sigma\to \Phi_g(x)\in \Sigma$ which satisfies $\Phi_\e(x)=x$, for all $x\in \Sigma$, and $\Phi_{g_1}\circ\Phi_{g_2}=\Phi_{g_1g_2}$.  In these notes, we consider actions that are defined on a normed vector space $\Ban$. If the $\Phi_g$  are linear,  one says $\Phi$ is a representation of the group. This will \emph{not} always be the case in these notes: actions may be  nonlinear. Furthermore, all actions considered will be at least continuous, and very often they will have additional smoothness properties.  In this appendix, where we deal with finite dimensional systems only, the actions are supposed to be separately $C^1$ in each of their two variables $g\in G$ and $u\in \Ban$. Appropriate technical conditions to deal with infinite dimensional spaces $\Ban$ are given in the main part of the text as needed. 

By definition, \emph{the} Lie algebra $\frak g$ of a Lie group G is the tangent space to the manifold $G$ at the unit element $\e\in G$:
$$
\frak g=T_eG.
$$
In other words, for each $\xi \in \frak g$, there exists $\gamma:t\in \R\to G$, a smooth curve with $\gamma(0)=\e=\mathrm{I}_N$,  and $\dot\gamma(0)=\xi$. Note that one should think of $\xi$ as a matrix, since for each $t$, $\gamma(t)$ is one. In addition, it turns out that, given $\xi\in \frak g$, 
$$
\exp(t\xi)\in G,
$$
for all $t\in\R$ where $\exp(t\xi)$ is to be understood as the exponential of the matrix $t\xi$. Indeed, given $\xi$ and $\gamma$ as above, for all $n\in\N$, $\gamma(\frac{t}n)\in G$ and so $\gamma(\frac{t}n)^n\in G$. Taking $n\to+\infty$, the result follows. A one-parameter subgroup of $G$ is, by definition, a smooth curve $\gamma : t\in \R\to \gamma(t)\in G$, which is also a group diffeomorphism: $\gamma(t+s)=\gamma(t)\gamma(s)$. What precedes shows that any such one-parameter group is of the form $t\to \exp(t\xi)$. So there is a one-to-one correspondence between the one-parameter subgroups of $G$ and its Lie-algebra, which starts to explain the importance of this latter notion. In addition, it turns out that, if $\xi, \eta\in T_\e G$, then so is their commutator (seen as matrices)
$$
[\xi, \eta]=\xi\eta-\eta\xi,
$$
which justifies calling $T_\e G$ a Lie algebra.
Indeed, consider, for each $s\in\R$,  the curve
$$
\gamma: t\in\R\to \exp(s\eta)\exp(t\xi)\exp(-s\eta)\in G.
$$
Clearly $\gamma(0)=\mathrm{I}_N$ and $\dot\gamma(0)=\exp(s\eta)\xi\exp(-s\eta)\in T_\e G$. So we have a curve
$$
s\in\R \to \exp(s\eta)\xi\exp(-s\eta)\in T_\e G.
$$
Taking the derivative with respect to $s$ yields $[\eta, \xi]\in T_\e G$:
\begin{equation}\label{eq:adjointaction}
\frac{\rd}{\rd s} \exp(s\eta)\xi\exp(-s\eta)_{\mid s=0} = [\eta, \xi].
\end{equation} 
As an example, the Lie algebra of SO$(N)$, denoted by so$(N)$, is given by
$$
\mathrm{so}(N)=\{A\in \mathcal{M}(N,\R) \mid A^T+A=0\},
$$
which is the space of all anti-symmetric $N\times N$ matrices. This is easily established by writing $\exp(tA^T)\exp(tA)=\mathrm{I}_N$ and taking a $t$-derivative at $t=0$. And it is obvious that the commutator of two anti-symmetric matrices is anti-symmetric. A basis for so$(3)$ is 
\begin{equation}\label{eq:so3basis}
e_1=
\begin{pmatrix}
0&0&0\\
0&0&-1\\
0&1&0
\end{pmatrix},
\ e_2=
\begin{pmatrix}
0&0&1\\
0&0&0\\
-1&0&0
\end{pmatrix},
\ e_3=
\begin{pmatrix}
0&-1&0\\
1&0&0\\
0&0&0
\end{pmatrix},
\end{equation}
and one readily checks that
\begin{equation}\label{eq:so3comrel}
[e_1, e_2]=e_3,\quad [e_2, e_3]= e_1, \quad [e_3, e_1]=e_2.
\end{equation}
One then identifies $\xi\in \ $so$(3)$ with $\xi\in\R^3$ via 
\begin{equation}\label{xirepr}
\xi=\sum_{i=1}^3 \xi_i e_i 
=
\begin{pmatrix}
0&-\xi_3&\xi_2\\
\xi_3&0&-\xi_1\\
-\xi_2&\xi_1&0
\end{pmatrix}.
\end{equation} 
Similarly, a basis for sl$(2,\R)$, the Lie algebra of SL$(2,\R)$, is 
\begin{equation}
e_0=
\begin{pmatrix}
1&0\\
0&-1
\end{pmatrix},
\ e_+=
\begin{pmatrix}
0&1\\
0&0
\end{pmatrix},
\ e_-=
\begin{pmatrix}
0&0\\
1&0
\end{pmatrix},
\end{equation}
and one has
\begin{equation}
[e_0, e_+]=2e_+,\quad [e_0, e_-]=-2e_-,\quad [e_-, e_+]=-e_0.
\end{equation}
In general, if $e_i$, $i=1,\dots, m$ is a basis of $\frak g$, there exists constants $c_{ij}^k$ so that
\begin{equation}\label{eq:structureconstants}
[e_i, e_j]=c_{ij}^ke_k,
\end{equation}
where the summation over $k$ is understood; the $c_{ij}^k$ are called the structure constants of $\frak g$. 

There exists a natural \emph{linear} action of $G$ on its Lie algebra, called the \emph{adjoint action}\index{adjoint action} or \emph{adjoint representation}\index{adjoint representation|seeonly{adjoint action}}, defined as follows, for all $g\in G, \xi\in T_\e G$:
$$
\mathrm{Ad}_g\xi=g\xi g^{-1}.
$$
Clearly $\mathrm{Ad}_{g_1g_2}=\mathrm{Ad}_{g_1}\mathrm{Ad}_{g_2}$. Note that for a commutative Lie group $G$, such as $\R^n$, it is trivial: $\mathrm{Ad}_g\xi=\xi$. It is instructive to compute some non-trivial adjoint actions explicitly.  For SO$(3)$, one finds, with the above (somewhat abusive) notation
\begin{equation}\label{eq:so3adjointaction}
\mathrm{Ad}_R\xi=
\begin{pmatrix}
0&-(R\xi)_3&(R\xi)_2\\
(R\xi)_3&0&-(R\xi)_1\\
-(R\xi)_2&(R\xi)_1&0
\end{pmatrix}=R\xi.
\end{equation}
We invite the reader to do the analogous computation for sl$(2,\R)$, determining the matrix of $\mathrm{Ad}_g$ in the basis given above. 

The dual of the Lie algebra $\frak g$ (as a vector space) is denoted by $\frak g^*$. It appears very naturally in the study of symplectic group actions arising in the study of Hamiltonian systems with symmetry, as we will see in Section~\ref{s:nother}. 
Given a basis $e_i$ of $\frak{g}$, we denote by $e_i^*$ the dual basis defined by $e_i^*(e_j)=\delta_{ij}$. 

Moreover, there is a natural action of $G$ on $\frak g^*$, obtained by dualization as follows. For all $\mu\in \frak g^*$, for all $\xi\in\frak g$, we define
\begin{equation}\label{eq:coadjoint}
\mathrm{Ad}^*_g\mu(\xi)=\mu(\mathrm{Ad}_{g^{-1}}\xi).
\end{equation}
This is called the co-adjoint action\index{co-adjoint action} of $G$. For later purposes, we define, for all $\mu\in\frak g^*$,
\begin{equation}\label{eq:stabilizer}
G_\mu=\{g\in G\mid \mathrm{Ad}^*_g\mu=\mu\},
\end{equation}
the so-called \emph{stabilizer}\index{stabilizer|seeonly{isotropy group}} or \emph{isotropy group}\index{isotropy group} of $\mu\in\frak g^*$.

As above, given a basis $e_i$ of $\frak g$, one identifies $\mu\in\frak g^*$ with $\mu=(\mu_1,\dots, \mu_m)\in\R^m$ by writing 
\begin{equation}\label{eq:gstaridentif}
\mu=\sum_{i=1}^m \mu_i e^*_i \mbox{ so that }  \mu(\xi)=\sum_{i=1}^m \mu_i\xi_i.
\end{equation}

Let $\mu\in\mathrm{so}(3)^*$; we write $\mu(\xi)=\sum_{i=1}^3\mu_i \xi_i$ and identify $\mu\in \mathrm{so}(3)^*$ with $\mu=(\mu_1,\mu_2,\mu_3)\in\R^3$. Again, one readily checks that
\begin{equation}\label{eq:so3coadjointaction}
{\mathrm{Ad_R}^*\mu}=R \mu.
\end{equation}

\begin{remark}\label{rem:euclidianstructure}
It is often useful to suppose there exists an Euclidian structure on $\frak{g}$ that is preserved by $\mathrm{Ad}_g$ for all $g\in G$. This is equivalent to supposing that there exists a basis $e_i$ of $\frak{g}$ so that the matrix of $\mathrm{Ad}_g$ in $e_i$ belongs to $\mathrm{O}(m)$. We will simply write $\mathrm{Ad}_g\in \mathrm{O}(m)$ in this case. It follows that the matrix of $\mathrm{Ad}^*_g$ in the dual basis $e_i^*$ belongs to $\mathrm{O}(m)$ as well. This implies that the natural Euclidian structure induced on $\frak{g}^*$ by the one on $\frak{g}$ is preserved by $\mathrm{Ad}^*_g$ for all $g\in G$. Such a structure always exists if the group $G$ is compact.
\end{remark}

Suppose now we have a $C^1$-action $\Phi:(g,u)\in G\times \Ban\to \Phi_g(u)\in \Ban$ of a Lie group $G$ on a normed vector  space $\Ban$.  Then, for all $\xi\in T_\e G$, one can define the vector field $X_\xi$ on $\Ban$, called \textit{generator}, via
\begin{equation}\label{eq:generator}\index{generator}
X_\xi(u)=\frac{\rd}{\rd t} \Phi_{\exp(t\xi)}(u)_{\mid t=0}.
\end{equation}

\begin{lemma} \label{lem:liealgantihom}
If $\Phi$ is a $C^2$-action, then for all $g\in G$, $\xi, \eta\in \frak g$, for all $u\in \Ban$, one has
\begin{eqnarray}
[X_\xi, X_\eta]&=&-X_{[\xi,\eta]},\label{eq:liealgrep}\\
X_{\mathrm{Ad}_g\xi}(\Phi_g(u))&=&D_u\Phi_g(X_\xi(u)).\label{eq:Adintertwine}
\end{eqnarray}
\end{lemma}
\begin{proof}
It follows from~\eqref{eq:commutator} that
$$
\frac{\partial^2}{\partial s\partial t}{\Phi_{\exp (s\eta)\,\exp (t\xi)\exp (-s\eta)}}_{\mid s=0=t}=\left[X_\xi, X_\eta\right].
$$
Now, by definition, 
$$
X_{\exp (s\eta)\xi\exp(-s\eta)}=\frac{\rd}{\rd t}{\Phi_{\exp (s\eta)\exp (t\xi)\exp(-s\eta)}}_{\mid t=0}
$$
and furthermore
$$
\frac{\rd}{\rd s}{X_{\exp (s\eta)\xi\exp(-s\eta)}}_{\mid s=0}=X_{[\eta, \xi]}.
$$
This proves~\eqref{eq:liealgrep}. For~\eqref{eq:Adintertwine}, note that the chain rule implies
$$
\frac{\rd}{\rd t}\Phi_g(\Phi_{\exp(t\xi)}(u))_{\mid t=0}=D_u\Phi_g(X_\xi(u)).
$$
On the other hand, $\Phi_{g\exp(t\xi)}(u)=\Phi_{g\exp(t\xi)\ g^{-1}}(\Phi_g(u))$. Hence
$$
\frac{\rd }{\rd t}\Phi_g(\Phi_{\exp(t\xi)}(u))_{\mid t=0}=X_{\mathrm{Ad}_g\xi}.
$$
\end{proof}
Lemma~\ref{lem:liealgantihom} shows that the map $\xi\in \frak g\to X_\xi$ is a Lie algebra anti-homomorphism.

\section{Hamiltonian dynamical system with symmetry in  finite dimension}\label{s:hammech}
We now turn to a very short description of Hamiltonian dynamical systems and their symmetries on a finite dimensional normed vector space $\Ban$. We present the theory in a simple but slightly abstract formalism that is well-suited for the generalization to the infinite dimensional situation needed for the main body of the text and presented in Section~\ref{s:hamdyninfinite}. The modern theory of finite dimensional Hamiltonian dynamical systems finds its natural setting in the theory of (finite dimensional) symplectic geometry~\cite{am, ar, ma, souriau1997}. We shall however have no need for this more general formulation in these notes. 
\subsection{Hamiltonian dynamical systems}\index{Hamiltonian flow}
The central object of the theory in its usual formulation is a symplectic form, that we now define. Let $\omega:\Ban\times \Ban\to\R$ be a bilinear form which is anti-symmetric, meaning 
\begin{equation*}
\forall u, u'\in \Ban, \ \omega(u,u')=-\omega(u', u),
\end{equation*}
and non-degenerate, meaning that, for all $u\in \Ban$, 
\begin{equation*}
\left(\forall u'\in \Ban, \ \omega(u,u')=0\right)\Rightarrow u=0.
\end{equation*}
Such a form is called a symplectic form. The standard example is $\Ban=\R^{n}\times\R^n$ with $u=(q,p)$ and 
\begin{equation}\label{eq:basicsymplectic}
\omega(u, u')=q\cdot p'- q'\cdot p,
\end{equation}
where $\cdot$ indicates the standard inner product on $\R^n$. 
Given a $C^1$-function $F:\Ban\to\R$, one defines the \emph{Hamiltonian vector field\index{Hamiltonian vector field} $X_F$ associated to $F$} as follows: for all $u\in \Ban$,
\begin{equation}\label{eq:hamfielddef}
\omega(X_F(u), u')=D_uF(u'),\quad \forall u'\in \Ban.
\end{equation}
We recall that $D_uF\in E^*$ is our notation for the Frechet derivative of $F$ at $u$. Observe that one can think of the map $u\in \Ban\to D_uF\in \Ban^*$ as a differential one-form on $\Ban$. The vector field $X_F$ is well-defined and unique, thanks to the non-degeneracy of the symplectic form. If $\omega$ were symmetric, rather than anti-symmetric, it would define an inner product on $E$, rather than a symplectic form, and~\eqref{eq:hamfielddef} would actually define the gradient of $F$; in analogy, one sometimes refers to $X_F$ as the symplectic gradient of $F$. We will see it has radically different features from the gradient. 

For later reference, we point out that
\begin{equation}\label{eq:XFzero}
X_F=0\Rightarrow \exists c\in\R,\ \forall u\in\Ban,\ F(u)=c.
\end{equation}

The flow of the Hamiltonian vector field $X_F$, for which we shall write $\Phi_t^F$, is obtained by integrating the differential equation
\begin{equation}\label{eq:hameqmotion}
\dot u(t)=X_F(u(t)), \quad u_0=u,
\end{equation}
referred to as the Hamiltonian equation of motion. One writes $\Phi_t^F(u)=u(t)$. In this section we suppose that~\eqref{eq:hameqmotion} admits a unique and global solution and that, for all $t\in\R$, $\Phi_t\in C(E,E)$.

As a typical example from elementary mechanics, let $V\in C^1(\R^3;\R)$  and define the function 
\begin{equation}\label{eq:potential}
H(q,p)=\tfrac12p^2+V(q)
\end{equation}
on $\Ban=\R^6$, with the symplectic form as above. The equations of motion corresponding to $H$ are then 
\begin{equation}\label{eq:hampotential}
\dot q(t)=p(t),\quad \dot p(t)=-\nabla V(q(t)).
\end{equation}
Note that they lead to Newton's force law in the form $\ddot q(t)=-\nabla V(q(t))$. 
More generally, in the example above, with $\Ban=\R^{2n}$, one finds
\begin{equation*}
X_F(q,p)=
\begin{pmatrix}
\partial_p F(q,p)\\ -\partial_q F(q,p)
\end{pmatrix},
\end{equation*}
which leads to the familiar Hamiltonian equations of motion:
\begin{align*}
\dot q(t)&=\partial_pF(q(t), p(t)),\quad&
\dot p(t)&=-\partial_qF(q(t), p(t)).
\end{align*}
We give several other explicit examples of such flows in the main part of these notes.

Let us return to the general situation. Given two functions $F_1, F_2:\Ban\to\R$, one defines their Poisson bracket\index{Poisson bracket} $\{F_1, F_2\}$ via
\begin{equation}\label{eq:poissonbracket1}
\{F_1, F_2\} =\omega(X_{F_1}, X_{F_2})=-\{F_2, F_1\} .
\end{equation}
Observe that, with the notation from~\eqref{eq:derivationX}, we have
\begin{equation}\label{eq:poissonbracketbis}
\widehat X_{F_1}(F_2)=DF_2(X_{F_1})=\omega(X_{F_2}, X_{F_1})=\{F_2, F_1\},
\end{equation}
\emph{i.e.} for all $u\in\Ban$,
\begin{equation*}
\widehat X_{F_1}(F_2)(u)=D_uF_2(X_{F_1}(u))=\omega(X_{F_2}(u), X_{F_1}(u))=\{F_2, F_1\}(u).
\end{equation*}
It is then immediate from what precedes that, for all $u\in\Ban$,
\begin{eqnarray*}
\frac{\rd}{\rd t} (F_2\circ \Phi_t^{F_1})(u)&=&D_{\Phi_t^{F_1}(u)}F_2(X_{F_1}(\Phi_t^{F_1}(u)))\nonumber\\
&=&\{F_2, F_1\}(\Phi_t^{F_1}(u))\nonumber
\end{eqnarray*}
which in turn yields:
\begin{theorem} \label{thm:nother1} Let $F_1, F_2\in C^1(\Ban, \R)$. Then $F_1\circ \Phi_t^{F_2}=F_1$ for all $t$ iff $F_2\circ \Phi_t^{F_1}=F_2$ for all $t$, iff $\{F_1, F_2\}=0$. 
\end{theorem}
When $F_1\circ \Phi_t^{F_2}=F_1$ for all $t$, one says either that the $\Phi_t^{F_2}$ form a symmetry group\footnote{See Definition~\ref{def:symgroup}.} for $F_1$ or that $F_1$ is a constant of the motion\footnote{Defined in~\eqref{eq:constantofmotion}.} for the flow $\Phi_t^{F_2}$.  The theorem, which is a Hamiltonian version of Noether's theorem\index{Noether's Theorem} (See~\cite{am, ar, ma, souriau1997} for a general treatment), can therefore be paraphrased by saying that $F_2$ is a constant of the motion for the flow $\Phi_t^{F_1}$ iff the flow $\Phi_t^{F_2}$ of $F_2$ forms a group of symmetries for $F_1$. Several instances and applications of this result appear in the main body of the text. It is typically used in the following manner. One wishes to study the dynamical flow $\Phi_t^{F_1}$. One has a simple and well-known one parameter group $\Phi_t^{F_2}$ for which one readily establishes with an explicit computation that $F_1\circ \Phi_t^{F_2}=F_1$. From this, one can then conclude that $F_2$ is a constant of the motion for the dynamical group $\Phi_t^{F_1}$. We will elaborate on this point in Section~\ref{s:nother}.

The radical difference between the properties of the symplectic gradient and the ``usual'' gradient is now apparent. The anti-symmetry of the Poisson bracket implies $\widehat X_F(F)=0$, that is, the symplectic gradient is~\emph{tangent} to the level surfaces of $F$ (See \eqref{eq:deftangentspace}), rather than orthogonal. Hence its flow $\Phi_t^F$ preserves these surfaces rather than moving points to increasing values of $F$ as does the usual gradient. These features, together with the Jacobi identity, are at the origin of all special properties of Hamiltonian systems. 

To prepare for the treatment of Hamiltonian dynamical systems in infinite dimension (see Section~\ref{s:hamdyninfinite}), we reformulate the above as follows. Given a symplectic form $\omega$ on a finite dimensional normed vector space $\Ban$, one can define a bijective linear map
$$
\Jcal: u\in \Ban\to \Jcal u\in \Ban^*
$$ 
by 
$
\Jcal u(v)=\omega(u, v).
$
It is clear that 
\begin{equation}\label{eq:Jantisym}
\Jcal u(v)=-\Jcal v(u).
\end{equation} 
With this notation, we find that
\begin{equation}\label{eq:HamfieldJ}
X_F=\Jcal^{-1}DF,\quad\mathrm{or}\quad \Jcal X_F=DF
\end{equation}
so that the Hamiltonian equations of motion~\eqref{eq:hameqmotion} can be equivalently rewritten as
\begin{equation}\label{eq:hameqmotionJ}
\Jcal \dot u(t)=D_{u(t)}F.
\end{equation}
This formulation is the one that we carry over to the infinite dimensional setting in the main body of these notes. Note that the Poisson bracket\index{Poisson bracket} of two functions can now be written as
\begin{equation}\label{eq:poissonbrackJ}
\{F,G\}=DF(\Jcal^{-1}DG).
\end{equation}
The point to make is that all objects of the theory can be expressed in terms of $\Jcal$. This is illustrated in the proof of the following result. 
\begin{lemma}\label{lem:poissonproperties} If $F_1, F_2, F_3\in C^2(\Ban,\R)$, then the Jacobi identity holds:
\begin{equation}\label{eq:jacobipoisson}
\{\{F_1, F_2\}, F_3\} + \{\{F_2, F_3\}, F_1\}+\{\{F_3, F_1\}, F_2\}=0
\end{equation}
If $F_1, F_2\in C^2(\Ban,\R)$, then
\begin{equation}\label{eq:compoisson}
X_{\{F_1, F_2\}}=-[X_{F_1}, X_{F_2}].
\end{equation}

\end{lemma}
\begin{proof} 
To prove~\eqref{eq:jacobipoisson}, one first easily checks that
\begin{eqnarray*}
&\ &\{\{F_1, F_2\}, F_3\}(u)= \\
&=&D^2_uF_1(\Jcal^{-1}D_uF_2, \Jcal^{-1}D_uF_3)+D_uF_1(\Jcal^{-1}D^2_uF_2(\cdot, \Jcal^{-1}D_uF_3))\\
&=&D^2_uF_1(\Jcal^{-1}D_uF_2, \Jcal^{-1}D_uF_3)-D^2_uF_2(\Jcal^{-1}D_uF_1, \Jcal^{-1}D_uF_3),
\end{eqnarray*}
where we used~\eqref{eq:Jantisym}. The result is then immediate. To prove~\eqref{eq:compoisson} we use~\eqref{eq:derivationX}--\eqref{eq:commutator3} to write
\begin{eqnarray*}
\widehat{[X_{F_1}, X_{F_2}]}(F_3)&=&\widehat X_{F_1}(\widehat X_{F_2}(F_3))-\widehat X_{F_2}(\widehat X_{F_1}(F_3))\\
&=&\widehat X_{F_1}(\{F_3, F_2\})-\widehat X_{F_2}(\{F_3, F_1\}))\\
&=&\{\{F_3, F_2\}, F_1\} -\{\{F_3, F_1\}, F_2\}\\
&=&\{\{F_1, F_2\}, F_3\}=-\widehat X_{\{F_1, F_2\}}(F_3),
\end{eqnarray*}
where we used the Jacobi identity in the last line. 
\end{proof}
For the case where $\Ban=\R^{2n}$ with the standard symplectic structure, one readily finds
\begin{equation}\label{eq:poisbrackstandard}
\{F_1, F_2\}=\partial_q F_1\cdot \partial_pF_2-\partial_p F_1\cdot \partial_qF_2.
\end{equation}
The above lemma then follows from a direct computation. 

The lemma implies that the vector space $C^\infty(\Ban,\R)$, equipped with the Poisson bracket, is a Lie algebra. In addition, it follows that the constants of the motion of a given function $F\in C^\infty(\Ban,\R)$ form a Lie subalgebra. Indeed, introducing the space of constants of the motion of $F$,
\begin{equation}
\mathcal{C}_F=\{G\in C^\infty(E,\R) \mid G\circ\Phi_t^F=G, \forall t\in\R\},
\end{equation}
which is clearly a vector space, it follows immediately from~\eqref{eq:jacobipoisson} that
$$
G_1,G_2\in \mathcal{C}_F\Rightarrow \{G_1, G_2\}\in \mathcal C_F,
$$
so that $\mathcal{C}_F$ is a Lie subalgebra of $C^\infty(\Ban,\Ban)$. 

We finally need to introduce symplectic transformations.
\begin{definition}\label{def:sympltransf} A symplectic transformation\index{symplectic transformation} on a symplectic space $(\Ban,\omega)$ is a $C^1$ diffeomorphism $\Phi:\Ban\to\Ban$ so that, for all $u,v,w\in \Ban$
\begin{equation}\label{eq:sympltransf}
\omega(D_u\Phi(v), D_u\Phi(w))=\omega(v,w).
\end{equation}
\end{definition}
This is often paraphrased by the statement  that ``$\Phi$ preserves the symplectic structure.''  To understand what this means, one should recall the interpretation of $D_u\Phi(v)$ as the ``push forward'' of $v$ by $\Phi$, explained in Section~\ref{s:diffgeom}.  Equation~\eqref{eq:sympltransf}  states that a diffeomorphism is symplectic if the symplectic form is left invariant by the ``push forward'' operation of its arguments. Note that, if $\Phi$ is linear,~\eqref{eq:sympltransf} reduces to $\omega(\Phi(v), \Phi(w))=\omega(v,w)$.  And if $\Ban=\R^{2n}$ with its standard symplectic structure, this then means that $\Phi\in\mathrm{Sp}(2n)$, defined in~\eqref{eq:spgroup}.

\begin{lemma}\label{lem:composeF} Let $F\in C^1(\Ban, \R)$ and let $\Phi\in C^1(\Ban, \Ban)$  be a symplectic transformation. Then, for all $u\in\Ban$,
\begin{equation}\label{eq:hampush}
D_u\Phi(X_{F\circ\Phi}(u))=X_F(\Phi(u)).
\end{equation}
Moreover, for all $t\in\R$,
\begin{equation}\label{eq:hamintertwine}
\Phi\circ\Phi_t^{F\circ\Phi}\circ\Phi^{-1}=\Phi_t^{F}.
\end{equation}
In particular, if $F\circ\Phi=F$, then $\Phi$ commutes with $\Phi_t^F$, for all $t\in\R$. And if $\Phi$ commutes with $\Phi_t^F$, for all $t\in\R$, then there exists $c\in\R$ so that $F\circ\Phi=F+c$. 
\end{lemma}
Equation~\eqref{eq:hampush} asserts that the push forward of the vector field $X_{F\circ \Phi}$ by $\Phi$ is $X_F$. 
\begin{proof}
For all $u,v\in\Ban$, one has
\begin{eqnarray*}
\omega(X_{F\circ\Phi}(u), v)&=&D_u(F\circ \Phi)(v)=D_{\Phi(u)}F(D_u\Phi(v))\\
&=&\omega(X_F(\Phi(u)), D_u\Phi(v)).
\end{eqnarray*}
Hence, since $\Phi$ is symplectic and since $D_{\Phi(u)}\Phi^{-1}D_u\Phi=\mathrm{Id}_\Ban=D_u\Phi D_{\Phi(u)}\Phi^{-1}$,
$$
\omega(D_{u}\Phi(X_{F\circ\Phi}(u)), v)=\omega(X_{F\circ\Phi}(u), D_{\Phi(u)}\Phi^{-1}(v))=
\omega(X_F(\Phi(u)), v)$$
which yields~\eqref{eq:hampush}. Next, for all $u\in\Ban$, one finds from the chain rule and~\eqref{eq:hampush}
\begin{eqnarray*}
\frac{\rd}{\rd t} \Phi(\Phi_t^{F\circ\Phi}(\Phi^{-1}(u)))&=&D_{\Phi_t^{F\circ\Phi}(\Phi^{-1}(u))}\Phi \left(X_{F\circ\Phi}(\Phi_t^{F\circ \Phi}(\Phi^{-1}(u))\right)\\
&=&X_F((\Phi\circ\Phi_t^{F\circ\Phi}\circ\Phi^{-1})(u)).
\end{eqnarray*}
This shows $t\in\R\to (\Phi\circ\Phi_t^{F\circ\Phi}\circ\Phi^{-1})(u)\in\Ban$ is a flow line of $X_F$. Since the latter are unique,~\eqref{eq:hamintertwine} follows.
\end{proof}
We end with a proof of a basic fact about Hamiltonian flows: if they are smooth, they are symplectic.
\begin{theorem}\label{thm:symplecticflow}
Let $F\in C^2(\Ban, \R)$. Suppose that the corresponding Hamiltonian flow $\Phi^F:\R\times \Ban\to \Ban$ is of class $C^2$. Then, for all $t\in\R$, $\Phi_t^F$ is a symplectic transformation.
\end{theorem}
\begin{proof} It will be sufficient to show that, for all $u,v,w\in \Ban$, and for all $t\in\R$,
$$
\frac{\rd}{\rd t} (\Jcal D_u\Phi_t^F v)(D_u\Phi_t^F w)=0.
$$
Using the group property of the flow, one sees it is enough to show this at $t=0$. Then
\begin{align*}
\frac{\rd}{\rd t}(\Jcal D_u\Phi_t^F v)(D_u\Phi_t^F w)_{\mid {t=0}}&=\Jcal (\Jcal^{-1}D_u^2F (v, \cdot))( w)+ (\Jcal v)(\Jcal^{-1}D_u^2 F (w,\cdot)).
\end{align*}
where we used the continuity of $\Jcal$, the Schwarz Lemma (exchange of partial derivatives) and the observation that
$$
\Jcal\frac{\partial \Phi^F}{\partial t}(u)=D_{u(t)}F\in E^*,
$$
and hence, at $t=0$, 
$$
\Jcal D_u\left(\frac{\partial \Phi^F}{\partial t}\right)(u)=\left(D^2_u F\right),
$$
which means that, for all $v\in \Ban$, 
$$
\Jcal D_u\left(\frac{\partial \Phi^F}{\partial t}\right)(u)v =\left(D^2_u F\right)(v, \cdot).
$$
Note that both sides are elements of $\Ban^*$ since $u\in\Ban\to \frac{\partial \Phi^F}{\partial t}\in E$ so that $D_u\left(\frac{\partial \Phi^F}{\partial t}\right)(u)\in\mathcal{B}(\Ban, \Ban)$. 
Using the anti-symmetry of $\Jcal$, one then finds
\begin{align*}
\frac{\rd}{\rd t} (\Jcal D_u\Phi_t^F v)(D_u\Phi_t^F w)_{\mid t=0}&=D_u^2F (v, w) -D_u^2 F (w,v)=0.
\end{align*}
\end{proof} 

\begin{remark}\label{rem:smoothness}
We point out that the proof, as it stands, is valid in infinite dimensional systems. Remark however that the conditions imposed on the flow $\Phi_t^F$ are very strong for systems in infinite dimension. Too strong actually to be of much use in that context. We use/need those conditions to apply the Schwarz Lemma at several points in the proof.  
Also, it is known that Hamiltonian flows in infinite dimension need not always be symplectic. In the framework of Section~\ref{s:hamdyninfinite} it is possible to give sufficient smoothness conditions on the restriction of the flow to $\Dcal$ that will guarantee the result, but we shall not need this. For  a different set of technical conditions guaranteeing the symplecticity of the flow, we refer to~\cite{chm}.
\end{remark}
\subsection{Symmetries and constants of the motion}\label{s:nother}
Hamiltonian dynamical systems have many special features, but the one important to us here is that there exists for them a special link between the symmetries of the dynamics and the constants of the motion. This link takes the form of a Hamiltonian version of Noether's Theorem, of which we already gave a simple version in Theorem~\ref{thm:nother1}, and has far-reaching consequences, some of which we further explore in this section. Again, a general treatment can for example be found in~\cite{am, ma}; we give just those few elements needed in these notes.  

We start with some notions on Hamiltonian Lie group actions on a symplectic vector space. 
\begin{definition}\label{def:globhamaction} Let $G$ be a Lie group and $\Phi: (g,x)\in G\times \Ban\to \Phi_g(x)\in \Ban,$ an action of $G$ on $\Ban$ with $\Phi_g\in C^1(E,E)$.  We will say $\Phi$ is globally Hamiltonian\index{globally Hamiltonian action} if $\Phi_g$ is symplectic for all $g\in G$ and if, for all $\xi\in\frak{g}$, there exists $F_\xi\in C^2(\Ban, \R)$ so that $\Phi_{\exp(t\xi)}=\Phi_t^{F_\xi}$. 
\end{definition}
In other words, an action is globally Hamiltonian if all $\Phi_g$ are symplectic and if all one parameter groups  are realized by Hamiltonian flows. In the notation of the previous sections this means that
$$
X_\xi=X_{F_\xi}.
$$
Here, the left hand side is the generator of the action, defined in~\eqref{eq:generator} and the right hand side is the Hamiltonian vector field associated to $F_\xi$.

\begin{remark}\label{rem:expo}
In view of Theorem~\ref{thm:symplecticflow}, if $g=\exp(\xi)$ for some $\xi\in\frak g$ and $\Phi_g$ can be written as $\Phi_{\exp(\xi)}=\Phi_1^{F_\xi}$ for some $F_\xi\in C^2(E,\real)$ such that $\Phi_1^{F_\xi}$ is $C^2$, then $\Phi_g$ is symplectic. This will obviously hold as well 
for all $g$ that can be written as a finite product of elements of the form $\exp(\xi)$, which is the case for all $g$ in the connected component of $G$ containing $e\in G$ (See~\cite{ma}, page 145, Proposition 2.10). 
So the assumption that $\Phi_g$ is symplectic is only needed for elements $g$ that are not
connected to $e\in G$. 
In infinite dimensional systems, as indicated in Remark~\ref{rem:smoothness} at the end of the previous section, the condition that all 
$\Phi_g$ must be symplectic is more restrictive. In practice, one often works with linear actions of the symmetry group, for which the symplectic property can be checked directly. 
\end{remark}

The above definition is a special case of the more general definition of \emph{globally Hamiltonian action} for 
infinite dimensional systems that we introduced in Definition~\ref{def:globhamactioninf}. 
It suffices to take $\Dcal=E$ in the latter to obtain the definition here. 

We shall now continue with the abstract theory where, in particular, we will see through a version
of Noether's Theorem\index{Noether's Theorem} that, if the Hamiltonian is invariant under a globally Hamiltonian action 
$\Phi$ as above, then the functions $F_\xi\in C^2(E,\real)$ are constants of the motion. The theory will be illustrated in Example~\ref{ex:SO(3)} at the end of the section, in the simple case where $\Ban=\R^6$ and $G=\mathrm{SO}(3)$.

\begin{theorem} \label{thm:nother2} Let $G$ be a Lie group and $\Phi$ a globally Hamiltonian action of $G$ on a symplectic vector space $\Ban$. Let $H\in C^1(\Ban,\R)$ and let $\Phi_t^H$ be the corresponding Hamiltonian flow. Suppose that
\begin{equation}\label{eq:haminvariant}
\forall g\in G, \quad H\circ \Phi_g=H.
\end{equation}
Then the following statements hold.
\begin{enumerate}[label=({\roman*})]
\item For all $\xi\in\frak g$, $\{H, F_\xi\}=0.$
\item For all $t\in\R$, $F_\xi\circ \Phi_t^H=F_\xi$. 
\item $G$ is an invariance group\footnote{See Definition~\ref{def:symgroup}} for $\Phi_t^H$. 
\end{enumerate}
\end{theorem}

\begin{proof}
This is an immediate consequence of Theorem~\ref{thm:nother1} and of Lemma~\ref{lem:composeF}.
\end{proof}

This result is useful because it is often easy to check~\eqref{eq:haminvariant}, whereas the conclusions (ii) and (iii) are statements about the flow $\Phi_t^H$, which is usually not explicitly known, and are therefore hard to check directly. In particular, (iii) says that if the Hamiltonian $H$ is $G$-invariant as a function, then $G$ is an invariance group of the dynamics\footnote{This is the point in the proof where the symplectic nature of the $\Phi_g$ is used, via Lemma~\ref{lem:composeF}.}. And (ii) ascertains that the group generators $F_\xi$ are then constants of the motion for $\Phi_t^H$. 

Let us point out that (iii) implies neither (i), (ii) or~\eqref{eq:haminvariant} (See Lemma~\ref{lem:composeF}.)

So the hypothesis that the Hamiltonian is invariant under the group action is strictly stronger than the statement that the Hamiltonian flow is invariant under $G$.  
The map
\begin{equation}\label{eq:momentmap0}
\xi\in\frak g\to F_\xi\in C^2(\Ban,\R)
\end{equation}
can be chosen to be linear. Indeed, if $e_i$, $i=1,\dots, d$ is a basis of $\frak g$, if we choose $F_i=F_{e_i}$, and if we write $\xi=\sum_i\xi_ie_i$, we can define
\begin{equation}\label{eq:basisdecompmoment}
F_\xi=\sum_i \xi_iF_i,
\end{equation}
by linearity. This allows one to define the \emph{momentum map} for the action $\Phi$, as follows:
\begin{equation}\label{eq:momentmap}
\mathcal F: u\in \Ban\to \mathcal F(u)\in \frak g^*,\quad \mathcal F(u)(\xi)=F_\xi(u).
\end{equation}
This, of course, is just a rewriting of~\eqref{eq:momentmap0}. 
In the main body of the text we shall always assume a basis has been chosen for $\frak g$, as above, so that we can identify $\frak g\simeq\R^m$. And we shall simply write 
\begin{equation}\label{eq:momentmap1}
F:u\in E\to (F_1(u), \cdots, F_m(u))\in\R^m\simeq\frak g^*.
\end{equation}
We shall refer to $\mathcal F$ or to $F$ as a momentum map\index{momentum map} for the action, indifferently. 

\begin{definition}\label{def:equivariant} Let $\Phi$ be a globally Hamiltonian action of $G$ on $\Ban$, with momentum map $F$. One says the momentum map is $\mathrm{Ad}^*$-equivariant\index{momentum map!equivariant} if, for all $g\in G$, for all $\xi\in\frak g$,
\begin{equation}\label{eq:adstarequiv}
F_\xi\circ\Phi_g=F_{\mathrm{Ad}_{g^{-1}}\xi}.
\end{equation}
\end{definition}
The terminology comes from the following observation. If~\eqref{eq:adstarequiv} holds, then
it follows from~\eqref{eq:momentmap} and~\eqref{eq:coadjoint} that
\begin{equation}
\mathcal F\circ \Phi_g=\mathrm{Ad}^*_g\circ\mathcal F.
\end{equation}
Since we identify $\frak g^*\simeq\R^m$, this can be written 
\begin{equation}\label{eq:adstarequivbis}
F\circ \Phi_g=\mathrm{Ad}^*_gF.
\end{equation}

We can now formulate the final result from the theory of invariant Hamiltonian systems that we need. It is an immediate consequence of~\eqref{eq:adstarequivbis} or, for the reader weary of duals, of~\eqref{eq:adstarequiv}.
\begin{proposition}\label{eq:reduction} Let $\Phi$ be a globally Hamiltonian, $\mathrm{Ad}^*$-equivariant action of a Lie group $G$ on a symplectic vector space $\Ban$. Let $\mu\in\frak g^*\simeq \R^m$ and define
\begin{equation}\label{eq:sigmamuapp}
\Sigma_\mu=\{u\in\Ban\mid {F}(u)=\mu\}
\end{equation}
Then ${G_\mu}=G_{\Sigma_\mu}$, where $G_\mu$ is the stabilizer of $\mu$,
defined in~\eqref{eq:stabilizer} and $G_{\Sigma_\mu}$ is defined in~\eqref{eq:gsimgamu}.
\end{proposition}
The situation we have in mind is the one where $G$ is such that $H\circ \Phi_g=H$, for all $g\in G$. 
By Theorem~\ref{thm:nother2},
the functions $F_i$ are then constants of the motion for the flow $\Phi_t^H$ and hence the surfaces $\Sigma_\mu$ are $\Phi_t^H$ invariant. We can therefore consider the dynamical system $(\Sigma_\mu, \Phi_t^H)$, which has $G_\mu$ as an invariance group ($G_\mu$ leaves invariant both $\Sigma_\mu$
and the flow $\Phi_t^H$). 
This viewpoint will prove useful in the study of orbital stability in several situations. 
\begin{definition}\label{def:regularpoint} Let $\Phi$ be a globally Hamiltonian action of a Lie group $G$ on a symplectic vector space $\Ban$.
Let $\mu\in\frak g^*$. We say $\mu$ is a regular point\index{regular point} of the momentum map $F$ if, for all $u\in \Sigma_\mu$, $D_u F$ is surjective. 
\end{definition}
This definition simply guarantees that $\Sigma_\mu$ is a co-dimension $m$ submanifold of $\Ban$, where $m$ is the dimension of $\frak g$.

\begin{example}\label{ex:SO(3)}
For the simple Hamiltonian system with spherical potentials considered in Section~\ref{s:spherpot} and Section~\ref{s:spherpotstab}, one has $E=\R^6$, 
$G=\mathrm{SO}(3)$, and it is not difficult to check that, for all
$u(q,p)\in \R^6$, $F(u)=L(q,p)\in\R^3\simeq\mathrm{so}(3)^*$ and 
$F_\xi(q,p)=\xi\cdot L(q,p)$, where we use the identifications \eqref{xirepr} and \eqref{eq:gstaridentif}.
Furthermore, for all $R\in \mathrm{SO}(3)$,
$$
L(Rq, Rp)=RL(q,p),
$$
which shows the action is Ad$^*$-invariant, in view of~\eqref{eq:so3coadjointaction}.
\end{example}

We end this section with some comments on the Poisson brackets of the components of the momentum map. Remark first that the momentum map of a globally Hamiltonian action is not unique since, for any choice of $\lambda\in \frak g^*$, 
$\tilde F_\xi=F_\xi+\lambda(\xi)$ also satisfies $X_\xi=X_{\tilde F_\xi}$. Note furthermore that, in view of~\eqref{eq:liealgrep} and~\eqref{eq:compoisson}, the momentum map satisfies, for all $\xi,\eta\in\frak g$,
$$
X_{F_{[\xi, \eta]}}=X_{[\xi,\eta]}=X_{\{F_\xi, F_\eta\}}.
$$ 
It then follows from~\eqref{eq:XFzero} that, for all $\xi, \eta\in\frak g$, there exists a constant $c(\xi,\eta)$ so that
$$
F_{[\xi, \eta]}=\{F_\xi, F_\eta\}+c(\xi,\eta).
$$
The following lemma is useful and an easy consequence of~\eqref{eq:adstarequivbis}:
\begin{lemma}\label{lem:adstarequiv}
Let $\Phi$ be a globally Hamiltonian action of $G$ on $\Ban$, with momentum map $F$. If $ F$ is $\mathrm{Ad}^*$-equivariant, then, 
for all $\xi, \eta\in\frak g$, 
\begin{equation}\label{eq:lieisom}
F_{[\xi, \eta]}=\{F_\xi, F_\eta\}.
\end{equation}
Conversely, if~\eqref{eq:lieisom} holds, then~\eqref{eq:adstarequiv} holds for all $g\in G$ of the form $g=\exp(\eta)$, for some $\eta\in \mathfrak{g}$ and then for all $g$ in the connected component of $e$.
\end{lemma}
What one has to remember here is this. In applications, we often wish to assure~\eqref{eq:adstarequiv} holds. The preceding lemma states this is essentially guaranteed by~\eqref{eq:lieisom}, at least for all $g=\exp\eta$, which, for many Lie groups, means all of $G$. Finally,~\eqref{eq:lieisom} is guaranteed by
\begin{equation}\label{eq:structuconstantspoisson}
\{F_i, F_j\}=c_{ij}^k F_k,
\end{equation}
where we used the notation introduced in~\eqref{eq:structureconstants} and~\eqref{eq:basisdecompmoment}. 
As an example, one may remark that the components of the angular momentum vector $L$ satisfy  the commutation relations of the Lie algebra of SO$(3)$, namely
$$
\{L_i, L_j\}=\epsilon_{ijk}L_k, \quad i,j,k=1,2,3.
$$
One may therefore show that an action is Ad$^*$-equivariant by showing~\eqref{eq:structuconstantspoisson} holds. However, in infinite dimension, this is not immediate since the necessary smoothness properties of the $F_i$'s and even of the corresponding Hamiltonian vector fields are not readily verified. 

Finally, an $Ad^*$-equivariant moment map may not exist. An easy example is $\Ban=\R^2$, $G=\R^2$ and $\Phi:\R^2\times \R^2\to \R^2$ given by $\Phi_{(a,b)}(q,p)=(q+a, p-b)$. Identifying $\frak g\simeq \R^2$ in the obvious way, this action has a moment map $F_1(q,p)=p, F_2(q,p)=q$ and $\{F_1, F_2\}=-1$. Since the group is commutative, it is clearly not $Ad^*$-equivariant. Ways to handle such situations exist, but we shall not deal with such complications in the main part of the text. We refer to~\cite{am, ma, souriau1997} for details. 




\newcommand{\etalchar}[1]{$^{#1}$}

\printindex

\end{document}